\documentclass[a4paper,reqno,times, 11pt]{amsart}
\usepackage{hyperref}
\usepackage[active]{srcltx}
\usepackage[curve]{xypic}
\usepackage{graphicx}
\usepackage{geometry}
\usepackage{longtable}
\usepackage{rotating}
\usepackage{multirow}
\usepackage[active]{srcltx}
\usepackage{epsfig,amsmath}
\usepackage[english]{babel}
\usepackage{amsmath,amsfonts,amssymb,amscd,color,epsfig,amsthm}
\usepackage{titletoc}
\usepackage{mathrsfs}
\usepackage{caption}

\usepackage{tikz}
\usetikzlibrary{arrows}
\tikzstyle{every picture}=[> = to]
\tikzset{cdlabel/.style={execute at begin node=$\scriptstyle,execute at end node=$}}
\tikzset{implication/.style={double equal sign distance, -implies}}
\tikzset{biimplication/.style={double equal sign distance, implies-implies}}

\newtheorem{theorem}{Theorem}[section]
\newtheorem{lemma}{Lemma}[section]
\newtheorem{proposition}{Proposition}[section]
\newtheorem{corollary}{Corollary}[section]
\newtheorem{definition}{Definition}[section]
\newtheorem{remark}{Remark}[section]
\newtheorem{conjecture}{Conjecture}[section]
\newtheorem{question}{Question}[section]

\newtheorem*{theoremA}{Theorem A}
\newtheorem*{theoremB}{Theorem B}
\numberwithin{equation}{section}

\def\cbar{\overline{\C}}
\def\R{\mbox{$\mathbb R$}}
\def\Z{\mbox{$\mathbb Z$}}

\def\D{\mbox{$\mathbb D$}}
\def\C{\mbox{$\mathbb C$}}

\def\VV{\mathscr V}
\def\DD{\mathscr D}
\def\EE{\mathscr E}
\def\BB{\mathscr B}
\def\GG{\mathscr G}

\def\AAA{{\mathcal A}}
\def\BBB{{\mathcal B}}

\def\DDD{{\mathcal D}}
\def\EEE{{\mathcal E}}

\def\GGG{{\mathcal G}}

\def\III{{\mathcal I}}

\def\KKK{{\mathcal K}}
\def\LLL{{\mathcal L}}
\def\MMM{{\mathcal M}}
\def\NNN{{\mathcal N}}

\def\SSS{{\mathcal S}}

\def\UUU{{\mathcal U}}
\def\VVV{{\mathcal V}}

\def\g{\gamma}
\def\G{\Gamma}

\def\De{\Delta}
\def\de{\delta}

\def\wt{\widetilde}
\def\wh{\widehat}
\def\ov{\overline}

\def\sm{\setminus}

\def\pf{post-critically finite}
\def\Sie{Sierpi\'{n}ski}

\title{Invariant graphs in Julia sets  and decompositions of rational maps}

\author{Guizhen Cui}
\address{Guizhen Cui, School of Mathematical Sciences, Shenzhen University, Shenzhen 518061, China}
\email{gzcui@math.ac.cn}

\author{Yan Gao}
\address{Yan Gao, School of Mathematical Sciences, Shenzhen University, Shenzhen 518061, China}
\email{gyan@szu.edu.cn}

\author{Jinsong Zeng}
\address{Jinsong Zeng, School of Mathematical Sciences, Shenzhen University, Shenzhen 518061, China}
\email{jinsongzeng@163.com}

\begin{document}

\begin{abstract}
In this paper, we prove that for any \pf\ rational map $f$ on the Riemann sphere $\cbar$, and for each sufficiently large integer $n$, there exists a finite and connected graph $G$ in the Julia set of $f$ such that $f^n(G) \subset G$. This graph contains all post-critical points in the Julia set, while every component of $\cbar\sm G$ contains at most one post-critical point in the Fatou set. The proof relies on the cluster-\Sie\ decomposition of \pf\ rational maps.

\end{abstract}

\subjclass[2010]{Primary 37F20; Secondary 37F10}
\keywords{rational map,  Julia set, invariant graph, decomposition}
\thanks{
}

\maketitle

\tableofcontents

\section{Introduction}

Let $f$ be a rational map on the Riemann sphere $\cbar$ with $\deg f\ge 2$. The Fatou set and Julia set of $f$ are denoted by $F_f$ and $J_f$, respectively. Their definitions and basic properties can be found in \cite{Mi1}.  The set of {post-critical points} of $f$ is defined by
$$
P_f=\bigcup_{n>0}\{f^n(c):f^\prime(c)=0\}.
$$
In particular, the map $f$ is called {\bf post-critically finite}, or simply {\bf PCF}, if $\#P_f<\infty$. Generally, a {\bf marked rational map} $(f,P)$ is a  PCF  rational map $f$  with a finite set $P\subset\cbar$ such that $P_f\subset P$ and $f(P)\subset P$.

 In complex dynamics, a fundamental problem is understanding the structure of Julia sets for rational maps. Significant progress has been made in this area for polynomials, largely since the Julia set of a polynomial is the boundary of its  basin of infinity. However, for a general rational map, it is not possible to observe the entire Julia set from only a single Fatou domain. Therefore, we need to consider not only the boundary of each Fatou domain, but also the arrangement  of distinct Fatou domains.

An effective approach to this problem is to construct a suitable invariant graph. In this paper, the term {\bf graph} refers to a finite and connected graph in $\cbar$. For PCF polynomials, the well-known Hubbard trees are invariant and completely characterize the dynamics of the polynomials \cite{DH2,Poi2}. Invariant graphs for Newton maps and critically fixed rational maps have been studied by several groups \cite{DMRS,DS,LMS1,LMS,R,WYZ,CGNPP,H}.

The first breakthrough in the general situation was made by Cannon, Floyd, and Parry \cite{CFP} and Bonk and Meyer \cite{BM} independently. They proved that
\begin{theoremA}[{\cite[Theorem 3.1]{BM}}]\label{thm:A}
	 Any marked rational map $(f,P)$ with $J_f=\cbar$ admits an $f^n$-invariant Jordan curve passing through all points of $P$ for each sufficiently large integer $n$.
\end{theoremA}
 The same conclusion was obtained for marked {\bf \Sie\ rational maps}, i.e., rational maps with \Sie\ carpet Julia sets, by Meyer, Ha\"{i}ssinsky and the last two authors of this paper \cite{GHMZ}. The following theorem is an enhanced version of \cite[Theorem 1.2]{GHMZ}.
 \begin{theoremB}[{\cite[Theorem 1.2]{GHMZ}}]\label{thm:B}
 Let $(f,P)$ be a marked \Sie\ rational map such that no points of $P$ lie on the boundaries of Fatou domains. Then for each sufficiently large integer $n$, there exists an $f^n$-invariant Jordan curve passing through all points of $P$, such that its intersection  with the closure of any Fatou domain is either empty or the union of two closed internal rays.
 \end{theoremB}  
 
   Recently, by extending the Bonk-Meyer method in \cite{BM}, the authors of this paper demonstrated that every PCF rational map $f$ admits an $f^n$-invariant graph containing $P_f$ for each sufficiently large integer $n$; see \cite[Theorem 1.1]{CGZ}.

However, not all invariant graphs are sufficient to capture the full complexity of the Julia set. For example, for a PCF polynomial without bounded Fatou domains, the union of external rays landing at the post-critical points forms an invariant graph. Unlike the Hubbard tree, this graph provides limited information about the Julia set. Therefore, to better address these limitations, we aim to confine the graphs within the Julia set.
\vspace{2pt}

The main result of this paper is as follows.

\begin{theorem}[Invariant graph in the Julia set]\label{thm:main}
Let $(f, P)$ be a marked rational map. Then, for each sufficiently large integer $n$, there exists a graph $G\subset J_f$ such that $f^n(G)\subset G$, $P\cap J_f\subset G$, and each component of $\cbar\sm G$ contains at most one point of $P$.
\end{theorem}

\begin{remark}\label{rem:main}
	{\rm
{(1)} Based on this theorem, we obtain an increasing sequence of invariant graphs  $\{f^{-kn}(G)\}_{k\ge 1}$  that approximate the Julia set from within.\vspace{2pt}

{(2)} Theorem \ref{thm:main} is essentially known for PCF polynomials. Specifically,  let $X$ be the union of $P_f$ and the branch points of the Hubbard tree $T$. 
If $f$ has no bounded Fatou domains, then $T$ itself serves as the desired graph. Otherwise, for each bounded Fatou domain $U$ that intersects $T$, if $\ov{U}\cap X\neq\emptyset$, we substitute $U\cap T$ with the Jordan curve $\partial U$; if $\ov{U}\cap X=\emptyset$, we replace the segment $U\cap T$ with a suitable choice of one of the two open arcs as the components of $\partial U\sm T$. The resulting graph satisfies the conditions of Theorem \ref{thm:main}.

\vspace{2pt}

{(3)} The proof of Theorem \ref{thm:main} is entirely independent of our earlier work \cite[Theorem 1.1]{CGZ} presented after Theorem B. Instead, \cite[Theorem 1.1]{CGZ} can be directly derived from Theorem \ref{thm:main}.

Indeed, we may mark one point on the boundary of each Fatou domain intersecting $P_f$ such that the union of these marked points, together with $P_f$, forms an $f$-invariant set, denoted by $P$.   By applying Theorem \ref{thm:main} to $(f,P)$, we obtain an $f^n$-invariant graph $G'\subset J_f$ such that $P\cap J_f\subset G'$, for each sufficiently large integer $n$. Thus, the union $G$ of $G'$ and all internal rays landing at points of $P$ is an $f^n$-invariant graph containing $P_f$.  }
\end{remark}
 
There exist several key ingredients in proving Theorem \ref{thm:main}, as outlined in the schematic diagram in Figure \ref{fig:laminationn} and summarized below.
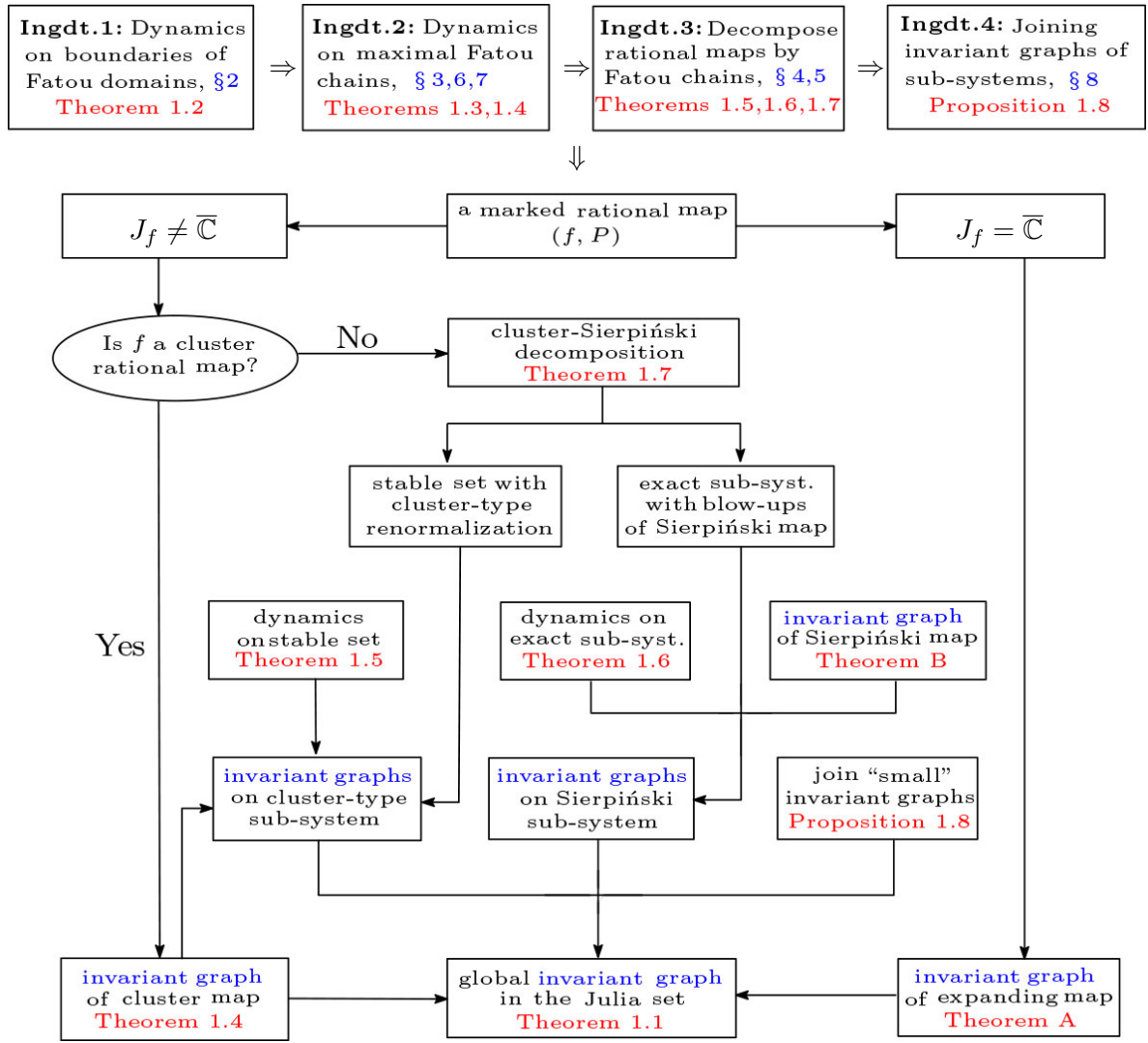
\begin{figure}[http]
	\begin{tikzpicture}
		\node at (0,0){\includegraphics[width=16cm]{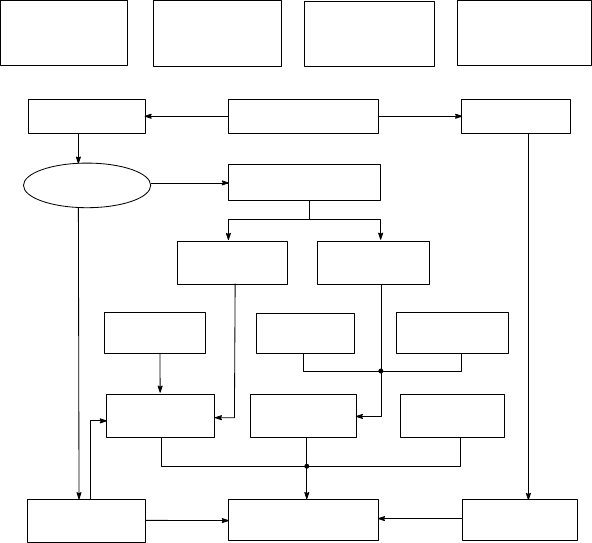}};
			\node at (-6.2,7){\footnotesize {\bf 1:} Dynamics on};
			\node at (-6.2,6.66){\footnotesize  boundaries of Fatou}; 
			\node at (-6.2,6.25){\footnotesize domains (Section 2)};
			\node at (-6.2,5.9){\footnotesize \bf \color{red} Theorem 1.2};
			
			\node at (-2.2,7){\footnotesize {\bf 2:} Dynamics on};
			\node at (-2.1,6.66){\footnotesize  maximal Fatou chains}; 
			\node at (-2.1,6.27){\footnotesize (Sections 3,\,6 and 7)};
			\node at (-2.1,5.9){\footnotesize \bf \color{red} Theorems 1.3--1.4};
			
				\node at (1.94,7){\footnotesize {\bf 3:} Decompose rational};
			\node at (2,6.65){\footnotesize  maps by Fatou chains}; 
			\node at (1.95,6.27){\footnotesize (Sections 4 and 5)};
			\node at (1.96,5.9){\footnotesize \bf \color{red} Theorems 1.5--1.7};
			
				\node at (6.1,7){\footnotesize {\bf 4:} Joining invariant};
			\node at (6.1,6.65){\footnotesize  graphs of sub-systems}; 
			\node at (6.1,6.27){\footnotesize (Section 8)};
			\node at (6.1,5.9){\footnotesize \bf \color{red} Proposition 1.1};
			
			\node at (0.1,4.33){\footnotesize A marked rational map};
			\node at (0.1,3.98){\footnotesize $(f,P)$};
			
			\node at(-5.62,2.45){\footnotesize Is $f$ a cluster};
			\node at(-5.62,2.15){\footnotesize rational map?};
			
			\node at(-6.25,-2){\small Yes};
			\node at(-3,2.6){\small No};
			
			\node at(0.3,2.65){\footnotesize Cluster-\Sie};
			\node at(0.3,2.38){\footnotesize decomposition};
			\node at(0.3,2.11){\footnotesize \color{red} \bf Theorem 1.7};
			
			\node at(-1.69,0.55){\footnotesize Stable sets with};
			\node at(-1.69,0.25){\footnotesize renormalizations};
			\node at(-1.69,-0.1){\footnotesize of cluster maps};
			
				\node at(2.05,0.53){\footnotesize Exact sub-systems};
			\node at (2.05,0.23){\footnotesize with blow-ups};
			\node at (2.05,-0.1){\footnotesize of \Sie\ maps};
			
			\node at (-3.75,-1.35){\footnotesize Dynamics on};
			\node at (-3.75,-1.65){\footnotesize a stable set};
			\node at (-3.75,-1.95){\footnotesize \color{red}\bf Theorem 1.5};
			
			\node at (0.28,-1.35){\footnotesize Dynamics on an};
			\node at (0.28,-1.67){\footnotesize  exact sub-system};
			\node at (0.28,-1.95){\footnotesize \color{red}\bf Theorem 1.6};
			
				\node at (4.22,-1.35){\footnotesize Invariant graphs};
			\node at (4.22,-1.67){\footnotesize  of \Sie\ maps};
			\node at (4.22,-1.95){\footnotesize \color{red}\bf Theorem B};
			
			\node at(-3.6,-3.6){\footnotesize Invariant graphs};
			\node at(-3.6,-3.92){\footnotesize on cluster-type};
			\node at(-3.6,-4.24){\footnotesize sub-systems};
			
			\node at(0.22,-3.6){\footnotesize Invariant graphs};
			\node at(0.22,-3.92){\footnotesize on \Sie};
			\node at(0.22,-4.24){\footnotesize sub-systems};
			
				\node at(4.17,-3.58){\footnotesize Joining ``small''};
			\node at(4.17,-3.9){\footnotesize invariant graphs};
			\node at(4.2,-4.24){\footnotesize\color{red}\bf Proposition 1.1};
			
			\node at(-5.58,-6.42){\footnotesize Invariant graphs};
			\node at(-5.58,-6.75){\footnotesize of cluster maps};
			\node at(-5.58,-7.05){\footnotesize \color{red}\bf Theorem 1.4};
			
				\node at(0.25,-6.42){\footnotesize Global invariant graphs};
			\node at(0.25,-6.73){\footnotesize in  Julia sets};
			\node at(0.25,-7.05){\footnotesize \color{red}\bf Theorem 1.1};
			
				\node at(6.1,-6.42){\footnotesize Invariant graphs};
			\node at(6.1,-6.75){\footnotesize of expanding maps};
			\node at(6.1,-7.05){\footnotesize \color{red}\bf Theorem A};
			
		\node at (-5.7,4.2){$J_f\not=\cbar$}; 
		\node at (6,4.2){$J_f=\cbar$};
		\node at (-4,6.2){$\Rightarrow$};
		\node at (0,6.2){$\Rightarrow$};
		\node at (4,6.2){$\Rightarrow$};
		\node at (0,5.){$\Downarrow$};
	\end{tikzpicture}
	\caption{An outline of the  procedure for proving Theorem \ref{thm:main}.}\label{fig:laminationn}
\end{figure}

 The first key ingredient refers to the invariant graphs on the boundaries of Fatou domains, serving as a semi-local counterpart to Theorem \ref{thm:main}.
Let $f$ be a PCF rational map, and let $U$ be a Fatou domain of $f$ with $f(U)=U$.
If $f$ is a polynomial, then  
$\partial U$ admits an invariant graph by  Remark \ref{rem:main}\,(2). It is natural to inquire whether this conclusion holds in general. 

The answer to this question is negative, as illustrated by a counterexample in Theorem \ref{thm:example}. On a positive note, we can construct an invariant graph associated with $\partial U$ within a larger invariant set, namely the {\bf Fatou chain generated by $\boldsymbol U$}, which is defined as
$
\ov{\bigcup_{k\ge 0} E_k},
$
where $E_k$ is the  component of $f^{-k}(\ov{U})$ containing $U$.


\begin{theorem}[Invariant graph associated with a Fatou domain]\label{thm:local}
Let $(f,P)$ be a marked rational map, and let $U$ be a fixed Fatou domain of $f$.  Then there exists a graph $G\subset J_f$ in the Fatou chain generated by $U$, such that $f(G)\subset G$ and $G$ is isotopic rel $P$ to a graph $G_0\subset \partial U$, which satisfies that $G_0\cap P=\partial U\cap P$, and that two points of $P$ lie in distinct components of $\cbar\setminus G_0$ provided that they belong to distinct components of $\cbar\setminus \partial U$.  
\end{theorem}

Theorem \ref{thm:local} is proved in Section \ref{sec:2}, based on an explicit study of the dynamics on $\partial U$. \vspace{3pt}

 We aim to extend the invariant graph in Theorem \ref{thm:local} to a broader setting. Inspired by the Fatou chain generated by a single Fatou domain, we introduce the concept of general Fatou chains. The second key ingredient involves constructing invariant graphs within  Fatou chains.

A {\bf continuum} is a connected and compact subset of $\cbar$ containing more than one point.
\begin{definition}
Let $f$ be a rational map with $J_f\neq\cbar$. A {\bf level-$\boldsymbol{0}$ Fatou chain} of $f$ is defined as the closure of a Fatou domain of $f$. A \emph{continuum}  $K\subset\cbar$ is a {\bf level-$\boldsymbol{1}$ Fatou chain} of $f$ if there exists a sequence of continua $\{E_k\}_{k\ge 0}$, each of which is the union of finitely many level-$0$ Fatou chains, such that
$$
E_k\subset E_{k+1}\quad\text{and}\quad K=\ov{\bigcup_{k\ge 0}E_k}.
$$

Inductively, a continuum $K\subset\cbar$ is a {\bf level-$\boldsymbol{(n+1)}$ Fatou chain} if there exists a sequence of continua $\{E_k\}$, each of which is the union of finitely many level-$n$ Fatou chains, such that $E_k\subset E_{k+1}$ and $K=\ov{\bigcup_{k\ge 0}E_k}$.\vspace{2pt}

A Fatou chain $K$ is {\bf maximal} if any Fatou chain intersecting $K$ is contained in $K$.
\end{definition}

By definition, a level-$n$ Fatou chain is also a level-$m$ Fatou chain if $n<m$, and the Fatou chain generated by a fixed Fatou domain is a level-$1$ Fatou chain. Moreover, for \Sie\ rational maps, any maximal Fatou chain is simply the closure of a Fatou domain, while for polynomials or Newton maps, the entire sphere is a maximal Fatou chain.

\begin{theorem}[Maximal Fatou chain]\label{thm:maximal}
Let $f$ be a rational map with $J_f\neq\cbar$. Then each Fatou domain of $f$ is contained within a maximal Fatou chain. Moreover, the image and components of the pre-image of a maximal Fatou chain under $f$ are also maximal Fatou chains.
\end{theorem}

 The proof of Theorem \ref{thm:maximal} is presented in Section \ref{sec:chain}. In Section \ref{sec:topology}, we  revisit maximal Fatou chains, exploring their combinatorial and topological properties. With these foundations,  the following result  will be proved in Section \ref{sec:7}.

\begin{theorem}[Invariant graphs on maximal Fatou chains]\label{thm:graph-maximal}
Let $(f,P)$ be a marked rational map with $J_f\neq\cbar$, and let $K$ be the intersection of $J_f$ with an $f$-invariant maximal Fatou chain. Then there exists a graph $G\subset K$ such that $f(G)\subset G$, $G\cap P=K\cap P$, and two points of $P$ lie in distinct components of $\cbar\setminus G$ provided that they belong to distinct components of $\cbar\setminus K$.
\end{theorem}

\begin{remark}
{\rm If a PCF rational map has a maximal Fatou chain equal to $\cbar$, then Theorem \ref{thm:main} follows directly from Theorem \ref{thm:graph-maximal} since every Fatou domain contains at most one marked point. From the perspective of Julia set configurations, such a map can be viewed as a generalization of polynomials and Newton maps, and is referred to as a {\bf cluster rational map}.}
\end{remark}

The third key ingredient concerns the decomposition of a marked rational map. According to Theorem \ref{thm:graph-maximal}, in order 
 to construct a global invariant graph, it is necessary to investigate the dynamics  outside the union of marked maximal Fatou chains. This approach  leads to a decomposition of marked rational maps by maximal Fatou chains, which we present in a generalized form.

\begin{definition}
Let $f$ be a rational map, and let $\KKK$ be a union of finitely many pairwise disjoint continua. We call $\KKK$ a {\bf stable set} of $f$ if $f(\KKK)\subset\KKK$ and each component of $f^{-1}(\KKK)$ is either  a component of $\KKK$ or  disjoint from $\KKK$.
\end{definition}

According to Theorem \ref{thm:maximal},  the union of all periodic maximal Fatou chains is a specific example of a stable set. By definition, each component of a stable set is eventually periodic. Thus, the following result describes the dynamics on a stable set.

\begin{theorem}[Renormalization]\label{thm:renorm}
Let $f$ be a   PCF  rational map, and let $K\not=\cbar$ be a connected stable set of $f$. Then $f$ is {\bf renormalizable} on $K$, i.e., there exist a  rational map $g$ and a quasiconformal map $\phi$ of $\cbar$ such that $J_{g}=\phi(\partial K)$ and $\phi\circ f=g\circ\phi$ on $K$.
Moreover, the rational map $g$ can be taken to be PCF and is unique up to conformal conjugacy. We call $g$ the {\bf renormalization} of $f$ on $K$.
\end{theorem}

Next, we consider the dynamics outside a stable set.


\begin{definition}\label{def:exact-system}
Let $(f,P)$ be a marked rational map, and let $\VVV_1\subset\VVV$ be open sets with $\partial\VVV\subset J_f$ such that each component of $\partial \VVV$ contains more than one point. We say $f:\VVV_1\to\VVV$ is an {\bf exact sub-system} of $(f,P)$ if
\begin{itemize}
\item [(1)]$\VVV$ has finitely many components, each of which is finitely connected;

\item [(2)] $\VVV_1$ is the union of some components of $f^{-1}(\VVV)$;

\item [(3)] each component of $\VVV\setminus\VVV_1$ is a  continuum disjoint from $P$.
\end{itemize}
\end{definition}

By definition, each component of $\VVV$ contains a unique component of $\VVV_1$.
Consequently, there exists a self-map $f_{\#}$ on the collection of components of $\VVV$ defined by $f_{\#}(V):=f(V_1)$, where $V_1$ is the unique component of $\VVV_1$ contained in $V$. Since $\VVV$ has finitely many components, every component of $\VVV$ is eventually $f_{\#}$-periodic.
Therefore, the dynamics of an exact sub-system is characterized by the following theorem.

\begin{theorem}[Blow-up]\label{thm:blow-up}
Let $(f,P)$ be a marked  rational map. Suppose that $f:V_1\to V$ is an exact sub-system of $(f,P)$ such that $V$ is connected. Denote
$$
V_n=(f|_{V_1})^{-n}(V)\quad\text{ and }\quad E=\bigcap_{n>0}\ov{V_n}.
$$
Then there exist a marked  rational map $(g,Q_g)$, a continuum $K_g\supset J_g$ with $g^{-1}(K_g)=K_g$, and a continuous onto map $\pi:\cbar\to\cbar$ such that
\begin{enumerate}
\item  components of $\cbar\setminus K_g$ are all Jordan domains with pairwise disjoint closures;
\item  $E=\pi(K_g)$ and $f\circ\pi=\pi\circ g$ on $K_g$;
\item  for any point $z\in\bigcap_{n>0} V_n$, the fiber $\pi^{-1}(z)$ is a singleton;
\item  for any component $B_n$ of $\cbar\setminus V_n$, the set $\pi^{-1}(B_n)$ is the closure of a component of $\cbar\sm K_g$;

\item  a point $x\in Q_g$ if and only if either $\pi(x)\in P\cap V$, or $x$ is the center in the B\"{o}ttcher coordinate of a component $D$ of $\cbar\setminus K_g$ such that $\pi(\ov{D})\cap P\not=\emptyset$. 
\end{enumerate}
\noindent Moreover, the marked rational map $(g,Q_g)$ is unique up to conformal conjugacy.
\end{theorem}

The marked rational map $(g,Q_g)$ is called the {\bf blow-up} of the exact sub-system $f:V_1\to V$ of $(f, P)$. Generally, if $f:\VVV_1\to\VVV$ is an exact sub-system of $(f,P)$, and $V$ is an $f_{ \#}$-periodic component of $\VVV$ with period $p$, then the blow-up of the exact sub-system $f^p:V_p\to V$ of $(f^p,P)$ is regarded as a {\bf blow-up} of $f:\VVV_1\to\VVV$ (associated with $V$). Here, $V_p$ denotes the component of $(f|_{\VVV_1})^{-p}(V)$ contained in $V$.

\vspace{2pt}

The primary result of the third key ingredient is the decomposition theorem below. 

\vspace{2pt}

A connected open or closed set $E$ is called {\bf simple-type} (rel $P$) if there is a simply connected domain $D\subset\cbar$ such that $E\subset D$ and $\#(D\cap P)\le 1$; or {\bf annular-type} if $E$ is not simple-type and there is an annulus $A\subset\cbar\setminus P$ such that $E\subset A$; or {\bf complex-type} otherwise.

\begin{theorem}[Cluster-\Sie\ decomposition]\label{thm:cluster-exact}
Let $(f,P)$ be a marked rational map with $J_f\neq\cbar$. Then there exists a stable set $\KKK\subset J_f$ such that
\begin{enumerate}
\item  for any periodic component $K$ of $\KKK$ with period $p$, the renormalization of $f^p$ on $K$ is a cluster rational map;
\item  either $\VVV=\emptyset$ or $f:\VVV_{1}\to\VVV$ is an exact sub-system of $(f,P)$, where $\VVV$ and $\VVV_1$ are the unions of complex-type components of $\cbar\sm\KKK$ and $\cbar\sm f^{-1}(\KKK)$, respectively.
\end{enumerate}
Moreover, each blow-up of $f:\VVV_{1}\to\VVV$ is a marked \Sie\ rational map. 
\end{theorem}

\begin{remark}\label{rem:decomposition}
	{\rm
By Theorem \ref{thm:cluster-exact},  the dynamics of $(f,P)$ is essentially inherited by the sub-systems $f:\KKK\to\KKK$ and $f:\VVV_1\to\VVV$. In fact, the complement of $\KKK\sqcup \VVV$ can be expressed as $\AAA\sqcup \SSS$, where $\AAA$ and $\mathcal{S}$ denote the unions of all annular-type and simple-type components of $\cbar\setminus\KKK$, respectively. 

The set $\AAA$ has finitely many components, each of which is an annulus (see Theorem \ref{thm:cluster-exact0}). Let $\AAA_1$ be the union of all annular-type components of $f^{-1}(\AAA)$. It follows that $\AAA_1\subset\AAA$ and $f:\AAA_1\to\AAA$ forms an annular sub-system. The dynamics of an annular sub-system is straightforward and has been extensively studied in \cite{CPT2}.

Additionally, the dynamics of $f$ associated with $\mathcal{S}$ is trivial by the shrinking lemma (see Lemma \ref{lem:expanding}) since each component of $\SSS$ contains at most one point of $P_f$. }\end{remark}

 Theorem \ref{thm:cluster-exact}\,(1) and (2) and Theorem \ref{thm:renorm} are established in Section \ref{sec:4}. Theorem \ref{thm:blow-up} is proved in Section \ref{sec:5}, which immediately implies the remaining part of Theorem \ref{thm:cluster-exact}.\vspace{3pt}

Now, according to Theorem \ref{thm:cluster-exact}, any marked rational map with a non-empty Fatou set can be decomposed into several marked cluster or \Sie\ rational maps. The invariant graphs for marked cluster rational maps are established in Theorem \ref{thm:graph-maximal}, while those for marked \Sie\ rational maps appear in Theorem B.

 In the fourth and final key ingredient, we 
will connect the invariant graphs associated with these sub-systems  to derive a global invariant graph.  This can be accomplished by identifying invariant arcs within the annular sub-system described in Remark \ref{rem:decomposition}. The process is encapsulated in the following proposition, which is proved in Section \ref{sec:invariant-graphs}.

A graph is called {\bf regulated} for a PCF rational map if its intersection with the closure of any Fatou domain of the map is either empty or the union of finitely many closed internal rays.
\begin{proposition}\label{pro:pre}
	Let $(f, P)$ be a marked rational map with $J_f\not=\cbar$, and let $\KKK,\VVV,\VVV_1$ represent the sets specified in Theorem \ref{thm:cluster-exact}. Suppose  each blow-up $(g, Q_g)$ of the exact sub-system $f:\VVV_1\to\VVV$ admits a $g$-invariant regulated graph containing $Q_g$.  Then there exists an $f$-invariant graph $G\subset J_f$ such that $P\cap J_f\subset G$ and each component of $\cbar\setminus G$ contains at most one point of $P$. 
\end{proposition}


\begin{proof}[Proof of Theorem \ref{thm:main}]
	If $J_f=\cbar$, then Theorem \ref{thm:main} follows immediately from Theorem A. 
	
	Suppose that $J_f\neq \cbar$.  Let $\KKK,\VVV$ and $\VVV_1$ represent the sets specified in Theorem \ref{thm:cluster-exact}.
	For every $n\geq1$, the stable set $\KKK$ induces a cluster-\Sie\ decomposition of $(f^n,P)$. In particular, $f^n:\VVV_n\to\VVV$ is an exact sub-system of $(f^n,P)$, where $\VVV_n$ denotes the union of all complex-type components of $f^{-n}(\VVV)$.
	We will compare the  blow-ups of $f:\VVV_1\to\VVV$ and those of $f^n:\VVV_n\to\VVV$.
	
	Let $V$ be any $f_{ \#}$-periodic component of $\VVV$ with period $p$. Denote $(g, Q_g)$ as the  blow-up of the exact sub-system $f^p:V_p\to V$, where $V_p$ refers to the unique component of $\VVV_p$ contained in $V$. Fix any integer $n\geq 1$. Let $m=m(n, V)$ be the least common multiple of $n$ and $p$. Then the period of $V$ under $(f^n)_{ \#}$ is $m/{n}$. Moreover,  the blow-up of $f^n:\VVV_n\to\VVV$ associated with $V$ is the  blow-up of the exact sub-system $f^m:V_m\to V$ of $(f^m,P)$, which is exactly $(g^{m/p},Q_g)$.
	
	Since  $m(n, V)\to\infty$ as $n\to\infty$, it follows from Theorem B that each blow-up $(g^{m/p}, Q_g)$ of $f^n:\VVV_n\to\VVV$ admits a $g^{m/p}$-invariant and regulated graph passing through $Q_g$ for each sufficiently large integer $n$. Therefore, by applying Proposition \ref{pro:pre} to $(f^n, P)$ and $\KKK$, we obtain an $f^n$-invariant graph $G$ with all the properties stated in Theorem \ref{thm:main}.
\end{proof}

The standard spherical metric is denoted by $\sigma(z)|dz|$ with $\sigma(z)=1/(1+|z|^2)$. Without emphasis, the distance, diameter, convergence, etc., are all considered under the spherical metric. Thus, we use  simplified notations such as ${\rm dist}(\cdot,\cdot)$ and ${\rm diam}(\cdot)$ instead of ${\rm dist}_\sigma(\cdot,\cdot)$ and ${\rm diam}_\sigma(\cdot)$.

Another metric used in this paper is the \emph{orbifold metric} $\omega$ with respect to a   PCF  rational map. Its definition and properties are given in Appendix \ref{app:1}.
Under this metric, we typically use the \emph{homotopic length} $L_\omega[\cdot]$ and the \emph{homotopic diameter} $\text{H-diam}_\omega(\cdot)$ instead of the usual length and diameter for a smooth curve and a connected set in $\ov\C\setminus P_f$, respectively; see Appendix \ref{app:1} for their definitions and detailed discussions.

In Appendix \ref{app:2}, we introduce an isotopy lifting lemma under rational maps and a well-known convergence result for a sequence of isotopies obtained by lifting. Appendix \ref{app:3} includes three topological results related to local connectivity.

\subsection{Related work}\label{sec:work}
Theorem \ref{thm:cluster-exact} is closely related to Theorem C in a recent work \cite{DHS} by
Dudko, Hlushchanka and Schleicher. We first became aware of their work in 2022 from a slide
by Hlushchanka, by which time the main results of our paper had already been completed.

In our opinion, these two decomposition theorems are essentially the same, but with quite different formulations and approaches. 
In \cite{DHS}, the decomposition is by means of stable multicurves,
as done by Pilgrim in \cite{Pil}; while our decomposition directly utilizes stable sets. Netherless, both
of the starting points are the maximal Fatou chains (called \emph{maximal clusters} in \cite{DHS}). Another
relevant work can be found in \cite{CYY}.

Recently, several interesting results about PCF cluster maps were announced. For example, this type of map has a zero-entropy invariant graph containing $P_f$ (see \cite[Theorem B]{DHS}), and the Ahlfors-regular conformal
dimension of its Julia set is equal to one (see \cite[Theorem A]{P}).

D. Thurston posed a question (see \cite[Question 1.19]{Th}) regarding the identification of a preferred ``best'' spine of $\cbar\setminus P_f$ for a hyperbolic PCF cluster rational map. In this case, the invariant graph obtained in Theorem \ref{thm:graph-maximal}  appears to
 be a good candidate.

The existence of invariant graphs has also been studied beyond the rational case.  A {\bf Thurston map} is a PCF  branched covering on the $2$-sphere. Bonk and Meyer \cite{BM} proved that any \emph{expanding} Thurston map $f$ admits an $f^n$-invariant Jordan curve passing through all post-critical points for each sufficiently large integer $n$. More broadly, a Thurston map is \emph{B\"{o}ttcher expanding} if it has a certain ``expansion property'' near its Julia set (see \cite{BD2}). The dynamics of such maps is investigated in a series of works, including \cite{BD1,BD2,BM,FPP1,FPP2}. In particular, Floyd, Parry, and Pilgrim \cite{FPP2} showed that a suitable iterate of a {B\"{o}ttcher expanding} Thurston map admits an isotopy-invariant graph  containing all post-critical points.

Invariant graphs are extensively used in the study of the dynamics of PCF rational maps and Thurston maps. For example, Meyer \cite{M} investigated the unmating of PCF rational maps with empty Fatou sets using invariant Peano curves. Hlushchanka and Meyer employ the invariant Jordan curves from Theorems A and B to calculate the growth of iterated monodromy groups for certain PCF rational maps. Additionally, based on Theorem A, Li established the thermodynamic formalism \cite{Li1,Li2} and, in collaboration with Zheng, the prime orbit theorems \cite{LZ1,LZ2,LZ3} for expanding Thurston maps.

\subsection{Future directions}
First, a natural question arises regarding whether the iterate is strictly necessary in Theorem \ref{thm:main}. To address this question, we propose the following conjecture.

\begin{conjecture}\label{conj.1}
	 For any marked rational map $(f,P)$, Theorem \ref{thm:main} holds with $n=1$. In other words, there exists an $f$-invariant graph $G\subset J_f$ such that $P\cap J_f\subset G$ and each component of $\cbar\sm G$ contains at most one point of $P$.
\end{conjecture}

According to Proposition \ref{pro:pre},  this conjecture is true if we can confirm that any marked rational map $(g,Q)$ with its Julia set equal to  either the sphere or the \Sie\ carpet admits a $g$-invariant and  regulated graph containing $Q$. 

Every PCF rational map with the Julia set equal to $\cbar$ is an expanding Thurston map. In addition, each PCF \Sie\ rational map $f$ can  descend to an expanding Thurston map $F$ by collapsing the closure of each Fatou domain to a point, and any graph in the $F$-plane can be lifted to a  regulated graph for $f$; see \cite[Sections 5 and 6]{GHMZ}. 
Therefore, Conjecture \ref{conj.1} is implicated by the following conjecture,  which appeared in \cite[Problem 2]{BM}. 

\begin{conjecture}
	For any marked expanding Thurston map $(F,Q)$, there exists an $F$-invariant graph  containing $Q$. 
\end{conjecture} 

Another direction concerns the renormalizability of a rational map on stable sets. 
A classical result by McMullen asserts that any rational map is renormalizable on each of its fixed Julia components \cite[Theorem 3.4]{Mc2}.
It is worth noting that every fixed Julia component is a specific connected stable set. On the other hand, Theorem \ref{thm:renorm} shows that if the rational map is PCF, then it is renormalizable on any connected stable set, due to the expansion property near the Julia set.
\begin{question}
	Is every rational map renormalizable on any connected stable set or on any fixed maximal Fatou chain of the map?
\end{question}

The next direction examines the invariant graphs derived from Theorem \ref{thm:main} from the perspective of entropy.
According to W. Thurston, the \emph{core entropy} of a polynomial is  the topological entropy on its Hubbard tree, which is a very useful tool for studying the bifurcation locus of polynomials \cite{GT,Th+,Ti1,Ti2}. However, there exists currently no definition for the core entropy of a rational map.

Consider a marked rational map $(f,P_f)$, and let $\GG$ denote the collection of all graphs obtained in Theorem \ref{thm:main}. For polynomials, the topological entropy of $f$ on the graphs in $\GG$ remains constant, which equals the maximum of the core entropy of $f$ and $\log d_{ U}/p_{ U}$ for all periodic Fatou domains $U$, where $p_{ U}$ denotes the period of $U$ and $d_{ U}$ denotes the degree of $f^{p_{ U}}:U\to U$.
Based on this observation, a potential candidate for the core entropy of $f$ is given by
\[h(f)=\inf_{G\in\GG}\{h_{top}(f^n|_{G})/n:f^n(G)\subset G,n\geq1\},\]
where $h_{top}(f^n|_{G})$ denotes the topological entropy of $f^n:G\to G$. Indeed,  a motivation for us to construct invariant graphs within the Julia set is to define the core entropy of a rational map.

Additionally, when $f$ is a polynomial, the graphs in $\GG$ are isotopic rel $P_f$ under some natural restrictions. However, in the general case, the elements of $\GG$ are far from unique up to isotopy. Therefore, it is important to seek invariant graphs with canonical conditions. From the perspective of entropy, we may ask
\begin{question}
	 Is there a $($unique$)$ $f^n$-invariant graph $G\in\GG$ such that $h(f)=h_{top}(f^n|_{G})/n$\,?
	 \end{question}

The final direction is to generalize Theorem \ref{thm:main} to the non-rational case, specifically to B\"{o}ttcher expanding Thurston maps as mentioned in Section \ref{sec:work}.
These maps also have  Julia and Fatou sets and share several similarities with PCF rational maps. Hence, it is plausible to expect that Theorem \ref{thm:main} applies to B\"{o}ttcher expanding Thurston maps as well.
\begin{question}
Do $($any of\,$)$ the theorems listed in the Introduction still hold for B\"{o}ttcher expanding Thurston maps after appropriate revisions?
\end{question}

\vskip 0.3cm

\noindent {\bf Acknowledgements.} The authors are grateful for insightful discussions with Zhiqiang Li, Xiaoguang Wang, Yunping Jiang, Dylan Thurston, and Luxian Yang. The first author is supported by the National Key R\&D Program of China (Grant no.\,2021YFA1003203)
and the NSFC (Grant nos.\,12131016 and 12071303).
The second author is supported by the NSFC (Grant no.\,12322104) and the NSFGD (Grant no.\,2023A1515010058).
The third author is supported by the NSFC (Grant no.\,12271115).

\section{Invariant graphs associated with fixed Fatou domains}\label{sec:2}
In this section, we study the dynamics of a rational map $f$ on the boundary of a fixed Fatou domain $U$ of $f$. We begin by examining the mapping behavior of $f$ on $\partial U$. Next, we construct an invariant continuum on $\partial U$ with nice topological properties, called the \emph{circle-tree}.  Finally, we present the proof of Theorem \ref{thm:local}.

\subsection{Circle-trees}
Let $U\subset\cbar$ be a simply connected domain such that $T_0:=\partial U$ is a locally connected continuum. The following lemma is classical (see \cite[Chapter 2]{DH2}). In this paper,  a {\bf circle}  means a Jordan curve, and a {\bf disk} means a Jordan domain in $\cbar$. An {\bf arc} is a continuous injective map from $[0,1]$ into $\cbar$, and its restriction to $(0,1)$ is called an {\bf open arc}. 

\begin{lemma}\label{lem:topology}
The following statements hold:
\begin{enumerate}
\item  Both $T_0$ and $\cbar\sm U$ are arcwise connected;
\item  All components of $\cbar\sm\ov{U}$ are disks, whose diameters converge to zero;
\item  Each circle $C\subset T_0$ is the boundary of a component of $\cbar\sm\ov{U}$.
\end{enumerate}
\end{lemma}

\begin{lemma}\label{lem:circle}
Let $C\subset T_0$ be a circle. If $E\subset T_0$ is a continuum, then $C\cap E$ is connected. If $C'\neq C$ is also a circle in $T_0$, then $\#(C\cap C')\le 1$.
\end{lemma}

\begin{proof}
Suppose, to the contrary, that $C\cap E$ is disconnected. Then $C\sm E$ has at least two components. Let $x$ and $y$ be two points  contained in two distinct components of $C\sm E$, respectively. Let $D$ be the component of $\cbar\sm C$ disjoint from $U$. Then there exist open arcs $\alpha\subset U$ and $\beta\subset D$, both joining the points $x$ and $y$. Now, $\alpha\cup\beta\cup\{x,y\}$ is a Jordan curve disjoint from $E$, and both of its complementary components intersect $E$. This contradicts  the connectivity of $E$.

 Suppose $C'\neq C$ is also a circle in  $T_0$. Then $I=C\cap C'$ is connected by the above discussion. If $I$ contains at least two points, then it contains an open arc $\gamma$. This implies that each point in $\g$ is an exterior point of $U$, which contradicts the fact that $\g\subset C\subset \partial U$.
\end{proof}

Motivated by the above results, we consider circles in $T_0$ as entire entities when discussing subsets of $T_0$.

\begin{definition}
A continuum $T\subset T_0$ is called a {\bf circle-tree} of $T_0$ if, for any circle $C\subset T_0$, either $C\subset T$ or $\#(C\cap T)\le 1$.
\end{definition}

Let $T$ be a circle-tree of $T_0$. A point $x\in T$ is a {\bf cut point} of $T$ if $T\sm\{x\}$ is disconnected. A circle $C\subset T$ is an {\bf end circle} of $T$ if $C$ contains at most one cut point of $T$. A point $x\in T$ is an {\bf endpoint} of $T$ if it is neither contained in a circle in $T$ nor a cut point. By an {\bf end}, we mean an endpoint or an end circle. We call $T$ a {\bf finite circle-tree} if $T$ has  finitely many ends.

In order to study circle-trees and their topology, one useful tool is the geodesic lamination introduced by W. Thurston. Let $\D$ denote the unit disk. Then there exists a conformal map $\phi: \C\sm\ov{\D}\to U$, which can be extended continuously to the boundary. For each point $x\in T_0$, denote by $H_{x}$ the convex hull within $\ov{\D}$ of $\phi^{-1}(x)$ under the Poincar\`{e} metric on $\D$. The basic observation of lamination theory is
$$
H_{x}\cap H_y=\emptyset\quad\text{ if }x\neq y.
$$
Note that $\partial H_x\cap\D$ consists of geodesics if it is non-empty. The {\bf lamination} $\LLL_{ U}$ induced by $U$ is defined as the union of all such geodesics, which are called {\bf leaves}. Then $\LLL_{U}$ is closed in $\D$, and the closure of a component of $\D\sm\LLL_{U}$ is a {\bf gap} of $\LLL_{U}$.

\begin{lemma}\label{lem:lamination}
Assume that $U$ is not a disk. Then the following statements hold:
\begin{enumerate}
\item For each gap $A$ of $\LLL_{U}$, $\phi(A\cap\partial\D)$ is either a point or a circle. Conversely, for any circle $C\subset T_0$, there exists a unique gap $A$ such that $\phi(A\cap\partial\D)=C$. Moreover, $C$ is an end circle of $T_0$ if and only if $A\cap\partial\D$ is connected.

\item A point $x\in T_0$ is an endpoint if and only if $\#\phi^{-1}(x)=1$, and there exists a sequence of leaves $\{L_n\}$ in $ \LLL_{U}$ converging to $\phi^{-1}(x)$, such that $L_{n}$ separates $L_{n-1}$ from $L_{n+1}$.

\item Let $x\in T_0$ be a point, and let $I_0$ be a component of $\partial\D\sm\phi^{-1}(x)$. Then either $\phi(\ov{I_0})$ is an end circle, or $\phi(I_0)$ contains an end.

\item Let $C\subset T_0$ be a circle, and let $I_0$ be a component of $\partial\D\sm\phi^{-1}(C)$. Then either $\phi(\ov{I_0})$ is an end circle, or $\phi(I_0)$ contains an end.
\end{enumerate}
\end{lemma}

\begin{figure}[http]
	\centering
	\begin{tikzpicture}
		\node at (0,0){ \includegraphics[width=6cm]{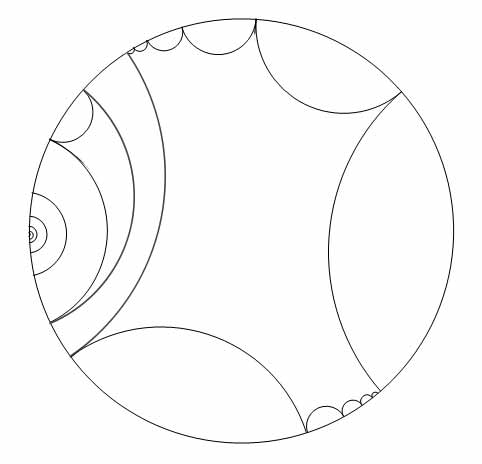}};
		\node at (-2.9,-0.1) {$t_0$};
		\node at (1.85,-0.125) {$A$};
	\end{tikzpicture}
	\caption{$\phi(\partial\mathbb{D}\cap A)$ is an end circle and $\phi(t_0)$ is an endpoint.}\label{fig:lamination}
\end{figure}
\begin{proof}
(1) Note that $\partial A$ is a Jordan curve. Define a map $\phi_{A}: \partial A\to T_0$ by $\phi_{A}=\phi$ on $\partial A\cap\partial\D$ and $\phi_{A}(L)=\phi(L\cap\partial\D)$ for any leaf $L\subset\partial A$. Then $\phi_{A}$ is continuous, and $\phi_{A}(\partial A)=\phi(A\cap\partial\D)$. Thus, $\phi_{A}(\partial A)\subset T_0$ is either a point or a closed curve. In the latter case, the curve is not self-intersecting since $\phi_{A}^{-1}(x)$ is connected for any  $x\in\phi_{A}(\partial A)$. Therefore, it is  a circle in $T_0$.

Conversely, let $C\subset T_0$ be a circle. For any point $x\in C$, $C\sm\{x\}$ is connected. Thus, $\phi^{-1}(C\sm\{x\})$ is contained in a component $A_{x}$ of $\ov{\D}\sm H_{x}$, and $C\subset\phi(\ov{A_{x}}\cap\partial \D)$. Let $A=\bigcap_{x\in C}\ov{A_{x}}$. Then $A$ is a gap, and $C\subset\phi(A\cap\partial\D)$. From the  discussion in the previous paragraph, $\phi(A\cap\partial\D)$ is either a point or a circle. Thus, we have $C=\phi(A\cap\partial\D)$.

If $A'\neq A$ is another gap, then there exists a leaf $L\subset\partial A$ that  separates the interior of $A$ from $A'$. Thus, $\phi(A\cap\partial\D)\cap\phi(A'\cap\partial\D)$ contains at most one point, and then $\phi(A'\cap\partial \D)\not=C$.

If $A\cap\partial\D$ is connected, then $\phi$ is injective in the interior of $A\cap\partial\D$, whose image contains no cut points, and $\phi$ maps the two endpoints of $A\cap\partial \D$ to a cut point. Thus, $C$ is an end circle.
Conversely, if $C$ is an end circle, let $x\in C$ be the unique cut point. Then $A\cap\partial\D=\ov{\phi^{-1}(C\sm\{x\})}$ is connected since $\phi^{-1}(y)$ is a point for $y\in C\sm\{x\}$.
\vspace{2pt}

(2) Denote $x_n=\phi(L_n\cap\partial\D)$. Let $B_n$ be the component of $T_0\sm\{x_n\}$ containing the point $x$. Then $B_{n+1}\subset B_n$, and the diameter of $B_n$ tends to $0$ as $n\to\infty$. Thus, $x$ is an endpoint.

Conversely, if $x\in T_0$ is an endpoint, then $\phi^{-1}(x)$ consists of a single point $t\in\partial\D$, and there exist no leaves landing on $t$. For each leaf $L$, denote by $|L|_t$ the length of the component of $\partial\D\sm L$ containing the point $t$. Assume, by contradiction,  that $\inf\{|L|_t\}>0$. Then there exists a leaf $L_0$ such that $|L_0|_t=\inf\{|L|_t\}$ since $\LLL_{ U}$ is closed. Let $D_0$ be the component of $\D\sm L_0$ whose boundary contains the point $t$. Then there exist no leaves in $D_0$ separating $L_0$ from the point $t$. Thus, there exists a gap $A$ containing the point $t$ and the leaf $L_0$. By statement (1), $\phi(A\cap\partial\D)$ is either a single point or a circle. Since $x\in\phi(A\cap\partial\D)$ is an endpoint, we obtain $x=\phi(A\cap\partial\D)$, which contradicts the condition that $\phi^{-1}(x)$ is a single point.
\vspace{2pt}

(3) By statement (1), the two endpoints of $I_0$ are connected by a leaf in $\LLL_{U}$. Denote by $\III$ the collection of all open arcs $I\subset I_0$ with $I\neq I_0$ such that the two endpoints of $I$ are connected by a leaf in $\LLL_{U}$. Then any two arcs in $\III$ are either disjoint or nested since any two distinct leaves are disjoint.

If $\III$ is empty, then $\phi(\ov{I_0})$ is an end circle by statement (1). If $|I|>|I_0|/2$ for all $I\in\III$, then there exists a unique arc $I^*\in\III$ such that $I^*\subset I$ for all $I\in\III$. This implies that $\phi(\ov{I^*})$ is an end circle. Otherwise, there exists an arc $I_1\in\III$ such that $|I_1|\le |I_0|/2$.

By iterating this process,  we have to either stop at some step, yielding an end circle, or  obtain an infinite sequence of arcs $\{I_n\}$ such that $I_{n+1}\subset I_n$ and $|I_{n+1}|\le |I_n|/2$. By the definition of lamination, at most two leaves share a common endpoint. Thus, $t=\bigcap I_n$ is a single point. By statement (2), $\phi(t)$ is an endpoint.

(4) The proof is similar to that of statement (3).
\end{proof}

The following result is a direct consequence of Lemma \ref{lem:lamination} (3) and (4).

\begin{corollary}\label{cor:end}
Let $x\in T_0$ be a point, and let $B$ be a component of $T_0\sm\{x\}$. Then either $\ov{B}$ is an end circle, or $B$ contains an end of $T_0$. Let $C\subset T_0$ be a circle, and let $B$ be a component of $T_0\sm C$. Then $\ov{B}\cap C$ is a singleton, and either $\ov{B}$ is an end circle or $B$ contains an end of $T_0$.
\end{corollary}

A circle-tree can be characterized by the lamination $\LLL_{U}$.

\begin{lemma}\label{lem:circle-tree}
A continuum $T\subset T_0$ is a circle-tree of $T_0$ if and only if each component of $\partial H_{T}\sm\partial\D$ is a leaf in $\LLL_{U}$, where $H_{T}$ is the convex hull of $\phi^{-1}(T)$ within $\ov{\D}$.
\end{lemma}

\begin{proof}
For any circle $C\subset T_0$, there exists a unique gap $A$ such that $\phi(A\cap\partial\D)=C$ by Lemma \ref{lem:lamination}\,(1). Since each component of $\partial H_{T}\sm\partial\D$ is a leaf, either $A$ is contained in $H_{T}$, or $A\cap H_{T}=\emptyset$, or $A\cap H_{T}$ is a leaf. Thus, either $C\subset T$ or $\#(T\cap C)\le 1$. Therefore, $T$ is a circle-tree of $T_0$.

Conversely, assume that $T$ is a circle-tree of $T_0$. Let $I=(s,t)$ be a component of $\partial\D\sm\phi^{-1}(T)$. Denote $\phi(s)=x$ and $\phi(t)=y$. Then $x,y\in T$.

If $x\not=y$, then $H_x\cap H_y=\emptyset$. Note that there exist no leaves of $\LLL_{U}$ in $\D\sm(H_x\cup H_y)$ separating $H_x$ from $H_y$, since such a leaf would have an endpoint in $I$, which contradicts the connectivity of $T$. Thus, there exists a gap $A$ such that $s, t\in A\cap\partial \D$. By Lemma \ref{lem:lamination}\,(1), $\phi(A\cap\partial\D)$ is a circle in $T_0$, which contains the points $x,y\in T$. Thus, it is contained in $T$ since $T$ is a circle-tree, which implies that $A\subset H_{T}$. Hence, $I$ is a component of $\partial\D\sm A$. This implies that $s,t$ are connected by a leaf in $\partial A$, and hence $x=y$, a contradiction.

Since $x=y$,  either there exists a leaf joining the points $s$ and $t$, or $H_x\cap I\neq\emptyset$. The latter cannot happen since $I\cap H_{T}=\emptyset$. Thus, $s$ and $t$ are connected by a leaf in $\LLL_{U}$.
\end{proof}

\begin{corollary}\label{cor:circle-tree}
Let $T$ be a circle-tree of $T_0$. Then $T$ is locally connected, and there exists a simply connected domain $V\subset\cbar$ such that $\partial V=T$.
\end{corollary}

\begin{proof}
Note that $\partial H_{T}$ is a Jordan curve. By Lemma \ref{lem:circle-tree}, each component of $\partial H_{T}\sm\partial\D$ is a leaf. Define a map $\phi_{T}: \partial H_{T}\to T_0$ by $\phi_{T}=\phi$ on $\partial H_{T}\cap\partial\D$ and $\phi_{T}(L)=\phi(L\cap\partial\D)$ for any leaf $L\subset\partial H_{T}$. Then $\phi_{T}$ is continuous, and $\phi_{T}(\partial H_{T})=T$. Thus, $T$ is locally connected.

Let $V$ be the component of $\cbar\sm T$ containing $U$. Then $V$ is a simply connected domain, and $\partial V\subset T$. On the other hand, $T\subset\ov{U}\subset\ov{V}$. Thus, $T\subset\partial V$. Hence, we have $\partial V=T$.
\end{proof}

The following result provides a basic tool for constructing circle-trees.

\begin{lemma}\label{lem:span}
Let $x,y\in T_0$ be two distinct points. Then there exists a unique circle-tree $T[x,y]$ of $T_0$ such that any circle-tree of $T_0$ containing $x$ and $y$ contains $T[x,y]$. Moreover, each end of $T[x,y]$ intersects $\{x, y\}$.
\end{lemma}
We call $T[x,y]$ the {\bf circle-tree spanned by $\boldsymbol{\{x,y\}}$}.
\begin{proof}
By Lemma \ref{lem:topology}\,(1), there exists an arc $\gamma: [0,1]\to T_0$ with $\g(0)=x$ and $\g(1)=y$. Let $T_1$ be the union of $\gamma$ and all  circles $C\subset T_0$ with $\#(C\cap\gamma)\ge 2$. By Lemma \ref{lem:topology}\,(2), $T_1$ is a continuum.

We will show that $T_1$ is a circle-tree. By definition, it suffices to prove that
 for any circle $C\subset T_0$ with $\#(C\cap T_1)\ge 2$, it holds that $\#(C\cap \g)\geq 2$.

 Suppose, to the contrary, that $\#(C\cap \g)\leq 1$. Let $x_1,x_2\in C\cap T_1$ be two distinct points, and let $\alpha$ be an arbitrary component of $C\setminus\{x_1,x_2\}$.

 If $C\cap \g=\emptyset$, then there exist two distinct circles
 $C_1, C_2\subset T_1$ such that $x_1=C\cap C_1$ and $x_2=C\cap C_2$. By the definition of $T_1$, there exists an arc $\g_0\subset \g$ such that $y_1:=\g_0(0)\in C_1,y_2:=\g_0(1)\in C_2$, and $\g_0(0,1)$ is disjoint from $C_1\cup C_2$. For $i=1,2$, let $\beta_i$ be a component of $C_i\setminus \{x_i,y_i\}$ such that $\beta_1\cap \beta_2=\emptyset$. Then $\alpha,\beta_1,\beta_2$, and $\g_0$ are pairwise disjoint. It follows that
$$
\alpha\cup\beta_1\cup\beta_2\cup\gamma_0\cup\{x_1, x_2,y_1,y_2\}
$$
is a circle in $T_0$, a contradiction to  Lemma \ref{lem:topology}\,(3).

If $\#(C\cap\gamma)=1$, we may assume $x_1$ to be this intersection point, and there exists a circle $C_2\subset T_1$ with $x_2=C\cap C_2$.  A similar argument as above will also lead to a contradiction to Lemma \ref{lem:topology}\,(3).  Now, we have proved that $T_1$ is a circle-tree.

Let $T_2$ be a circle-tree containing the points $x$ and $y$. Then there exists an arc $\g'\subset T_2$ joining $x$ and $ y$. For any component $\g_1$ of $\g\setminus \g'$, denote by $\g_1'$ the sub-arc of $\g'$ with the same endpoints as those of $\g_1$. Thus, $\g_1\cup \g_1'$ is a circle in $T_0$. Since $\g_1'\subset T_2$, it follows that $\g_1\cup\g_1'\subset T_2$, and hence $\g\subset T_2$. By the definition of $T_1$, we have $T_1\subset T_2$. This implies the uniqueness of $T_1$.

By definition, any point of $T_1$ belongs to either $\g$ or a circle in $T_0$. Thus, an endpoint of $T_1$ must be $x$ or $y$.
If $C$ is an end circle of $T_1$ disjoint from $\{x,y\}$, then $T_1':=(T_1\setminus C)\cup\{z\}\subset T_1$ is a circle-tree containing $x$ and $y$, where $z$ is the unique cut point of $T_1$ on $C$. The uniqueness implies $T_1'=T_1$, a contradiction.
\end{proof}

\begin{lemma}\label{lem:operation}
Let $T_1$ and $T_2$ be circle-trees of $T_0$ such that $T_1\cap T_2\neq\emptyset$. 
\begin{enumerate}
\item  $T_1\cap T_2$ is either a singleton or a circle-tree of $T_0$.
\item  $T_1\cup T_2$ is a circle-tree of $T_0$, and each end of $T_1\cup T_2$ is an end of $T_1$ or $T_2$.
\end{enumerate}
\end{lemma}

\begin{proof}
(1) For any two distinct points $x,y\in T_1\cap T_2$, $T[x,y]\subset T_1\cap T_2$ by Lemma \ref{lem:span}. Thus, $T_1\cap T_2$ is a continuum. For any circle $C\subset T_0$ with $\#(C\cap T_1\cap T_2)\ge 2$, we have $\#(C\cap T_1)\ge 2$ and $\#(C\cap T_2)\ge 2$. Thus, $C\subset T_1\cap T_2$. Therefore, $T_1\cap T_2$ is a circle-tree of $T_0$.\vspace{2pt}

(2) By Lemma \ref{lem:circle-tree}, each component of $\partial H_{T_1}\sm\partial\D$ and $\partial H_{T_2}\sm\partial\D$ is a leaf in $\LLL_{U}$. Since any two distinct leaves are disjoint in $\D$, each component of $\partial H_{T_1\cup T_2}\sm\partial\D$ is a leaf in $\LLL_{U}$. Thus, $T_1\cup T_2$ is a circle-tree of $T_0$.

Let $x\in T_1\cup T_2$ be a point disjoint from any circle in $T_1\cup T_2$. Assume $x\in T_1$.  If $x$ is a cut point of $T_1$, then there exists a Jordan curve in $U\cup\{x\}$ that separates  $T_1\sm\{x\}$. Thus, $x$ is a cut point of $T_1\cup T_2$. Therefore, if $x$ is an endpoint of $T_1\cup T_2$, then it is an endpoint of $T_1$ or $T_2$.

Let $C\subset T_1\cup T_2$ be an end circle. Then either $C\subset T_1$ or $C\subset T_2$. Assume $C\subset T_1$. If $C$ contains two distinct cut points $x$ and $y$ of $T_1$, then $x$ and $y$ are also cut points of $T_1\cup T_2$. This is a contradiction. Thus, $C$ is an end circle of $T_1$.
\end{proof}

For any finite set $\{x_1,\ldots, x_n\}\subset T_0$ with $n\ge 2$, denote
$$
T[x_1,\ldots, x_n]=T[x_1, x_2]\cup\cdots\cup T[x_1, x_n].
$$
 Furthermore, let $\{x_1,\ldots,x_n,C_1,\ldots,C_m\}$ be a collection of points $x_i$ and circles $C_j$ in $T_0$. Pick two distinct points $y_j, z_j\in C_j$ for each circle $C_j$. Denote
$$
T[x_1,\ldots,x_n,C_1,\ldots,C_m]=T[x_1,\ldots,x_n,y_1,\ldots,y_m,z_1,\ldots,z_m].
$$
By Lemmas \ref{lem:span} and \ref{lem:operation}, it is a finite circle-tree and also the minimal circle-tree of $T_0$ containing $x_1,\ldots,x_n,C_1,\ldots,C_m$.
 We call it the {\bf circle-tree spanned by $\boldsymbol{\{x_1,\ldots,x_n,C_1,\ldots,C_m\}}$}.

\begin{lemma}\label{lem:coincide}
Let $T$ be a finite circle-tree of $T_0$, and let $T_1$ be the circle-tree spanned by the ends of $T$. Then $T_1=T$.
\end{lemma}

\begin{proof}
By Lemma \ref{lem:span}, $T_1\subset T$. Assume that $x\in T\sm T_1$ is a point disjoint from all circles in $T$. Since $x$ is not an endpoint of $T$, there exists a component $T'$ of $T\sm\{x\}$ disjoint from $T_1$. By Corollary \ref{cor:end}, $\ov{T'}$ contains an end of $T$, a contradiction.

Assume that $C\subset T$ is a circle such that $C\cap T_1$ contains at most one point. Then $C$ is not an end circle of $T$. Thus,
$T\sm C$ has a component $T'$ disjoint from $T_1$. By Corollary \ref{cor:end}, $\ov{T'}$ contains an end of $T$, also a contradiction.
\end{proof}

Let $T$ be a finite circle-tree of $T_0$. By Corollary \ref{cor:circle-tree}, there exist a component $V$ of $\cbar\sm T$ and a conformal map $\psi: \C\sm\ov{\D}\to V$, which can be extended continuously to the boundary such that $\psi(\partial\D)=\partial V=T$. For each point $x\in T$, denote
$$
\mu_{T}(x)=\#\psi^{-1}(x).
$$
A point $x\in T$ is called either a {\bf cut point} of $T$ if $\mu_{T}(x)\ge 2$, or a {\bf branched point} of $T$ if $\mu_{T}(x)\ge 3$, or a {\bf locally branched point} of $T$ if, for any sufficiently small neighborhood $W$ of $x$, $(T\cap W)\sm\{x\}$ has at least three components. For any circle $C\subset T$, denote
$$
{\mu}_{ T}(C)=\#\{y\in C: \mu_{T}(y)\ge 2\}.
$$
A circle $C\subset T$ is called a {\bf cut circle} of $T$ if $\mu_{T}(C)\ge 2$, or a {\bf branched circle} of $T$ if $\mu_{T}(C)\ge 3$.

When $x\in T$ is not contained in any circle in $T$, then $x$ is a branched point if and only if it is a locally branched point. When $x\in T$ is contained in a circle in $T$, then $x$ is a locally branched point if and only if it is a cut point of $T$. If a circle $C\subset T$ contains no branched points of $T$, then ${\mu}_{ T}(C)$ is the number of components of $T\sm C$. In general, ${\mu}_{ T}(C)$ is the number of components of $\ov{T\sm C}$.
Refer to Figure \ref{fig:regular-set} for an example of finite circle-trees, where $p_1$ is an endpoint, $p_2$ is a cut point, and $p_3$ is a branched point; $C_1$ and $C_2$ are end circles, $C_3$ and $C_4$ are cut circles, and $C_5$ is a branched circle.

\begin{figure}[http]
	\centering
	\begin{tikzpicture}
		\node at (0,0){\includegraphics[width=5.5cm]{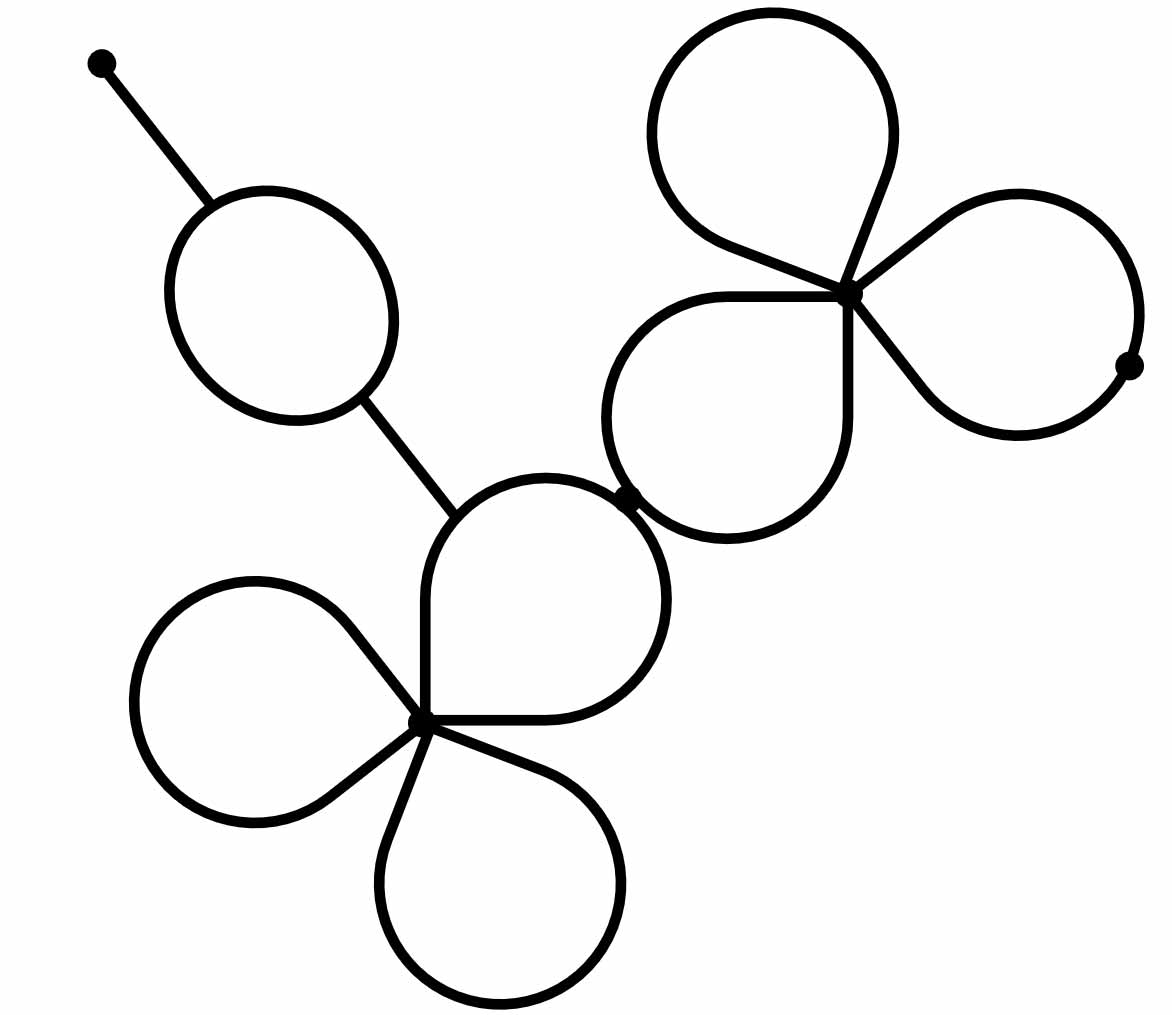}};
		\node at (-2.5,2){$p_1$};
		\node at (0.5,0.25){$p_2$};
		\node at (-0.43,-0.6){$p_3$};
		\node at (2.75, 0.5){$p_4$};
		\node at (-2, 0.5){$C_3$};
		\node at (0,1){$C_4$};
		\node at (-2.35,-1){$C_1$};
		\node at (0.5,-1.8){$C_2$};	
		\node at (0.6,-0.7){$C_5$};
	\end{tikzpicture}
\caption{Classification of points and circles in a circle-tree}\label{fig:regular-set}
\end{figure}
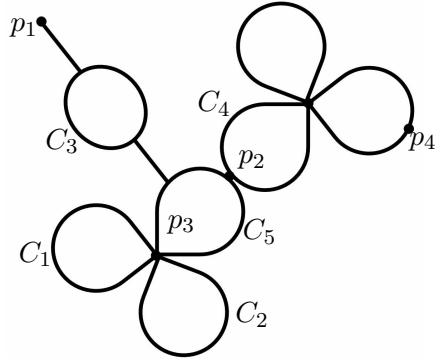


Note that any circle-tree $T\subset T_0$ has at least one end by Corollary \ref{cor:end}. If $T$ has only one end, then it is a circle.

\begin{lemma}\label{lem:number}
Let $T$ be a finite circle-tree of $T_0$ with $n\ge 2$ ends. Then $T$ has exactly $k$ branched points $\{x_i\}$ and $l$ branched circles $\{C_j\}$ such that
$$
\sum_{i=1}^k(\mu_{T}(x_i)-2)+\sum_{j=1}^l(\mu_{T}(C_j)-2)=n-2.
$$
\end{lemma}

\begin{proof}
If $n=2$, the circle-tree $T$ has neither branched points nor branched circles. In fact, if $z\in T$ is a branched point, then $T$ has at least three ends by Corollary \ref{cor:end}, a contradiction. Similarly, we obtain that $T$ has no branched circles.

Assume, by induction, that the lemma holds for an integer $n \ge 2$. Let $T$ be a circle-tree of $T_0$ with $n +1$ ends $X_0,\ldots, X_{n }$. Denote $T'=T[X_1,\ldots,X_{n }]$.

If $X_0\cap T'\neq\emptyset$, then $X_0$ is an end circle, and $T'$ intersects $X_0$ at a single point $y$.

 If $X_0\cap T'=\emptyset$, then there exists an arc $\g:[0,1]\to T$ such that $\g(0)\in X_0$, $y=\g(1)\in T'$, and $\g(t)\not\in T'$ for $t\in [0,1)$. We claim that $T[X_0,y]\cap T'=\{y\}$.

By the definition of $T[X_0,y]$ in the proof of Lemma \ref{lem:span}, it suffices to verify that for any circle $C\subset T_0$ with $\#(C\cap\g)\ge 2$, either $C\cap T'=\emptyset$ or $C\cap T'=\{y\}$. Since $\g[0,1)$ lies in a component of $T\setminus \{y\}$ disjoint from $T'$, there exists an open arc $\beta\subset U$
such that $$\lim_{t\to0}\beta(t)=\lim_{t\to 1}\beta(t)=y$$ and  $\ov{\beta}$ separates $\g[0,1)$ from $T'\setminus\{y\}$. Note that $C\subset T$ and $C\cap \g[0,1)\not=\emptyset$. Then $C\setminus\{y\}$ and $T'\setminus\{y\}$ are contained in distinct components of $\ov{\C}\setminus\ov{\beta}$.
 Thus, the claim is  proved.

In both cases, $y$ is not an endpoint of $T'$. If $y$ is a cut point of $T'$, then
$$
\mu_{T}(y)=\mu_{T'}(y)+1.
$$
Otherwise, $y$ is contained in a circle $C\subset T'$ that is not an end circle of $T'$. Thus,
$$
\mu_{T}(C)=\mu_{T'}(C)+1.
$$
For any branched point $x$ of $T'$ with $x\neq y$, it is also a branched point of $T$ with $\mu_{T}(x)=\mu_{T'}(x)$. If $C_1\not=C$ is a branched circle of $T'$, then it is also a branched circle of $T$ with $\mu_{T}(C_1)=\mu_{T'}(C_1)$. Finally, by the claim above, $T\setminus T'=T[X_0,y]\setminus\{y\}$, which contains neither branched points nor branched circles of $T$. Thus, the lemma is proved.
\end{proof}

\subsection{Images of circle-trees}
Let $f:\cbar\to\cbar$ be a branched covering, and let $U,V\subset\cbar$ be simply connected domains such that $U$ is a component of $f^{-1}(V)$ and $\partial V$ is locally connected. In particular, these conditions hold if $f$ is a rational map with a connected and locally connected Julia set, and $U$ is a Fatou domain of $f$.

A continuum $E\subset\cbar$ is {\bf full} if $\cbar\sm E$ is connected.
\begin{lemma}\label{lem:circle-image}
Let $C\subset\partial U$ be a circle. Then $f(C)$ is a finite circle-tree of $\partial V$. Moreover, each endpoint of $f(C)$ is a critical value of $f$, and if $f:C\to f(C)$ is not a homeomorphism, then each end circle of $f(C)$ either contains a critical value or separates a critical value  from $V$. 
\end{lemma}

\begin{proof}
Let $C'\subset\partial V$ be a circle such that $\#(f(C)\cap C')\ge 2$. Denote $I_1=\{x\in C: f(x)\in C'\}$ and $I_0=C\sm I_1$. Denote by $\{\alpha_i\}$ the components of $I_0$. Then each $\alpha_i$ is an open arc, and $f(\alpha_i)$ is contained in a component $B_i$ of $\partial V\sm C'$. By Corollary \ref{cor:end}, $\ov{B_i}\cap C'$ consists of a single point, and hence $f(x_i)=f(x'_i)$, where $x_i$ and $x'_i$ are the endpoints of $\alpha_i$.

\begin{figure}[http]
\centering
\includegraphics[width=12cm]{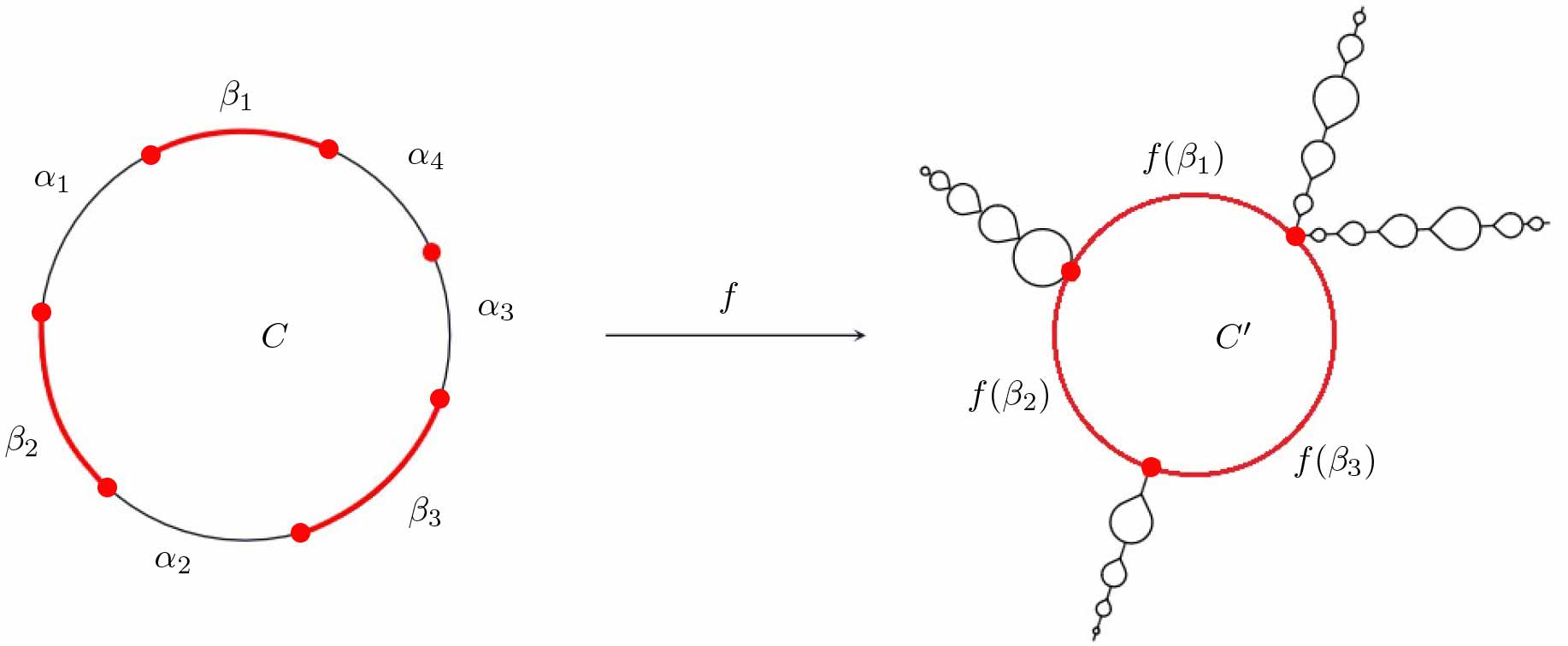}
\caption{The image of a circle. }\label{fig:circle-image}
\end{figure}

Let $E_i$ be the component of $(\cbar\sm V)\sm C'$ containing $B_i$. Then $\ov{E_i}$ is a full continuum, and $\ov{E_i}\cap C'=\{f(x_i)\}$. Moreover, $\ov{E_i}\cap\ov{E_j}=\emptyset$ if $f(x_i)\neq f(x_j)$. We claim that $\ov{E_i}$ contains critical values of $f$. Otherwise, there exists a disk $W\subset\cbar$ disjoint from the critical values of $f$ such that $\ov{E_i}\subset W$. Thus, $f$ is a homeomorphism on each component of $f^{-1}(W)$, which contradicts the assumption that $f(x_i)=f(x'_i)$.

Denote by $Z$ the set  of points $f(x_i)$ for all components $\alpha_i$. Since $\ov{E_i}\cap\ov{E_j}=\emptyset$ if $f(x_i)\neq f(x_j)$, we obtain $\# Z\le 2d-2$ by the above claim, where $d=\deg f$. For each point $z\in Z$, there exist at most $d$ components $\alpha_i$ such that $f(x_i)=z$. Therefore, $I_0$ has at most $d(2d-2)$ components. Consequently, $I_1$ has at most $d(2d-2)$ components.

By Lemma \ref{lem:circle}, $f(C)\cap C'$ is a continuum since $\#(f(C)\cap C')\ge 2$. Then at least one component $\beta_j$ of $I_1$ is an arc. Since $f: \beta_j\to C'$ preserves the orientation induced by $U$ and $V$, respectively, we obtain $f(I_1)=C'$. Thus, $C'\subset f(C)$, and hence $f(C)$ is a circle-tree of $\partial V$.

  Assume that $f:C\to f(C)$ is not a homeomorphism. Then each endpoint of $f(C)$ is a critical value of $f$. Let $C'$ be an end circle of $f(C)$. We claim that $C'$ either contains a critical value or separates critical values from $V$. If this claim is false, each component of $f^{-1}(C')$ is a Jordan curve on which the restriction of $f$ is injective. As above, denote $I_1=\{x\in C: f(x)\in C'\}$. Since $C'$ is an end circle of $f(C)$, $I_1$ has exactly one component $\beta$ that is not a single point. Thus, $f(\beta)=C'$. Since $f$ is injective on each component of $f^{-1}(C')$, it follows that $\beta=C$, and $f:C\to C'$ is a homeomorphism, a contradiction. The claim is proved.

There may exist infinitely many  circles in $\partial V$ containing critical values of $f$. However, for each critical value $v$ of $f$, there exist at most $\deg f$  circles of $\partial V$ containing $v$, which  are contained in $f(C)$. Therefore, $f(C)$ is a finite circle-tree.
\end{proof}

\begin{lemma}\label{lem:tree-image}
Let $T$ be a finite circle-tree of $\partial U$. Then $f(T)$ is a finite circle-tree of $\partial V$. Each endpoint of $f(T)$ is either the image of an endpoint of $T$ or a critical value of $f$. Each end circle of $f(T)$ either is the image of an end circle of $T$, or contains a critical value of $f$, or separates a critical value of $f$ from $V$.
\end{lemma}

\begin{proof}
Let $C'\subset\partial V$ be a circle such that $\#(C'\cap f(T))\ge 2$. We claim that there exists a circle $C\subset T$ such that $C'\subset f(C)$. By this claim, $C'\subset f(T)$, and then $f(T)$ is a circle-tree in $\partial V$.

To prove the claim, denote by $I_0\subset T$ the set of points that are not contained in any circle in $\partial U$. Then $f(I_0)\cap C'=\emptyset$, for otherwise, there exists an open arc $\beta\subset \cbar\sm\ov{V}$ that joins a point in $f(I_0)$ to a point in $\cbar\sm\ov{V}$. Thus, $f^{-1}(\beta)$ has a component in $\cbar\sm\ov{U}$ that joins a point in $I_0$ to a point in $\cbar\sm\ov{U}$, which is impossible.

Denote $I_1=T\sm I_0$. Then each point of $I_1$ is contained in a circle of $\partial U$.

Assume, by contradiction,  that $C'\not\subset f(C)$ for any circle $C\subset T$. It follows that $\#(C'\cap f(C))\le 1$ since $f(C)$ is a circle-tree. Thus, $C'\cap f(I_1)$ is a countable set, as $\partial U$ has only countably many circles. Since $C'\cap f(I_0)=\emptyset$, we know that $C'\cap f(T)=C'\cap f(I_1)$ is a countable set. On the other hand, by Lemma \ref{lem:circle}, $C'\cap f(T)$ is a continuum since $\#(C'\cap f(T))\ge 2$, a contradiction. Thus, the claim is proved.

Immediately, each endpoint of $f(T)$ is either a critical value of $f$ or  the image of an endpoint of $T$. Let $C'$ be an end circle of $f(T)$. By the claim above, there exists a circle $C\subset T$ such that $C'\subset f(C)$. Then $C'$ is also an end circle of $f(C)$. By Lemma \ref{lem:circle-image},  either $f:C\to C'$ is a homeomorphism, or $C'$ contains a critical value, or $C'$ separates a critical value from $V$.

The number of circles $C'$ in the last case is clearly finite since $f$ has a finite number of critical values. The circles $C'$ in the first case must be the images of end circles of $T$, and hence their number is finite. Note that there exist finitely many circles in $T$ containing a pre-image of the critical values of $f$. Then the number of $C'$ in the second case is also finite.
Therefore, $f(T)$ is a finite circle-tree in $\partial V$.
\end{proof}

\subsection{Invariant circle-trees}
Let $(f,P)$ be a marked rational map, and let $U$ be a fixed Fatou domain of $f$. We will construct an $f$-invariant and finite circle-tree of $\partial U$. The process is similar to the construction of the Hubbard tree for   PCF  polynomials \cite{DH2}.

We say a continuum $E$ {\bf separates} $P$ if there exist two points of $P$ in distinct components of $\cbar\setminus E$.
A circle $C\subset\partial U$ is called a {\bf marked circle} (rel $P$) if $C$ either intersects or separates $P$.

\begin{lemma}\label{lem:finite}
Any eventually periodic point in $\partial U$ receives finitely many internal rays in $U$. Consequently, there exist  finitely many marked circles in $\partial U$.
\end{lemma}

\begin{proof}
It suffices to prove the lemma for a fixed point $z\in\partial U$. Let $\Theta\subset\partial\D$ be the set of angles corresponding to the internal rays in $U$ landing at $z$. Then $\Theta$ is compact,  and $p_d:\Theta\to\Theta$ is injective, where $p_d(z)=z^d$ and $d=\deg f|_{ U}$. By \cite[Lemma 18.8]{Mi1}, $\Theta$ is a finite set.

To show the finiteness of marked circles in $\partial U$, it suffices to prove that at most finitely many circles in $\partial U$ pass through an eventually periodic point $z\in \partial U$. According to the previous discussion, $\partial U\sm\{z\}$ has finitely many components, each of which, together with the point $z$, contains at most one circle in $\partial U$ passing through the point $z$. Thus, the lemma is proved.
\end{proof}

 For two continua $E_0\subset E$, we call $E_0$  a {\bf skeleton} of $E$ (rel $P$) if $E_0\cap P=E\cap P$ and any two points of $P$ in distinct components of $\ov{\C}\setminus E$  also lie in distinct components of $\ov{\C}\setminus E_0$.

\begin{theorem}\label{thm:invariant-CT}
Let $T$ be the finite circle-tree of $\partial U$ spanned by $P\cap\partial U$ together with all marked circles in $\partial U$. Then
\begin{enumerate}
\item each end of $T$ is a marked point or a marked circle;
\item $f(T)\subset T$, and  $T$ is a skeleton of $\partial U$ rel $P$.
\end{enumerate}

\end{theorem}

\begin{proof} By Lemmas \ref{lem:span} and \ref{lem:operation}, each endpoint of $T$ is contained in $P\cap\partial U$, and each end circle of $T$ is a marked circle. By Lemma \ref{lem:tree-image}, for each endpoint $y$ of $f(T)$, either $y$ is a critical value, or there exists an endpoint $x$ of $T$ such that $f(x)=y$. In both cases, we have $y\in P\cap\partial U$. For each end circle $C$ of $f(T)$, either $C$ is a marked circle, or $C$ is the image of an end circle of $T$. In the latter case, $C$ is also a marked circle. Therefore, each end of $f(T)$ is contained in $T$. Thus, $f(T)\subset T$ by Lemma \ref{lem:coincide}.

Immediately, $T\cap P=\partial U\cap P$. If two points $a,b\in P$ are contained in distinct components of $\cbar\sm\partial U$, then there exists a unique circle $C\subset\partial U$ separating $a$ from $b$. Thus, $C\subset T$ since $C$ is a marked circle. It follows that $T$ is a skeleton of $\partial U$.
\end{proof}

The invariant circle-tree $T$ obtained in Theorem \ref{thm:invariant-CT} attracts every circle in $\partial U$.

\begin{lemma}\label{lem:eventually}
For any circle $C\subset\partial U$, there exists an integer $n\ge 0$ such that $f^n(C)\subset T$.
\end{lemma}

\begin{proof}
By Lemma \ref{lem:circle-image} and Theorem \ref{thm:invariant-CT}, either $f(C)$ is still a circle in $\partial U$, or $f(C)\subset T$. Thus, it suffices to show that  $f^{N}(C)$ is a marked circle for some integer $N\ge 0$, under the assumption that $f^n(C)$ is always a circle for every $n\ge 0$. Otherwise, let $D_n$ be the disk bounded by $f^n(C)$ and disjoint from $U$ for $n\ge 0$. Then $\ov{D_n}\cap P=\emptyset$. Thus, $f^{n}(D)=D_n$, which implies $D$ is a Fatou domain of $f$. Consequently, there exists an integer $N\ge 0$ such that $f^N(D)$ is a periodic Fatou domain. Then $f^N(C)$ is a marked circle, a contradiction.
\end{proof}

As a by-product, we obtain the following result regarding the locally branched points on the boundaries of Fatou domains. This generalizes a well-known fact for polynomials.

A circle $C\subset T$ is called {\bf regular} if it is neither a marked circle nor a branched circle of $T$. Note that $T$ has only finitely many irregular circles.

\begin{theorem}\label{thm:eventually}
Every locally branched point of $\partial U$ is eventually periodic.
\end{theorem}

\begin{proof}
Let $x$ be any locally branched point of $\partial U$.
We first claim that there exists an integer $N>1$ such that $f^N(x)$ is either a locally branched point of $T$ or a point in $P_f\cap T$.

If $x$ is contained in a circle $C$ of $\partial U$, then there exists a component $E$ of $\partial U\sm C$ such that $\ov{E}\cap C=\{x\}$. Since $\bigcup_{n>0}(f^{-n}(T)\cap \partial U)$ is dense in $\partial U$, there exists a point $y\in E$ such that $f^{n_0}(y)\in T$ for some integer $n_0>0$. Then $x$ is a locally branched point of $T_1=T[y,C]$. By Lemma \ref{lem:eventually}, there exists an integer $N\ge n_0$ such that $f^N(x)\in f^N(C)\subset T$. It follows from Lemma \ref{lem:tree-image} that  $f^N(T_1)$ is a circle-tree whose ends are contained in $T$, and thus $f^N(T_1)\subset T$ by Lemma \ref{lem:span}.  Therefore, the claim holds.

If $x$ avoids any circle in $\partial U$, then $x$ is a branched point of $\partial U$. Thus, $\partial U\sm\{x\}$ has at least three components $E_1$, $E_2$, and $E_3$.
By a similar argument as above,  there exist a point $y_i\in E_i$ and an integer $n_i>0$ for each $i=1,2,3$ such that $f^{n_i}(y_i)\in T$, and the circle-tree $f^N(T_1)$ is contained in $T$ with $T_1:=T[y_1,y_2,y_3]\ni x$ and $N:=\max\{n_1,n_2,n_3\}$. Thus, the claim still holds.

Since $T$ has only finitely many branched points by Lemma \ref{lem:number}, it follows from the above claim that either $x$ is eventually periodic, or $f^n(x)$ is a locally branched point but not a branched point of $T$ for every $n\ge N$. It suffices to consider the latter. In this situation, each $f^{n}(x)$ is a cut point of $T$ and contained in a circle $C_n$ of $T$ for $n\ge N$.

  If $C_{n_i}=C$ for an infinite sequence $\{n_i\}$, then $x$ is eventually periodic since each circle contains finitely many cut points of $T$ by Lemma \ref{lem:number}. Thus, we may further assume that $C_n,n\geq N$ are pairwise different circles of $T$.

Since $T$ has finitely many irregular circles,  the circle $C_n$ is  regular for each sufficiently large integer $n$. For a regular circle $C$, there exists a dichotomy: either $D_{C}$ contains a component of $f^{-1}(U)$, or $f:\ov{D_{C}}\to f(\ov{D_{C}})$ is a homeomorphism, where $D_{C}$ denotes the component of $\cbar\setminus C$ disjoint from $U$. Clearly, there exist finitely many regular circles of the first type in $T$. It follows that
 $C_{n+1}=f(C_n)$ and $D_{C_{n+1}}=f(D_{C_n})$ for every sufficiently large integer $n$.  This implies the existence of wandering Fatou domains, a contradiction.
\end{proof}

\subsection{A Fatou domain without invariant graphs on the boundary}
In this subsection, we give an example of a   PCF  rational map with a fixed Fatou domain $U$, such that $\partial U$  admits no invariant graphs.

Let $X\subset \cbar$ be a compact set. A continuous map $\phi:\cbar\times [0,1]\to\cbar$ is an {\bf isotopy rel $\boldsymbol{X}$} if each map $\phi_s=\phi(\cdot,s)$ is a homeomorphism of $\cbar$ and $\phi_s(z)=z$ for every $z\in X$ and $s\in[0,1]$. In this case, we say the homeomorphisms $\phi_0$ and $\phi_1$ are isotopic rel $X$. Sometimes, we  write the isotopy $\phi$ as $\{\phi_s\}_{s\in [0,1]}$. 

Moreover, we say two subsets $E_1$ and $E_2$ of $\cbar$ are {\bf isotopic rel $\boldsymbol{X}$} if there exists a homeomorphism $h:\cbar\to\cbar$ that is isotopic to the identity map rel $X$ such that $h(E_1)=E_2$. In this paper, $E_1$ and $E_2$  are typically considered  Jordan curves, (open) arcs, or graphs.

\begin{theorem}\label{thm:example}
There exist a cubic   PCF  rational map $f$ and a fixed Fatou domain $U$ of $f$ such that $\partial U$ contains infinitely many circles, and for any arc $\gamma\subset\partial U$, $f^n(\gamma)=\partial U$ for some integer $n\ge 1$. Consequently, there exist no invariant graphs on $\partial U$.
\end{theorem}

Let $g(z)=z^2-2$. Its Julia set is $[-2,2]$. Let $D$ be the disk with diameter $[-2,0]$, and let $B$ be the domain bounded by the three external rays landing at the points $0$ and $-2$. Then there exists a homeomorphism $\varphi$ from $B\sm\ov{D}$ to $B\sm [-2,0]$, and  $\varphi$ can be  continuously extended to the boundary such that $\varphi=id$ on the three external rays and $\varphi(x+\textbf{i}y)=x$ on $\partial D$.

\begin{figure}[http]
\centering
\includegraphics[width=9.5cm]{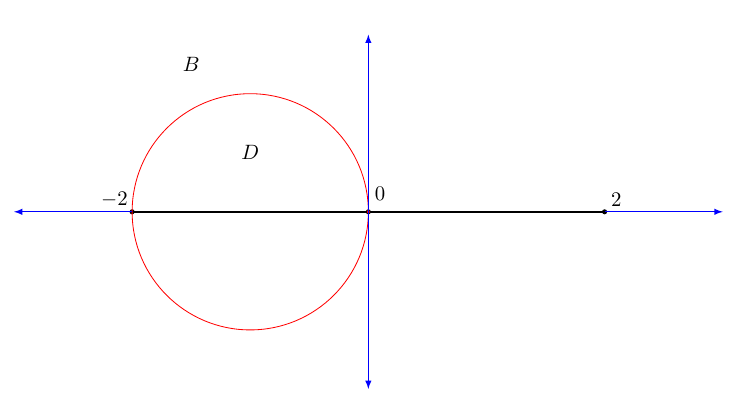}
\caption{The construction of $\tilde f$.}\label{figure:4}
\end{figure}

Let $h: D\to\cbar\sm [-2,2]$ be a homeomorphism such that $h=g\circ\varphi$ on $\partial D$. Define
$$
\tilde f=
\begin{cases}
g & \text{ on }\cbar\sm B, \\
g\circ\varphi & \text{ on }B\sm\ov{D}, \\
h & \text{ on }\ov{D}.
\end{cases}
$$
Then $\tilde f$ is a branched covering of $\cbar$ with $\deg\tilde f=3$. It has three critical points $-2$, $0$, and $\infty$, with $\deg(\tilde f|_{\tilde z=0})=3$ and $\deg(\tilde f|_{\tilde z=-2})=\deg(\tilde f|_{\tilde z=\infty})=2$. Its post-critical set is $P_{\tilde f}=\{-2,2,\infty\}$. Thus, $\tilde f$ is combinatorially equivalent to a rational map $f$ by the Thurston theorem (see \cite{DH1} or \cite{Mc3}). This means there exists a pair of orientation-preserving homeomorphisms $(\phi_0,\phi_1)$ of $\cbar$ such that $\phi_1$ is isotopic to $\phi_0$ rel $P_{\tilde f}$, and $f:=\phi_0\circ\tilde f\circ\phi_1^{-1}$ is a rational map.

Denote the $\phi_0$-image of $-2$, $0$, $2$, and $\infty$ by $a$, $b$, $a_1$, and $c$, respectively. Then  the critical points  of $f$ are $a,b,c$ with $\deg( f|_{z=b})=3$ and $\deg(f|_{z=a})=\deg(f|_{z=c})=2$. Moreover,
\[f(b)=a,\quad f(a)=a_1=f(a_1),\quad  \text{and} \quad f(c)=c.\] Thus, $P_f=\{a, a_1, c\}$. The map $f$ has exactly one periodic Fatou domain $U$ containing $c$. Then $f(U)=U$ and $\deg(f|_U)=2$. Thus, $f^{-1}(U)$ has another component, $U'$, in addition to $U$.

\begin{proposition}\label{prop:concrete}
The lamination $\LLL_{U}$ of $U$ consists of leaves $L_n,n\geq1$ such that the endpoints of $L_n$ are $e^{{\pi\textup{\bf i}}/{2^n}}$ and $e^{-\pi  \textup{\bf i}/{2^n}}$.
\end{proposition}

\begin{proof}
Let $W_0$ be a round disk under the B\"{o}ttcher coordinate of $U$ that is compactly contained in $U$, and let $W_n$ be the component of $f^{-n}(W_0)$ containing the fixed point $c$ for $n\ge 1$. Then $W_n\subset W_{n+1}$, and $\bigcup_{n\ge 0}W_n=U$.

Denote by $R_f(\theta)$ the internal ray of $f$ in $U$ with angle $\theta\in(-\pi,\pi]$, and by $R_g(\theta)$ the external ray of $g$ with angle $\theta\in(-\pi,\pi]$. Then
$\tilde f(R_g(0))=R_g(0)$. We may assume that $\phi_0(R_g(0))$ coincides with $R_f(0)$ in $W_0$. Then $\phi_1(R_g(0))$ coincides with $R_f(0)$ in $W_1$ since $f(\phi_{1}(R_g(0))=\phi_0(R_g(0))$. Thus, there exists an isotopy $\{\phi_s\}_{s\in[0,1]}$ rel $P_{\tilde f}$  such that $\phi_1=\phi_0$ on $R_g(0)\cap W_0$.

Lifting the isotopy $\{\phi_s\}_{s\in[0,1]}$ inductively by Lemma \ref{lem:lift}, we get a sequence of homeomorphisms $\{\phi_n\}$ of $\cbar$ such that $\phi_{n+1}$ is isotopic to $\phi_n$ rel $P_{\tilde f}$ and $f\circ\phi_{n+1}=\phi_n\circ\tilde f$. Thus, $\phi_{n+1}(R_g(0))$ coincides with $R_f(0)$ in $W_{n+1}$, and $f(\phi_{n+1}(R_g(0))=\phi_n(R_g(0))$. By Lemma \ref{lem:expanding}, $\phi_n(R_g(0))$ converges to $R_f(0)$. Thus, $R_f(0)$ lands at the point $a_1$.

Since $f^{-1}(a_1)=\{a,a_1\}$, the ray $R_f(\pi)$ lands at the point $a$, and $f^{-1}(R_f(0))$ has a component in $U'$ that joins the point $a$ and the  unique point $c'$ of $f^{-1}(c)$ in $U'$. Since $f^{-1}(a)=b$, both $R_f(\pm {\pi}/{2})$ land at the point $b$, and  a component of $f^{-1}(R_f(\pi))$ in $U'$ connects $c'$ and the critical point $b$. Consequently, $a,b\in\partial U\cap\partial U'$. It follows that $R_f(\theta_1)$ and $R_f(\theta_2)$ land at distinct points if $\theta_1\in(\pi/2,\pi)$ and $\theta_2\in(-\pi, -\pi/2)$.

Consider the simply connected domain bounded by $R_f(\pi)$ and $R_f(\pm{\pi}/{2})$. It contains no critical values of $f$. Thus, its pre-image has three components, one of which is bounded by $R_f(\pm{\pi}/{2})$ and $R_f(\pm{\pi}/{4})$. Thus, $R_f(\pm{\pi}/{4})$ land at the same point. Moreover, $R_f(\theta_1)$ and $R_f(\theta_2)$ land at distinct points if $\theta_1\in(\pi/4,\pi/2)$ and $\theta_2\in(-\pi/2, -\pi/4)$.

Inductively taking pre-images as above, the rays $R_f(\pm{\pi}/{2^n})$ land at the same point, but $R_f(\theta_1)$ and $R_f(\theta_2)$ land at distinct points if $\theta_1\in(\pi/2^n,\pi/2^{n-1})$ and $\theta_2\in(-\pi/2^{n-1}, -\pi/2^n)$ for $n\ge 2$.

 Now, we have proved that $L_n$ is a leaf of $\LLL_{U}$, and there exists no leaf that joins $e^{\textup{\bf i}\theta_1}$ to $e^{ \textup{\bf i}\theta_2}$ if $\theta_1\in(\pi/2^n,\pi/2^{n-1})$ and $\theta_2\in(-\pi/2^{n-1}, -\pi/2^n)$ for $n\ge 1$. It follows that if $L$ is a leaf of $\LLL_{U}$ that joins $e^{ \textup{\bf i}\theta_1}$ to $e^{ \textup{\bf i}\theta_2}$, then $|\theta_1-\theta_2|<\pi/2$.

Assume that $L$ is a leaf of $\LLL_{U}$ that joins $e^{\textup{\bf i}\theta_1}$ to $e^{\textup{\bf i}\theta_2}$. Then there exists a leaf of $\LLL_{U}$ that joins $e^{2^n  \textup{\bf i}\theta_1}$ to $e^{2^n  \textup{\bf i}\theta_2}$ for $n\ge 1$, except when $2^n(\theta_1-\theta_2)\equiv 0~\textup{mod}~{2\pi}$. In particular, there exists an integer $n\ge 1$ such that $\pi/2<2^n|\theta_1-\theta_2|\le\pi$. This is a contradiction.
\end{proof}

\begin{proof}[Proof of Theorem \ref{thm:example}]
Denote by $\phi:\D\to U$ the inverse of the B\"{o}ttcher coordinate for $U$. It can be extended continuously to the boundary. For any arc $\gamma\subset\partial U$,  Proposition \ref{prop:concrete} implies that $\phi^{-1}(\gamma)$ must contain a non-trivial interval. Thus, $f^n(\gamma)=\partial U$ for some integer $n\ge 1$.
\end{proof}

Up to conformal conjugacy, the rational map $f$ constructed above has the form
$$
f(z)=(z^2-6z+9-8/z)/3
$$
with the critical points $-1,2,$ and $\infty$; see Figure \ref{fig:dia} for its Julia set.

\begin{figure}[http]
	\centering
	\includegraphics[width=12cm]{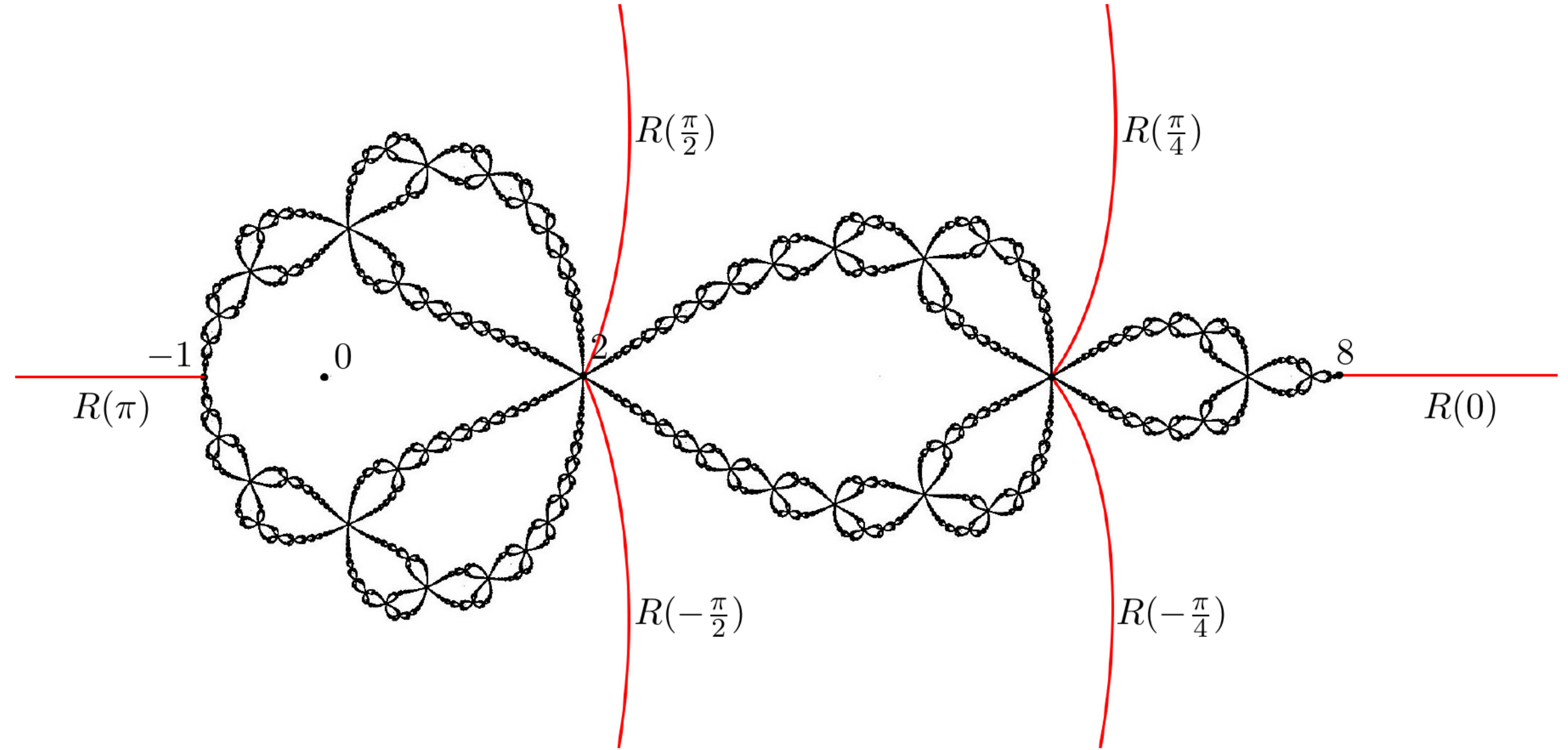}
\caption{The Julia set of $f(z)=(z^2-6z+9-8/z)/3$.}\label{fig:dia}
\end{figure}

\subsection{Proof of Theorem \ref{thm:local}}
Let $(f,P)$ be a marked rational map, and let $U$ be a fixed Fatou domain of $f$. Let $T\subset\partial U$ be the $f$-invariant circle-tree obtained in Theorem \ref{thm:invariant-CT}. Our proof strategy is as follows. First, we will find a graph $G_1$ serving as a skeleton of $T$ rel $P$ such that $f^{-1}(G_1)$ contains a graph $G_2$ that is isotopic to $G_1$ rel $P$. Then, by lifting, we obtain a sequence of graphs $\{G_n\}$, and finally, we will prove that $\{G_n\}$ converges to an invariant graph $G$. 

Let $X_0\subset T$ be the union of $P$ together with the set of cut points of $T$. Then $X_0$ is a compact set containing all endpoints of $T$, and $f(X_0)\subset X_0$. Each component of $T\sm X_0$ is an open arc in a circle of $T$. Denote $X_n:=f^{-n}(X_0)$ for $n\ge 0$. Then $X_n\subset X_{n+1}$.

Recall that a circle $C\subset T$ is regular if it is neither a marked circle nor a branched circle of $T$. Thus, each regular circle $C$ contains exactly two points of $X_0$, which cut $C$ into two open arcs $C^{+}$ and $C^-$.  Set
$$
G_1:=T\sm\bigcup_{C}C^-,
$$
where $C$ ranges over all regular circles in $T$. Then $G_1$ is a graph  since there exist finitely many irregular circles in $T$, and $G_1$ is a skeleton of $\partial U$ by Theorem \ref{thm:invariant-CT}.

To construct $G_2\subset f^{-1}(G_1)$, we need to go beyond $\partial U$. Let $\alpha_1$ be a component of $G_1\sm X_1$. Its image $f(\alpha_1)$ is a component of $T\sm X_0$. Thus, there exists a circle $C\subset T$ such that $f(\alpha_1)$ is a component of $C\sm X_0$. If $C$ is irregular, then $f(\alpha_1)\subset C\subset G_1$. If $C$ is regular, then $f(\alpha_1)$ equals either $C^+$ or $C^-$.
\begin{itemize}
\item If $f(\alpha_1)=C^+$, we still have $f(\alpha_1)\subset G_1$. 
\item If $f(\alpha_1)=C^-$, since $C^+$ and $C^-$ are isotopic rel $X_0$, there exists  a unique component $\alpha_1^+$ of $f^{-1}(C^+)$  isotopic to $\alpha_1$ rel $X_1$. Let $B(\alpha_1)$ denote the closed disk bounded by $\alpha_1$ and $\alpha_1^+$ disjoint from $U$. Then $B(\alpha_1)\cap G_1=\ov{\alpha_1}$ and $B(\alpha_1)\cap X_1=\{\alpha_1(0),\alpha_1(1)\}$. Such a component $\alpha_1$ of $G_1\setminus X_1$ is called a {\bf deformation arc} of $G_1$. 
\end{itemize}

Define the graph $G_2$ as 
$$
G_2:=\left(G_1\sm\bigcup\alpha_1\right)\cup\bigcup\alpha^+_1,
$$
where the union is taken over all deformation arcs of $G_1$. From the previous discussion, we have $f(G_2)\subset G_1$, and
 there exists an isotopy $\Theta^1:\cbar\times [0,1]\to\cbar$ rel $P$ such that $\Theta^1_t:=\Theta^1(\cdot,t)$ satisfies
\begin{enumerate}
\item $\Theta^1_0=id$ on $\cbar$;

\item $\Theta^1_t(z)=z$ on a neighborhood of the attracting cycles of $f$ for $t\in [0,1]$;

\item if $z\in G_1$ is not in any deformation arc, then $\Theta^1_t(z)=z$ for $t\in [0,1]$; and

\item if $\alpha_1$ is a deformation arc of $G_1$, then $\Theta^1_1(\alpha_1)=\alpha_1^+$, and $\Theta^1(\ov{\alpha_1}\times[0,1])=B(\alpha_1)$. 
\end{enumerate}

\noindent Consequently, we have $\theta_1(G_1)=G_2$ with $\theta_1:=\Theta^1_1$.

By inductively applying Lemma \ref{lem:lift}, we obtain an isotopy $\Theta^n:\cbar\times [0,1]\to\cbar$ rel $P$ and a graph $G_{n+1}$ for each $n\geq1$, such that $\Theta^n_0=id$ and  $\Theta^{n}_t\circ f(z) =f\circ \Theta^{n+1}_t(z)$  for all $z\in\ov\C$ and $t\in [0,1]$, and that $G_{n+1}=\theta_n(G_n)$ with $\theta_n:=\Theta^{n}_1$.
Thus, $f(G_{n+1})\subset G_{n}$. In addition,  there exist some components of $G_n\setminus X_n$, called the {\bf deformation arcs} of $G_n$ (under $\Theta^n$), such that
\begin{itemize}
\item [(a)] if $z\in G_n$ is not in any deformation arc of $G_n$, then $\Theta^n_t(z)=z$ for $t\in [0,1]$;
\item [(b)] if $\alpha_n$ is a deformation arc of $G_n$, then 
 the deformation  of $\ov{\alpha_n}$ under $\Theta^n$, denoted by $B(\alpha_n)$, is a closed disk satisfying $B(\alpha_n)\cap G_n=\ov{\alpha_n}$ and $B(\alpha_n)\cap X_n=\{\alpha_n(0),\alpha_n(1)\}$.
 \end{itemize}

\begin{figure}[http]
\centering
\includegraphics[width=11cm]{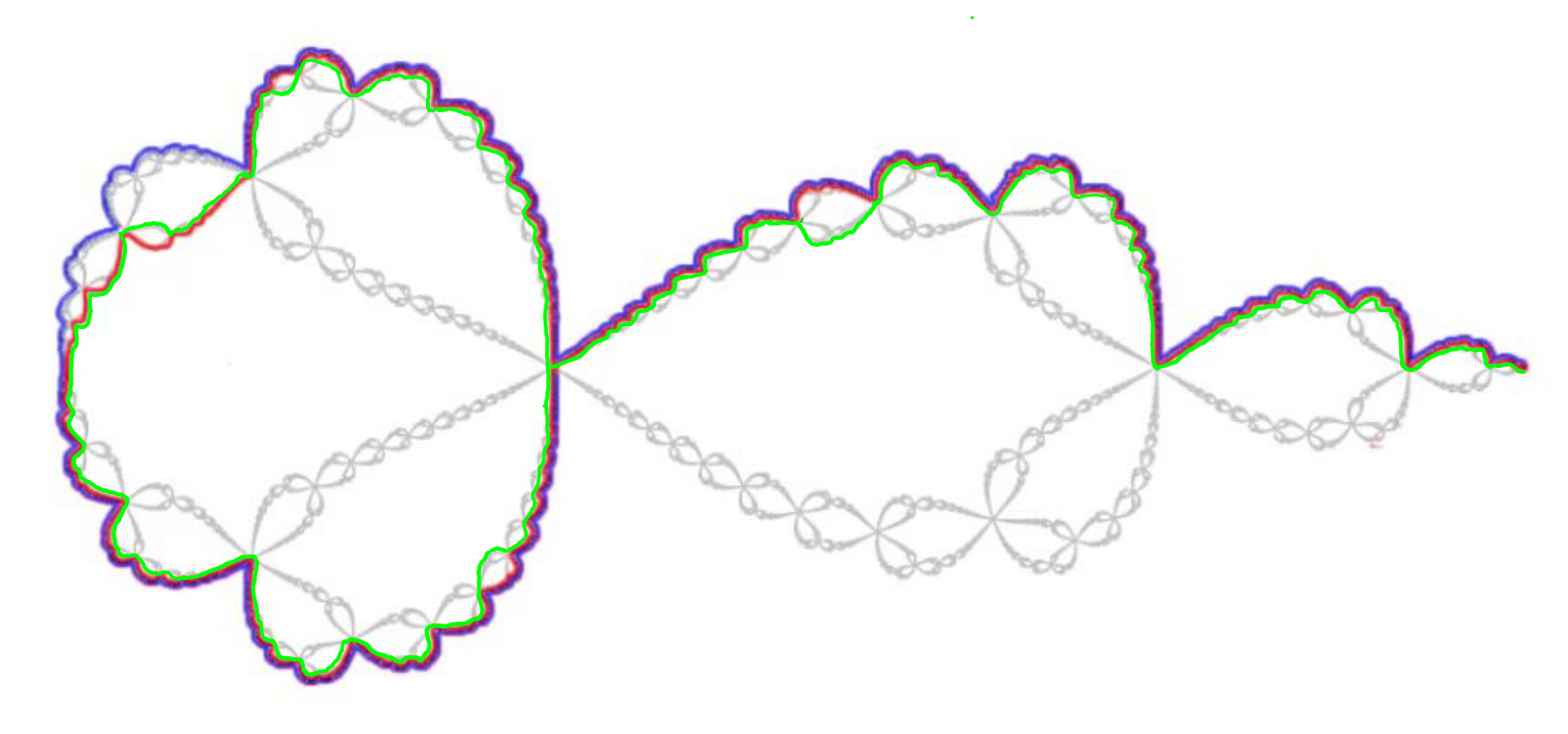}
\caption{The images of $\{G_n\}$.}
\end{figure}

Denote $\phi_n=\theta_{n-1}\circ\cdots\circ\theta_0$ for $n\ge 1$ with $\theta_0:=id$. Then $G_{n}=\phi_n(G_1)$. By Lemma \ref{thm:isotopy}, $\{\phi_n\}$ uniformly converges to a quotient map $\varphi$ of $\cbar$. Consequently, $f(G)\subset G$, where $G$ is defined as  $G:=\varphi(G_1)$.
In order to show that $G$ is a graph, we need to clarify the relation between the deformation arcs of $G_m$ and $G_n$ for $m>n\ge 1$.

Fix a deformation arc $\alpha_n$ of $G_n$ with $n\geq1$. Set $\alpha_{n-k}:=f^k(\alpha_n)$ for $0\leq k\leq n$. From the lifting construction of $\Theta^n$, it follows that, for $0\leq k\leq n-1$, $\alpha_{n-k}$ is a deformation arc of $G_{n-k}$ and  $f^k(B(\alpha_n))=B(\alpha_{n-k})$, and that $\alpha_0=C^-$ for a regular circle $C$ of $T$ and
 $f^n:B(\alpha_n)\to B(\alpha_0)$ is a homeomorphism. Here, $B(\alpha_0)=B(C^-)$ refers to the closure of the component of $\overline{\mathbb{C}}\sm C$ disjoint from $U$.

\begin{proposition}\label{prop:pearl}
Let $\alpha_m$ and $\beta_n$ be distinct deformation arcs of $G_m$ and $G_n$, respectively,  with $m\geq n\geq1$. Then either $B(\alpha_m)\subset B(\beta_n)$, or $B(\alpha_m)\cap B(\beta_n)=\emptyset$, or $B(\alpha_m)$ intersects $B(\beta_n)$ at a single point of $X_n$.
\end{proposition}

\begin{proof}
Set $\beta_0:=f^n(\beta_n)$ and $\alpha_{m-n}:=f^n(\alpha_m)$. By definition, $B(\beta_0)$ is the closure of a component of $\cbar\setminus \ov{U}$, and the interior of $B(\alpha_{m-n})$ is contained in a component $D$ of $\ov{\C}\setminus \ov{U}$. Then by Lemma \ref{lem:topology},  either $\ov{D}=B(\beta_0)$, or $\ov{D}\cap B(\beta_0)=\emptyset$, or $\ov{D}\cap B(\beta_0)$ is a singleton in $X_0$. It follows that either $B(\alpha_{m-n})\subset B(\beta_0)$, or $B(\alpha_{m-n})\cap B(\beta_0)=\emptyset$, or $B(\alpha_{n-m})$ intersects $B(\beta_0)$ at a single point of $X_0$.  Thus, this proposition can be proved by a pullback argument.
\end{proof}

\begin{proposition}\label{prop:hn}
Let $m>n\geq1$ be integers, and let $\alpha_n$ be any deformation arc of $G_n$.
\begin{enumerate}
\item Let $x\in G_1$ be a point such that $\phi_n(x)\in \alpha_n$. Then $\phi_{m}(x)\in B(\alpha_n)$. Consequently, if $\phi_m(x)$ is contained in a deformation arc $\alpha_m$ of $G_m$,  then $B(\alpha_{m})\subset B(\alpha_{n})$.
\item Let $\alpha\subset G_1$ be an open arc such that $\phi_n(\alpha)=\alpha_n$. Then $G_{m}\cap B(\alpha_{n})=\phi_{m}(\ov\alpha)$.
\end{enumerate}

\end{proposition}

\begin{proof}
(1) Let $n=n_1<\cdots<n_s< n_{s+1}:=m$ be all integers such that $\phi_{n_i}(x)$ belongs to a deformation arc $\alpha_{n_i}$ of $G_{n_i}$ for $i=1,\ldots, s$. For each $i\in\{1,\ldots,s\}$ and any $n_i< k\leq n_{i+1}$, it follows from the definition of $\phi_n$ and properties (a) and (b) of $\Theta^n$ that
$$\phi_k(x)=\theta_{k-1}\circ\cdots\circ\theta_{n_i}\circ\phi_{n_i}(x)=\theta_{n_i}\circ\phi_{n_i}(x)\in\theta_{n_i}(\alpha_{n_i})\subset B(\alpha_{n_i}).$$
Thus, $\phi_{n_{i+1}}(x)\in B(\alpha_{n_i})$, and furthermore $\phi_{n_{i+1}}(x)\in B(\alpha_{n_i})\cap B(\alpha_{n_{i+1}})$ for $i\in\{1,\ldots,s-1\}$.
This implies $B(\alpha_{n_{i+1}})\subset B(\alpha_{n_i})$  for $i\in\{1,\ldots,s-1\}$  by Proposition \ref{prop:pearl}, since
$\phi_{n_{i+1}}(x)\in \alpha_{n_{i+1}}$, which is disjoint from $X_{n_{i+1}}$.  Therefore, $\phi_m(x)=\phi_{n_{s+1}}(x)\in B(\alpha_{n_s})\subset\cdots\subset B(\alpha_{n})$.

\vspace{2pt}

(2) By statement (1), we  immediately get that $\phi_{m}(\ov\alpha)\subset B(\alpha_n)$. Therefore, to prove $\phi_m(\ov{\alpha})=G_m\cap B(\alpha_n)$, it suffices to show that $\phi_m(z)\not\in B(\alpha_n)$ for any $z\in G_1\setminus \ov\alpha$.

 First, note that $\phi_{n}(z)\notin B(\alpha_n)$ since $G_{n}\cap B(\alpha_n)=\ov\alpha_n$. If $\phi_{k}(z)$ does not belong to any deformation arc of $G_k$ for every $n\leq k<m$, then $$\phi_{m}(z)=\phi_{m-1}(z)=\cdots=\phi_{n}(z)\notin B(\alpha_n).$$ 
 Otherwise, let $n_1\in[n,m)$ be the smallest integer such that $\phi_{n_1}(z)$ belongs to a deformation arc $\alpha_{n_1}$ of $G_{n_1}$. Then $\phi_{n}(z),\phi_m(z)\in B(\alpha_{n_1})$ by statement (1). Since $\phi_n(z)\not\in B(\alpha_n)$, it follows from Proposition \ref{prop:pearl} that $B(\alpha_n)\cap B(\alpha_{n_1})$ is either empty or a singleton in $X_n$. Note also that $B(\alpha_n)\cap X_n=\{\alpha_n(0),\alpha_n(1)\}$ by property (b) above. Thus, $\phi_m(z)\not\in B(\alpha_n)$.
\end{proof}

\begin{corollary}\label{cor:C}
For each point $z\in G\sm\partial U$, there exist an integer $n\ge 1$ and a component $D$ of $\cbar\sm\ov{U}$, such that $f^n(z)\in\ov{D}$ and $\partial D$ is a regular circle of $\partial U$.
\end{corollary}

\begin{proof}
Let  $x\in G_1$ be a point such that $\varphi(x)=z$. Since $z\not\in \partial U$, there exists a smallest integer $n_0\ge 1$ such that $\phi_{n_0}(x)$ belongs to a deformation arc $\alpha_{n_0}$ of $G_{n_0}$. It then follows from Proposition \ref{prop:hn}\,(1) that $z=\varphi(x)\in B(\alpha_{n_0})$. By the discussion before Proposition \ref{prop:pearl}, $f^{n_0}(B(\alpha_{n_0}))$ is the closure of a component of $\ov{\C}\setminus \ov{U}$ bounded by a regular circle of $\partial U$.
\end{proof}

The following result is a key part of the proof of Theorem \ref{thm:local}.

\begin{proposition}\label{prop:convergence}
For any two distinct points $x,y\in G_1$ with $\varphi(x)=\varphi(y)$, there exists an arc $\beta\subset G_1$ connecting $x$ and $y$ such that $\varphi(\beta)=\varphi(x)$.
\end{proposition}
\begin{proof}
A point $z\in G_1$ is called \emph{finitely deforming} (\emph{under $\{\phi_n\}$}) if there exists an integer $n(z)\geq1$ such that $\phi_n(z)$ does not belong to any deformation arc of $G_n$ for every $n\geq n(z)$.
Thus, if $z\in G_1$ is \emph{infinitely deforming}, we can find an increasing sequence $\{n_i\}_{i\ge 1}$ such that $\phi_{n_i}(z)$ belongs to a deformation arc $\alpha_{n_i}$ of $G_{n_i}$ for all $i\ge 1$. In this case, it holds that $B(\alpha_{n_{i+1}})\subset B(\alpha_{n_i})$ by Proposition \ref{prop:hn}\,(1). According to Lemma \ref{lem:orbifold}, the homotopic diameters of $B(C)$ for all regular circles $C$ of $T$ are bounded above. Thus, Lemma \ref{lem:expanding} implies  $\bigcap_{i\ge 1}B(\alpha_{n_i})=\{\varphi(z)\}$.

Since $\varphi(x)=\varphi(y)$, at least one of $\{x, y\}$, say $x$, is infinitely deforming. As above, there exist an  increasing sequence $\{n_i\}_{i\ge 1}$ and a deformation arc $\alpha_{n_i}$ of $G_{n_i}$ for each $i\geq1$, such that $\phi_{n_i}(x)\in\alpha_{n_i}$, $B(\alpha_{n_{i+1}})\subset B(\alpha_{n_i})$, and $\bigcap_{i\ge 1}B(\alpha_{n_i})=\{\varphi(x)\}$.
\vskip 0.15cm

Case 1. The point $y$ is finitely deforming. 
\vskip 0.15cm

In this case, we have $\varphi(y)=\phi_n(y)$ for every $n\ge n(y)$. Since $\varphi(x)=\varphi(y)$, it follows that $\phi_{n_i}(y)=\varphi(y)\in B(\alpha_{n_i})$ for $n_i>n(y)$. Then $\phi_{n_i}(y)\in G_{n_i}\cap B(\alpha_{n_i})=\ov{\alpha_{n_i}}$. Hence, $\phi_{n_i}(y)$ is an endpoint of $\alpha_{n_i}$.

Let $\g_{i}$ be the sub-arc of $\alpha_{n_i}$ connecting $\phi_{n_i}(y)$ and $\phi_{n_i}(x)$. Then $\beta_{i}:=\phi_{n_i}^{-1}(\g_{i})$ is an arc in $G_1$ connecting $x$ and $y$. Since there exist only finitely many distinct arcs in $G_1$ connecting $x$ and $y$, by passing to a subsequence of $\{i\}$, we have $\beta=\beta_i$ and $\phi_{n_i}(\beta)=\gamma_i\subset \alpha_{n_i}$ for every $i\geq1$. This implies  $\varphi(\beta)=\varphi(x)$.
\vskip 0.15cm

Case 2. The point $y$ is infinitely deforming.
\vskip 0.15cm

 In this case, we obtain another  increasing sequence $\{m_j\}_{j\ge 1}$ and a deformation arc $\delta_{m_j}$ of $G_{m_j}$ for each $j$, such that $\phi_{m_j}(y)\in\delta_{m_j}$, $B(\delta_{m_{j+1}})\subset B(\delta_{m_j})$, and $\{\varphi(y)\}=\bigcap_{j\ge 1} B(\delta_{m_j})$.  Since $\varphi(x)=\varphi(y)$, it follows from Proposition \ref{prop:pearl} and Proposition \ref{prop:hn}\,(1) that, if $m_j\ge n_i$,  either $B(\delta_{m_j})\subset B(\alpha_{n_i})$, or $B(\delta_{m_j})$ intersects $B(\alpha_{n_i})$ at a single point in $X_{n_i}$.
\vskip 0.15cm

Case 2.1. There exist $m_j\geq n_i$ such that $B(\delta_{m_j})\cap B(\alpha_{n_i})$ is a singleton $w\in X_{n_i}$. \vskip 0.15cm

Since $\phi_n(x)\in B(\alpha_{n_i})$ and $\phi_n(y)\in B(\de_{m_j})$ for each sufficiently large integer $n$ by Proposition \ref{prop:hn} (1), it follows that $\varphi(x)\in B(\alpha_{n_i})$ and $\varphi(y)\in B(\delta_{m_j})$. Thus, $\varphi(x)=\varphi(y)=w$. Assume $\phi_{n_i}(z)=w$. Then $\varphi(z)=\phi_n(z)=w$ for every $n\ge n_i$. By applying Case 1 to $\{x,z\}$ and $\{z,y\}$, respectively, we obtain the required arc $\beta$.
\vskip 0.15cm

Case 2.2. For each pair $m_j\geq n_i$, it holds that $B(\delta_{m_j})\subset B(\alpha_{n_i})$.\vskip 0.15cm

Let $\g_i:=\phi_{n_i}^{-1}(\alpha_i)\subset G_1$ be the arc containing $x$. For any pair $m_j\geq n_i$, by Proposition \ref{prop:hn} (2), we have $\phi_{m_j}(x)\in \phi_{m_j}({\overline{\g_i}})= G_{m_j}\cap B(\alpha_{n_i})$. Note also that  $\phi_{m_j}(y)\in G_{m_j}\cap B(\de_{m_j})\subset G_{m_j}\cap B(\alpha_{n_i})$. Thus, $\phi_{m_j}(x), \phi_{m_j}(y)\in \phi_{m_j}({\overline{\g_i}})$. This implies that there exists a sub-arc $\beta_i\subset \g_i$ joining $x$ to $y$ such that $\phi_{m_j}(\beta_i)\subset B(\alpha_{n_i})$.
Since there exist finitely many arcs in $G_1$ joining $x$ to $y$, by passing to a subsequence, we may assume that $\beta_i=\beta$ for all $i\geq1$. Then $\varphi(\beta)=\lim_{j\to\infty}\phi_{m_j}(\beta)$ coincides with $\bigcap_{i\geq 1} B(\alpha_{n_i})=\{\varphi(x)\}$.
\end{proof}

\begin{proof}[Proof of Theorem \ref{thm:local}]
Clearly, $G=\lim_{n\to\infty} G_n=\varphi(G_1)$ is an $f$-invariant continuum. Note that $G_{n+1}$ lies in the component $E_n$ of $f^{-n}(\ov{U})$ containing $\ov{U}$. Then $G\subset K_{U}=\ov{\bigcup_{n\geq1} E_n}$.

 We claim that $\varphi(\alpha)$ is not a singleton for any component $\alpha$ of $G_1\setminus X_1$. If $\alpha$ has two distinct endpoints, then the claim is immediate since $\varphi=id$ on $X_1\cap G_1$. In the remaining case, $\ov\alpha$ is a circle in $G_1$. If $\phi_n(x)$ does not belong to the deforming arcs of $G_n$ for any $x\in\alpha$ and every $n\geq 1$, we have $\varphi(\ov\alpha)=\ov\alpha$, and the claim holds.  Otherwise, there exist a point $x\in\alpha$ and a smallest integer $n_0\geq 1$ such that $\phi_{n_0}(x)$ belongs to a deformation arc $\alpha_{n_0}$ of $G_{n_0}$. Then $\alpha_{n_0}\subset \alpha$, and $\phi_{n_0}=id$ on $\alpha_{n_0}$. This implies that $\varphi=id$ on the two endpoints of $\alpha_{n_0}$. Thus, the claim is proved.

Since $\varphi$ is the identity on $X_1\cap G_1$, which divides $G_1$ into open arcs, by Proposition \ref{prop:convergence} and the claim above, the pre-image of each point of $G$ under $\varphi|_{G_1}$ is either a singleton or an arc in $G_1$. This implies that $G$ is a graph homeomorphic to $G_1$.

Finally, to prove that $G$ is isotopic to $G_1$ rel $P$, it suffices to show $G_1\cap P=G\cap P$, as $G_n$ is isotopic to $G_1$ rel $P$ for every $n\geq1$. Since $\varphi$ is the identity on $P$ and $G_1\cap P=\partial U\cap P$, it follows that
$G_1\cap P=\partial U\cap P\subset G\cap \partial U\cap P\subset \partial U\cap P=G_1\cap P$. On the other hand, we have $(G\setminus \partial U)\cap P=\emptyset$ by Corollary \ref{cor:C}. Thus, $G_1\cap P=G\cap P$.
\end{proof}

\section{Fatou chains}\label{sec:chain}
In this section, we  establish some basic properties of Fatou chains and prove Theorem \ref{thm:maximal}.

Throughout this section, let $f$ be a rational map with $J_f\not=\cbar$. Recall that a level-$0$ Fatou chain of $f$ is the closure of a Fatou domain of $f$. By induction, define a continuum  $K\subset\cbar$ as a level-$(n+1)$ Fatou chain of $f$ if there exists a sequence $\{E_k\}_{k\ge 0}$ of continua, each composed of finitely many level-$n$ Fatou chains, such that
$$
E_k\subset E_{k+1}\quad\text{and}\quad K=\ov{\bigcup_{k\ge 0}E_k}.
$$

\begin{definition}\label{def:extremal}
 A level-$n$ $(n\ge 0)$ Fatou chain $K$ is called a level-$n$ {\bf extremal (Fatou) chain}  if any level-$n$ Fatou chain that intersects $K$ at a point in $F_f$ is contained in $K$.
\end{definition}

By definition, each level-$0$ extremal  chain is the closure of a Fatou domain.

\begin{lemma}\label{lem:E-chain}
For every $n>0$ and any Fatou domain $U$ of $f$, there exists a unique level-$n$ extremal chain $K$ containing $U$. Moreover, there exists a sequence $\{E_k\}$ of continua, each of which is the union of finitely many level-$(n-1)$ extremal chains,  such that $E_k\subset E_{k+1}$ and $K=\ov{\bigcup_{k\geq0} E_k}$.
\end{lemma}

\begin{proof}
We first prove the lemma in the case of $n = 1$.

Let $\Sigma(U)$ denote the collection of Fatou domains $U'$ for which both $U$ and $U'$ are contained in a continuum $E(U, U')$ consisting of finitely many level-$0$ chains.

Enumerate the elements of $\Sigma(U)$ by $U_i$, $i \geq 0$, and fix $E(U, U_i)$ for each $i$. For every $k \geq 0$, define
\[
E_k = \bigcup_{0 \leq i \leq k} E(U, U_i) \quad \text{and} \quad K = \ov{\bigcup_{k \geq 0} E_k}.
\]
Then $K$ is a level-$1$ Fatou chain by definition. It remains to verify that $K$ is extremal.

Now, consider any other level-$1$ Fatou chain $K'$ such that $(K' \cap K) \cap F_f \neq \emptyset$. Then $K' \cap K$ contains a Fatou domain $V$. By definition, assume $K' = \ov{\bigcup_{k \geq 0} E'_m}$, where $E'_m$ is the union of a finite number of level-$0$ Fatou chains, and $E'_m \subset E'_{m+1}$ for every $m \geq 0$.

Since $V \subset K'$, it follows that $V \subset E'_m$ for any sufficiently large  integer $m$. Similarly, we have $V \in \Sigma(U)$.  
Hence, each level-$0$ Fatou chain in $E'_m$ is contained in $\Sigma(U)$.  
By the construction of $E_k$, we obtain $E'_m \subset E_k$ for a sufficiently large integer $k$. This implies $K' \subset K$. Therefore, $K$ is a level-$1$ extremal chain.

Assume that the lemma holds for some $n \geq 1$. Then there exists a unique level-$n$ extremal chain $\sigma$ containing $U$. Similarly, as in the case of $n = 1$, let $\Sigma(\sigma)$ be the collection of all level-$n$ extremal chains $\sigma'$ for which both $\sigma$ and $\sigma'$ are contained in a continuum $E(\sigma, \sigma')$ consisting of finitely many level-$n$ extremal chains.

Note that $\Sigma(\sigma)$ is a finite or countable collection. Thus, $\Sigma(\sigma) = \{\sigma_i\}_{i \geq 0}$. Fix $E(\sigma, \sigma_i)$ for each $\sigma_i$. For every $k \geq 0$, define
\[
E_k = \bigcup_{0 \leq i \leq k} E(\sigma, \sigma_i) \quad \text{and} \quad K = \ov{\bigcup_{k \geq 0} E_k}.
\]
By definition, $K$ is a level-$(n+1)$ Fatou chain.  
Finally, applying a similar argument as in the case of $n = 1$, we can show that $K$ is an extremal chain of level-$(n+1)$.
\end{proof}

Here are some examples of extremal chains.
For a polynomial, the entire Riemann sphere $\cbar$ is its level-$1$ extremal chain.
On the other hand, any level-$n$ extremal chain ($n\geq0$) of a \Sie\ rational map is the closure of a Fatou domain.

If $f$ is a Newton map, the union of the attracting basins for all attracting fixed points is contained in a level-$1$ extremal chain of $f$. This chain contains $J_f$. Thus, $\cbar$ is a level-$2$ extremal  chain of $f$.

\begin{lemma}\label{lem:chain-map}
Let $K\subset\cbar$ be a level-$n$ extremal chain $(n\geq0)$ of $f$. Then 
\begin{enumerate}
\item $f(K)$ is also a level-$n$ extremal chain; and
\item $f^{-1}(K)$ has a unique decomposition $f^{-1}(K)=\bigcup_{i=1}^m K_i$ such that each $K_i$ is a level-$n$ extremal chain with $f(K_i)=K$.
\end{enumerate}
Moreover, $\deg(f|_{K_i}):=\#(f^{-1}(w)\cap K_i)$ is constant if $w\in K\cap F_f$ is not a critical value.
\end{lemma}

\begin{proof}
 If $n=0$, the lemma holds since any level-$0$ extremal chain is the closure of a Fatou domain.

Suppose that the lemma holds for level-$n$ extremal chains with $n\geq0$. Let $K$ be a level-$(n+1)$ extremal chain. By Lemma \ref{lem:E-chain}, there exists a sequence of continua $\{E_k\}$ such that each $E_k$  consists of finitely many level-$n$ extremal chains, $E_k\subset E_{k+1}$, and $K=\ov{\bigcup_{k\geq0} E_k}$.
\vspace{2pt}

(1) By induction, each $f(E_k)$ consists of finitely many level-$n$ extremal chains. Then $f(K)=\ov{\bigcup_{k\geq0}f(E_k)}$ is a level-$(n+1)$ Fatou chain and is contained in a  level-$(n+1)$ extremal chain, denoted by $K'$. Lemma \ref{lem:E-chain} implies $$K'=\ov{\bigcup_{j\ge 0} E'_j},$$ where each $E'_j$ consists of finitely many level-$n$ extremal chains and $E'_j\subset E'_{j+1}$.
Thus, there exists an integer $j_0\ge 0$ such that $f(E_0)\subset E'_j$ for $j\ge j_0$.

Let $ E_j''$ be the component of $f^{-1}(E'_j)$ containing $E_0$. By induction, the continuum $ E_j''$ consists of finitely many level-$n$ extremal chains and thus forms a level-$(n+1)$ Fatou chain.  Since $K$ is extremal, we have ${E}_j''\subset K$. Consequently, $E'_j=f(E_j'')\subset f(K)$ for all $j\ge j_0$. It follows that $f(K)=K'$ is a level-$(n+1)$ extremal chain.\vspace{2pt}

(2) Let $m(k)$ denote the number of components of $f^{-1}(E_k)$. Then $m(k)$ is decreasing. Thus, there exists an integer $k_0 \ge 0$ such that $m(k) = m$ is constant for $k \ge k_0$. Let $E_{i,k}$, $1 \le i \le m$, be the components of $f^{-1}(E_k)$ such that $E_{i,k} \subset E_{i,k+1}$. It follows that $d_i := \deg(f|_{E_{i,k}})$ is constant for $k \ge k_0$.

Set $K_i := \ov{\bigcup_{k \ge k_0} E_{i,k}}$. Then $f^{-1}(K) = \bigcup_{i=1}^m K_i$, and $f(K_i) = K$. By induction, each $E_{i,k}$ is the union of finitely many level-$n$ extremal chains, so $K_i$ is a level-$(n+1)$ Fatou chain.

Let $K'_i$ denote the level-$(n+1)$ extremal chain containing $K_i$. Then $f(K'_i) \supset f(K_i) = K$. By statement (1), the continuum $f(K'_i)$ is a level-$(n+1)$ extremal chain. Thus, $f(K'_i) = f(K_i) = K$, which implies $\bigcup_{i=1}^m K'_i = \bigcup_{i=1}^m K_i$. Since $E_{i,k}$ is disjoint from $E_{j,k}$ if $i \neq j$, any level-$n$ extremal chain in $K_i$ is disjoint from that in $K_j$ if $i \neq j$. Thus, we obtain $K'_i = K_i$ for $1 \le i \le m$.

Finally, let $w$ be a point in $K \cap F_f$. Then $w \in E_k$ for every sufficiently large integer $k$. Furthermore, if $w$ is not a critical value, we have
$$
\#(f^{-1}(w) \cap K_i) = \#(f^{-1}(w) \cap E_{i,k}) = \deg(f|_{E_{i,k}}) = d_i.
$$
Thus, the lemma is proved.
\end{proof}

According to Lemma \ref{lem:chain-map}, every level-$n$ extremal chain is eventually periodic. Moreover, for any level-$n$ extremal chain $K \neq \cbar$, its boundary and interior are contained in the Julia set and Fatou set of $f$, respectively. To see this, first note that $\partial K \subset J_f$. If the interior of $K$ contains a point in the Julia set, then $f^m(K) = \cbar$ for a sufficiently large integer $m$. Since $f^m(K)$ is a level-$n$ extremal chain, we obtain $K = f^m(K) = \cbar$ by Definition \ref{def:extremal}.

The following result provides a dynamical construction of  periodic extremal chains.

\begin{lemma}\label{lem:dyn-def}
Let $K$ be a periodic level-$(n+1)$ extremal chain of $f$ with period $p\ge 1$, and let $E_0$ be the union of all periodic level-$n$ extremal chains in $K$. Then $E_0$ is connected, $f^p(E_0)=E_0$, and
$$
K=\ov{\bigcup_{k\ge 0}E_k},
$$
where $E_k$ is the component of $f^{-kp}(E_0)$ containing $E_0$.
\end{lemma}

\begin{proof}
First, note that $f^p(E_0)=E_0$ since the image of a periodic level-$n$ extremal chain is also a periodic level-$n$ extremal chain. By Lemma \ref{lem:E-chain}, $E_0$ is contained in a continuum $E\subset K$ that is the union of finitely many level-$n$ extremal chains. Since $f^p(E_0)=E_0$, it follows that $E_0\subset f^{kp}(E)$ for every $k>0$. On the other hand, since each level-$n$ extremal chain is eventually periodic, we obtain $f^{k_0p}(E)\subset E_0$ for some integer $k_0\ge 0$. Therefore, $E_0=f^{k_0p}(E)$ is connected. By Lemma \ref{lem:chain-map}\,(2), each $E_k$ is a level-$(n+1)$ Fatou chain, and $E_0\subset E_k$ contains Fatou domains. Thus, $\ov{\bigcup_{k\ge0}E_k}\subset K$ by the definition of extremal chains.

Conversely, for any level-$n$ extremal chain $\sigma\subset K$, there exists a continuum $E'$ such that $E_0\cup\sigma\subset E'$ and $E'$ is the union of finitely many level-$n$ extremal chains. As above, we have $f^{k_1p}(E')\subset E_0$ for an integer $k_1>0$. Then $\sigma\subset E'\subset E_{k_1}$, and  therefore $K\subset\ov{\bigcup_{k\ge0}E_k}$.
\end{proof}

By definition, every level-$n$ extremal chain is contained in a level-$(n+1)$ extremal chain. The following result shows that the growth of extremal chains will stop at a certain level.

\begin{lemma}\label{lem:maximal}
There exists an integer $N\ge 0$ such that any level-$n$ extremal chain of $f$ is a level-$N$ extremal chain for $n\geq N$.
\end{lemma}

\begin{proof}
Let $k(n)$ denote the number of periodic level-$n$ extremal chains of $f$. Then $k(n)$ is decreasing. Thus, there exists an integer $n_0$ such that $k(n)$ is constant for $n\ge n_0$.
This implies that two distinct periodic level-$n$ extremal chains are disjoint for $n\ge n_0$.

For each periodic Fatou domain $U$ of $f$ with period $p\ge 1$, denote by $K_n(U)$ the level-$n$ extremal chain containing $U$. Then $f^p(K_n(U))=K_n(U)$, and $K_n(U)$ is the unique periodic level-$n$ extremal chain contained in $K_{n+1}(U)$ for $n\ge n_0$. If $K_{n}(U)$ is not a component of $f^{-p}(K_{n}(U))$, we have
$$
\deg(f^p|_{K_{n+1}(U)})>\deg(f^p|_{K_{n}(U)})
$$
by Lemmas \ref{lem:chain-map} and  \ref{lem:dyn-def}. On the other hand, since $\deg(f|_{K_{n+1}}(U))\le\deg f$, there exists an integer $n(U)\ge n_0$ such that $\deg(f^p|_{K_{n}(U)})$ is constant for $n\ge n(U)$. Thus, $K_{n}(U)$ must be a component of $f^{-p}(K_{n}(U))$ for $n\ge n(U)$. It then follows from Lemma \ref{lem:dyn-def} that $K_{n+1}(U)=K_n(U)$ for $n\ge n(U)$.

Let $N_1$ be the maximum of $\{n(U)\}$ for all periodic Fatou domains $U$ of $f$. Then every periodic level-$n$ extremal chain is a level-$N_1$ extremal chain for $n\geq N_1$.

For any level-$N_1$ extremal chain $K$,  there exists an integer $q\ge 0$ such that $f^q(K)$ is a periodic level-$N_1$ extremal chain. Let $K_{i}$ denote the level-$(N_1+i)$ extremal chain containing $K$ for $i>0$. Then $f^q(K_{i})$ is a periodic level-$(N_1+i)$ extremal chain containing $f^q(K)$, and hence $f^q(K_{i})=f^q(K)$.  Applying Lemma \ref{lem:chain-map}\,(2) to $f^q$, we obtain that $K_{i}=K_{1}$ for $i\ge 1$. Therefore, the lemma holds if we define $N:=N_1+1$.
\end{proof}

\begin{proof}[Proof of Theorem \ref{thm:maximal}]
By Lemma \ref{lem:maximal}, there exists an integer $N\ge 0$ such that any level-$n$ extremal chain is a level-$N$ extremal chain for every $n\ge N$. For any Fatou domain $U$ of $f$, let $K(U)$ denote the level-$N$ extremal chain containing $U$. If a Fatou chain $K$ intersects $K(U)$, then $K\cup K(U)$ is contained in an extremal chain of level $N+1$. This implies $K\subset K(U)$. Thus, $K(U)$ is a maximal Fatou chain.
By Lemma \ref{lem:chain-map}, the image and components of the pre-image of a maximal Fatou chain are still maximal Fatou chains.
\end{proof}

\section{Decompositions of rational maps}\label{sec:4}

In this section, we establish the \emph{cluster-exact decomposition} (Theorem \ref{thm:cluster-exact0}) for marked rational maps. This  decomposition theorem corresponds to Theorem \ref{thm:cluster-exact}\,(1) and (2), and the remaining part of Theorem \ref{thm:cluster-exact} follows from Theorem \ref{thm:blow-up}, which
will be proved in the next section.

In Section 4.1, we study the combinatorics of planar continua and domains by their \emph{branched numbers}. In Section 4.2, we characterize the dynamics of stable sets by proving Theorem \ref{thm:renorm}. In Section 4.3, we obtain an important result, called the \emph{exact decomposition}, which serves as a key step toward the cluster-exact decomposition. Finally, we complete the proof of the cluster-exact decomposition in Section 4.4.

\subsection{Branched numbers}\label{sec:branched-number}
Let $P\subset\cbar$ be a finite marked set, and let $E\subset\cbar$ be a connected open or closed set.
Recall that $E$ is {simple-type} (rel $P$) if there exists a simply connected domain $D\subset\cbar$ such that $E\subset D$ and $\#(D\cap P)\le 1$; or {annular-type} if $E$ is not simple-type and there exists an annulus $A\subset\cbar\setminus P$ such that $E\subset A$; or {complex-type} otherwise.

The {\bf branched number} of $E$ (rel $P$) is defined by
$$
b(E):=\#(E\cap P)+\kappa(E),
$$
where $\kappa(E)$ is the number of components of $\cbar\sm E$ that intersect $P$. By definition, $E$ is complex-type if and only if $b(E)\ge 3$, and $b(E)=2$ if $E$ is annular-type.

Let $K_0\subset K$ be continua in $\cbar$. Recall that $K_0$ is a skeleton of $K$ (rel $P$) if $K_0\cap P=K\cap P$ and any two points of $P$ in distinct components of $\cbar\sm K$ are contained in distinct components of $\cbar\sm K_0$. It is easy to verify that
\begin{equation}\label{eq:skeleton}
\text{$K_0$ is a skeleton of $K$ $\Longleftrightarrow$ $b(K_0)=b(K)$ and $\#(K_0\cap P)=\#(K\cap P)$}.
\end{equation}

\begin{lemma}\label{lem:closed-open}
The following statements hold:
\begin{enumerate}
\item For any continuum $E\subset\cbar$, there exists a domain $U\supset E$ such that $b(U)=b(E)$;
\item For any domain $U\subset \cbar$, there exists a continuum $E\subset U$ such that $b(U)=b(E)$.
\end{enumerate}
\end{lemma}

\begin{proof}
(1) Let $V_i$, $1\le i\le n$, be the components of $\cbar\sm E$ containing points of $P$. Then there exists a full continuum $K_i\subset V_i$ such that $P\cap K_i=P\cap V_i$. Set $U=\cbar\sm\bigcup_{i=1}^n K_i$. Then $U\supset E$ is a domain, and $b(U)=b(E)$.

(2) Let $E_j$, $1\le j\le m$, be the components of $\cbar\sm U$ that intersect $P$. Then there exist disks $V_j\supset E_j$ with pairwise disjoint closures such that $\partial V_j\subset U$ and $P\cap E_j=P\cap V_j$. Since $U$ is a domain, there exists a graph $E\subset U$ containing $P\cap U$ and all $\partial V_j$, $j=1,\ldots,m$. It follows that  $b(U)=b(E)$.
\end{proof}

\begin{lemma}\label{lem:complexity}
Suppose that $V\subset\cbar$ is a complex-type domain and $\KKK\subset V$ is a compact set. Let $\EE$ be the collection of all complex-type components of either $V\sm\KKK$ or $\KKK$. Then
$$
\sum_{ E\in\EE} (b(E)-2)=b(V)-2.
$$
\end{lemma}

\begin{proof}
There exist at most $\#P$ elements of $\EE$ intersecting $P$ and  at most $\#P-2$ elements  disjoint from $P$ since each divides $P$ into at least three parts. Thus, $\EE$ is a finite collection.

In order to prove the equality, define a graph $T$ as  follows. Let $\EE_1$ be the collection of all components of $\cbar\sm V$ intersecting $P$.  There exists a bijection $v$ from $\EE_1\cup \EE$ onto the set of vertices of $T$. Two vertices $v(E_1)$ and $v(E_2)$ of $T$ are connected by an edge if and only if $E_1$ and $E_2$ are \emph{adjacent}, i.e., no elements of $\EE$ separate $E_1$ from $E_2$. Then $T$ is a tree.

Note that for any element $E\in\EE_1\cup \EE$, the number of edges of $T$ connecting to the vertex $v(E)$ is exactly $\kappa(E)$, i.e., the number of components of $\ov{\C}\setminus E$ intersecting $P$. Thus, $v(E)$ is an endpoint of $T$ precisely if $\kappa(E)=1$. In particular, $v(E)$ is an endpoint if $E\in \EE_1$.

Let $k_0\ge 0$ denote the number of elements of $\EE$ with $\kappa(E)=1$. Then $T$ has exactly $\kappa(V)+k_0$ endpoints. Since $T$ is a tree, we have
$$
\kappa(V)+k_0-2=\sum(\kappa(E)-2),
$$
where the summation is taken over all elements of $\EE$ with $\kappa(E)\geq 2$.
It follows immediately that $$\kappa(V)-2=\sum(\kappa(E)-2),$$
where the summation is taken over all elements of $\EE$. Thus, the lemma holds if $V\cap P=\emptyset$.

In the general case, without loss of generality, we  assume that all marked points in $\KKK$ are  interior points of $\KKK$. Then there exists a small number $r>0$ such that $\D(z,3r)\subset V$
for each point $z\in P\cap V$, and  $\D(z,3r)\subset \KKK$ for $z\in P\cap\KKK$.

Set $V':=V\setminus \bigcup_{z\in P\cap V}\ov{\D(z,r)}$ and $\KKK':=\KKK\setminus\bigcup_{z\in P\cap \KKK}\D(z,2r)$. Let $\EE'$ be the collection of all complex-type components of either $V'\sm\KKK'$ or $\KKK'$. It follows that
\begin{itemize}
\item $\sum_{E'\in \EE'}(b(E')-2)=b(V')-2$ since $V'\cap P=\emptyset$; and \vspace{1pt}

\item $b(V)=b(V')$ and each $E'\in\EE'$ is contained in a unique element $E \in\EE$ with $b(E')=b(E)$.\vspace{2pt}
\end{itemize}
\noindent Therefore, we have $\sum_{E\in \EE}(b(E)-2)=b(V)-2$. The lemma is proved.
\end{proof}

\begin{corollary}\label{cor:monotone}
The following statements hold:
\begin{enumerate}
\item  Let $K_0\subset K$ be continua in $\cbar$. Then $b(K_0)\le b(K)$.
\item  Let $\{K_n\}$ be a sequence of continua in $\cbar$ such that $K_{n}\subset K_{n+1}$ for all $n\ge 0$. Then there exists $N\geq0$ such that $b(K_n)=b(K_N)$, and $K_N$ is a skeleton of $K_n$ for every $n\geq N$.
\item  Let $\{K_n\}$ be a sequence of continua in $\cbar$ such that $K_{n+1}\subset K_n$ for all $n\ge 0$, and set $K:=\bigcap_{n\ge 1}K_n$. Then $b(K)=b(K_n)$ for sufficiently large $n$.
\end{enumerate}
\end{corollary}

\begin{proof}
(1) By Lemma \ref{lem:closed-open}, there exists a domain $U\subset\cbar$ such that $b(U)=b(K)$. It follows from Lemma \ref{lem:complexity} that $b(K_0)\le b(U)=b(K)$.

(2) Note that the numbers $b(K_n)$ and $\#(K_n\cap P)$ are increasing and bounded above by $\# P$. Thus, there exists an integer $N\geq0$ such that both $b(K_n)$ and $\#(K_n\cap P)$ are constant for every $n\geq N$. By relation \eqref{eq:skeleton}, $K_N$ is a skeleton of $K_n$ for every $n\geq N$.

(3) By statement (1), the number $b(K_n)$ is decreasing. Thus, $b(K_n)$ becomes a constant $b\ge 1$ for sufficiently large $n$. Since $K$ is a connected closed set, we have $b(K)\le b$. On the other hand, by Lemma \ref{lem:closed-open}, there exists a domain $U\supset K$ such that $b(U)=b(K)$. Since $K_n\subset U$ for every sufficiently large  integer $n$, it follows from Lemma \ref{lem:complexity} that $b(K)=b(U)\ge b(K_n)=b$.
\end{proof}

 Now, let $(f,P)$ be a marked rational map. Since $f(P)\subset P$, we immediately obtain the following \emph{pullback principle}.
\begin{lemma}\label{lem:pullback}
 Let $(f, P)$ be a marked rational map. Suppose that $E\subset \ov{\C}$ is a connected open or closed set. If $E$ is simple-type, then each component of $f^{-1}(E)$ is simple-type. If $E$ is annular-type, then each component of $f^{-1}(E)$ is either annular-type or simple-type.
\end{lemma}

\begin{lemma}\label{lem:deg}
Let $(f, P)$ be a marked rational map. Let $E\subset E'$ be connected open or closed sets in $\cbar$ with $b(E)=b(E')$. Let $E_1'$ be a component of $f^{-1}(E')$. Then  $E_1:=E_1'\cap f^{-1}(E)$ is connected.  Moreover, if $E$ is a skeleton of $E'$, then $E_1$ is a skeleton of $E'_1$.
\end{lemma}

\begin{proof}
By Lemma \ref{lem:closed-open}, there exist a domain $V\supset E'$ and a compact connected set $K\subset E$ such that $b(V)=b(K)$. Let $V_1$ be the component of $f^{-1}(V)$ containing $E_1'$.

According to Lemma \ref{lem:complexity}, each component $U$ of $V\sm K$ is either simple-type or annular-type, and $\partial U$ has exactly one component contained in $K$. Consequently, any component of $f^{-1}(U)$ is either simple-type or annular-type by Lemma \ref{lem:pullback}, and its boundary has exactly one component contained in $f^{-1}(K)$. This implies that $V_1$ contains exactly one component $K_1$ of $f^{-1}(K)$ and $b(V_1)=b(K_1)$. Thus, the former part of the lemma holds.

 Furthermore, if $E$ is a skeleton of $E'$, then $E\cap P=E'\cap P$, which implies  $E_1\cap P=E_1'\cap P$. Note also that $b(K_1)\le b(E_1)\leq b(E_1')\leq b(V_1)=b(K_1)$. Thus, $E_1$ is a skeleton of $E_1'$ by \eqref{eq:skeleton}.
\end{proof}

\subsection{Stable sets}
 Recall that a stable set $\KKK$ of a rational map $f$ is a non-empty and finite disjoint union of continua such that $f(\KKK)\subset\KKK$ and each component of $f^{-1}(\KKK)$ is either a component of $\KKK$ or disjoint from $\KKK$. By definition, each component of $\KKK$ is eventually periodic, and $\partial\KKK$ is also a stable set of $f$ provided that $\KKK\not=\cbar$.

Throughout this subsection, let $f$ be a given PCF  rational map.

\begin{lemma}\label{lem:boundary}
Let $K\subsetneq\cbar$ be a connected stable set of $f$. Then $\partial K\subset J_f$.
\end{lemma}

\begin{proof}
 Choose a domain $W\supset K$ such that $b(K)=b(W)$. Then each component of $f^{-1}(W)$ contains exactly one component of $f^{-1}(K)$ by Lemma \ref{lem:deg}. In particular, the component $W_1$ of $f^{-1}(W)$ containing $K$ is disjoint from $f^{-1}(K)\sm K$.

 Suppose, to the contrary, that $\partial K\cap F_f\neq\emptyset$. Since $K$ is a component of $f^{-1}(K)$, we have $f(\partial K)=\partial K$. Thus, there exists a super-attracting periodic point $a\in\partial K$.  Without loss of generality, we may assume $f(a)=a$. Let $U$ be the Fatou domain containing $a$. Then there exists a disk $\Delta\subset U$ such that it is a round disk in the B\"{o}ttcher coordinate and $\Delta\subset  W$. This implies that if $z\in K\cap\Delta$, then $f^{-1}(z)\cap U\subset K$.

Let $\g_t\subset\Delta$ be the Jordan curve corresponding to the circle with radius $t\in (0,1)$ in the B\"{o}ttcher coordinate.  Since $K$ is connected and $a\in K$,  there exists a point $t_0\in (0,1)$ such that $\g_{t_0}\cap K\neq\emptyset$ and $\g_{t_0}\subset\Delta$. It follows that $\g_{t}\cap K\neq\emptyset$ for all $t\in (0,t_0)$ since  $\g_t$ separates $\g_{t_0}$  from  $a$. In particular, given any $t\in (0,t_0)$, $f^{k}(\g_t)\cap K\neq\emptyset$ for all $k\ge 1$.

Pick a point $z_k\in f^{k}(\g_t)\cap K$. Then $f^{-k}(z_k)\cap U\subset\g_t\cap K$. Since $\g_t\cap K$ is compact and $\bigcup_{k\ge 1}(f^{-k}(z_k)\cap U)$ is dense in $\g_t$, we obtain $\g_t\subset K$ for all $t\in (0,t_0)$, a contradiction.
\end{proof}

The following lemma offers a way to obtain stable sets.
\begin{lemma}\label{lem:new-stable}
	 Let $\{V_n\}_{n\geq 0}$ be a sequence of domains in $\ov{\mathbb{C}}$ such that $V_{n+1}\subset V_n$ and $f:V_{n+1}\to V_n$ is proper. If, for any $n\geq 0$, there exists an integer $m>n$ such that  $\ov{V_m}\subset V_n$, then $K=\bigcap_{n>0} V_n$ is a stable set of $f$ when $K$ is not a singleton.
\end{lemma}
\begin{proof}
	It follows from the known condition that  $K$ is a component of $f^{-1}(K)$. Hence, $K$ is a stable set unless it is a singleton.
\end{proof}

\begin{proof}[Proof of Theorem \ref{thm:renorm}]
 Let $\wh{K}$ be the union of $K$ and all  components of $\cbar\sm K$ disjoint from $P_f$. If $\wh K=\cbar$, then $f^{-1}(K)=K$, and thus $\wh{K}=K=\cbar$, which contradicts the condition that $K\neq \cbar$.

Now, assume $\wh{K}\not=\ov{\C}$. Let $\DD$ denote the collection of components of $\cbar\sm\wh K$. Define a self-map $f_*$ on $\DD$ as follows. If $D\in\DD$ is disjoint from $f^{-1}(K)$,  then $f(D)\in\DD$ and we set $f_*(D):=f(D)$. Otherwise, let $D'$ be the component of $D\sm f^{-1}(K)$ such that $\partial D'\supset\partial D$. In this case, $f(D')$ is an element of $\DD$, and we define $f_*(D):=f(D')$.

Since $\DD$ is a finite collection, each of its elements is eventually periodic under $f_*$. Assume that $D_i,0\le i<p,$ forms a cycle in $\DD$ with $D_i=f_*^i(D_0)$ and $D_0=f_*^p(D_0)$. Since $f$ is expanding in a neighborhood of $J_f$ under the orbifold metric, and $\partial K\subset J_f$ by Lemma \ref{lem:boundary}, for each $0\le i< p$, there exists an annulus $A_{D_i}=A_i\subset D_i\sm P_f$ with $\partial D_i\subset\partial A_i$, such that $\ov{A_i^1}\subset A_i\cup\partial D_i$, where $A_i^1$ is the component of $f^{-1}(A_{i+1})$ (with $A_p=A_0$) such that  $\partial A_i^1\supset\partial D_i$. Applying a similar argument, we can assign an annulus $A_{D}$ to every periodic element $D\in\DD$.

 If $D'\in\DD$ is not $f_*$-periodic but $f_*(D')=D$ is periodic, we assign an annulus $A_{D'}\subset D'\sm P_f$ with $\partial D'\subset\partial A_{D'}$, such that $\ov{A_D^1}\subset A_{D'}\cup\partial D'$, where $A_{D}^1$ is the component of $f^{-1}(A_{D})$ with $\partial D'\subset A_{D}^1$. Repeating this process, we assign an annulus $A_{D}$ to each element $D\in\DD$.

Let $V$ be the union of $\wh K$ and $A_{D}$ for all $D\in\DD$. Then $V$ is a finitely connected domain with $V\cap P_f=K\cap P_f$. Moreover, the component $U$ of $f^{-1}(V)$ containing $K$ is compactly contained in $V$ by the construction of $A_{D}$.   Since $K$ is not a singleton, it follows from \cite[Lemma 18.8]{Mi1} that $\deg f|_K\ge 2$. Thus, $f: U\to V$ is a rational-like map (see \cite[Definition 4]{CPT2}). Then the theorem follows directly from \cite[Theorem 5.2]{CPT2}.
\end{proof}

\begin{lemma}\label{lem:nest}
Let $\{\KKK_n\}_{n\ge 0}$ be a sequence of stable sets of $f$ such that $\KKK_{n+1}\subset\KKK_n$. Then there exists an integer $N\ge 0$ such that $\KKK_n=\KKK_N$ for every $n\ge N$.
\end{lemma}

\begin{proof}
By the pullback principle (Lemma \ref{lem:pullback}), we can split each stable set $\KKK_n$ into two stable sets, $\KKK^0_n$ and $\KKK'_n$, such that each periodic component of $\KKK^0_n$ is either simple-type or annular-type, and each periodic component of $\KKK'_n$ is complex-type. Then $\KKK'_{n+1}\subset\KKK'_n$ by Corollary \ref{cor:monotone}\,(1).

We first assume that the components of $\KKK'_n$ are all complex-type for every $n\ge 0$. The branched number of $\KKK_n'$ is defined by
$$
b(\KKK'_n)=\sum (b(K)-2)+2,
$$
where the summation is taken over all components of $\KKK_n'$. Then $b(\KKK'_{n+1})\le b(\KKK'_n)$ by Lemma \ref{lem:complexity}. Thus, there exists an integer $n_1\ge 0$ such that $b(\KKK'_n)$ is constant for $n\ge n_1$. This implies that, for $n\ge n_1$, each component of $\KKK'_n$ contains at least one component of $\KKK'_{n+1}$.

Let $k(n)$ be the number of components of $\KKK'_n$ for $n\ge n_1$. As argued above,  $k(n)$ is increasing. However, Lemma \ref{lem:complexity} implies $k(n)\le \#P_f-2$. Thus, there exists an integer $n_2\ge n_1$ such that $k(n)$ is constant for $n\ge n_2$. Consequently, each component $K_n$ of $\KKK'_n$ contains exactly one component $K_{n+1}$ of $\KKK'_{n+1}$ for $n\ge n_2$. Since $b(\KKK'_n)$ is constant for $n\ge n_2$, it follows that $b(K_n)=b(K_{n+1})$.

To complete the proof, we need to show that for each periodic component $K_n$ of $\KKK'_n$, it holds that $K_{n+1}=K_n$ for sufficiently large $n>n_2$. Without loss of generality, we may assume $f(K_n)=K_n$. Then $f(K_{n+1})=K_{n+1}$.

By Theorem \ref{thm:renorm} and Lemma \ref{lem:deg}, we know that $\bigcup_{k\ge 0} (f|_{K_n})^{-k}(\partial K_{n+1})=\partial K_{n+1}$ is dense in $\partial K_n$. Hence, $\partial K_{n+1}=\partial K_n$. If $K_{n+1}\neq K_n$, it implies that $K_n\sm K_{n+1}\subset F_f$. Since $f$ has at most $2\deg f-2$ cycles of Fatou domains, the inequality $K_{n+1}\neq K_n$ can  occur only finitely many times. Hence, there exists an integer $n_3\ge n_2$ such that $\KKK'_n=\KKK'_{ n_3}$ for $n\ge n_3$.

In general, let $\KKK''_n$ be the union of all complex-type components of $\KKK'_n$. Then $\KKK''_n$ is also a stable set of $f$, and $\KKK''_{n+1}\subset\KKK''_n$ for all $n\ge 0$. Based on the previous discussion, we can find an integer $N_0\ge 0$ such that $\KKK''_n=\KKK''_{ N_0}$ for every $n\ge N_0$.

Note that $\KKK_n''$ contains all periodic components in $\KKK_n'$, which means that any component of $\KKK_n'$ is eventually iterated into $\KKK_n''$. Thus, for any $m\geq N_0$ and any component $K$ of $\KKK_m'$, either $K$ is a component of $\KKK_n'$ for every $n\geq m$, or $K\cap \KKK_n'=\emptyset$ for sufficiently large $n$.
Consequently,  the number $l(n)$ of components of $\KKK'_n$ (for $n\ge N_0$) is decreasing. Therefore, there exists an integer $N\geq N_0$ such that $l(n)=l(N)$ for every $n\geq N$. This implies $\KKK'_n=\KKK'_N$ for $n\ge N$.

Since $\KKK'_{n}=\KKK'_N$ for $n\ge N$, it follows that $\KKK^0_{n+1}\subset\KKK^0_n$ for $n\ge N$. For any periodic component $K$ of $\KKK^0_n$, the renormalization of $f^p$ on $K$ is conformally conjugate to either $z\mapsto z^d$ or $z\mapsto 1/z^d$ with $d\ge 2$. Thus, $K$ is either a Jordan curve or the closure of a periodic Fatou domain of $f$. In the former case, the cycle of $K$ contains no other stable set of $f$ except itself. In the latter case, the cycle of $\partial K$ is the unique stable set of $f$ properly contained in the cycle of $K$. Thus, we have $\KKK^0_{n+1}=\KKK^0_n$ for sufficiently large $n\ge N$.
\end{proof}

\subsection{Exact decomposition}
Let $(f,P)$ be a marked rational map.
Suppose that $\KKK$ is a stable set of $f$. Let $\VVV$ and $\VVV_1$ be the union of all  complex-type components of $\cbar\sm\KKK$ and $\cbar\sm f^{-1}(\KKK)$, respectively. By the pullback principle (Lemma \ref{lem:pullback}), it holds that $f(\VVV_1)\subset \VVV$.

We say $\KKK$ induces an {\bf exact decomposition} of $(f,P)$ if either $\VVV=\emptyset$, or $f:\VVV_1\to\VVV$ is an exact sub-system of $(f,P)$, i.e., each component of $\VVV\setminus \VVV_1$ is a full continuum disjoint from $P$; see Definition \ref{def:exact-system}.

The following result serves as a key step toward the cluster-exact decomposition.
By an {\bf exceptional stable set}, we mean a stable set containing the Julia set.

\begin{theorem}[Exact decomposition]\label{thm:decomposition}
Let $(f,P)$ be a marked rational map, and let $\KKK_0$ be a non-exceptional stable set of $f$. Then there exists a non-exceptional stable set $\KKK\supset\KKK_0$  that induces an exact decomposition of $(f,P)$. Moreover, if each component of $\KKK_0$ intersects or separates $P$ $($as defined before Lemma \ref{lem:finite}$)$, then so does each component of $\KKK$.
\end{theorem}

The condition that each component of $\mathcal{K}_0$ intersects or separates $P$ is equivalent to $\kappa(U)=\#\textup{Comp}(\partial U)$ for any component $U$ of $\cbar\setminus \mathcal{K}_0$. In particular, annular-type components of $\cbar\setminus \mathcal{K}_0$ are annuli. Recall that $\kappa(U)$ denotes the number of components of $\cbar\setminus U$ intersecting $P$, and ${\rm Comp}(\cdot)$ denotes the collection of all components of the corresponding set.

We can always choose an $f$-invariant and finite set $P_1\supset P$ such that $P_1\setminus P\subset \KKK_0$ and each component of $\KKK_0$ intersects or separates points of $P_1$. Immediately, any complex-type domain rel $P$ is still complex-type rel $P_1$. By definition,  if $\KKK$ induces an exact decomposition of $(f,P_1)$,  it also induces an exact decomposition of $(f,P)$.
Thus, it suffices to prove the theorem for $(f,P_1)$. Therefore, we can  assume that each component of $\KKK_0$ intersects or separates  $P$.

For any stable set $\BBB$ of $f$, denote by $\BBB^n$  the union of all components of $f^{-n}(\BBB)$ that intersect or separate $P$. By Lemma \ref{lem:pullback}, each $\BBB^n$ is a stable set of $f$, and $\BBB^n\subset\BBB^{n+1}$. 

For each $n\geq0$, let $\UUU_n$ be the union of all complex-type components of $\cbar\setminus\KKK^n_0$. It follows immediately that $\UUU_{n+1}\subset \UUU_n$.

\begin{lemma}\label{lem:stable}
Assume that $\UUU_n\neq\emptyset$ for all $n\ge 0$. Then there exists a positive integer $N_0$ such that any component $U_{N_0}$ of $\UUU_{N_0}$ contains a unique component $U_n$ of $\UUU_n$ for every $n\geq N_0$. Moreover, it holds that $$\#(U_n\cap P)=\#(U_{N_0}\cap P)  \quad \text{and} \quad \#{\rm Comp}(\partial U_n)=\#{\rm Comp}(\partial U_{N_0}).$$
\end{lemma}

\begin{proof}
Let $k(n)$ denote the number of complex-type components of $\KKK_0^n$. Then $k(n)$ is increasing, and $k(n)\leq \#P-2$
 by Lemma \ref{lem:complexity}. Thus, there exists an integer $n_0$ such that $k(n)=k(n_0)$ for all $n\ge n_0$. Therefore, $\UUU_{n_0}$  contains no complex-type components of $\KKK_0^n$ for all $n>n_0$.

Fix a component $U_n$ of $\UUU_n$ with $n\ge n_0$. Since $U_n$ contains no complex-type components of $\KKK_0^m$ for $m>n$, it follows from Lemma \ref{lem:complexity} that
$$
\sum(b(U)-2)=b(U_n)-2>1,
$$
where the summation is taken over all  components of $\UUU_m$ contained in $U_n$. Thus, $U_n$ contains at least one component of $\UUU_m$. Consequently, the number $v(n)$ of components of $\UUU_n$ is increasing for $n\ge n_0$.

 Note that $\#(\UUU_n\cap P)$ is decreasing. Then there exists an integer $n_1\ge n_0$ such that both $v(n)$ and $\#(\UUU_n\cap P)$ are constant for $n\ge n_1$. Thus, each component $U_{n_1}$ of $\UUU_{n_1}$ contains a unique component $U_n$ of $\UUU_{n}$ for every $n>n_1$ such that
$\#(U_n\cap P)=\#(U_{n_1}\cap P)$. Since $b(U_n)$ is decreasing, there exists an integer $N_0>n_1$ such that $b(U_n)=b(U_{N_0})$ for all $n\geq N_0$.

Finally, since each component of $\KKK_0^n$ intersects or separates $P$,  all complementary components of $U_n$ intersect $P$, i.e., $\#{\rm Comp}(\partial U_n)=\kappa(U_n)$. It follows that $\#{\rm Comp}(\partial U_n)=b(U_n)-\#(U_n\cap P)$ is constant for $n\geq N_0$ by the choice of $N_0$.
\end{proof}

According to Lemma \ref{lem:stable}, any component $U_{N_0}$ of $\UUU_{N_0}$ and any component $\lambda_{N_0}$ of $\partial U_{N_0}$ determine a sequence of pairs $(U_n, \lambda_n)$ for $n\geq N_0$, where $U_n$ is the component of $\UUU_{n}$ contained in $U_{N_0}$, and $\lambda_n$ is the component of $\partial U_n$ such that either $\lambda_{n+1}=\lambda_{n}$, or $\lambda_{n+1}$ is disjoint from $\lambda_n$ but separates $\lambda_n$ from $U_{n+1}$.

Since $\UUU_{N_0}$ has  finitely many components, all of which are finitely connected, there exists an integer $N \ge N_0$ such that, for any determined sequence $\{(U_n,\lambda_n), n\geq N\}$, exactly one of the following two cases occurs:
\begin{itemize}
\item $\lambda_n=\lambda_N$ for all $n\ge N$;
\item for any $n\ge N$, there exists an integer $m>n$ such that $\lambda_m$ is disjoint from $\lambda_n$ and separates $\lambda_n$ from $U_m$.\vspace{2pt}
\end{itemize}
\noindent We call $\lambda_N$  an {\bf exact boundary component} of $U_N$ in the first case.

\vskip 0.1cm
From now on,  write $\VVV=\UUU_N$, and denote by $\VVV_n$ the union of all complex-type components of $f^{-n}(\VVV)$. Then $\VVV_n$ coincides with the union of all complex-type components of  $\cbar\setminus f^{-n}(\KKK^N_0)$. This implies $\VVV_n\subset \UUU_{N+n}$.

Note that  any component of $f^{-n}(\KKK^N_0)\setminus \KKK^{N+n}_0$ neither intersects nor separates $P$, while each component of $\partial\UUU_{N+n}$ intersects or separates $P$. It follows that $\UUU_{N+n}\setminus \VVV_n$ consists of pairwise disjoint full continua disjoint from $P$.
Therefore,
\begin{enumerate}
\item each component $V=U_N$ of $\VVV$ contains a unique component $V_n$ of $\VVV_n$ such that  $U_{N+n}\setminus V_n$ consists of pairwise disjoint full continua that avoid $P$;

\item for any boundary component $\lambda$ of $V$, there exists a unique boundary component $\lambda_n$ of $V_n$ {\bf parallel to} $\lambda$ in the sense that either $\lambda_n=\lambda$ or $\lambda_n$  separates $\lambda$ from $V_n$.\vspace{2pt}
\end{enumerate}

We say $V$ is an {\bf exact} (resp., {\bf renormalizable}) component of $\VVV$ if all components of $\partial V$ are exact (resp., non-exact) boundary components of $V$; see Figure \ref{fig:system} (where the pants represent $V$, and the domains colored yellow correspond to $V_1$).

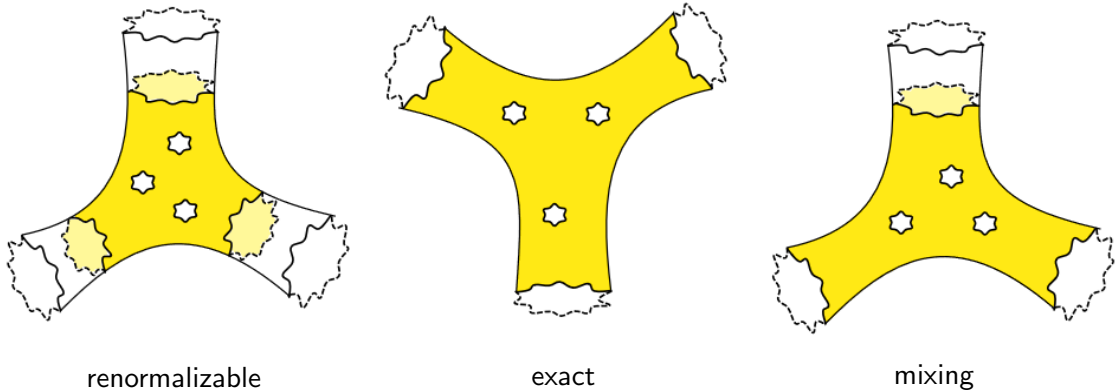
\begin{figure}[http]
\centering
\begin{tikzpicture}
\node at (0,0){ \includegraphics[width=15cm]{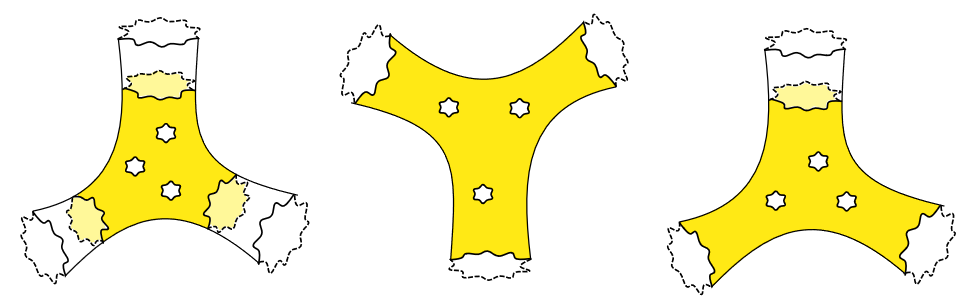}};
\node at (-5,-2.8) {\sf renormalizable};
\node at (0.125,-2.8) {\sf exact};
\node at (5,-2.8) {\sf mixing};
\end{tikzpicture}
\caption{Classification of components of a sub-system. }\label{fig:system}
\end{figure}

If $V=U_N$ is exact, then $V=U_{N+1}$. By this point and statement (1) above, it follows that $V\setminus V_1$  consists of full continua disjoint from $P$. This immediately implies the following:
\begin{proposition}\label{pro:exact}
The stable set $\KKK^N_0$ induces an exact decomposition of $(f,P)$  if every component of $\VVV$ is exact.
\end{proposition}

Let $\VV$ be the collection of all components of $\VVV$. Then $f:\VVV_1\to \VVV$ induces a self-map $f_{\#}$ on $\VV$, defined by $f_{\#}(V):=f(V_1)$, where $V_1$ is the unique component of $\VVV_1$ contained in $V$. Since $\VV$ is a finite collection, each component of $\VVV$ is eventually $f_{\#}$-periodic.

The map $f:\VVV_1\to\VVV$ also induces a self-map $f_*$ on the collection $\partial\VV$ of the boundary components of $V$ for all $V\in\VV$. This self-map is defined by $f_*(\lambda):=f(\lambda_1)$, where $\lambda_1$ is the unique boundary component of $V_1$ parallel to $\lambda$.  Since $\partial\VV$ is a finite collection, its elements are eventually $f_*$-periodic.

\begin{proposition}\label{prop:exact}
Let $V$ be a component of $\VVV$, and let $\lambda$ be a component of $\partial V$.
Then $\lambda$ is an exact boundary component of $V$ if and only if $f_*(\lambda)$ is an exact boundary component of $f_{\#}(V)$.
Consequently, if $V$ is non-exact, then $f_{\#}(V)$ is also non-exact.
\end{proposition}

\begin{proof}
For each $n\geq0$, denote by $V_n$ the unique component of $\VVV_n$ contained in $V$, and by $\lambda_n$ the unique boundary component of $\partial V_n$ parallel to $\lambda$. 
Set $W=f_{\#}(V)$ and $\eta=f_*(\lambda)$.  Similarly, we  define $W_n$ and $\eta_n$ for $n\geq0$. By definition, it holds that $f(V_1)=W$ and $f(\lambda_1)=\eta$.

If $\lambda$ is exact, then $\lambda_{n+1}=\lambda$ and $\eta=f(\lambda_{n+1})=\eta_n$ for all $n\geq 0$. Thus, $\eta$ is exact.

If $\lambda$ is non-exact, there exists an $n\geq0$ such that $\lambda_{n+1}\cap \lambda_1=\emptyset$. 
Choose an annulus $A\subset W\setminus P$ that is bounded by $\eta$ and a Jordan curve in $W_n$. Since $b(W_{n})=b(W)$, it follows from Lemma \ref{lem:deg} that $f^{-1}(W_n)\cap V_1=V_{n+1}$. 

Let $A_1\subset V_1$ be the component of $f^{-1}(A)$ with $\lambda_1\subset \partial A_1$. Then $A_1$ is an annulus disjoint from $P$ and the boundary component of $A_1$ other than $\lambda_1$ is contained in $V_{n+1}$. Since $\lambda_{n+1}\cap \lambda_1=\emptyset$, we have $\lambda_{n+1}\subset A_1$. It follows that $A$ contains a boundary component of $W_n$ parallel to $\eta$, which can only be $\eta_n$. Thus, $\eta$ is non-exact
by the choice of $N$.
\end{proof}

According to Proposition \ref{pro:exact}, if all components of $\VVV$ are exact, then Theorem \ref{thm:decomposition} holds by defining $\KKK=\KKK^N_0$.  If the components of $\VVV$ are either exact or renormalizable, denote by $\VVV'$  the union of all renormalizable components of $\VVV$, and by $\VVV_n'$ the union of all components of $\VVV_n$ within $\VVV'$. By Proposition \ref{prop:exact}, the map $f_{\#}$ is invariant on both the collection of all renormalizable components  and the collection of all exact components of $\VVV$. Thus, $f:\VVV_1\setminus \VVV_1'\to \VVV\setminus \VVV'$ is an exact sub-system, and $\KKK':=\bigcap_{n\ge 1}\VVV_n'$ is a  stable set of $f$ disjoint from $\KKK_0$ by Lemma \ref{lem:new-stable}. Therefore, Theorem \ref{thm:decomposition} holds if we set $\KKK:=\KKK^N_0\cup\KKK'$.

However, $\VVV$ might contain components that are neither exact nor renormalizable; see Figure \ref{fig:system}. In this case, we need to combine these components to obtain a renormalization domain.

\begin{lemma}\label{lem:new}
Suppose that $V$ is an $f_{\#}$-periodic and non-exact component of $\VVV$.  Then there exists a non-exceptional stable set $\KKK'$ of $f$, whose components are all complex-type, such that
$
\bigcap_{n\geq0}\ov{V_n}\subset\KKK',
$
where  $V_n$ denotes the component of $\VVV_n$ contained in $V$. Moreover, each component of $\KKK_0$ is either contained in $\KKK'$ or disjoint from $\KKK'$.
\end{lemma}

We can quickly deduce Theorem \ref{thm:decomposition} from Lemma \ref{lem:new}. 

\begin{proof}[Proof of Theorem \ref{thm:decomposition}]
We adhere to the notations mentioned above.
If $\VVV=\emptyset$ or  $\VVV$ contains only exact components,  the theorem holds by taking $\KKK=\KKK^N_0$, according to Proposition \ref{pro:exact}. Otherwise, $\VVV$ has an $f_{\#}$-periodic and non-exact component $V$ by Proposition \ref{prop:exact}. 

Let $\KKK'$ be the non-exceptional stable set obtained in Lemma \ref{lem:new}. Then there exists a sufficiently large integer $N'$ such that $(\KKK')^{ N'+1}\sm (\KKK')^{ N'}$ is disjoint from $\KKK_0$.

Set $\KKK_1=\KKK_0\cup (\KKK')^{ N'}$. It is a non-exceptional stable set of $f$, and its components all intersect or separate $P$.  Since $\bigcap_{n\ge 0}\ov{V_n}$ is a complex-type continuum (by Corollary \ref{cor:monotone}\,(3)) not contained in $\KKK_0$, it follows from Lemma \ref{lem:complexity} that
$$
b(\KKK_0):=\sum (b(K)-2)+2<b(\KKK_1):=\sum (b(K_1)-2)+2,
$$
where the first and second summations are taken over all complex-type components of $\KKK_0$ and $\KKK_1$, respectively.

If $\KKK_1^{ N_1}$ induces an exact decomposition of $(f,P)$ for an integer $N_1$, the theorem holds by taking $\KKK=\KKK_1^{ N_1}$. Otherwise, we can repeat the argument above by replacing $\KKK_0$ with $\KKK_1$ and obtain a non-exceptional stable set $\KKK_2\supset\KKK_1$  such that $b(\KKK_2)>b(\KKK_1)$ and each component of $\KKK_2$ intersects or separates $P$.

By iterating this process, we obtain an increasing sequence of non-exceptional stable sets $\{\KKK_n\}$ such that $b(\KKK_{n+1})>b(\KKK_n)$. Since $b(\KKK_n)\le\#P$ by Lemma \ref{lem:complexity}, this process must stop after a finite number of steps. This completes the proof.
\end{proof}

\begin{proof}[Proof of Lemma \ref{lem:new}]
According to Proposition \ref{prop:exact}, there exists an $f_*$-periodic and non-exact boundary component $\lambda$ of $V$. Its period is denoted by $p$.

For each $0\le i<p$, set $V_{i,0}:=f_{\#}^i(V)$ and  $\lambda_i:=f_*^i(\lambda)$. Then $f^p_{\#}(V_{i,0})=V_{i,0}$, and each $\lambda_i$ is a non-exact boundary component of $V_{i,0}$ by Proposition \ref{prop:exact}. For every $n\geq0$, denote by $V_{i,n}$ the unique complex-type component of $f^{-np}(V_{i,0})$ contained in $V_{i,0}$. Equivalently, $V_{i,n}$ is the component of $\VVV_{np}$ contained in $V_{i,0}$.

Let $D_{i,0}$ be the component of $\cbar\sm\lambda_i$ containing $V_{i,0}$. Then $f^{-p}(D_{i,0})$ has a unique component $D_{i,1}$ containing $V_{i,1}$, and  $\ov{{D}_{i,1}}\subset D_{i,0}$ since $\lambda_i$ is non-exact. Inductively, for each $n\ge 1$, $f^{-p}(D_{i,n})$ has a component $D_{i,n+1}$ containing $V_{i,n+1}$, and $\overline{D_{i,n+1}}\subset D_{i,n}$.  By Corollary \ref{cor:monotone},
$$
K_i:=\bigcap_{n\ge 1}D_{i,n}
$$
is a complex-type continuum. Moreover, it is a stable set of $f^p$ by Lemma \ref{lem:new-stable}, and $K_i\not\supset J_f$ since $\lambda_i$ is disjoint from $\ov{D_{i,k}}$ for a sufficiently large integer $k$. Thus, $\partial K_i\subset J_f$ by Lemma \ref{lem:boundary}.

Let $r\in[1,p]$ be the smallest integer such that $K_0=K_r$.  From the above construction, we obtain that $K_{i+1}=f(K_i)$ and $K_{i+r}=K_i$ for every $i\in\{0,\ldots,p-1\}$. Then each of $K_0,\ldots,K_{r-1}$ is a stable set of $f^r$, and $r$ is a factor of $p$. Moreover, $K_0,\ldots,K_{r-1}$ are pairwise distinct. In order to obtain a stable set of $f$, we need to consider the intersections of $K_i$ with $K_j$.

\begin{proposition}\label{prop:intersection}
Suppose $K_i\cap K_j\neq\emptyset$ for distinct $i,j\in\{0,\ldots,r-1\}$. Then
\begin{enumerate}
\item $\lambda_j\subset D_{i,0}$ and $\lambda_i\subset D_{j,0}$;
\item $V_{i,n}\cup V_{j,n}\subset D_{i,n}\cap D_{j,n}$ for all $n\ge0$; and
\item if $K_\ell$ intersects $K_i$ for some $\ell\in\{0,\ldots,r-1\}$, then $K_\ell$ also intersects $K_j$.
\end{enumerate}
\end{proposition}

\begin{proof}
We first claim that $D_{i,n}\nsubseteq D_{j,0}$ for any $n\geq 0$. Assume, by contradiction,  that $D_{i,m}\subseteq D_{j,0}$ for some $m\geq 0$. Then, for all $n\geq 1$, $D_{i,m+n}$ lies in a component of $f^{-np}(D_{j,0})$. This component must be $D_{j,n}$, for otherwise, it would contradict the condition that $K_i\cap K_j\neq \emptyset$. Therefore, we have $D_{i,m+n}\subset D_{j,n}$ for all $n$. This implies $K_i\subset K_j$.

Since $\deg(f^p|_{{K_i}})=\deg(f^p|_{{K_j}})$ and  both $K_i$ and $K_j$ are stable sets of $f^p$, we have
$$\bigcup_{n>0}(f^p|_{K_j})^{-n}(K_i)=K_i.$$
 Furthermore, since $f^p:\partial K_j\to \partial K_j$ is quasi-conformally conjugate to the restriction of a rational map on its Julia set (Theorem \ref{thm:renorm}), the set $\bigcup_{k>0}(f^p|_{K_j})^{-k}(\partial K_i)$ is dense in $\partial K_j$. This implies  $\partial K_i=\partial K_j$.
 Then each component of $K_j\setminus K_i$, if it exists, would be a Fatou domain. However, since $\overline{D_{i,n+1}}\subset D_{i,n}$, no component of $\partial D_{i,n}$ for any $n\geq0$ forms the boundary of a Fatou domain in $K_j\setminus K_i$. Thus, $K_i=K_j$.
 The claim is proved.\vspace{2pt}
	
(1) Since $K_i\cap K_j\neq\emptyset$, we have  either $D_{i,0}\subset D_{j,0}$, or $D_{j,0}\subset D_{i,0}$, or $\lambda_j\subset D_{i,0}$ and $\lambda_i\subset D_{j,0}$. Then statement (1) follows directly from the above claim by setting $n=0$.
\vspace{2pt}
	
(2) It suffices to show that $V_{i,n}\subset D_{j,n}$ for all $n\geq0$. By statement (1), we have $V_{i,0}\subset D_{j,0}$. Consequently, for each $n>0$, either $V_{i,n}\subset D_{j,n}$ or $V_{i,n}\cap D_{j,n}=\emptyset$.
If $V_{i,n}\cap D_{j,n}=\emptyset$ for some $n>0$, according to the construction of $V_{i,n}$ and $D_{j,n}$, there exists a component $\eta$ of $\partial D_{j,n}$ that separates $D_{j,n}$ from $V_{i,n}$. In particular, $\eta$ separates $D_{j,n}$ from $\lambda_i$. By statement (1), it follows that $D_{j,n}\subset D_{i,0}$, which contradicts the claim above.
\vspace{2pt}

(3) Without loss of generality, we assume that $K_\ell$ is distinct from both $K_i$ and $K_j$. Then by applying statement (2) to $\{K_i,K_j\}$ and $\{K_i,K_\ell\}$, we obtain that $V_{i,n}\subset D_{j,n}\cap D_{\ell,n}$ for all $n>0$. This implies $K_j\cap K_\ell\not=\emptyset$.
\end{proof}



Let $s\in[1,r]$ be the smallest integer such that $K_0\cap K_s\not=\emptyset$. Then $s$ is a factor of $r$.  Set $Z:=\{ks:0\leq k <r/s\}$. By Proposition \ref{prop:intersection}\,(3), we have
\begin{itemize}
\item [(a)] $K_i\cap K_j\not=\emptyset$ for any pair $i,j\in Z$; and
\item [(b)] $K_i\cap K_\ell =\emptyset$ if $i\in Z$ and $\ell\in\{0,\ldots,r-1\}\setminus Z$.
\end{itemize}

 Let $D_0$ be the intersection of all $D_{i,0}$ with $i\in Z$. Applying Proposition \ref{prop:intersection}\,(1) to each pair $\{K_i,K_j\}$ with distinct $i,j\in Z$,  we conclude that $D_0$ is the domain with boundary components $\{\lambda_i:i\in Z\}$, and $V_{i,0}\subset D_0$ for every $i\in Z$.

 For every $n\geq1$,  denote by $D_{n}$  the component of $f^{-pn}(D_0)$ containing $V_{0,n}$. By point (a) above and Proposition \ref{prop:intersection}\,(2),
 it holds that $\bigcup_{i\in Z}V_{i,n}\subset \bigcap_{i\in Z} D_{i,n}$ for every $n\geq0$. Moreover, since
$f^{np}\big(\bigcap_{i\in  Z}D_{i,n}\big)\subset \bigcap_{i\in  Z}D_{i,0}=D_0$ and $f^{np}(D_{n})=D_0$, it follows that
$$\bigcup_{i\in Z}V_{i,n}\subset \bigcap_{i\in  Z}D_{i,n}\subset D_n$$ for all $n\geq0$.
This inclusion also implies $D_n\subset D_{i,n}$ for any $i\in Z$ and $n\geq0$. Thus
\begin{itemize}
	\item [(c)] for every $n\geq 0$, the equality $\bigcap_{i\in  Z}D_{i,n}= D_n$ holds.
\end{itemize}

This equality implies $\overline{{D}_{n_1}}\subset D_{n_2}$ for sufficiently large $n_2-n_1$.  Then $$E:=\bigcap_{n\geq0}D_{n}=\bigcap_{n\geq0}\ov{D_{n}}$$ is a stable set of $f^p$ by Lemma \ref{lem:new-stable}. Moreover, $\partial D_n$ is disjoint from $\KKK_0$ for every sufficiently large integer $n$. Thus, each component of $\KKK_0$ is either contained in $E$ or disjoint from $E$.  Since $E$ contains $\bigcap_{n\geq0} \ov{V_{0,n}}$, it follows from Corollary \ref{cor:monotone} that $E$ is complex-type. Additionally, since $\lambda_0=\lambda\subset J_f$ is disjoint from $E$, we have  $J_f\not\subset E$.

Finally, point (c)  implies  $E=\bigcap_{i\in Z} K_i$. Therefore, $f^s(E)\subset E$, and hence $E$ is also a stable set of $f^s$.
Combining this with point (b) above,  we deduce that $E,f(E),\ldots, f^{s-1}(E)$ are pairwise disjoint.  Thus, $\KKK':=\bigcup_{i=0}^{s-1} f^i(E)$ is a stable set of $f$ and satisfies all the conditions of Lemma \ref{lem:new} according to the previous discussion.
\end{proof}

\subsection{Cluster-exact decomposition}
Let $(f,P)$ be a marked rational map. A continuum $K\subset J_f$ is called a {\bf cluster} if it is a stable set of $f^p$ for some  $p\ge 1$, and the renormalization of $f^p$ on $K$ is a cluster rational map, i.e., the sphere is a Fatou chain of this rational map.

\begin{theorem}[Cluster-exact decomposition]\label{thm:cluster-exact0}
Let $(f,P)$ be a marked rational map with $J_f\neq\cbar$, and let $\MMM_f$ be the intersection of $J_f$ with the union of all maximal Fatou chains of $f$ intersecting $P$. Then there exists a stable set $\KKK$ of $f$ with $\MMM_f\subset \KKK\subset J_f$ such that 
\begin{enumerate}
\item every periodic component of $\KKK$ is a cluster; and
\item $\KKK$ induces an exact decomposition of $(f,P)$.
\end{enumerate}
\noindent Moreover, each component of $\KKK$ intersects or separates $P$. 
\end{theorem}

\begin{proof}
If $J_f= \MMM_f$,  the theorem holds by taking $\KKK=J_f$. Thus, we assume $\MMM_f\subsetneq J_f$. Note that $\MMM_f$ is a stable set of $f$.
Then by applying Theorem \ref{thm:decomposition} to $\KKK_0=\MMM_f$, we obtain a stable set $\KKK_1$ with $ \MMM_f\subset \KKK_1\subsetneq J_f$ such that $\KKK_1$ induces an exact decomposition of $(f,P)$, and each component of $\KKK_1$ intersects or separates $P$. 

If every periodic component of $\KKK_1$ is a cluster, the theorem holds by taking $\KKK=\KKK_1$.

Now, suppose that $K_*$ is a periodic component of $\KKK_1$ with period $p\ge 1$ such that $K_*$ is not a cluster. By Theorem \ref{thm:renorm}, there exist a marked rational map $(g,Q)$ and a quasiconformal map $\phi$ of $\cbar$, such that $J_{g}=\phi(K_*)$ and $\phi\circ f^p=g\circ\phi$ on $K_*$.
Here, $Q$ is the union of $\phi(P\cap K_*)$ together with all centers of Fatou domains $U$ of $g$ such that $\phi^{-1}(U)$ contains a point of $P$. In particular, $g$ is not a cluster rational map.

 As before, we can define $\MMM_g$ for $(g,Q)$. Then $\MMM_g\subsetneq J_g$.
By applying Theorem \ref{thm:decomposition} to $(g,Q)$ and $\MMM_g$, we obtain a stable set $\KKK_g$ of $g$ with $\MMM_g\subset \KKK_g\subsetneq J_g$ such that $\KKK_g$ induces an exact decomposition of $(g,Q)$, and each component of $\KKK_g$ intersects or separates $Q$. Set $\EEE=\phi^{-1}(\KKK_g)$. Then $\EEE\subsetneq K_*$ is a stable set of $f^p$, and  we have the following commutative diagram:
\begin{equation}\label{eq:commutative}
\begin{array}{ccc}
(K_*,\EEE) &\xrightarrow[]{\ \ f^p\ \ }  & (K_*,\EEE) \\
\phi\Big\downarrow &&\Big\downarrow \phi  \\
(J_g,\KKK_g) &  \xrightarrow[]{\ \ g\ \ } & (J_g,\KKK_g)\vspace{-0.1cm}.
\end{array}
\end{equation}
\vspace{1mm}

From the choice of $Q$, it follows that each component of $\EEE$ intersects or separates $P$. It is worth noting that $ \MMM_f\cap K_*$ is also a stable set of $f^p$.

For any continuum $E\subset \cbar$,  denote by $\wh{E}$ the union of $E$ and all components of $\ov{\C}\setminus E$ disjoint from $P$.

\begin{proposition}\label{prop:B}
Both $\MMM_f\cap K_*$ and $\partial\wh{K}_*$ are contained in $\EEE$.
\end{proposition}

\begin{proof}
It suffices to prove that $\phi(\MMM_f\cap K_*)$ and $\phi(\partial\wh K_*)$ are contained in $\MMM_g\, (\subset \KKK_g)$. Recall that $\phi$ sends a  component of $\ov{\C}\setminus K_*$ onto a Fatou domain of $g$.

Let $B$ be a marked maximal Fatou chain of $(f,P)$ such that $\partial B$ is a component of $\MMM_f$ contained in $K_*$. Note that each component of $B\setminus \partial B$ is a Fatou domain of $f$, and hence a  component of $\cbar\setminus K_*$. This implies that $\phi(B)$ lies in a marked maximal Fatou chain of $(g,Q)$. Hence,  $\phi(\partial B)=\partial\phi(B)\subset \MMM_g$.

For any point $z\in\partial \wh K_*$, there exists a component $D$ of $\cbar\setminus \wh K_*$ with $z\in\partial D$, and such a $D$ must intersect $P$.  Then  $\phi(\partial D)$ is the boundary of a marked Fatou domain of $(g,Q)$. It follows immediately that $\phi(z)\in\MMM_g$.
\end{proof}

Let $K_1,\ldots,K_m$ be all components of $\KKK_1$ whose orbits pass through $K_*$. For each $K_i$, there exists a smallest integer $k_i\ge 0$ such that $f^{k_i}(K_i)=K_*$. Thus, $K_i$ is a component of $f^{-k_i}(K_*)$. Let $\EEE_i$ denote the union of all components of $f^{-k_i}(\EEE)\cap K_i$ that either intersect or separate $P$. Then both $ \MMM_f\cap K_i$ and $\partial\wh{K}_i$ are contained in $\EEE_i$ for each $i\in\{1,\ldots,m\}$ by Proposition \ref{prop:B}. Set
$$
\KKK_2=\bigg(\KKK_1\sm\bigcup_{i=1}^m K_i\bigg)\cup\bigcup_{i=1}^{m}\EEE_i.
$$
The previous discussion shows that $\KKK_2$ is a stable set of $f$ with $ \MMM_f\subset\KKK_2\subsetneq J_f$, and each component of $\KKK_2$ intersects or separates $P$. Moreover, it holds that
\begin{equation}\label{eq:444}
\bigcup_{K\in{\rm\, Comp}(\KKK_1)}\partial \wh K\ \subset\ \KKK_2\subsetneq\KKK_1.
\end{equation}

\begin{proposition}\label{prop:also-exact}
The stable set $\KKK_2$ induces an exact decomposition of $(f,P)$.
\end{proposition}

\begin{proof}
Suppose that $\BBB$ is a stable set of $f$. From the definitions, we deduce the following:
\begin{enumerate}
\item The stable set $\BBB$ induces an exact decomposition of $(f, P)$ if and only if, for any complex-type component $V$ of $\ov{\C}\sm \mathcal{B}$, whenever a component $B_1$ of $f^{-1}(\mathcal{B})$  lies in $V$, it neither intersects nor separates $P$;
\item  For any component $B$ of $\mathcal{B}$, a component $B_1$ of $f^{-1}(\mathcal{B})$ that intersects $\wh{B}$ is either equal to $B$ or contained in a component of $\wh{B}\sm B$, which is simply connected and avoids $P$.
\end{enumerate}

We shall use statement (1) to prove this proposition.

Let $V$ be any complex-type component of $\cbar\setminus \KKK_2$.	
By the construction of $\KKK_2$ and the inclusion relation \eqref{eq:444}, the domain $V$ is either  a  complex-type component of $\ov{\C}\sm\KKK_1$ or a complex-type component of $\wh{K}_i\setminus \EEE_i$ for some $i\in\{1,\ldots,m\}$.

Let $E$ be a component of $f^{-1}(\KKK_2)$ that lies in $V$. 
 Since $\KKK_2\subset \KKK_1$, the continuum $E$ is contained in a component of $f^{-1}(\KKK_1)$, denoted by $K(E)$. The purpose is to verify that $E$ neither intersects nor separates $P$.
	
Case 1. The domain $V$ is also a component of $\ov{\C}\sm \KKK_1$. 
Since $\KKK_1$ induces an exact decomposition of $(f,P)$, by statement (1) above, $K(E)$ neither intersects nor separates $P$. So does $E$.

Case 2. The domain $V$ is a complex-type component of $\wh{K}_i\setminus \EEE_i$ for some $1\leq i\leq m$.
In this case, $K(E)$ intersects $\wh K_i$.
Then by statement (2),  either $K(E)=K_i$, or $K(E)$ is contained in a component $D$ of $\wh K_i\setminus K_i$. The domain $D$ is simply connected and disjoint from $P$. Moreover, we have $D\subset V$ since $E\subset V$. Thus, it suffices to consider the former case.

 The equality $K(E)=K_i$ implies that $E\subset K_i$ and $f(E)\subset f(K_i)=K_{j}$ for some $j$. Thus, $f(E)$ is a component of $\EEE_j\subset K_j$. 
 Since $E\subset V $ is disjoint from $\EEE_i$, by the definition of $\EEE_i$, exactly one of the following two situations occurs:
 \begin{itemize}
\item $K_i\not= K_*$, and  $E$ neither intersects nor separates $P$;
\item $K_i=K_*$, and $E$ is a component of $(f^p|_{K_*})^{-1}(\EEE)$ that lies in $V$.
\end{itemize}
\indent Thus, it suffices to deal with the second situation.

By the commutative diagram \eqref{eq:commutative}, $\phi(E)$ is a component of $g^{-1}(\KKK_g)$.
Note also that $\phi(V)$ is a complex-type component of $\ov{\C}\sm \KKK_g$. Since $\KKK_g$ induces an exact decomposition of $(g, Q)$, it follows from statement (1) that $\phi(E)$  neither intersects nor separates $Q$. Thus, $E$ neither intersects nor separates $P$.
\end{proof}

By Proposition \ref{prop:also-exact}, if every periodic component of $\KKK_2$ is a cluster, then Theorem \ref{thm:cluster-exact0} holds by choosing $\KKK=\KKK_2$. Otherwise, we can repeat the above argument by replacing $\KKK_1$ with $\KKK_2$ and obtain a 
 stable set  $\KKK_3$ with $\MMM_f\subset\KKK_3\subsetneq\KKK_2$ such that $\KKK_3$ induces an exact decomposition of $(f,P)$, and each component of $\KKK_3$ intersects or separates $P$. 

By iterating this process, we obtain a sequence of stable sets $\{\KKK_n\}$ with $ \MMM_f\subset\KKK_{n}\subsetneq\KKK_{n-1}$. This process must stop after a finite number of steps by Lemma \ref{lem:nest}. This completes the proof of Theorem \ref{thm:cluster-exact0}.
\end{proof}

The subsequent corollary of Theorem \ref{thm:cluster-exact0} will be used in Section 8.
\begin{corollary}\label{cor:desired}
	Let $(f, P)$ be a marked rational map with $J_f\neq\cbar$. Then there exist an $f$-invariant and finite set $P'\supset P$ and a stable set $\KKK'\subset J_f$ such that 
	\begin{enumerate}
	
\item  the stable set $\KKK'$ induces a cluster-exact decomposition of $(f,P')$, and each of its components intersects $P'$;
	
\item every complex-type component of $\cbar\setminus \KKK'$ rel $P'$ is disjoint from attracting cycles of $f$;
	
\item every simple-type component of $\cbar\setminus\KKK'$ rel $P'$ is a simply connected domain; and 
	
\item every annular-type component $A$ of $\cbar\setminus\KKK'$ rel $P'$ is an annulus, and moreover, if $A\cap f^{-1}(\KKK')\neq \emptyset$, then $A$ contains an annular-type component of $f^{-1}(\KKK')$. 
	\end{enumerate}
\end{corollary}
\begin{proof}
	Let $\KKK$ be the stable set obtained in Theorem \ref{thm:cluster-exact0}.
	Consider a finite and $f$-invariant set $Q_0\subset \KKK$ such that each component of $\KKK$ contains at least two points of $Q_0$. It is important to note that the complex-type components of $\cbar\setminus\KKK$ rel $P$ coincide with those rel $P\cup Q_0$.
Hence, items (1)--(3) and the former part of $(4)$ hold for the stable set $\KKK$ rel $P\cup Q_0$.
	 
	 If the latter part of item (4) is false for an annular-type component $A$ of $\cbar\setminus \KKK$ rel $P\cup Q_0$,  let $K_{A}$ be a component of $f^{-1}(\KKK)\cap A$.  We can select two points from $f^{-1}(Q_0)$ within $K_{A}$ and denote by $Q_1$ the union of these two points  with $Q_0$. Then the stable set $\KKK_1:=\KKK\cup K_{A}$ satisfies items (1)--(3) and the former part of $(4)$ rel $P\cup Q_1$.
	 Moreover, the number of annular-type components of $\cbar\setminus \KKK_1$ rel $P\cup Q_1$ is bounded above by that of $\cbar\setminus \KKK$ rel $P\cup Q_0$.

	  If  the latter part of item (4) is still false for $\KKK_1$ rel $P\cup Q_1$, we can repeat the argument above, replacing $\KKK$ and $Q_0$ with $\KKK_1$ and $Q_1$, respectively. Thus, we obtain a sequence of stable sets $\{\KKK_n\}$ and a sequence of  $f$-invariant finite sets $\{Q_n\}$ such that $\KKK_n$ satisfies items (1)--(3) and the former part of $(4)$ rel $P\cup Q_n$, and the number of annular-type  components of $\cbar\setminus \KKK_n$ rel $P\cup Q_n$ is strictly decreasing as $n$ increases.
	  Consequently, this process must stop after $N$ steps for an integer $N\geq0$. Then $\KKK'=\KKK_N$ and $P'=P\cup Q_N$ satisfy items (1)--(4).
\end{proof}

\section{Blow-up of an exact sub-system}\label{sec:5}
In this section, we will prove Theorem \ref{thm:blow-up} and complete the proof of Theorem \ref{thm:cluster-exact}.

Throughout this section, let $(f,P)$ be a marked rational map, and let $V\subset\cbar$ be a domain such that $\partial V\subset J_f$ consists of finitely many pairwise disjoint continua. We also assume that $f:V_1\to V$ is an exact sub-system of $(f,P)$, i.e.,  $V_1$ is a component of $f^{-1}(V)$ contained in $V$, and each component of $V\setminus V_1$ is a full continuum disjoint from $ P $.

For two topological spaces $X$ and $Y$, a {\bf homotopy} from $X$ to $Y$ is a continuous map $\xi:X\times [0,1]\to Y$. We usually write the homotopy as $\{\xi_t\}_{t\in[0,1]}$.

\subsection{Construction of the blow-up map}\label{sec:blow-up}
Let $\lambda$ be a component of $\partial V$. Since $V\setminus V_1$ is compact, we have $\lambda\subset\partial V\subset\partial V_1$.  Thus, $f(\lambda)$ is also a component of $\partial V$. Let $E_{\lambda}$ be the component of $\cbar\setminus V$ containing $\lambda$. If $E_{f(\lambda)}$ is disjoint from $ P $, then $f(E_{\lambda})=E_{f(\lambda)}$, and $E_{\lambda}$ is also disjoint from $ P $.

Let $\lambda$ be a periodic component of $\partial V$ with period $p\ge 1$. Since $f$ is  expanding in a neighborhood of $J_f$ under the orbifold metric, there exists an annulus $A\subset V\setminus  P $ such that $\lambda$ is a component of $\partial A$, and $\ov{A_1}\subset A\cup\lambda$, where $A_1$ is the component of $f^{-p}(A)$ with $\lambda\subset\partial A_1$. A folklore argument implies that $E_{\lambda}$ is locally connected and $E_{\lambda}\cap  P \neq\emptyset$.
Since each component $\lambda$ of $\partial V$ is eventually periodic, it follows that each component of $\cbar\sm V$ is locally connected.

\vskip 0.1cm
Now, we begin to construct the blow-up map. Let $\chi$ be a conformal map from $V$ onto a circular domain $\hat\Omega\subset\cbar$, i.e., each component of $\cbar\sm\hat\Omega$ is a closed round disk in $\C$. Let $\hat\Omega_1:=\chi(V_1)$. Then
$$
\hat{g}:=\chi\circ f\circ\chi^{-1}: \hat{\Omega}_1\to\hat\Omega
$$
is a holomorphic and proper map, which can be continuously extended to $\partial\hat\Omega$ such that $\hat{g}(\partial\hat\Omega)\subset\partial\hat\Omega$. By the symmetry principle and the expanding property of $f$, the map $\hat g$ is holomorphic and expanding in a neighborhood of $\partial\hat\Omega$.

Denote $\hat\DDD=\cbar\setminus\hat\Omega$. Define a map $\wp: \hat\DDD\to\hat\DDD$ by $\wp(\hat{D}_i)=\hat{D}_j$ if $\hat{g}(\partial \hat{D}_i)=\partial \hat{D}_j$, where $\hat{D}_i$ and $\hat{D}_j$ are components of $\hat{\DDD}$, and
$$
\wp(z)=r_j\left(\frac{z-a_i}{r_i}\right)^{d_i}+a_j\quad\text{ if }z\in \hat{D}_i,
$$
where $a_i$ and $r_i$ are the center and the radius of the closed round disk $\hat{D}_i$, respectively, and $d_i=\deg(\hat{g}|_{{\partial \hat{D}_i}})$. Since $\hat{g}$ is expanding on $\partial\hat{\Omega}=\partial\hat{\DDD}$, if $\partial \hat{D}_i$ is periodic with period $p_i\ge 1$, then there exists a quasi-symmetric map $w_i: \partial \hat{D}_i\to\partial \hat{D}_i$ such that $\wp^{p_i}\circ w_i=w_i\circ {\hat{g}}^{p_i}$ on $\partial \hat{D}_i$. By pullback, we obtain a quasi-symmetric map $w:\partial\hat{\Omega}\to\partial\hat{\Omega}$ such that
$$
\wp\circ w=w\circ {\hat{g}}\quad \text{ on }\partial\hat{\Omega}.
$$

Consider the conformal welding induced by $w$. There exist two conformal maps $\zeta: \hat{\Omega}\to\tilde{\Omega}\subset \cbar\textup{ and }\eta: {\rm int}(\hat{\DDD})\to {\rm int}(\tilde{\DDD})$ such that $\zeta=\eta\circ w$ on $\partial\hat{\Omega}$, where $\tilde{\DDD}:=\cbar\setminus\tilde{\Omega}$, and the notation ${\rm int}(\cdot)$ represents the interior of the corresponding set. Define
$$
\tilde{g}_0:=\begin{cases}
	\zeta\circ\hat{g}\circ\zeta^{-1} & \text{ on } \zeta(\hat{\Omega}_1)\subset \tilde{\Omega}, \\
	\eta\circ\wp\circ\eta^{-1}&\text{ on }\eta(\hat{\DDD})=\tilde{\DDD}.
\end{cases}
$$
Then $\tilde{g}_0$ is a holomorphic map on $\zeta(\hat{\Omega}_1)\cup\eta(\hat{\DDD})$. Set $\xi_0:=\chi^{-1}\circ\zeta^{-1}: \tilde{\Omega}\to V$, and continuously extend it to a quotient map (defined in \ref{app:2}) of $\cbar$, due to the local connectivity of $\partial V$. Then
$$
\xi_0\circ\tilde{g}_0=f\circ\xi_0\quad\text{ on }\tilde{\Omega}_1^*:=\zeta(\hat{\Omega}_1).
$$

For each $n\geq1$, set $V_n:=(f|_{V_1})^{-1}(V)$. Then $f:V_{n+1}\to V_n$ is an exact sub-system for each $n\geq1$.
By replacing $V$ with some $V_n$, we may assume that $V\sm V_1$ is disjoint from $f^{-1}( P )$. This means that $f$ sends a neighborhood of each component of $V\sm V_1$ homeomorphically onto a neighborhood of a complementary component of $V$.

For each component of $\tilde{\Omega}\sm\tilde{\Omega}_1^*$, we pick a small disk in $\tilde{\Omega}\sm\xi_0^{-1}( P \cap V)$ as a neighborhood of this component, such that these disks have pairwise disjoint closures. Let $\NNN$ denote their union. Then $\tilde{g}_0$ is injective on $\partial{\NNN}$.  Define a new map $\tilde{g}:\cbar\to\cbar$ such that $\tilde{g}$ is continuous and injective on $\NNN$, and $ \tilde{g}(z)=\tilde{g}_0(z)$ for all  $z\in\cbar\sm\NNN$. 

It is easy to verify that $ \tilde{g}$ is a   PCF  branched covering with ${\rm deg}(\tilde{g})=\deg f|_{ V_1}$ and it is holomorphic on $\cbar\sm \ov{\NNN}$. Note that the interior of each component $\tilde{D}$ of $\cbar\sm\tilde{\Omega}$ contains a unique eventually periodic point $z(\tilde{D})$ of $ \tilde{g}$. Set 
\[\tilde{Z}=\{z(\tilde{D}): \xi_0(\tilde{D})\cap  P \neq \emptyset\}\quad \text{ and }\quad \tilde{Q}=\xi_0^{-1}( P \cap V)\cup \tilde{Z}.\]
It follows that $\tilde{g}(\tilde{Z})\subset \tilde{Z}$, $\tilde{g}(\tilde{Q})\subset \tilde{Q}$, and $P_{\tilde{g}}\subset \tilde{Q}$.

\begin{figure}[http]
\centering
\begin{tikzpicture}
	\node at (0,0){\includegraphics[width=13cm]{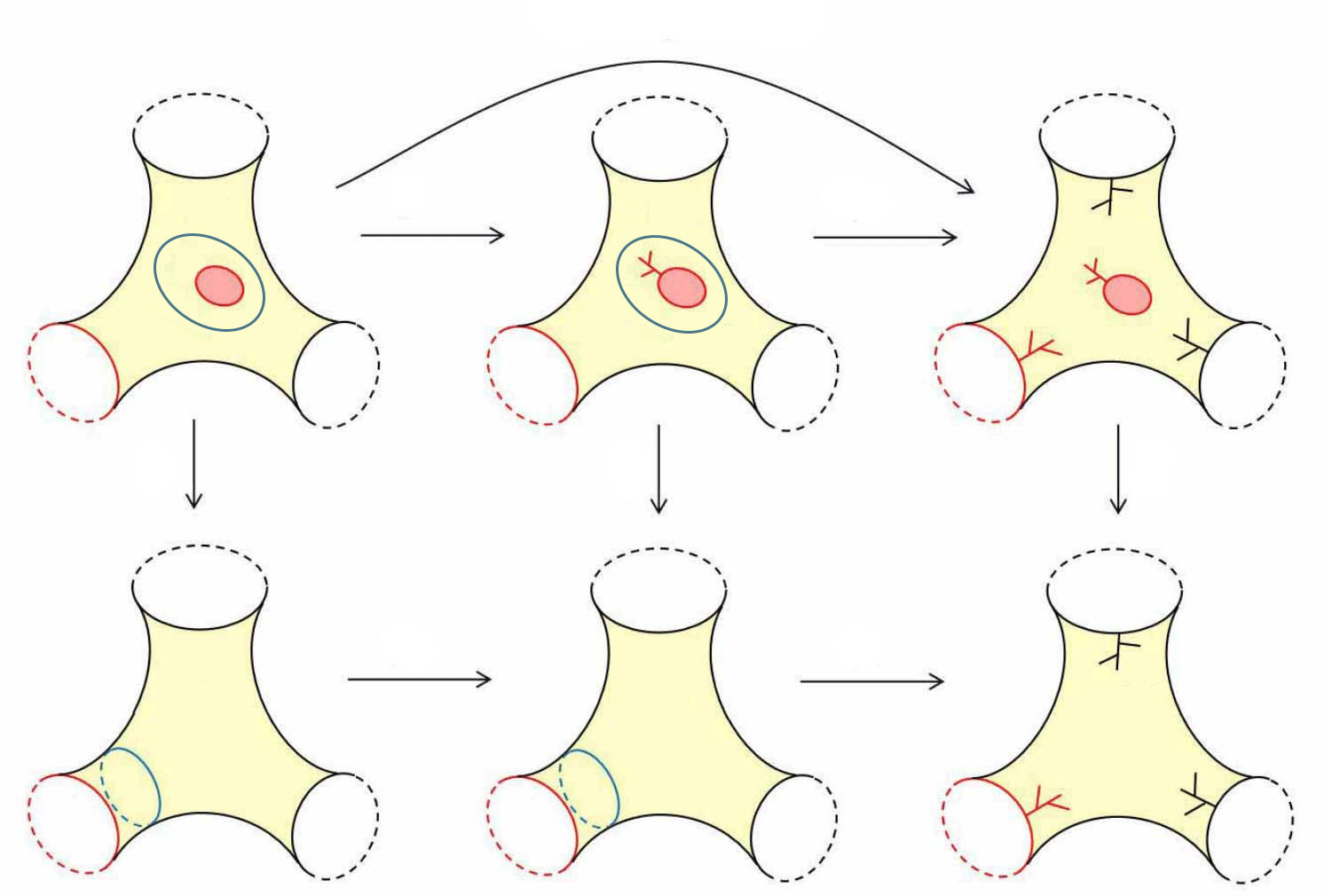}};
	\node at (-2.5,2.45){$\theta$};
	\node at (2, 2.45){$\xi_0$};
	\node at (-2.5,-2){$\textup{id}$};
	\node at (2,-2){$\xi_0$};
	\node at (0, 4.25){$\xi_1=\xi_0\circ\theta$};
	\node at (4.8, -0.2){$f$};
	\node at (-0.3, -0.2){$\tilde{g}_0$};
	\node at (-5,-0.2){$\tilde{g}$};
	\node at (4.75, 2.1){$V_1$};
	\node at (4.5, -2.5){$V$};
	\node at (0.1,2.35){$\tilde{\Omega}_1^*$};
	\node at (0.1, -2.5){$\tilde\Omega$};
	\node at (-4.5,2.3){$\tilde{\Omega}_1$};
	\node at (-4.5, -2.5){$\tilde{\Omega}$};
	\node at (0.1,1.7){$\mathcal{N}$};
	\node at (0.15,-3.25){$\tilde{g}_0(\partial\mathcal{N})$};
\end{tikzpicture}
	\caption{The construction of $\tilde{g}_0$, $\tilde{g}$, $\xi_0$, and $\xi_1$.}\label{fig:xi}
\end{figure}
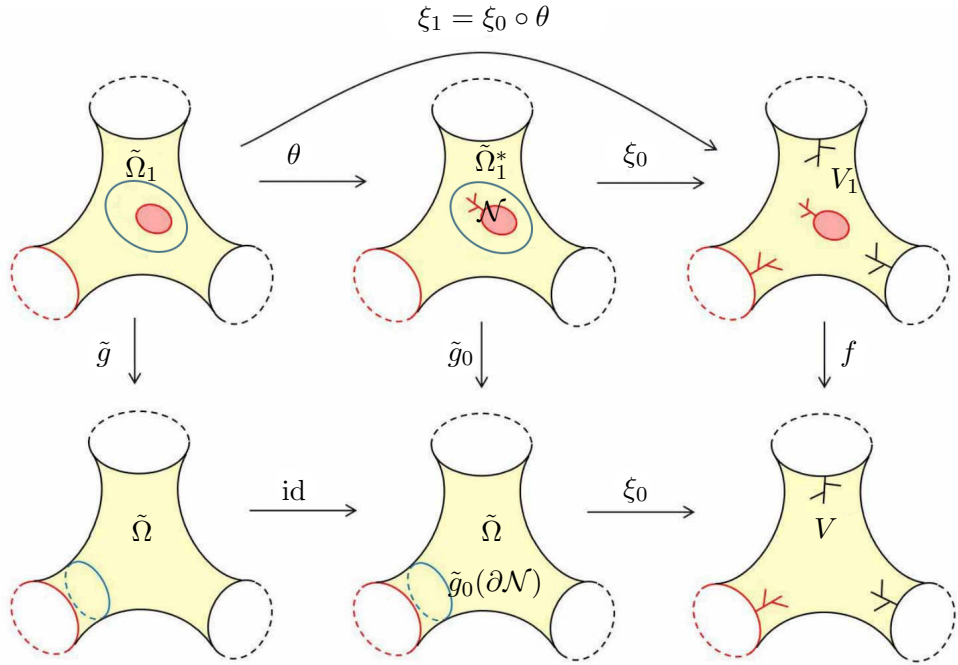

Denote $\tilde{\Omega}_1=\tilde{g}^{-1}(\tilde{\Omega})$. Then $\tilde{\Omega}\sm \tilde{\Omega}_1$ consists of pairwise disjoint  closed disks in $\NNN$. Moreover, by lifting, there exists a homeomorphism $\theta:\tilde{\Omega}_1\to\tilde{\Omega}_1^*$ such that $\theta=id$ on $\tilde{\Omega}\sm\NNN$ and $\tilde{g}=\tilde{g}_0\circ\theta$ on $\tilde{\Omega}_1$; see Figure \ref{fig:xi}. Since each component of $\partial\tilde{\Omega}_1$ is a Jordan curve and $ \tilde{g}$ is injective on $\partial\tilde{\Omega}_1\sm\partial\tilde{\Omega}$, we can continuously extend $\theta$ to a quotient map of $\cbar$. This extended map, still denoted by $\theta$, sends $\tilde{\Omega}\sm\tilde{\Omega}_1$ onto $\tilde{\Omega}\sm\tilde{\Omega}_1^*$.

Define $\xi_1:=\xi_0\circ\theta$. Then $\xi_1$ is a quotient map of $\cbar$ such that $\xi_1(\tilde{\Omega}_1)=V_1$, $\xi_1=\xi_0$ on $\cbar\sm\NNN$, and
$$
\xi_0\circ \tilde{g}=f\circ\xi_1\quad\text{ on }\tilde{\Omega}_1.
$$
Moreover, there exists a homotopy $\xi_t:\cbar\to\cbar$, $t\in[0,1]$, such that $\xi_t$ is a quotient map of $\cbar$ and $\xi_t(z)=\xi_0(z)$ for all $z\in\cbar\sm\NNN$ and $t\in [0,1]$. In particular, $\xi_t(\tilde{Q}\cap \tilde{\Omega})= P \cap V$.

Since $ \tilde{g}: \tilde{\Omega}_1\setminus  \tilde{g}^{-1}(\tilde{Q})\rightarrow \tilde{\Omega}\setminus \tilde{Q}$ and $f: V_{1}\setminus f^{-1}( P )\rightarrow V\setminus  P $ are both coverings, and
 $$\{\xi_t^{-1}(z):t\in[0,1]\}$$ 
 is a singleton in $\tilde{Q}\cap \tilde{\Omega}$ for every $z\in  P \cap V$, the homotopy $\xi_t:\tilde{\Omega}\setminus \tilde{Q}\to V\setminus  P $ can be lifted by $f$ and $ \tilde{g}$ to a homotopy $\xi_t:\tilde{\Omega}_1\setminus  \tilde{g}^{-1}(\tilde{Q})\to V_1\setminus f^{-1}( P )$, $t\in [1,2]$, by the general homotopy lifting theorem; see \cite[Proposition 1.30]{Ha}. Furthermore, this homotopy can be extended to a homotopy
$\xi_t: \cbar\to\cbar$, $t\in[1,2],$ such that  each $\xi_t$ is a quotient map and $\xi_t(z)=\xi_1(z)$ on $\ov\C\setminus  \tilde{g}^{-1}(\NNN)$ for every $t\in[1,2]$.

Inductively applying the above argument, we obtain a sequence of quotient maps $\{\xi_n\}$ of $\cbar$ such that $\xi_n(\tilde{\Omega}_n)=V_n$, $\xi_{n+1}=\xi_{n}$ on $\cbar\sm  \tilde{g}^{-n}(\NNN)$, and
$$
\xi_n\circ \tilde{g}=f\circ\xi_{n+1}\quad\text{ on }\tilde{\Omega}_{n+1},
$$
where $\tilde{\Omega}_n= \tilde{g}^{-n}(\tilde{\Omega})$ and $V_n=(f|_{V_1})^{-n}(V)$.

\begin{proposition}\label{prop:realization}
	The marked branched covering $( \tilde{g}, \tilde{Q})$ is combinatorially equivalent to a marked rational map $(g, Q)$.
\end{proposition}

\begin{proof}
	Let $\Gamma=\{{\g_k}\}$ be a multicurve of $( \tilde{g},\tilde{Q})$. Its transition matrix $(a_{kl})$ is defined by
	$$
	a_{kl}=\sum\frac 1{\deg \tilde{g}:\delta\to{\g}_l},
	$$
	where the summation is taken over all components $\delta$ of $ \tilde{g}^{-1}({\g}_l)$ isotopic to ${\g}_k$ rel $ \tilde{Q}$.
	
	Since each component of $\cbar\setminus\tilde{\Omega}$ contains at most one point of $ \tilde{Q}$, we may require that each curve in $\G$ is contained in $\tilde{\Omega}\sm\NNN$. Thus
	$\xi_0\circ \tilde{g}=f\circ\xi_0$ on $ \tilde{g}^{-1}({\g_k})$ for each ${\g}_k\in{\G}$. By the choice of $ \tilde{Q}$, the collection of curves $\xi_0(\G)=\{\xi_0({\g}_k)\}$ forms a multicurve of the rational map $f$. Moreover, each entry of the transition matrix of $\xi_0(\G)$ under $f$ is greater than or equal to the corresponding entry of the transition matrix of $\Gamma$ under $( \tilde{g}, \tilde{Q})$.  Then $( \tilde{g},  \tilde{Q})$ has no Thurston obstruction since $f$ has no Thurston obstruction by \cite[Theorem 2.3]{BCT}. Therefore, $( \tilde{g},  \tilde{Q})$ is combinatorially equivalent to a marked rational map $(g,Q)$
	by \cite[Theorem 2.1]{BCT}.
\end{proof}

\subsection{Dynamics of the blow-up map}
According to Proposition \ref{prop:realization}, there exists an isotopy $\phi_t:\cbar\to\cbar$ rel $ \tilde{Q}$, $t\in[0,1]$ such that $\phi_0( \tilde{Q})=Q$ and $g\circ \phi_1=\phi_0\circ \tilde{g}$ on $\ov{\C}$. Recall that $\tilde{Z}= \tilde{Q}\setminus \tilde{\Omega}$ and set $Z=\phi_0(\tilde{Z})$.

\begin{proposition}\label{prop:disk}
	Each Fatou domain of $g$ with the center in $Z$ is a disk whose boundary is disjoint from $Q$, and any two such Fatou domains  have disjoint closures. In particular, $g$ is a \Sie\ rational map if its attracting periodic points are all contained in $Z$.
\end{proposition}

To prove this proposition, we need a combinatorial criterion to determine whether the boundary of a Fatou domain contains marked points, whether it is a Jordan curve, and whether two Fatou domains have disjoint closures.

\begin{lemma}\label{lem:criterion}
	Let $R$ be a   PCF  rational map, and let $U$ be a periodic Fatou domain of $R$ with center $a$.
	\begin{enumerate}
\item  A repelling periodic point $b$ lies in $\partial U$ if and only if there exists an open arc $\beta\subset\cbar\sm P_{R}$ joining $a$ and $b$, such that $R^{-p}(\beta)$ has a component isotopic to $\beta$ rel $P_{R}$ for some  $p\ge 1$.
\item  Let $U'\subset\cbar$ be another periodic Fatou domain of $R$ with center $a'$. Then $\partial U\cap\partial U'\neq\emptyset$ if and only if there exists an open arc $\beta\subset\cbar\sm P_{R}$ joining $a$ and $a'$, such that $R^{-p}(\beta)$ has a component isotopic to $\beta$ rel $P_{R}$ for some integer $p\ge 1$.
\item  Assume that $\partial U\cap P_{R}=\emptyset$. Then $U$ is not a disk if and only if there exists an open arc $\beta\subset\cbar\sm P_{R}$ that joins $a$ to itself, such that $\ov{\beta}$ separates $P_{R}$, and $R^{-p}(\beta)$ has a component isotopic to $\beta$ rel $P_{R}$ for some integer $p\ge 1$.
	\end{enumerate}
\end{lemma}

\begin{proof}
	(1) If $b\in\partial U$,  the internal ray in $U$ that lands at $b$ satisfies the condition.
	
	Conversely, the arc $\beta$ can be decomposed into two sub-arcs $\beta=\alpha\cup\delta$, such that $\alpha\subset U$ and $\ov{\delta}$ is disjoint from the super-attracting cycles of $R$. By successive lifting, $R^{-kp}(\beta)$ has a component $\beta_k$ isotopic to $\beta$ rel $P_{R}$, and  $\beta_k$ has a  decomposition  $\beta_k=\alpha_k\cup\delta_k$ such that $R^{kp}(\alpha_k)=\alpha$ and $R^{kp}(\delta_k)=\delta$. Observe that $\alpha_k\subset U$, and ${\rm diam}(\delta_k)\to 0$ as $k\to\infty$ by Lemma \ref{lem:expanding}. Hence, $b\in\partial U$.
	
	(2) First, assume that $\partial U\cap\partial U'\neq\emptyset$. We choose an open arc $\beta'$ that joins $a$ and $a'$ and passes through a point $z\in\partial U\cap\partial U'$ such that $\beta'\setminus \{z\}$ consists of two internal rays in $U$ and $U'$, respectively.
	
	If $R^k(z)\notin P_{R}$ for all $k\ge 1$, since $\# P_{R}<\infty$, there exist integers $q,p\ge 1$ such that  $R^{q+p}(\beta')$ is isotopic to $ R^{p}(\beta')$ rel $P_{R}$. Let $\beta=R^{q+p}(\beta')$. Then $R^{-p}(\beta)$ has a component isotopic to $\beta$ rel $P_{R}$.
	
	If $R^k(z)\in P_{R}$ for some integer $k\ge 1$, then by Lemma \ref{lem:finite}, there exist integers $q,p\ge 1$ such that $R^{q+p}(\beta')=R^{q}(\beta')$. Note that $R^q(z)$ is a repelling periodic point in $P_{R}$. Let $\beta$ be an open arc obtained by modifying $R^{q}(\beta')$ in a small neighborhood of the point $R^q(z)$ such that $R^q(z)\notin\beta$. Then $R^{-2p}(\beta)$ has a component isotopic to $\beta$ rel $P_{R}$.
	
	Conversely, we decompose $\beta$ into three sub-arcs $\beta=\alpha\cup\delta\cup\alpha'$, such that $\alpha\subset U$, $\alpha'\subset U'$, and $\ov{\delta}$ is disjoint from the super-attracting cycles of $g$. By successive lifting, $R^{-kp}(\beta)$ has a component $\beta_k$ isotopic to $\beta$ rel $P_{R}$, and $\beta_k$ can be decomposed as $\beta_k=\alpha_k\cup\delta_k\cup\alpha_k'$ such that $R^{kp}(\alpha_k)=\alpha$, $R^{kp}(\de_k)=\delta$, and $R^{kp}(\alpha_k')=\alpha'$. Observe that $\alpha_k\subset U$, $\alpha_k'\subset U'$, and  ${\rm diam}(\delta_k)\to 0$ as $k\to\infty$ by Lemma \ref{lem:expanding}. Thus, $\partial U\cap\partial U'\neq\emptyset$.
	
	(3) First, assume that $U$ is not a disk. Then there exist two internal rays in $U$ landing at a common point $z\in\partial U$. Let $\beta'$ be the union of these two internal rays together with the point $z$. For simplicity, we assume $R(U)=U$. Since $\partial U\cap P_{R}=\emptyset$, it follows that all $R^{k}(\beta')$ are open arcs in $\cbar\setminus P_{R}$ with the same endpoints $a$.
	
	If $\cbar\setminus R^{k+1}(\ov{\beta'})$ has a component $D_{k+1}$ disjoint from $P_{R}$, then $\cbar\setminus R^{k}(\ov{\beta'})$ also has a component $D_k$ disjoint from $P_{R}$, and $R(D_k)=D_{k+1}$. It follows that $R^k(\ov{\beta'})$ separates $P_{R}$ for each sufficiently large integer $k$. Otherwise, there would be a sequence $\{k_n\}$ of integers tending to $\infty$ such that $R^{k_n}(D_1)\cap P_{R}=\emptyset$ for all $n\ge 1$. This is impossible as $D_1\cap J_{R}\neq\emptyset$.
	
	Since $\# P_{R}<\infty$, there exist integers $q,p\ge 1$ such that  $R^{q+p}(\beta')$ is isotopic to $ R^{p}(\beta')$ rel $P_{R}$. Let $\beta=R^{q+p}(\beta')$. Then $R^{-p}(\beta)$ has a component isotopic to $\beta$ rel $P_{R}$.
	
	Conversely, by a similar argument as in the proof of statement (2), we can obtain two distinct internal rays in $U$ with the same landing point. Hence, $U$ is not a disk.
\end{proof}

\begin{proof}[Proof of Proposition \ref{prop:disk}]
	To prove the proposition, it suffices to verify the combinatorial conditions in Lemma \ref{lem:criterion} for the  branched covering $ \tilde{g}$. Let $a\in \tilde{Z}$ be a periodic point of $ \tilde{g}$.
	
	Let $\beta\subset\cbar\sm \tilde{Q}$ be an open arc joining the point $a$ to a repelling periodic point $b\in \tilde{Q}$ that belongs to $\tilde{\Omega}$. Assume, by contradiction,  that $ \tilde{g}^{-p}(\beta)$ has a component $\beta_1$ isotopic to itself rel $\tilde{Q}$ for some integer $p\ge 1$. By  isotopy lifting, $ \tilde{g}^{-kp}(\beta)$ has a component $\beta_{k}$ isotopic to $\beta$ rel $\tilde{Q}$.
	
	We  adjust the arc $\beta$ within its isotopic class so that  $\beta=\alpha\cup\delta$ with $\alpha\subset\cbar\sm\tilde{\Omega}$ and $\delta\subset\tilde{\Omega}$. This allows us to write $\beta_{k}=\alpha_{k}\cup\delta_{k}$ with $\alpha_{k}\subset\cbar\sm \tilde{g}^{-kp}(\tilde{\Omega})$ and $\delta_k\subset \tilde{g}^{-kp}(\tilde{\Omega})$, where $ \tilde{g}^{kp}(\alpha_k)=\alpha$ and $ \tilde{g}^{kp}(\delta_k)=\delta$. In particular, one endpoint of $\delta_k$ lies in $\partial\tilde{\Omega}$ and the other is $b$.
	
	Recall that $\{\xi_n\}$ is a sequence of quotient maps of $\cbar$ such that $\xi_0(\tilde{Q}\cap \tilde{\Omega})= P \cap V$, $\xi_n(\tilde{\Omega}_n)=V_n$, $\xi_{n+1}=\xi_{n}$ on $\cbar\sm\tilde{\Omega}_n$, and
	$$
	\xi_n\circ \tilde{g}=f\circ\xi_{n+1} \quad\text{ on }\tilde{\Omega}_{n+1},
	$$
	where $\tilde{\Omega}_n= \tilde{g}^{-n}(\tilde{\Omega})$ and $V_n=(f|_{V_1})^{-n}(V)$. Thus, $\xi_{kp}(\delta_{k})$ is a component of $f^{-kp}(\xi_0(\delta))$, such that one endpoint of $\xi_{kp}(\delta_k)$ lies in $\partial V$ and the other is $\xi_0(b)$. By Lemma \ref{lem:expanding}, the diameter of $\xi_{kp}(\delta_{k})$ tends to $0$ as $k\to\infty$. It follows that  $\xi_0(b)\in\partial V$, which contradicts the assumption that $b\in\tilde{\Omega}$. Hence, condition (1) holds.
	
	The verification of conditions (2) and (3) is similar. Thus, we omit the details.
\end{proof}

\subsection{Fibers of the semi-conjugacy}
Recall that $\tilde{\DDD}=\cbar\setminus \tilde{\Omega}$ consists of pairwise disjoint closed disks, and 
 $\tilde{g}$ is holomorphic in a neighborhood of $\tilde{\DDD}$ with $\tilde{g}(\tilde{\DDD})\subset \tilde{\DDD}$. Each component of ${\rm int}(\tilde{\DDD})$ contains a unique preperiodic point of $\tilde{g}$.
 Moreover, there exists a small neighborhood $\tilde{\NNN}_a$ of the  attracting cycles of $ \tilde{g}$ that are contained in $\tilde{\Omega}$ such that $ \tilde{g}:\tilde{\NNN}_a\to \tilde{\NNN}_a$ is holomorphic.
 
 Recall also that the marked branched covering $(\tilde{g},\tilde{Q})$ is combinatorially equivalent to a marked rational map $(g,Q)$ by a pair of homeomorphisms $\phi_0,\phi_1$ of $\cbar$, which are connected by an isotopy
  $\{\phi_t\}_{t\in [0, 1]}$  rel $ \tilde{Q}$. 
 
  By Proposition \ref{prop:disk}, the homeomorphism $\phi_0$ sends the preperiodic points of $\tilde g$ in ${\rm int}(\tilde\DDD)$ to the centers  of some Fatou domains of $g$, which are disks with pairwise disjoint closures. Note that the closure $\DDD$ of the union of these Fatou domains is invariant under $g$.

  We may specify the isotopy $\phi_t$ such that $\phi_0$ is holomorphic in  $\tilde{\NNN}_a\cup {\rm int}(\tilde{\DDD})$ with $\phi_0(\tilde{\DDD})=\DDD$, and $\phi_t=\phi_0$ on $\tilde{\NNN}_a\cup\tilde{\DDD}$ for $t\in[0, 1]$.

By successively applying Lemma \ref{lem:lift}, for every $n\geq0$, we have an isotopy $\{\phi_t\}_{t\in[n, n+1]}$ rel $ \tilde{g}^{-n}(\tilde{\DDD}\cup\tilde{\NNN}_a\cup \tilde{Q})$, such that $\phi_{n}\circ  \tilde{g}=g\circ\phi_{n+1}$ on $\ov{\C}$. Set $\Omega_n:=\phi_n(\tilde{\Omega}_n)$.

Recall that in Section \ref{sec:blow-up}, we obtained a homotopy $\{\xi_t\}_{t\in[n,n+1]}$ on $\cbar$ for every $n\geq0$, such that $\xi_n( \tilde{\Omega}_n)=V_n$, $\xi_n=\xi_{n+1}$ on $\cbar\setminus \tilde{\Omega}_n$, and
$
\xi_n\circ \tilde{g}=f\circ\xi_{n+1}\text{ on } \tilde{\Omega}_{n+1},
$
where $ \tilde{\Omega}_n= \tilde{g}^{-n}( \tilde{\Omega})$ and $V_n=(f|_{V_1})^{-n}(V)$. Then we have the following commutative diagram:
\begin{equation*}\label{eq:diagram}
	\xymatrix{
		\Omega_{n+1}\ar[d]_{g} & \ar[l]_{\phi_{n+1}} \tilde{\Omega}_{n+1}\ar[d]_{ \tilde{g}}\ar[r]^{\xi_{n+1}} & V_{n+1}\ar[d]^{f} \\
		\Omega_n &               \ar[l]_{\phi_n} \tilde{\Omega}_n\ar[r]^{\xi_n} & V_n}
\end{equation*}

Set $\BBB_n:=\cbar\sm V_n$, $ \DDD_n:=\cbar\setminus \Omega_n$,  and $\NNN_a:=\phi_0( \tilde{\NNN}_a)$.
Then for every $n\geq0$, the family of maps $\{h_t:=\xi_t\circ \phi_t^{-1}\}_{t\in[n, n+1]}$ is a homotopy on $\cbar$ such that the following conditions hold:
\begin{enumerate}
\item  $h_t(z):\ov\C\to\ov\C$ is a quotient map;
\item  $h_{t}(z)=h_n(z)$ for $z\in \DDD_n\cup g^{-n}(\NNN_a)\cup g^{-n}(Q)$;
\item  $h_t^{-1}(\BBB_n)=\DDD_n$;
\item  $h_{n}\circ g=f\circ h_{n+1}$ on $\Omega_{n+1}$.
\end{enumerate}

\begin{proposition}\label{prop:quotient}
	The sequence of  maps $\{h_n\}$ uniformly converges to a quotient map of $\cbar$.
\end{proposition}

\begin{proof}
	The argument is similar as in \cite[Theorem 1.1]{CPT}. By \cite[Lemma 3.1]{CPT}, the limit of a sequence of quotient maps is still a quotient map. Thus, it suffices to show that there exist constants $M>0$ and $\rho>1$ such that $
	\text{dist}(h_{n+1}(z),h_{n}(z))\le M\rho^{-n}
	$ for every $n\geq1$.
	
	Recall that the homotopic length of a curve $\g$ is the infimum among the lengths  of smooth curves homotopic to $\g$ rel $ P $ with endpoints fixed under the orbifold metric; see Appendix \ref{app:1}.
	
	For any point $z\in \Omega\sm(\NNN_a\cup Q)$,  define a curve $\g_z:[0,1]\to V\setminus  P $ as $\g_z(t):=h_t(z)$ for $t\in[0,1]$.
	Since the homotopic length of $\g_z$ is continuous with respect to $z$ and converges to zero as $z\to\partial\Omega\cup\partial \NNN_a\cup Q$, it is bounded above by a constant $M_1$ for all points $z\in \Omega\sm(\NNN_a\cup Q)$.
	
	Fix an integer $n\geq1$ and a point $z\in \ov\C$. If $z\in \DDD_n\cup g^{-n}(\NNN_a)\cup g^{-n}(Q)$, then $${\rm dist}(h_n(z),h_{n+1}(z))=0$$ by point (2) above.
	If $z\in \Omega_n\setminus(g^{-n}(\NNN_a)\cup g^{-n}(Q))$, then $w=f^n(z)\in \Omega\setminus(\NNN_a\cup Q)$.
	In this case, the curve $\beta=\{h_t(z):t\in[n,n+1]\}$ is a lift of $\g_w$ by $f^n$ based at $h_n(z)$. Consequently,
	$$
	\text{dist}(h_{n}(z),h_{n+1}(z))\le C\cdot L_\omega[\beta]\leq CM_1\rho^{-n}
	$$
	by \eqref{eq:777} and Lemma \ref{lem:expanding}. This completes the proof of Proposition \ref{prop:quotient}.
\end{proof}

\begin{proof}[Proof of Theorem \ref{thm:blow-up}]
	Let $\pi$ be the limit quotient map of the sequence $\{h_n\}$, and set  $K_g=\bigcap_{n>0}\ov{\Omega_n}$. By Proposition \ref{prop:quotient}, we have $\pi(\ov{\Omega_n})\subset\ov{V_n}$, $\pi({\DDD_n})=\BBB_n$, and $\pi(\partial\DDD_n)=\partial\BBB_n$ for all $n>0$. It follows that $\pi(K_g)\subset E:=\bigcap_{n\geq0}\ov{V_n}$. Since $\pi$ is surjective, we obtain $\pi(K_g)= E$. Moreover, the properties of $h_n$ also imply that $\pi\circ g=f\circ\pi$ on $K_g$ and that $\pi:K_g\cap F_g\to E\cap F_f$ is a conformal homeomorphism.
	
	Suppose that $B$ is a component of $\BBB$ such that $f^p(\partial B)=\partial B$. Due to the properties of $\pi$ mentioned above, there exists a unique component $D$ of $\DDD$ such that $\partial D\subset \pi^{-1}(\partial B)\cap K_g$, and $\pi^{-1}(\partial B)\cap K_g\subset J_g$ is a stable set of $g^p$ of simple type. Then by Theorem \ref{thm:renorm}, $\pi^{-1}(\partial B)\cap K_g$ is the boundary of a Fatou domain of $g$, which implies $\pi^{-1}(\partial B)\cap K_g=\partial D$. Since $\pi({D})=B$, it follows that $\pi^{-1}(B)=D$. By pullback, we obtain $\pi^{-1}(\BBB_n)=\DDD_n$ for every $n>0$.
	
	Now, consider an arbitrary point $z\in \bigcap_{n>0}V_n$. Then $\pi^{-1}(z)\subset\bigcap_{n>0} \Omega_n$ is a full and connected compact set of simple type. If $z\in F_f$, then $\pi^{-1}(z)$ is a singleton. If $z\in J_f$ is eventually periodic, then $\pi^{-1}(z)\subset J_f$ is eventually periodic under $g$, and thus a singleton by Lemma \ref{lem:expanding}.
	
	Assume that $z\in J_f$ is wandering, i.e., $f^i(z)\not=f^j(z)$ for any $i\neq j\geq 0$. Then the $\omega$-limit set $\omega(z)$ contains  infinitely many points. Otherwise, since $f(\omega(z))\subset \omega(z)$, the orbit of $z$ would converge to repelling cycles, a contradiction. Thus, we may choose a point $z_{\infty}\in\omega(z)\setminus  P $ and a subsequence $\{f^{n_k}(z)\}$ such that $f^{n_k}(z)\to z_{\infty}$ as $k\to\infty$.
	
	Let $U$ be a disk such that $z_{\infty}\in U$ and $\ov{U}\cap  P =\emptyset$. Then $f^{n_k}(z)\in U$ for every sufficiently large integer $k$.  It follows that $g^{n_k}(\pi^{-1}(z))\subset\pi^{-1}(\ov{U})$ for every sufficiently large integer $k$. Since $\pi^{-1}(\ov{U})$ is a full continuum disjoint from $P_g$,  by Lemma \ref{lem:expanding}, the diameters of components of $g^{-n}(\pi^{-1}(\ov{U}))$ tend to $0$ as $n\to\infty$.  Thus, $\pi^{-1}(z)$ is a singleton.

	Finally, the uniqueness of the rational map $g$ is deduced directly from \cite[Theorem 1]{DH1}. Then we complete the proof of Theorem \ref{thm:blow-up}.
\end{proof}

\begin{proof}[Proof of Theorem \ref{thm:cluster-exact}]
	By Theorem \ref{thm:cluster-exact0}, there exists a stable set $\KKK$ of $f$ that induces a cluster-exact decomposition of $(f,P)$. Moreover, the union $\VVV$ of all complex-type components of $\cbar\setminus \KKK$ avoids the attracting cycles of $f$.  It then follows from Theorem \ref{thm:blow-up} that each blow-up of the induced exact sub-system $f:\VVV_1\to\VVV$ has the  \Sie\ carpet Julia set.
\end{proof}

\section{Topology of growing continua}\label{sec:topology}
To construct invariant graphs in extremal chains, we first study their topology.

Let $f$ be a  rational map with $J_f\not=\cbar$. Suppose that $K$ is a periodic level-$(n+1)$ ($n\ge 0$) extremal chain of $f$ with period $p\ge 1$, and $E$ is the union of all periodic level-$n$ extremal chains contained in $K$. By Lemma \ref{lem:dyn-def}, $E$ is an $f^p$-invariant continuum, and $K$ is generated by $E$ in the sense that
$
K=\ov{\bigcup_{k\ge 0}E_k},
$
where $E_k$ is the component of $f^{-kp}(E)$ containing $E$.

Due to the  inductive construction mentioned above, all results about extremal chains can be proved by induction on levels. To improve the clarity of the proofs and ensure wider accessibility, we will adopt a more general framework for our discussions in this section.

By a {\bf growing continuum} of $f$, we mean a continuum $K\subset\cbar$ together with a continuum $E\subset\cbar$ such that $\partial E\subset J_f$, $f(E)\subset E$, and
\begin{equation}\label{eq:123}
	K=\ov{\bigcup_{k\ge 0}E_k},
\end{equation}
where $E_k$ is the component of $f^{-k}(E)$ containing $E$. We call $E$ the {\bf generator} of $K$.

Let $P$ be a finite marked set. Since $E_{k}\subset E_{k+1}$, according to Corollary \ref{cor:monotone}\,(2), there exists an integer $k_0\ge 0$ such that $E_{k_0}$ is a skeleton of $E_k$ rel $P$ for all $k>k_0$. Note that $f(E_{k_0})\subset E_{k_0}$. Then $K$ is also a growing continuum generated by $E_{k_0}$. Therefore, we may always assume that $E$ is a skeleton of $E_k$ for all $k>0$.

\subsection{Local connectivity of extremal chains}
Let $f$ be a   PCF  rational map. By Theorem \ref{thm:renorm}, the maximal Fatou chains of $f$ are locally connected since they are stable sets. In this subsection, we aim to prove the local connectivity of extremal chains, or more generally, growing continua.

\begin{lemma}\label{lem:top}
	Let $K\subset\cbar$ be a growing continuum generated by  $E$. Suppose that $E$ is locally connected. Then $K$ is locally connected.
\end{lemma}

According to Lemma \ref{lem:orbifold}, we need to consider the components of $\cbar\sm K$.  It is worth noting that any component of $\cbar\sm K$ is contained in a unique component of $\cbar\sm E_k$ for every $k\geq 0$.

A nested sequence $\{\Omega_k\}$ is called an {\bf end} of $K$ if $\Omega_k$ is a component of $\cbar\sm E_k$ and $\Omega_{k+1}\subset\Omega_k$ for every $k\geq0$. An end $\{\Omega_k\}$ is called {\bf marked} if $\Omega_k\cap P_f\not=\emptyset$ for all $k\ge 0$. There exist finitely many marked ends.

Since $E_{k+1}$ is a component of $f^{-1}(E_{k})$, for each component $\Omega_{k+1}$ of $\cbar\sm E_{k+1}$, there exists a unique component $\Omega_k'$ of $\cbar\sm E_k$ such that $f(\partial\Omega_{k+1})=\partial\Omega_k'$. Moreover, $f:\Omega_{k+1}\to\Omega_k'$ is a homeomorphism if $\Omega_k'\cap P_f=\emptyset$.

\begin{proposition}\label{prop:end-map}
	Let $\{\Omega_k\}$ be an end of $K$. For each $k\geq0$, let $\Omega_k'$ be the component of $\cbar\sm E_k$ such that $f(\partial\Omega_{k+1})=\partial\Omega_k'$. Then $\Omega'_{k+1}\subset\Omega'_k$ for every sufficiently large integer $k$.
\end{proposition}

\begin{proof}
	There exists an integer $k_0\ge 0$ such that, either $\Omega_{k+1}$ avoids $f^{-1}(E_{k})$ for each $k\geq k_0$ and hence $f(\Omega_{k+1})=\Omega'_{k}$, or $\Omega_{k+1}$ contains a component of $f^{-1}(E_{k})$ for each $k\geq k_0$. 
	
	In the former case, it is clear that $\Omega'_{k+1}\subset\Omega'_{k}$ for all $k\ge k_0$. 
	
	In the latter case, let $W_k$ be the component of $\Omega_{k+1}\sm f^{-1}(E_k)$ whose boundary contains $\partial\Omega_{k+1}$. Then $f:W_k\to\Omega'_{k}$ is proper, and $W_k$ contains critical points of $f$. Note that there exists an integer $k_1\ge k_0$ such that each $W_k$ contains the same critical points of $f$ for all $k\ge k_1$. Thus, all $\Omega'_{k}$ share common critical values of $f$. This implies that $\Omega'_{k+1}\subset\Omega'_k$ for $k\ge k_1$.
\end{proof}

By Proposition \ref{prop:end-map}, we obtain a self-map $f_\star$  on the collection of ends of $K$. This map is defined by $f_\star\{\Omega_k\}=\{\Omega_k'\}$ if $f(\partial\Omega_{k+1})=\partial\Omega_k'$ for each sufficiently large integer $k$. The proof of Proposition \ref{prop:end-map} shows that the image of a marked end remains marked. Hence, marked ends are eventually $f_\star$-periodic. Moreover, if $\{\Omega_k'\}=f_\star^N\{\Omega_k\}$ is not marked, then for each  sufficiently large integer $k$, the map $f^N:\Omega_{k+ N}\to\Omega'_{k}$ is conformal. 

\begin{lemma}\label{lem:eventually-periodic}
	There exist constants $M>0$ and $\rho>1$ with the following properties. Let $\{\Omega_k\}$ be an end of $K$ such that $f_\star^N\{\Omega_k\}$ is not marked for an integer $N\geq 1$. Then
	$$
	\textup{diam}\bigg(\bigcap_{k\ge 0} \ov{\Omega_k}\bigg)\le M\rho^{-N}.
	$$
	Consequently, $\bigcap_{k\ge 0}\ov{\Omega_k}$ is a singleton if $\{\Omega_k\}$ is $f_\star$-wandering.
\end{lemma}

\begin{proof}
	Recall that $E$ is a skeleton of each $E_k$ rel $P_f$.
	By Lemma \ref{lem:orbifold} and the fact that $E_1$ is locally connected, the homotopic diameters of the components of $\cbar\sm E_1$ that avoid $P_f$ are bounded above by a constant $M_1$. Since $f_\star^N\{\Omega_k\}$ is not marked, there exists an integer $k_0\geq1$ such that $f^N(\Omega_k)\cap P_f=\emptyset$ for every $k\geq k_0$.
	
	Fix any integer $k>k_0$.
	For each $0\leq i\leq k$, we denote $W_i$ as the component of $\cbar\sm E_i$ such that $\partial W_i=f^{k-i}(\partial \Omega_k)$. Let 
	 $n_k\geq 1$ be the minimal integer with $W_{n_k}\cap P_f=\emptyset$, and let $D_1$ be the component of $\cbar\sm E_1$ containing $W_{n_k}$. 
	
	We claim that $D_1\cap P_f=\emptyset$. If $n_k=1$, then $D_1=W_{n_k}$, and this claim is true. If $n_k>1$, we have $W_{n_k-1}\cap P_f\neq \emptyset$ by the choice of $n_k$. Let $D$ denote the component of $\cbar\setminus E$ containing $W_{n_k-1}$. Since $E$ is a skeleton of $E_{n_k-1}$, it follows that $D\cap P_f=W_{n_k-1}\cap P_f$. Thus, there exist an annulus $A\subset D\setminus P_f$ bounded by $\partial D$ and a Jordan curve in $W_{n_k-1}$. Let $A_1$ be the component of $f^{-1}(A)$ containing $\partial W_{n_k}$. Then $A_1\cap P_f=\emptyset$ and $A_1\cup W_{n_k}=D_1$. The claim is proved.

By this claim, the homotopic diameter of $D_1$ is bounded above by $M_1$. Due to the choices of $k$ and $n_k$, the map $f^{k-n_k}:\Omega_k\to W_{n_k}$ is conformal, and  $k-n_k\geq N$. Thus, this lemma follows directly from Lemma \ref{lem:expanding}.
\end{proof}

\begin{proof}[Proof of Lemma \ref{lem:top}]
Given any component $D$ of $\cbar\sm K$, let $\{\Omega_k(D)\}$ be the end of $K$ such that $D\subset\Omega_k(D)$ for all $k\ge 0$. By Lemma \ref{lem:eventually-periodic}, the end $\{\Omega_k(D)\}$ is eventually $f_\star$-periodic and marked.
	
	First, assume that $\{\Omega_k\}=\{\Omega_k(D)\}$ is periodic under $f_\star$. Without loss of generality, we may assume that the period is one and that $f(\partial\Omega_k)=\partial\Omega_{k-1}$ for every $k\geq1$. Let $\g_0\subset\Omega_0$ be a Jordan curve separating $\partial\Omega_0$ from $P_f\cap\Omega_0$. Then there exists a unique component $\gamma_1$ of $f^{-1}(\gamma_0)$ contained in $\Omega_1$ that separates $\partial\Omega_1$ from $P_f\cap\Omega_1=P_f\cap \Omega_0$. Thus, there exists a homeomorphism $\theta_0:\cbar\to\cbar$ isotopic to $id$ rel $P_f$, such that $\theta_0(\g_0)=\g_1$. By lifting (Lemma \ref{lem:lift}), we obtain a sequence of homeomorphisms $\{\theta_k\}$ of $\cbar$ isotopic to $id$ rel $P_f$, such that
	$$
	f\circ\theta_{k+1}=\theta_k\circ f\quad\text{ on }\cbar.
	$$
	Set $\phi_k=\theta_k\circ\cdots\circ\theta_0$. Then $\g_{k+1}=\phi_k(\g_0)$. By Lemma \ref{thm:isotopy}, $\{\phi_k\}$ uniformly converges to a quotient map $\varphi$ of $\cbar$. Denote $\g=\varphi(\g_0)$. Then $f(\g)=\g$, and $\g$ is locally connected.
	
	According to Lemma \ref{lem:expanding}, the Hausdorff distance between $\partial\Omega_k$ and $\g_k$ converges to zero. Consequently, $\partial\Omega_k\to\g$ as $k\to\infty$. Thus, $\g\subset K$. Then $D$ lies in a component of $\cbar\setminus\g$.
	
	We claim that $D$ is simply a component of $\cbar\setminus\gamma$. If this is false, there exist a point $z\in \partial D$ not  in $\gamma$ and a neighborhood $W$ of $z$ disjoint from $\partial \Omega_k$ for every sufficiently large integer $k$. Since $W\cap D\neq\emptyset$, it follows that $W\subset \Omega_k$ for every  $k\geq 0$. In particular, $W$ is disjoint from every $E_k$, and hence avoids $K=\ov{\bigcup_{k\geq0} E_k}$. Thus, $W\subset D$, a contradiction.
	
	This claim implies that  $\partial D$ is locally connected since $\g$ is locally connected.

	Now, suppose that $\{\tilde{\Omega}_k\}=\{\Omega_k(D)\}$ is strictly eventually periodic under $f_\star$. Let $q>0$ be the smallest integer such that $\{\Omega_k\}=f_\star^q(\{\tilde{\Omega}_k\})$ is periodic.
	
	Let $\tilde{\gamma}_q$ be the component of $f^{-q}(\g_0)$ contained in $\tilde{\Omega}_q$ that separates $\partial\tilde{\Omega}_q$ from $\tilde{\Omega}_q\cap f^{-q}(P_f)$. For all $k\geq 0$, define a homeomorphism $\tilde{\phi}_{k}:=\theta_{q+k}\circ\cdots\circ\theta_q$. 
	Then 
\begin{enumerate}
\item $f^q\circ\tilde{\phi}_{k}(z)={\phi}_k\circ f^q(z)$ for every $z\in\cbar$;
\item $\tilde{\g}_{q+k+1}:=\tilde{\phi}_k(\tilde{\g}_q)$ is contained in $\tilde{\Omega}_{q+k+1}$ and isotopic to $\tilde{\g}_q$ rel $f^{-q}(P_f)$.
\end{enumerate}

	By a similar argument as in the periodic case, we can prove that the map $\tilde{\phi}_k$ uniformly converges to a quotient map $\tilde{\varphi}$, and $D$ is a component of $\cbar\setminus\tilde{\varphi}(\tilde{\g}_q)$. Thus, $\partial D$ is locally connected.

	It remains to show that the diameters of the components of $\ov\C\setminus K$ tend to $0$.
	
	Given any $\epsilon>0$,  there exist only  finitely many ends $\{\Omega_k\}$ with ${\rm diam}(\bigcap_{k\geq0}\ov{\Omega_k})\geq \epsilon$ by Lemma \ref{lem:eventually-periodic}. Therefore, we simply need to consider the components $D$ of $\cbar\setminus K$ for which $\{\Omega_k(D)\}$ are such  ends.
	As shown above, $D$ is a complementary component of a curve $\g_{D}=\lim_{k\to\infty}\partial \Omega_k(D)$. Since there exist finitely many curves $\g_{D}$, and only finitely many components of $\cbar\setminus \g_{D}$ have diameters larger than $\epsilon$, we complete the proof of the lemma.
\end{proof}

\begin{theorem}\label{thm:locally-connected}
	Every extremal  chain of a PCF rational map is locally connected.
\end{theorem}

\begin{proof}
	Every level-$0$ extremal chain of a PCF rational map $f$ is clearly locally connected. Inductively, for $n\geq0$, assume that level-$n$ extremal chains are locally connected. If $K$ is a periodic level-$(n+1)$ extremal chain, then it is locally connected by Lemma \ref{lem:top} and the induction.

	Now, suppose that $K'$ is a strictly preperiodic level-$(n+1)$ extremal  chain such that $f^q(K')=K$, which is periodic with period $p$. Let $E$ be the union of all periodic level-$n$ extremal  chains contained in $K$, and let $E_k$ denote the component of $f^{-pk}(E)$ containing $E$ for every $k\geq 0$.
	We may assume that $E$ is a skeleton of every $E_k$ rel $P_f$. Then for each $k\ge 0$, there exists a unique component $E_k'$ of $f^{-q}(E_k)$ contained in $K'$ such that  $E_{k}'\subset E_{k+1}'$ and $K'=\ov{\bigcup_{k\ge 0} E_k'}$.\vspace{1pt}
	
The ends for $K'$ can  be similarly defined as in the periodic case. If $\{\Omega_k'\}$ is an end of $K'$, then there exists a unique end $\{\Omega_k\}$ of $K$ such that $f^q(\partial\Omega_k')=\partial\Omega_k$ for every sufficiently large integer $k$. Therefore, applying a similar argument as in the proof of Lemma \ref{lem:top}, we can establish the local connectivity of $K'$. The details are omitted.
\end{proof}

\subsection{Growing curves}
Let $f$ be a   PCF  rational map, and let  $K$ be a growing continuum generated by an $f$-invariant continuum $E$. As before, $E_k$ denotes the component of $f^{-k}(E)$ containing $E$, and $E$ is assumed to be a skeleton of $E_k$ (rel $P_f$) for every $k\geq0$.

A curve $\g: [0,1]\to K$ is called a {\bf growing curve} if, for any small number $\epsilon>0$, there exists an integer $k\geq0$ such that $\g[0, 1-\epsilon]\subset E_k$. The point $\g(1)$ is called the {\bf terminal} of $\g$.

By definition, any curve in $E_k$ is growing, including the trivial ones. Here, a curve is {\it trival} if its image is a singleton. Moreover, the image or lift of a growing curve in $K$ under $f$  is also a growing curve.

Growing curves will be crucial in constructing invariant graphs on a maximal Fatou chain in the next section. To this end, we aim to establish their existence through the following lemma.

\begin{lemma}\label{lem:growing}
 Suppose that $E$ is locally connected. Then the following statements hold:
\begin{enumerate}
\item Any point of $K$ is the terminal of a growing curve in $K$;

\item For any two points $a$ and $b$ in distinct components of $\cbar\sm K$, there exist two growing curves $\de_\pm\subset K$ with the same terminal, such that $E\cup\de_+\cup\de_-$ separates $a$ from $b$.
\end{enumerate}
\end{lemma}

Let $\g_1,\g_2:[0,1]\to \ov{\C}$ be two curves with $\g_1(1)=\g_2(0)$. The {\bf concatenation} $\g_1\cdot\g_2$  is a curve parameterized by
\[\g_1\cdot\g_2(t)=\left\{
                     \begin{array}{ll}
                       \g_1(2t) & \hbox{if $t\in[0,1/2]$,} \\
                       \g_2(2t-1) & \hbox{if $t\in[1/2,1]$.}
                     \end{array}
                   \right.
\]
If  $\g_1,\ldots,\g_n$ can  be successively concatenated,  their concatenation is parameterized by
\begin{equation}\label{eq:parameterize}
\text{$\g_1\cdot\g_2\cdots\g_n(t):=\g_1\cdot(\g_2\cdot(\cdots(\g_{n-1}\cdot\g_n)))(t)$,\quad $t\in[0,1]$}.
\end{equation}

\begin{proposition}\label{pro:continuous}
Suppose that $E$ is locally connected. Then there exists a family $\G$ of growing curves in $K$ such that any point of $K$ is the terminal of an element in $\G$, and that $\G$ is sequentially compact under uniform convergence, i.e., any infinite sequence in $\G$ has a convergent subsequence whose limit is also in $\G$.
\end{proposition}

\begin{proof}
 Since $E_1$ is locally connected, each point $w\in E_1$ can be joined to $E$ by a curve  $\beta_w\subset E_1$ with the following conditions: if $w\in E$, then $\beta_w\equiv w$; otherwise, it holds that $\beta_w(0)\in E$ and $\beta_w(0,1]\cap E=\emptyset$. By Lemma \ref{lem:equicontinuous},  we can require that $\G_0=\{\beta_w:w\in E_1\}$  is equicontinuous. Thus, the homotopic diameters of curves in $\G_0$ are bounded above by a constant.

For any integer $k\geq1$ and any point $z\in E_{k+1}$, set $w:=f^{k}(z)\in E_1$. If $w\in E_0$, define $\beta_z\equiv z$. Otherwise, since $E$ is a skeleton of $E_k$ rel $P_f$, we have $\beta_w(0,1]\cap P_f=\emptyset$. This implies that $\beta_w$ has a unique lift by $f^{k}$ based at $z$, which is defined as $\beta_z$.  Since $\G_0$ is equicontinuous, the collection $\G_k:=\{\beta_z,z\in E_{k+1}\}$ is also equicontinuous. According to Lemma \ref{lem:expanding}, each curve in $\G_k$ has a diameter bounded above by $M/\rho^k$ for some constants $M>0$ and $\rho>1$.

Now, for every $k\geq1$ and any point $z\in E_{k+1}$, we obtain a growing curve $\g_z:=\beta_0\cdot\beta_1\cdots \beta_k$ that joins $E$ to $z$ such that $\beta_i\in\G_i$ for every $i=0,\ldots,k$. By its parameterization given in \eqref{eq:parameterize}, it follows that
\begin{equation}\label{eq:growing}
\g_z\bigg[0,1-\frac{1}{2^k}\bigg]\subset E_k\quad\text{ for every }k\geq1.
\end{equation}

We claim that the family of curves $\G_\infty:=\{\g_z: z\in\bigcup_{k\geq1} E_k\}$ is equicontinuous.
Given any $\epsilon>0$, there exists an integer $N>0$ such that $M/(\rho^{ N-1}(\rho-1))<\epsilon$. Moreover, for every $k\geq0$, there exists $\de_k>0$ such that $|\beta(t_1)-\beta(t_2)|<\epsilon$ if $|t_1-t_2|<\de_k$ for any curve $\beta\in \G_k$.  Set $\de:=\min\{\de_0,\ldots,\de_N\}$. Let $\g=\beta_0\cdot\beta_1\cdots \beta_k$ be any element in $\G_\infty$. If $k\leq N$, according to the parameterization of $\g$, we have
\begin{equation}\label{eq:case1}
|\g(t_1)-\g(t_2)|<2\epsilon \quad \text{ as }  |t_1-t_2|<\de/2^{ N+1}.
\end{equation}
In the case of $k> N$,  the diameter of $\g[1-1/2^N,1]=\beta_N\cdots \beta_k$ is bounded above by $M/\rho^N+\cdots+M/\rho^k<M/(\rho^{ N-1}(\rho-1))<\epsilon$. Thus, $|\g(t_1)-\g(t_2)|<\epsilon$ when $t_1,t_2\ge 1-1/2^N$. If $t_1,t_2\in [0,1-1/2^{ N+1}]$, then  \eqref{eq:case1} holds. Thus, the claim is proved.

Let $\G$ be the union of $\G_\infty$ and the limit of every uniformly convergent sequence in $\G_\infty$. Then $\G$ is also equicontinuous. By the Ascoli-Arzel\`{a} theorem, $\G$ is a normal family. If $\g$ is the
limit of a uniformly convergent sequence in $\G$, then there exists a sequence of curves in $\G_\infty$ that also
uniformly converges to $\g$. Thus, $\G$ is sequentially compact. By \eqref{eq:growing}, for any $\gamma\in\Gamma$, we have $\g[0, 1-1/2^k]\subset E_k$ for every $k\geq0$. Hence, $\G$ consists of growing curves in $K$.

Fix a point $z\in K$. If $z\in E_k$ for some $k\ge0$, a curve in $\G_\infty$ joins $E$ to $z$. Otherwise, there exists a point $z_k\in E_k$ for every $k$ such that $z_k\to z$ as $k\to\infty$. For each $k$, let $\g_k$ be a curve in $\G_\infty$ joining $E$ to $z_k$. By taking a subsequence if necessary, the curve $\g_k$ uniformly converges to a curve $\g\in\G$, which joins $E$ to $z$.
\end{proof}

\begin{proof}[Proof of Lemma \ref{lem:growing}]
Statement (1) follows directly from Proposition \ref{pro:continuous}.

(2) If $a$ and $b$ belong to distinct components of $\cbar\sm E_m$ for some $m\ge 0$, we can choose the required curves $\de_\pm$ in $E_m$ since $E_m$ is locally connected. Thus, we assume that there exists an end $\{\Omega_k\}$ of $K$ such that $a,b\in\Omega_k$ for every $k\geq0$.

Let $U_a$ be the component of $\ov{\C}\setminus K$ containing $a$. Then $U_a$ is contained in each $\Omega_k$. Since $K$ is locally connected by Lemma \ref{lem:top}, it follows that $\partial U_a$ is locally connected. Let $\eta:\R/\Z\to \partial U_a$ be a parameterization of $\partial U_a$.

A curve $\g$ with endpoints in $E$ is said to {\bf split $\{a,b\}$} ({\bf rel $E$}) if $E$ contains a curve $\alpha$ with the same endpoints as those of $\g$ such that $\g\cdot \alpha^{-1}$ is not contractible in $\ov{\C}\setminus\{a,b\}$. Note that if $\g$ splits $\{a,b\}$, then $\g\cdot \alpha^{-1}$ is not contractible in $\ov{\C}\setminus \{a,b\}$ for \emph{any} curve $\alpha\subset E$ with the same endpoints as those of $\g$.

 According to Proposition \ref{pro:continuous}, for any  $t\in \R/\Z$, there exists a growing curve $\de_{t}\in\G$ with $\de_t(0)\in E$ and $\de_t(1)=\eta(t)\in \partial U_a$. Then for every $t\in\R/\Z$, we have two curves (see Figure \ref{fig:two-case})
$$
\ell^-_t:=\de_{0}\cdot \eta[0,t]\cdot \de_{t}^{-1}\text{\quad and\quad }\ell^+_t:=\de_{t}\cdot \eta[t,1]\cdot \de_{0}^{-1}.
$$
Since $\ell_t^-\cdot\ell_t^+=\de_0\cdot \eta\cdot \de_0^{-1}$, which splits $\{a,b\}$,  at least one of $\ell_t^{+}$ and $\ell_t^-$ splits $\{a,b\}$.

Note that $\ell_1^{-}=\de_0\cdot \eta\cdot \de_0^{-1}$, which splits $\{a,b\}$. Let $t_*$ denote the infimum of $t\in[0,1]$ such that $\ell_t^-$ splits $\{a,b\}$. Then there exists a sequence of decreasing numbers $\{t_n\}\subset [t_*,1]$ such that $t_n\to t_*$ and $\ell_{t_n}^-$ splits $\{a,b\}$.  Let $\{s_n\}\subset [0,t_*]$ be a sequence of increasing numbers converging to $t_*$. It follows that each $\ell_{s_n}^+$ splits $\{a,b\}$. Here, $t_n$ or $s_n$ are possibly constant for sufficiently large  $n$.
\begin{figure}[http]
\centering
\includegraphics[width=9.5cm]{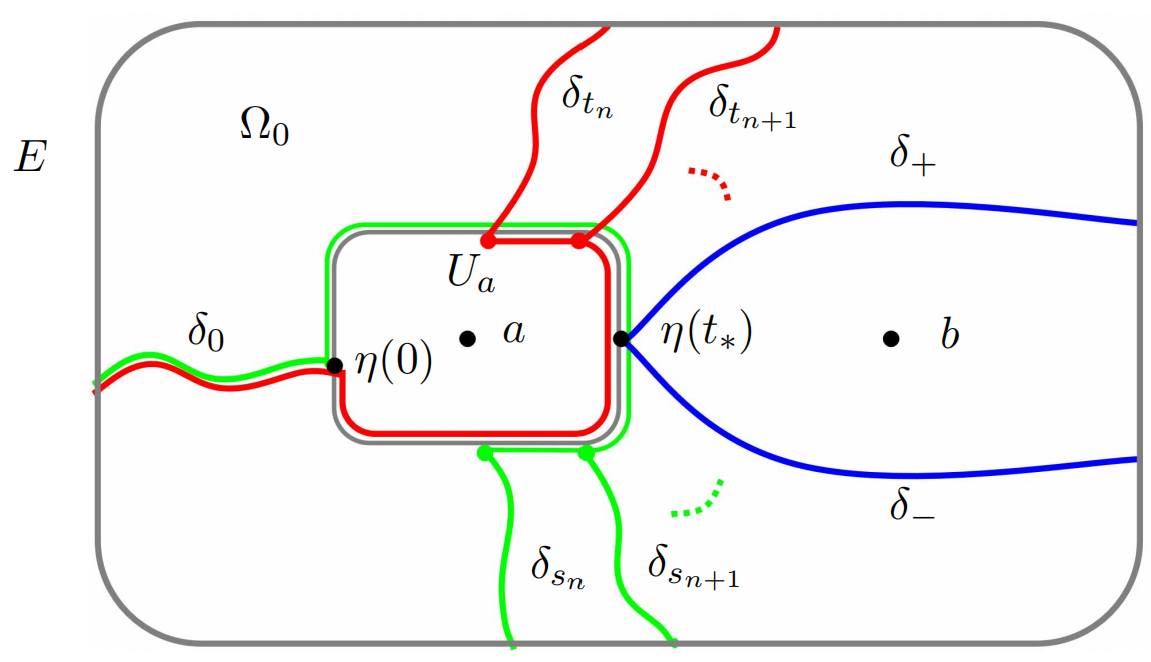}
\caption{Split $\{a,b\}$ by growing curves.}\label{fig:two-case}
\end{figure}

We claim that the curve $\de_{s_{n}}\cdot \eta[s_n,t_n]\cdot \de_{t_{n}}^{-1}$ splits $\{a,b\}$ for each $n\geq1$; see Figure \ref{fig:two-case}. Otherwise, since $$
\ell_{t_n}^-=\de_0\cdot \eta[0,t_n]\cdot\de_{t_{n}}^{-1}=(\de_0\cdot \eta[0,s_n]\cdot\de_{s_{n}}^{-1})\cdot(\de_{s_{n}}\cdot \eta[s_n,t_n]\cdot
\de_{t_{n}}^{-1})=\ell_{s_n}^{-}\cdot(\de_{s_{n}}\cdot \eta[s_n,t_n]\cdot \de_{t_{n}}^{-1})
$$
splits $\{a,b\}$, it follows  that $\ell_{s_n}^{-}$ splits $\{a,b\}$, which contradicts the choice of $t_*$.

 Since $\{\de_{s_n}\}$ and $\{\de_{t_n}\}$ are selected from a sequentially compact family  ${\G}$ of growing curves by Proposition \ref{pro:continuous}, we may assume that  $\{\de_{s_n}\}$ and $\{\de_{t_n}\}$ uniformly converge to growing curves $\de_-$ and $\de_+$, respectively. Consequently, both $\de_\pm$ join $E$ to $\eta(t_*)$, and  the curves $\de_{s_n}\cdot\eta[s_n,t_*]\cdot \de_-^{-1}$ and $\de_{+}\cdot\eta[t_*,t_n]\cdot\de_{t_n}^{-1}$ do not split $\{a,b\}$ for each sufficiently large integer $n$. Moreover, since
$$
\de_{s_{n}}\cdot \eta[s_n,t_n]\cdot \de_{t_{n}}^{-1}=(\de_{s_n}\cdot\eta[s_n,t_*]\cdot \de_-^{-1})\cdot(\de_-\cdot\de_+^{-1})\cdot(\de_{+}\cdot\eta[t_*,t_n]\cdot\de_{t_n}^{-1})
$$
splits $\{a,b\}$ by the claim above, it follows that $\de_-\cdot\de_+^{-1}$ splits $\{a,b\}$, and the lemma is proved.
\end{proof}

\subsection{Accesses within a growing continuum}
In order to construct invariant graphs within extremal chains, we need a sufficient number of preperiodic growing arcs. These arcs will be constructed in this and the next subsections.  

Let $(f,P)$ be a marked rational map. Suppose that $K$ is a growing continuum generated by an $f$-invariant and locally connected continuum $E$. We continue to assume that $E$ is a skeleton (rel $P$) of all $E_k$, where $E_k$ denotes the component of $f^{-k}(E)$ containing $E$. 

Let $P_0=P\setminus E$. Then $P_0\cap E_k=\emptyset$ for every $k\geq0$ since $E$ is a skeleton of $E_k$. Two growing curves $\alpha_1$ and $\alpha_2$ in $K$ with a common terminal $z$ are called {\bf equivalent} if there exist an integer $k\geq0$ and a curve $\de\subset E_k$  that joins $\alpha_1(0)$ to $\alpha_2(0)$, such that the closed curve $\g:=\alpha_1^{-1}\cdot\de\cdot\alpha_2$ is contractible in $\ov{\C}\setminus P_0$, i.e., there exists a continuous map $H: \mathbb{R}/\mathbb{Z}\times [0,1]\to \cbar$ such that the family of curves $\{H_s=H(\cdot, s), s\in[0, 1]\}$ satisfies
$$H_0=\gamma,\quad H_1\equiv\{z\}, \quad H_s(0)=z, \quad\textup{ and }\quad H_s(0, 1)\cap P_0=\emptyset,\quad \forall s\in(0, 1).$$
 This is clearly an equivalence relation. Note that $\g$ possibly passes through some points in $P\cap E$.
\begin{figure}[http]
	\centering
	\begin{tikzpicture}
	\node at (0,0){\includegraphics[width=8.5cm]{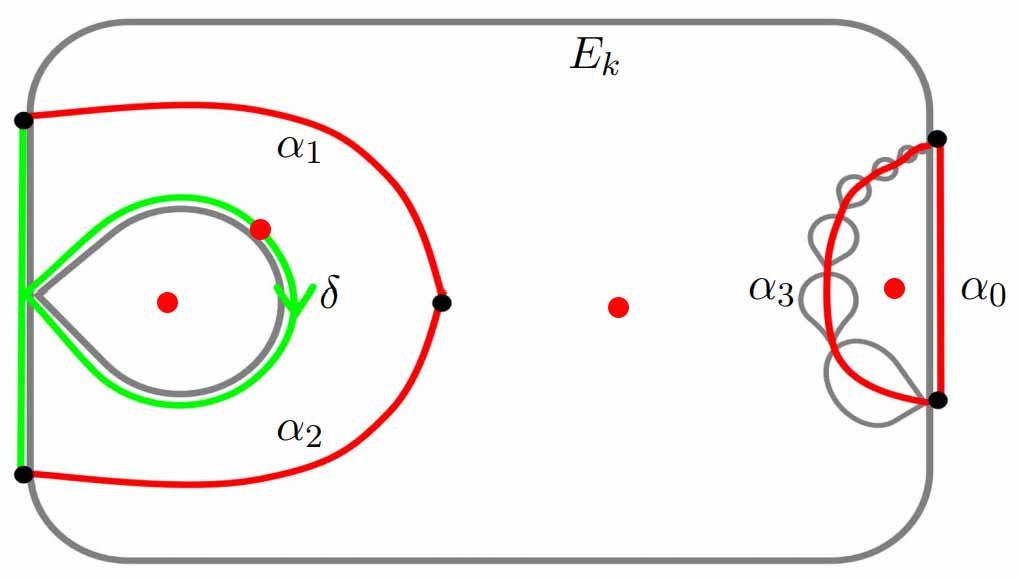}};
\node at (-0.32,-0.15){$z$};
\node at (3.85,1.2){ $z_1$};
	\end{tikzpicture}
  \caption{The equivalent growing curves $\alpha_1$ and $\alpha_2$, with marked points colored  red.}\label{fig:access}
\end{figure}

For each $k\geq0$, any two growing curves in $E_k$ with a common terminal are equivalent. A growing curve $\alpha$ is called {\bf infinitely growing} if it is not equivalent to any curve (including trivial ones) in $E_k$ for every $k\geq0$. By definition, infinitely growing curves cannot be trivial. In Figure \ref{fig:access},  the curve $\alpha_3$ is infinitely growing to $z_1$, while $\alpha_0$ is not.
An {\bf access} to $z$ is  an equivalence class of all infinitely growing curves to $z$.

By the \textbf{interior} of a curve $\g:[0, 1]\to \cbar$, we mean the set $\g(0, 1)$. The {sub-curve} $\g|_{[t_1,t_2]}$ of
$\g$ means a curve whose image equals $\g[t_1,t_2]$. An (open) arc $\g$ is called a \textbf{crosscut} of a domain $U\subset \cbar$ if $\g\subset \overline{U}$ with only the two endpoints in $\partial U$.

 Recall that two curves $\g_0,\g_1:[0,1]\to\cbar$ are   homotopic rel $P$ with endpoints fixed if there exists a continuous map $H:[0,1]\times [0,1]\to \cbar$ such that $H_0=\g_0,H_1=\g_1$, and each curve $H_s,s\in[0,1],$ has the same endpoints as  $\g_0$ with its interior disjoint from $P$.

\begin{proposition}\label{pro:disjoint}
	Let $\alpha,\alpha'\subset K$ be two growing curves with a common terminal $z$.
	\begin{enumerate}
	\item The curves $\alpha$ and $\alpha|_{[t,1]}$ are equivalent for any $t\in(0,1)$.
	
	\item If $\alpha(t, 1)\cap\alpha'(t, 1)\neq \emptyset$ for any $t\in(0, 1)$, then $\alpha$ and $\alpha'$ are equivalent.

	\item If $\alpha$ is infinitely growing, then for every sufficiently large integer $k$, there exists a number $t_k\in (0,1)$  such that $\alpha(t_k)\in E_k$ and $\alpha(t_k,1)\cap E_k=\emptyset$.  Moreover, the curve $\alpha|_{[t_k, 1]}$ contains an arc $\beta_k$ that is homotopic to $\alpha|_{[t_k, 1]}$ rel $P$ with endpoints fixed. In particular, $\beta_k$ lies in the same access to $z$ as $\alpha$.
	
\item  Suppose that $\alpha$ and $\alpha'$ belong to the same access to $z$, with their interiors disjoint from $P$. Then there exist an integer $m\geq0$ and a continuous family of curves $\{\alpha_s\}_{s\in [0, 1]}$ such that $\alpha_0=\alpha$, $\alpha_1=\alpha'$, and  each $\alpha_s$  joins $E_m$ to $z$ with its interior disjoint from $P$.
\end{enumerate}
\end{proposition}
\begin{proof}
We fix a disk $W$ such that $z\in W$ and $(W\setminus\{z\})\cap P=\emptyset$.

(1) The curve $\alpha|_{[0,t]}\subset E_k$ for some $k$, and $\alpha^{-1}\cdot\alpha|_{[0,t]}\cdot\alpha|_{[t,1]}$ is  contractible.\vspace{1pt}

(2)  There exist some $t, t'\in(0, 1)$ such that $\alpha(t)=\alpha'(t')$ and $\alpha|_{[t, 1]},\alpha'|_{[t', 1]}$ lie in $W$. It follows that $\alpha|_{[t, 1]}$ is equivalent to $\alpha'|_{[t', 1]}$, and thus $\alpha$ and $\alpha'$ are equivalent by statement (1). \vspace{1pt}


(3) To prove the existence of such $t_k$'s, suppose, to the contrary, that $\alpha(s_n)\in E_k$ for a sequence $\{s_n\}\subset(0,1)$ that converges to $1$ and a certain $k\geq0$. Then $z\in E_k$. Since $E_k$ is locally arcwise connected by Lemma \ref{lem:milnor}, there exists a curve $\gamma\subset E_k\cap W$ (possibly trivial) joining a certain $\alpha(s_n)$ to $z$. Thus, $\gamma^{-1}\cdot\alpha|_{[s_n, 1]}$ is contractible, which contradicts the assumption that $\alpha$ is infinitely growing.

By this statement, we can find $k_0>0$ such that $\alpha|_{[t_k,1]}\subset W$ and $z\not\in \alpha[t_k,1)$ for each $k>k_0$.
It follows that  $\alpha|_{[t_k,1]}$ contains an arc $\beta_k$ with endpoints $\alpha(t_k)$ and $z$. Then $\beta_k\subset W$, and its interior avoids $P$.
Hence, $\beta_k$ is homotopic to $\alpha|_{[t_k,1]}$ rel $P$ with endpoints fixed. 

(4)	
If $\alpha'$ is a sub-curve of $\alpha$, the conclusion is immediate. Thus, it suffices to prove the statement for a pair of sub-curves $\alpha|_{[t,1]}$ and $\alpha'|_{[t',1]}$ of $\alpha$ and $\alpha'$, respectively.

If $\alpha(t,1)\cap \alpha'(t,1)\not=\emptyset$ for any $t\in(0,1)$, then there exist $t,t'\in(0,1)$ such that $\alpha(t)=\alpha'(t')$ and $\alpha|_{[t,1]},\alpha'|_{[t',1]}\subset W$. Since the interiors of $\alpha$ and $\alpha'$ avoid $P$, it follows that $\alpha|_{[t,1]}$ and $\alpha'|_{[t',1]}$ are homotopic rel $P$ with endpoints fixed. Hence, statement (4) holds in this case. 

Otherwise, by statement (3), replacing $\alpha,\alpha'$ with their sub-curves, we can assume that $\alpha$ and $\alpha'$ are arcs with disjoint interiors such that $\alpha(0),\alpha'(0)\in E$ and $\alpha(0,1),\alpha'(0,1)\subset \cbar\setminus E$.  

Let $D$ and $D'$ be the components of $\cbar\setminus E$ containing $\alpha(0,1)$ and $\alpha'(0,1)$, respectively.
We claim that $D=D'$.
 If $z\not\in E$, the claim is immediate. Assume $z\in E$. Since  $\alpha$ and $\alpha'$ are infinitely growing,  each component of $D\setminus \alpha$ and $D'\setminus \alpha'$ contains marked points. This implies $D=D'$ since $\alpha$ and $\alpha'$ belong to the same access. The claim is proved.

Since $\alpha$ and $\alpha'$ are arcs with disjoint interiors and belong to the same access, there exists a simply connected domain $D_*$ of $D\setminus (\alpha\cup\alpha')$ such that $D_*\cap P=\emptyset$ and $\alpha,\alpha'\subset \partial D_*$. Then the desired family of curves $\{\alpha_s\}$ can be easily chosen within $\ov{D_*}$.
\end{proof}

\begin{proposition}\label{pro:alternative}
	Suppose that $G$ is a locally connected  skeleton of $E$. Let $\alpha_0,\alpha_1\subset K$ be two infinitely growing curves in the same access to $z$, with their initial points on $G$ and their interiors disjoint from $P$.
	Then there exists a continuous family of curves $\{\alpha_s\}_{s\in[0,1]}$ joining $G$ to $z$, such that the interior of each $\alpha_s$ is disjoint from $P$.
\end{proposition}
\begin{proof}
	Let $\{\beta_s\}_{s\in[0, 1]}$ be the family of curves derived from Proposition \ref{pro:disjoint}\,(4) such that $\alpha_0=\beta_0$ and $\alpha_1=\beta_1$. Then the curve $\de$ defined by $\delta(s):=\beta_s(0)$ lies in a certain $E_m$. We will construct a continuous family of curves $\{\eta_s\}_{s \in [0,1]}$ such that
\[
\eta_s(0) = \delta(s), \quad \eta_s(1) \in G, \quad\text{ and }\quad
\begin{cases}
\eta_s \equiv \eta_s(0) & \text{if } \eta_s(0) \in G, \\
\eta_s[0,1) \text{ avoids } P & \text{otherwise}.
\end{cases}
\]

		Then Proposition \ref{pro:alternative} holds by taking $\alpha_s:=\eta_s^{-1}\cdot\beta_s,s\in[0,1]$.\vspace{1pt}
	
Set $X=\{s\in[0,1]:\de(s)\in G\}$. Since $\de(0),\de(1)\in G$, each component of $[0,1]\setminus X$ is an open interval. If $s\in X$,  define $\eta_s\equiv\de(s)$. Let $(s_1,s_2)$ be a component of $[0,1]\setminus X$. Then there exists a component $D$ of $\cbar\setminus G$ such that $\de(s_1),\de(s_2)\in\partial D$ and $\de(s_1,s_2)\subset D$.

Since $\de\subset E_m$ and $G$ is a skeleton of $E_m$, it follows that $\de(s_1,s_2)$ avoids $P$ and does not separate $P$.  Consequently, there exists a disk $D'$ compactly contained in $D$ such that $P\cap D\subset D'$ and $\de(s_1,s_2)$ is contained in the annulus $D\setminus \ov{D'}$.
Thus, we can choose a continuous family of curves $\{\eta_s\}_{s\in[s_1,s_2]}$ such that $\eta_s(0)=\de(s)$, $\eta_s(1)\subset \partial D\subset G$, and $\eta_s(0,1)\subset D\setminus \ov{D'}$ for any $s\in(s_1,s_2)$, and that $\eta_{s_i}\equiv\de(s_i)$ for $i=1,2$. This completes the construction of $\{\eta_s\}_{s\in[0,1]}$.
\end{proof}

One main result of this subsection is the finiteness of accesses.

\begin{lemma}\label{lem:finite-access}
For any $z\in K$, there exist finitely many accesses to $z$.
\end{lemma}
\begin{proof}
Let $\De$ be a finite collection of infinitely growing curves in $K$ that lie in pairwise distinct  accesses to $z$. It suffices to show that $\# \De\leq (\# P)^2$.

By Proposition \ref{pro:disjoint}\,(1)--(3), we may assume that all elements in $\De$ are arcs with pairwise disjoint interiors, such that $\alpha(0, 1)\subset D_\alpha$ and $\alpha(0)\in\partial D_\alpha$ for every $\alpha\in\De$, where $D_\alpha$ is a component of $\cbar\setminus E_m$ and $m$ is a sufficiently large integer independent of $\alpha$.
		
Note that every component $D_\alpha$ must intersect $P$. Thus, there exist at most $\# P$ such components. Suppose that a certain $D_\alpha$ contains the interiors of $k$ arcs in $\De$. Then these arcs divide $D_\alpha$ into $k$ or $k+1$ simply connected domains, each intersecting $P$. It follows that $k\leq \# P$. Therefore, we have $\#\Delta\leq (\# P)^2$.
\end{proof}

In the following, we will construct numerous preperiodic growing arcs in $K$ based on the above lemma. We first prove a lifting property for accesses.

\begin{lemma}\label{lem:lifting}
	Let $\alpha\subset K$ be an infinitely growing curve with terminal $z$. Then
	\begin{enumerate}
\item the curve $f\circ\alpha$ is also infinitely growing with terminal $f(z)$;
\item  if $\beta$ and $f\circ\alpha$ lie in the same access to $f(z)$, then there exists a curve $\tilde{\beta}$ in the same access as $\alpha$ such that $f\circ\tilde{\beta}=\beta$.
\end{enumerate}
\end{lemma}

\begin{proof}
	(1) To the contrary, suppose that $f\circ\alpha$ is not infinitely growing. Then $z$ must be contained in some $E_{k_0}$. 
	By Proposition \ref{pro:disjoint}\,(3), for each sufficiently large integer $k$, there exists a number $t_k\in (0,1)$ such that $\alpha(t_k)\in E_{k}$ and $\alpha(t_k,1)\cap E_{k}=\emptyset$. It follows that $f\circ \alpha(t_k)\in E_{k-1}$ and $f\circ \alpha(t_k,1)\subset D_{k-1}$ for a component $D_{k-1}$ of $\ov{\C}\setminus E_{k-1}$.
	
	Note that the diameter of $f\circ\alpha(t_k,1)$ tends to $0$ as $k\to\infty$. Then  there exists an arc  $\g\subset f\circ\alpha([t_m,1])$ that is homotopic to $f\circ \alpha|_{[t_m,1]}$ rel $P$ with endpoints fixed for a sufficiently large integer $m$. In particular, $\gamma$ is a crosscut of $D_{m-1}$.
	By homotopy lifting, we obtain a lift $\tilde{\g}$ of $\g$ by $f$ that is homotopic to $\alpha|_{[t_m,1]}$ rel $P$ with endpoints fixed. Thus, $\tilde{\gamma}$ is  infinitely growing.
	
	On the other hand, since $f\circ\alpha$ is assumed not to be  infinitely growing, one of the two components of $D_{m-1}\setminus \g$, denoted by $D_*$, avoids $P$. Thus, there exists a component $\tilde{D}_*$ of $f^{-1}(D_*)$ with $\tilde{\gamma}\subset \partial \tilde{D}_*$. Since $\tilde{D}_*\cap P=\emptyset$ and $\partial \tilde{D}_*\setminus \tilde{\gamma}\subset E_{m}$,  $\tilde{\gamma}$ is not infinitely growing, a contradiction.\vspace{2pt}
	
	(2) By statement (1), both $\beta$ and $f\circ\alpha$ are infinitely growing. Then by Proposition \ref{pro:disjoint}\,(3), we can find numbers $t_0, t_1\in(0, 1)$ such that $f\circ \alpha(t_0, 1)$ and $\beta(t_1, 1)$   are disjoint from $P$. By Proposition \ref{pro:disjoint}\,(4), there exists a continuous family of curves $\{\g_s\}_{s\in [0, 1]}$ joining some $E_m$ to $z$ such that $\g_0=f\circ\alpha|_{[t_0, 1]}$, $\g_1=\beta|_{[t_1, 1]}$, and the interior of each $\g_s$ is disjoint from $P$.
	
	For any $t\in(0, 1)$,  the curve $\{\g_s(t):s\in[0,1]\}$ has a unique lift based at the point $\alpha|_{[t_0,1]}(t)$. Thus, by the continuity of $f$, we obtain a continuous family of lifts $\{\tilde{\g}_s\}$ of $\{\g_s\}$ such that each $\tilde{\g}_s$ joins $E_{m+1}$ to $z$ with its interior avoiding $P$. This implies  that  $\tilde{\g}_0=\alpha|_{[t_0,1]}$ and $\tilde{\g}_1$ lie in the same access to $z$. 
	Since $f\circ\tilde{\g}_1=\beta|_{[t_1,1]}$, there exists a growing curve $\tilde{\beta}$ such that $f(\tilde{\beta})=\beta$ and $\tilde{\beta}|_{[t_1,1]}=\tilde{\g}_1$. Then $\alpha$ and $\tilde{\beta}$ lie in the same access by Proposition \ref{pro:disjoint}\,(1).
\end{proof}

\begin{proposition}\label{lem:curve-to-arc}
	Suppose that $E\subset J_f$ and $G$ is a locally connected and $f$-invariant continuum serving as a skeleton of $E$ rel $P$. Let $\alpha\subset K$ be an infinitely growing curve joining $G$ to a preperiodic point $z$. Then  there exists a growing arc $\beta$ in $K$ such that
\begin{enumerate}
\item the arc $\beta$ joins $G$ to $z$ and lies in the same access as $\alpha$;
\item for any $t\in(0,1)$, there exists an integer $n_t>0$ such that $f^{n_t}(\beta[0, t])\subset G$;
\item there exist two integers $q\geq0$ and $p\ge1$, such that $f^{q+p}(\beta)\subset f^q(\beta)\cup G$ and the growing curves $f^i(\beta),i=0,\ldots, q+p-1,$ lie in pairwise distinct accesses.
\end{enumerate}
\end{proposition}

\begin{proof}
	By Lemma \ref{lem:lifting}\,(1), the curves $f^i(\alpha),i\geq0,$ are all infinitely growing, with initial points in $G$.
	According to Lemma \ref{lem:finite-access}, there exist minimal integers $q\geq 0$ and $p\geq 1$ such that $f^{q+p}(\alpha)$ and $f^q(\alpha)$ lie in the same access to $w=f^q(z)$. Set ${\alpha}_0:=f^{p+q}(\alpha) $ and ${\alpha}_1:=f^q(\alpha)$. Then $f^p(\alpha_1)=\alpha_0$. By Lemma \ref{lem:lifting}\,(2), we may assume  the interior of $\alpha_0$ is disjoint from $P$. Then $\alpha_1$ joins $G_p$ to $z$, and its interior is also disjoint from $P$. For simplicity, set $G=G_p$ and $E=E_p$. 
	
	By Proposition \ref{pro:alternative}, we have a continuous family of curves $\{{\alpha}_s\}_{s\in[0,1]}$ joining $G$ to $w$ such that $\alpha_s(0,1)\cap P=\emptyset$ for all $s\in[0,1]$. Define a curve $\de_0:[0,1]\to G$ by $\de_0(s):={\alpha}_s(0)$. As shown in the proof of Lemma \ref{lem:lifting}, there exists a continuous family of curves $\{{\alpha}_{s+1}\}_{s\in[0,1]}$ joining $G_p$ to $w$ such that $f^p\circ {\alpha}_{s+1}={\alpha}_s$. Thus, $\alpha_1$ and $\alpha_2$ lie in the same access to $w$, and we obtain a curve $\de_1:[0,1]\to G_p$ defined by $\de_1(s):={\alpha}_{s+1}(0)$ such that $f^p\circ \de_1=\de_0$.
	
	Inductively, for every $k\geq1$, there exist a curve $\de_k\subset G_{pk}$ and a growing curve ${\alpha}_{k}$ such that
\begin{enumerate}
\item $f^p\circ\de_{k+1}=\de_{k}$ and $\de_{k}(1)=\de_{k+1}(0)$;
	
\item ${\alpha}_k(0)=\de_k(0)$, ${\alpha}_k(1)=w$, and $f^p\circ{\alpha}_{k+1}={\alpha}_k$; 
	
\item $\alpha_k$ lies in the same access as $\alpha_0$.\vspace{2pt}
\end{enumerate}
	For every $m\geq1$, define a growing curve $\ell_m:=\de_0\cdots\de_{m-1}\cdot\alpha_m$. By Proposition \ref{pro:disjoint}\,(1) and point (3) above, the curves $\ell_m$ and $\alpha_0$ lie in the same access to $w$ for every $m\geq1$.
	
	By Lemma \ref{lem:expanding}, the diameters of $\de_k$ and ${\alpha}_k$ exponentially decrease to $0$. Then ${\alpha}_k\to w$ as $k\to\infty$, and $\ell_m$ uniformly converges to a growing curve $\beta_{q+p}\subset K$ with terminal $w$ as $m\to\infty$. Clearly, $f^p(\beta_{q+p})\subset \beta_{q+p}\cup G$, and the curves $\beta_{q+p}$ and $\alpha_0$ lie in the same access.

By successively applying Lemma \ref{lem:lifting}, for each $i=1,\ldots,q+p$, there exists a curve $\beta_{q+p-i}$ joining $G_{i}$ to $f^{q+p-i}(z)$ such that $f^{i}(\beta_{q+p-i})=\beta_{q+p}$ and  $\beta_{q+p-i}$ and $\alpha_{q+p-i}$ lie in the same access to $f^{(q+p-i)}(z)$. By replacing $G$ with $G_{q+p}$, the curve $\beta_0$ satisfies all  requirements of the proposition, except that it may not be an arc.
	
	To complete the proof, it suffices to find an arc  $\beta\subset \beta_0$ joining $G$ to $z$ such that $f^{q+p}(\beta)\subset f^q(\beta)\cup G$. Without loss of generality, we can assume that $q=0$.
	
	\begin{figure}[http]
		\centering
		\includegraphics[width=10cm]{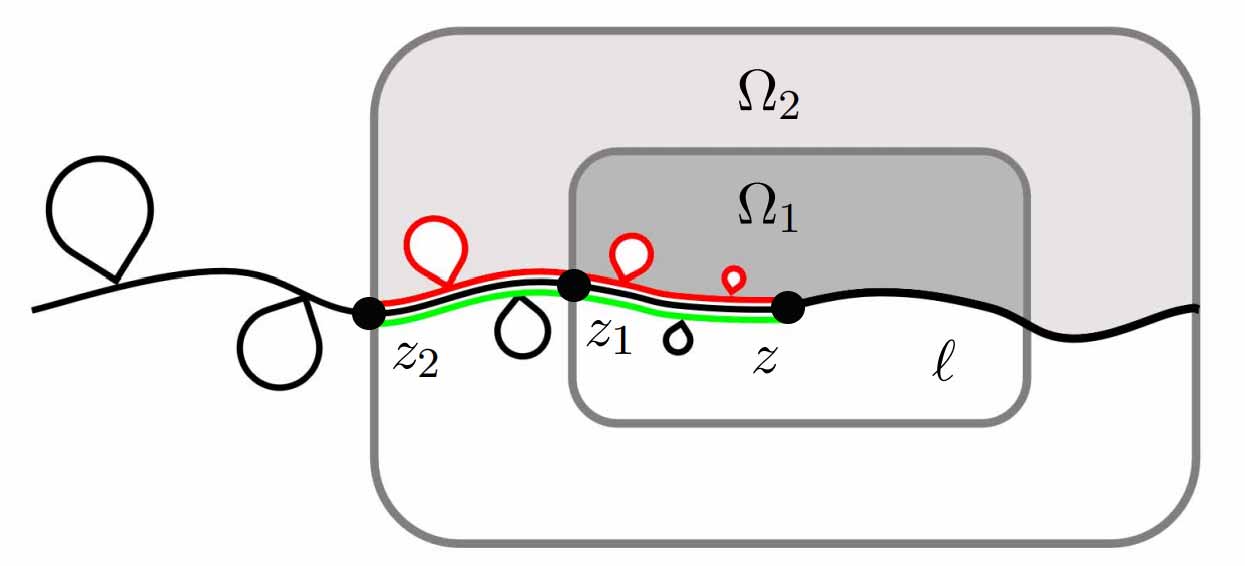}
		\caption{The curves $\eta_1$ and $\eta_2$ are shown in red, and the arcs $\lambda_1$ and $\lambda_2$ in green.}
\label{fig:arc}
	\end{figure}
	
	Take two small disks $D_1$ and $D_2$ containing $z$ such that $\ov{D_1}\subset D_2$ and $g=f^p: D_1\to D_2$ is a homeomorphism. Let $Y_i$ be the closure of the component of $D_i\cap \beta_0$ containing $z$ for $i=1,2$. Clearly, $Y_1\subset Y_2$. Let $\ell\subset D_2\sm Y_2$ be an open arc joining $z$ to a point in $\partial D_2\sm \beta_0$. Then for each $i$, the curves $\partial D_i$, $Y_i$, and $\ell$ bound a simply connected domain $\Omega_i$ with a locally connected boundary such that $\Omega_1\subset \Omega_2$; see Figure \ref{fig:arc}. Let $\eta_i=Y_i\cap \partial\Omega_i$ be the curve joining $z$ to some point $z_i\in\partial D_i$. Then $\eta_1$ is the closure of a component of $\eta_2\sm \{z_1\}$.
	Since $\beta_0$ is locally $g$-invariant near $z$, the map $g$ sends $Y_1, \eta_1$, and $z_1$ homeomorphically onto $Y_2, \eta_2$, and $z_2$, respectively.
	
	We claim that there exists a unique arc $\lambda_i\subset \eta_i$ joining $z$ and $z_i$ for $i=1, 2$. The existence of such an arc follows from the local connectivity of $\eta_i$. The curve $(\partial \Omega_i\sm\eta_i)\cup \lambda_i$ bounds a disk $W_i$ containing $\Omega_i$. Clearly, $\eta_i\subset\ov{W_i}$. Suppose $\lambda_i'$ is another such arc. Then $\partial W'_i\subset \ov{W_i}$ and  $\partial W_i\subset \ov{W'_i}$. Thus, $W_i=W_i'$, which implies  $\lambda_i=\lambda_i'$.
	
	Note that $g(\lambda_1)\subset \eta_2$ is an arc joining $z$ and $z_2$. By the uniqueness of $\lambda_1$ and $\lambda_2$, we have that $g(\lambda_1)=\lambda_2$ and $\lambda_1$ is the sub-arc of $\lambda_2$ from $z$ to $z_1$.
	
	Choose a sufficiently large integer $N$ such that $G_N$ contains $\lambda_2\setminus \lambda_1$, and define $\beta:=\lambda_1$. Then $\beta\subset \beta_0$ is an arc satisfying $f^p(\beta)\subset \beta\cup G_N$. The proof is completed by replacing $G$ with $G_N$.
\end{proof}

\subsection{Links between growing continua}\label{sec:max-inv}

In the previous subsection, we proved that if $z\in K$ is a preperiodic point, then there exists a preperiodic growing arc within any access to $z$. In this final part of Section \ref{sec:topology}, we aim to find abundant preperiodic points as terminals of growing curves.

Let $K_\pm$ be  growing continua generated by $f$-invariant and locally connected continua $E_\pm$, respectively, such that $E_-\cap E_+=\emptyset$. This implies that $E_{-,k}\cap E_{+,k'}=\emptyset$ for any $k, k'\geq 0$, where $E_{\pm, k}$ are the components of $f^{-k}(E_\pm)$ containing $E_\pm$, respectively. We continue to assume that $E_\pm$ serve as skeletons of $E_{\pm,k}$ (rel $P$) for every $k\ge0$.

A {\bf link} between $K_-$ and $K_+$ is a curve $\g$ with $\g(0)\in E_{-,k}$ and $\g(1)\in E_{+,k}$ for some $k\geq0$, such that one of the following two cases occurs:
\begin{itemize}
\item $\g$ is a growing curve in either $K_-$ or $K_+$ (\emph{one-sided link}); or 
\item $\g=\alpha_-\cdot\alpha_+^{-1}$, where $\alpha_\pm$ are growing curves in $K_\pm$, respectively, with a common terminal disjoint from both $P$ and any $E_{\pm,m}$ for $m\geq0$ (\emph{two-sided link}).
\end{itemize}

The  unique terminal $z$ of the growing curves in $\g$ is called
the {\bf infinity point} of the link $\g$. By definition, $\#\g^{-1}(z)=1$, and  it holds for a two-sided link that $\alpha_+\cap \alpha_-=\{z\}$. Moreover, a link $\g$ is  one-sided if and only if the infinity point is contained in a certain $E_{\pm, k}$, if and only if the infinity point is an endpoint of $\g$.

\begin{figure}[http]
	\centering
	\begin{tikzpicture}
		\node at (0,0){\includegraphics[width=14.5cm]{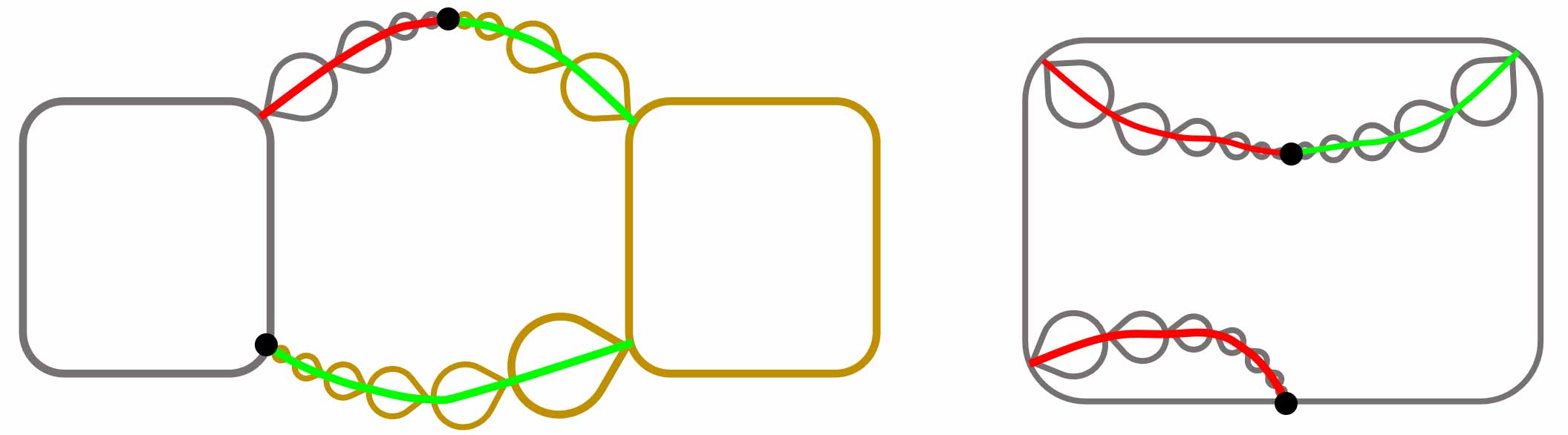}};
		\node at (4.7,-2.5){Self-links of $K$};
		\node at (-3,-2.5){Links between $K_\pm$};
		\node at (-3.15, -1.15){$\gamma'$};
		\node at (3.5, -0.5){$\gamma'$};
		\node at (4.725, 0.15){$\gamma$};
		\node at (-3.15, 1.25){$\gamma$};
	\end{tikzpicture}
	\caption{Two types of links}\label{fig:definition}
\end{figure}

The left image in Figure \ref{fig:definition} illustrates two types of links: the curve  $\g'$ is a one-sided link, while $\g$ is a two-sided link.

Set $P_0=P\setminus (E_+\cup E_-)$. Then $P_0$ is disjoint from $E_{\pm,m}$ for every $m\geq0$ since $E_\pm$ serve as skeletons of $E_{\pm,m}$, respectively. Two links $\g_1$ and $\g_2$ between $K_\pm$ are said to be {\bf equivalent} if there exist two curves $\de_\pm\subset E_{\pm,k}$ for some $k$, such that 
 $\de_-$ joins $\g_1(0) ,  \g_2(0)\in E_{-,k}$, 
$\de_+$ joins $\g_1(1) ,  \g_2(1)\in E_{+,k}$, and 
 the closed curve $\de_-\cdot\g_2\cdot\de_+^{-1}\cdot \g_1^{-1}$ is contractible in $\ov{\C}\setminus P_0$. \vspace{1pt}

This is an equivalence relation. 
Moreover,  link-equivalence is closely related to the concept of access defined in the previous subsection as follows:
\begin{itemize}
\item [(i)] If $\g\subset K_-$ is a one-sided link between $K_\pm$, then it must be an infinitely growing curve in $K_-$ since no $E_{-,k}$ contains the infinity point of $\g$. Moreover, any growing curve in the same access as $\g$ is also a link and  equivalent to $\g$ as a link. However, the converse does not hold since two equivalent one-sided links  may have distinct terminals.
\item [(ii)] If $\g=\alpha_-\cdot\alpha_+^{-1}$ is a two-sided link between $K_\pm$, then both $\alpha_\pm$ are infinitely growing. Moreover, if $\beta_\pm\subset K_\pm$ are growing curves in the same accesses as $\alpha_\pm$, respectively, then $\g':=\beta_-\cdot\beta_+^{-1}$ is a two-sided link equivalent to $\g$. 
\end{itemize}

Corresponding to Proposition \ref{pro:disjoint}, we have the following result for links.

\begin{proposition}\label{pro:disjoint1}
	Let $\g$ be a link between $K_\pm$. Then the following statements hold:
\begin{enumerate}
\item Any sub-curve of $\g$ joining $E_{\pm,k}$ for an integer $k$ is a link equivalent to $\g$.
\item For every sufficiently large integer $k$, there exist two numbers $t_{\pm,k}\in[0,1]$ such that $\g(t_{\pm, k})\in E_{\pm, k}$, respectively, and $\g(t_{-,k},t_{+,k})$ is disjoint from $E_{-,k}\cup E_{+,k}$. Moreover, $\g[t_{-,k}, t_{+,k}]$ contains an arc $\beta_k$ homotopic to $\g|_{[t_{-,k}, t_{+,k}]}$ rel $P$ with endpoints fixed. In particular, $\beta_k$ is a link between $K_\pm$ that is equivalent to $\g$ and has the same infinity point as $\g$.
\item Suppose that $\gamma$ and $\gamma'$ are equivalent links between $K_\pm$, with their interiors disjoint from $P$. Then there exist an integer $m\geq0$ and a continuous family of curves $\{\g_s\}_{s\in [0, 1]}$ such that $\gamma_0=\g$, $\g_1=\g'$, and  each $\g_s$ joins $E_{-,m}$ to $E_{+,m}$ with its interior  disjoint from $P$.		
\end{enumerate}
\end{proposition}

\begin{proof}
 According to the relationship between link-equivalence and access stated above this proposition, statements (1) and (2) follow directly from Proposition \ref{pro:disjoint}\,(1)--(3).

To prove statement (3), suppose first that
the infinity points of $\g$ and $\g'$ coincide. Then $\g$ and $\g'$ are either both one-sided links in one of $K_{\pm}$, or both two-sided links. In this case, statement (3) is an immediate consequence of Proposition \ref{pro:disjoint}\,(4).

If the infinity points of $\g$ and $\g'$ are distinct, by statements (1) and (2), we may assume  that $\g$ and $\g'$ are disjoint arcs serving as crosscuts of the unique annular component $A$  of $\cbar\setminus(E_{-}\cup E_{+})$.   Since $\g$ and $\g'$ are equivalent, there exists a simply connected component $D_*$ of $A\setminus (\g\cup \g')$ such that $\g,\g'\subset \partial D_*$ and $D_*\cap P=\emptyset$. The required curves $\{\g_s\}$ can be chosen within $\ov{D_*}$.  
\end{proof}

Based on this proposition, we can prove our desired result. 

\begin{proposition}\label{pro:criter}
Suppose that $K_\pm\subset J_f$ and $\gamma$ is a link between $K_\pm$. If the infinity point of $\g$ is wandering, then there exists a curve $\ell=\beta_-\cdot\beta_+^{-1}$ such that
\begin{enumerate}
\item  $\beta_\pm$ are growing curves in $K_\pm$, respectively, and their common terminal is preperiodic;
\item there exists a sequence of curves $\{\ell_k\}$ such that each $\ell_k$ is homotopic to $\g$ rel $P_0$ with endpoints fixed and $\ell_k\to \ell$ as $k\to\infty$.
\end{enumerate}
\end{proposition}

Note that the curve $\ell$ is not necessarily a link between $K_\pm$ since the common terminal of $\beta_\pm$ may be a marked point.  

\begin{proof}
We first claim that the links between $K_\pm$ belong to finitely many equivalence classes.

Let $\Sigma$ be a finite collection of links between $K_\pm$  in pairwise distinct equivalence classes. To prove the claim, it suffices to show that $\#\Sigma\leq (\#P)^6$.	
By Proposition \ref{pro:disjoint1},  we may assume \vspace{2pt}

$\bullet$ each curve in $\Sigma$ is an arc that serves as a crosscut of some component of $\ov\C\setminus (E_{-,m_0}\cup E_{+, m_0})$;
	
$\bullet$ if two arcs in $\Sigma$ have distinct infinity points, then they are disjoint.\vspace{2pt}
	
	Let $Z$ denote the set of infinity points of links in $\Sigma$. Decompose $\Sigma$ as $\Sigma=\bigcup_{z\in Z}\Sigma_z$, where $\Sigma_z$ is the collection of links in $\Sigma$ with infinity point $z$. Pick a representative element in each $\Sigma_z$ and denote their collection by $\Sigma_1$. Then $\#\Sigma_1=\# Z$ and the links in $\Sigma_1$ are disjoint.
By a similar argument as in the proof of Lemma \ref{lem:finite-access}, we have $\#\Sigma_1\leq (\#P)^2$. 
	
Fix  $z\in Z$. By the relationships (i) and (ii) between link-equivalence and access as stated before Proposition \ref{pro:disjoint1},
it follows from Lemma \ref{lem:finite-access} that $\#\Sigma_z\leq (\#P)^4$. Therefore, $\#\Sigma\leq (\# P)^6$, which proves the claim.

Since the infinity point $z$ of $\g$ is wandering, it cannot be iterated into $P$. Thus, for each $i\geq0$, the curve $f^i(\g)$ is a link between $K_\pm$.  By the claim above, there exist integers $q\geq 0$ and $p\geq 1$ such that $f^q(\g)$ and $f^{q+p}(\g)$ are equivalent. Set $\g_0:=f^{q+p}(\g)$ and $\g_1:=f^q(\g)$. 

By Proposition \ref{pro:disjoint1}\,(1), we may assume, by taking sub-curves if necessary,  that the interiors of $\gamma_0$ and $\gamma_1$ are disjoint from $P$. Then by Proposition \ref{pro:disjoint1}\,(3), there exists a continuous family $\{\gamma_s\}_{s\in[0, 1]}$ of curves joining $E_{\pm, k_0}$, with their interiors disjoint from $P$. Define two curves $\delta_{\pm,0}$  by $\delta_{-,0}(s):=\g_s(0)$ and  $\delta_{+,0}(s):=\g_s(1)$, $s\in[0,1]$. Then $\de_{\pm,0}\subset E_{\pm, k_0}$, respectively. 

Since $f^p(\gamma_1)=\gamma_0$, for any $t\in (0, 1)$, the curve $\{\g_s(t):s\in[0,1]\}$ has a unique lift by $f^p$ based at $\g_1(t)$, denoted by $\{\g_{s+1}(t):s\in[0,1]\}$. Therefore, we obtain a continuous family of curves $\{\gamma_{s+1}\}_{s\in[0, 1]}$ such that $f^p\circ \g_{s+1}=\g_s$. Consequently, $\g_2$ is a link between $K_\pm$ and equivalent to $\g_1$. Define two curves $\de_{\pm,1}$ by $\delta_{-,1}(s):=\gamma_{s+1}(0)$ and $\delta_{+,1}(s):=\gamma_{s+1}(1)$, $s\in[0,1]$. Then  $\de_{\pm,1}\subset E_{\pm, k_0+p}$ and  $f^{p}(\delta_{\pm, 1})=\delta_{\pm, 0}$, respectively.

Inductively applying the argument above, for each $k\geq1$, we obtain
\begin{itemize}
\item two equivalent links $\g_k$ and $\g_{k+1}$ between $K_\pm$ such that $f^p(\g_{k+1})=\g_k$;
\item a curve $\delta_{-, k}\subset E_{-, k_0+kp}$ joining $\g_k(0)$ to $\g_{k+1}(0)$ such that $f^p(\de_{-,k})=\de_{-,k-1}$; and 
\item a curve $\delta_{+, k}\subset E_{+, k_0+kp}$ joining $\g_k(1)$ to $\g_{k+1}(1)$ such that $f^p(\de_{+,k})=\de_{+,k-1}$.
\end{itemize}

Without loss of generality, we may assume that $q=0$. For each $m\geq1$, let $\beta_{-,m}$ and $\beta_{+,m}$ denote the concatenations of $\{\de_{-,k}\}_{k=1}^m$ and $\{\de_{+,k}\}_{k=1}^m$, respectively.
 By Lemma \ref{lem:expanding},  the diameters of $\g_k$ and $\delta_{\pm,k}$ exponentially decrease to $0$. It follows that ${\g}_k$ converges  to a point $x$ with $f^p(x)=x$, and that $\beta_{\pm,m}$ uniformly converges to growing curves $\beta_\pm$ in $K_\pm$, respectively, such that $\beta_\pm$ have the common terminal $x$.

For each $m\geq1$, define $\ell_m:=\beta_{-,m}\cdot{\g}_{m+1}\cdot\beta_{+,m}^{-1}$. Then $\ell_m$ is homotopic to $\g_1$ rel $P_0$ with endpoints fixed.  Immediately, $\ell_m$ converges to  $\ell:=\beta_-\cdot\beta_+^{-1}$ as $m\to\infty$.
\end{proof}

Finally, let $K$ be a growing continuum generated by an $f$-invariant and locally connected continuum $E$. Similar to the notion of links between $K_\pm$, we can define \emph{self-links of $K$}.

A {\bf self-link} of $K$ is a curve $\g\subset K$ with $\g(0),\g(1)\in E_{k}$  for some $k\geq0$ such that one of the following two cases occurs:
\begin{itemize}
\item $\g$ is an infinitely growing curve in $K$ (\emph{one-sided self-link}); or 
\item $\g=\alpha_-\cdot\alpha_+^{-1}$, where $\alpha_\pm$ are infinitely growing curves in  distinct accesses with a  common terminal that avoids both $P$ and every $E_k$ for $k\geq 0$ (\emph{two-sided self-link}).
\end{itemize}
 The unique terminal of the growing curves in $\g$ is called
the {\bf infinity point} of the self-link $\g$; see the right image of Figure \ref{fig:definition}.

Let $P_0=P\setminus E$. Two self-links $\g_1$ and $\g_2$ are called {\bf equivalent} if there exist two curves $\de_\pm\subset E_{k}$ for some $k$, such that 
$\de_-$ joins $\g_1(0)$ to $\g_2(0)$,
$\de_+$ joins $\g_1(1)$ to $\g_2(1)$, and 
the closed curve $\de_-\cdot\g_2\cdot\de_+^{-1}\cdot \g_1^{-1}$ is contractible in $\ov{\C}\setminus P_0$. 

Let $\g$ be a self-link of $K$, and let $z$ be the infinity point of $\g$. It is worth noting that $f\circ\g$ is also a self-link provided that $f(z)\notin P_0$. Indeed, if $\g$ is a one-sided self-link, this result  holds by Lemma \ref{lem:lifting}\,(1). In the case where $\g=\alpha_-\cdot\alpha_+^{-1}$ is a two-sided self-link, if the conclusion were false, then $f\circ\alpha_\pm$ would lie in the same access. Since $f$ is injective near $z$, it follows from Lemma \ref{lem:lifting}\,(2) that $\alpha_\pm$ lie in the same access to $z$, a contradiction.

With these definitions and a parallel argument, we can apply a similar argument as in the proof of Proposition \ref{pro:criter} to derive the following result. Details are omitted. 

\begin{proposition}\label{pro:criter1}
Suppose that $K\subset J_f$ and $\gamma$ is a self-link of $K$. If the infinity point of $\g$ is wandering, then there exists a curve $\ell=\beta_-\cdot\beta_+^{-1}$ such that
\begin{enumerate}
\item $\beta_\pm$ are growing curves in $K$, and their common terminal is preperiodic;
\item there exists a sequence of curves $\ell_k$ such that each $\ell_k$ is homotopic to $\g$ rel $P_0$ with endpoints fixed and $\ell_k\to \ell$ as $k\to\infty$.
\end{enumerate}
\end{proposition}

\section{Invariant graphs in maximal Fatou chains}\label{sec:7}
In this section, we prove that every periodic level-$n$ extremal chain admits an invariant graph on the Julia set if $n\geq 1$. Our proof relies on the inductive construction and the topology of extremal chains established in Sections \ref{sec:chain} and \ref{sec:topology}, respectively.

\subsection{Invariant graphs associated with level-\texorpdfstring{0}{ } Fatou chains}
Let $(f,P)$ be a marked rational map. We will analyze the dynamics of $f$ on the union of periodic level-$0$ Fatou chains.

Suppose that $E$ is a component of the union of all periodic level-$0$ Fatou chains with period $p$. Let $K$ be the level-$1$ extremal chain containing $E$.  The main result of this subsection is as follows, which  generalizes Theorem \ref{thm:local}.

\begin{proposition}\label{prop:graph1}
	There exists a graph $G\subset K\cap J_f$ such that $f^p(G)\subset G$, which is isotopic to a skeleton of $\partial E$ rel $P$. Moreover, for each point $z\in G\sm E$, there exist an integer $n_0\geq 1$ and a component $D$ of $\cbar\sm E$ with $\ov{D}\cap P=\emptyset$ such that $f^{n_0p}(z)\in \ov{D}$.
\end{proposition}

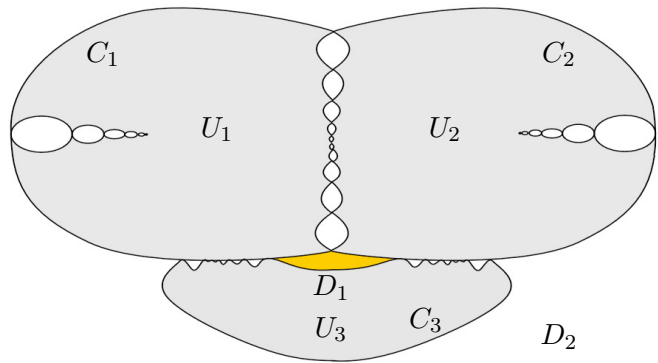
\begin{figure}[http]
	\centering
	\begin{tikzpicture}
		\node at (0,0){ \includegraphics[width=9cm]{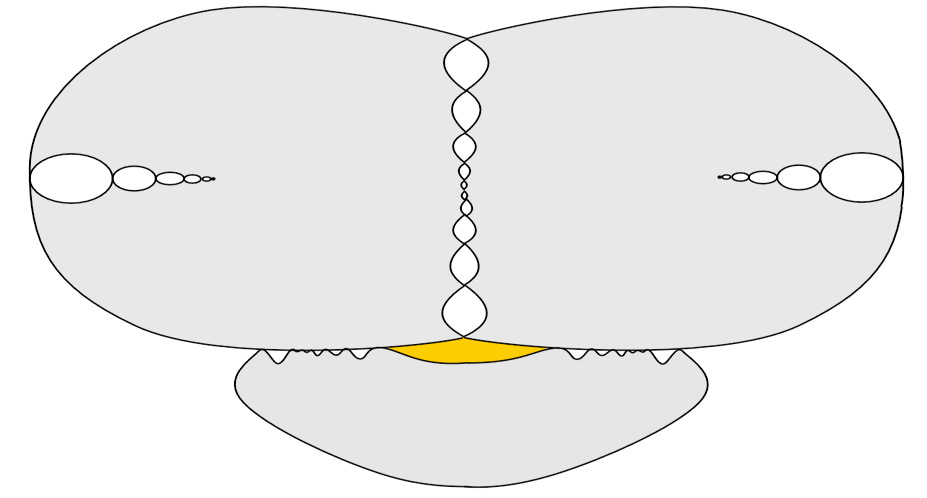}};
		\node at (-1.5,0.75) {$U_1$};
		\node at (1.5,0.75) {$U_2$};
		\node at (0,-1.9){$U_3$};
		\node at (-3,1.75) {$C_1$};
		\node at (3,1.75){$C_2$};
		\node at (1.25, -1.78){$C_3$};
		\node at (0, -1.35){$ D_1$};
		\node at (3,-2){$D_2$};
	\end{tikzpicture}
	\caption{The circle-graph $T$ of $E=\ov{U_1\cup U_2\cup U_3}$. The circles $C_i, i=1,2,3,$ are intersection circles, and the irregular boundary circles of $T$ are $\partial D_1$ and $\partial D_2$.}\label{fig:circle-intersection}
\end{figure}

If $E$ contains exactly one Fatou domain,  this proposition is a combination of Theorem \ref{thm:local} and Corollary \ref{cor:C}. Thus, we assume that $E$ contains $m\geq 2$ Fatou domains. The proof of Proposition \ref{prop:graph1} follows a similar approach as in that of Theorem \ref{thm:local}, with the distinction being the presence of intersection points between boundaries of different Fatou domains.

A point $x\in\partial E$ is called an {\bf intersection point} if it lies on the boundaries of at least two distinct Fatou domains within $E$. A circle $C\subset\partial E$ is called an {\bf intersection circle} if it lies on the boundary of a Fatou domain $U\subset E$ and  separates $U$ from another Fatou domain in $E$. Recall that a circle $C\subset\partial U$ is marked if it either intersects or separates $P$. Thus, every intersection circle is marked; see Figure \ref{fig:circle-intersection}.

By definition, each intersection point of $E$ is contained in an intersection circle, and conversely, each intersection circle of $E$ contains intersection points. Note that there exist at most $2(m-1)$ distinct intersection circles in $E$. Moreover, a component of $\cbar\sm E$ is not a disk if and only if its boundary contains an intersection circle. On the other hand, for each intersection circle $C$, there exists at most one component $D$ of $\cbar\sm E$ such that $C\subset\partial D$. Therefore, there exist at most $2(m-1)$ components of $\cbar\sm E$ that are not disks.

For each Fatou domain $U\subset E$, we denote by $T_{U}\subset\partial U$ the finite circle-tree spanned by $\partial U\cap P$ and all marked circles in $\partial U$; see Lemma \ref{lem:coincide} for background. Set
$$
T:=\bigcup_{U\subset E}T_{U}.
$$
Since the intersection points of $E$ are contained in the intersection circles, which are all marked,
it follows that $T$ is connected. By Lemmas \ref{lem:tree-image} and \ref{lem:coincide}, we also have $f^p(T)\subset T$. Moreover, $T$ is a skeleton of $\partial E$ (rel $P$) since each $T_{U}$ is a skeleton of $\partial U$.

Let $X_0$ be the union of $P$ together with all intersection points of $E$ and all cut points of $T_{U}$ for all Fatou domains $U\subset E$. Then $X_0$ is compact, and $f^p(X_0)\subset X_0$. Moreover, each component of $T\sm X_0$ is an open arc contained in a circle on the boundary of a Fatou domain in $E$.

There exist $m$ components of $\ov{\C}\setminus T$, each containing a Fatou domain in $E$. Let $T_*$ denote the union of $T$ and these $m$ components. Since $T_*$ contains all intersection circles of $E$, by the same reasoning as before, there exist at most $2(m-1)$ components of $\ov\C\setminus T_*$ that are not disks. Therefore, $T$ has at most $2(m-1)+m$ complementary components that are not disks.

By a {\bf boundary circle} of $T$, we mean the boundary of a component of $\ov\C\setminus T_*$ that is a disk. A boundary circle $C$ of $T$ is called {\bf regular} if $\# (C\cap X_0)=2$ and $\ov{D}\cap P=\emptyset$, where $D$ is the component of $\ov\C\setminus T$ with $\partial D=C$, and is called {\bf irregular} otherwise.

\begin{lemma}\label{lem:regular}
There exist finitely many irregular boundary circles of $T$.
\end{lemma}

\begin{proof}

	Let $D$ be a component of $\cbar\sm T_*$ that is a disk. Then either $D$ is a component of $\ov\C\setminus \ov{U}$ for a Fatou domain $U\subset E$, or the boundary $\partial D$ is composed of at least two arcs, which are sub-arcs of distinct intersection circles.
	
	In the former case, if $\partial D$ is a regular circle of $T_{U}$, then it is a regular boundary circle of $T$. Since $T_{U}$ contains finitely many irregular circles, there exist finitely many irregular boundary circles of $T$ of this type.
	
	In the latter case, the circle $\partial D$ of $T$ contains at least two intersection points, say $z_1$ and $z_2$. If $\partial D$ is irregular, then either $\ov{D}\cap P\neq \emptyset$; or $\partial D\sm\{z_1, z_2\}$ consists of two open arcs $\alpha_i\subset C_i$, $i=1,2$, where $C_i$ is a circle of $T_{U_i}$ for a Fatou domain $U_i\subset E$, such that $\alpha_1$ or $\alpha_2$ contains cut points of $T_{U_1}$ or $T_{U_2}$, respectively;
	 or $\partial D\cap X_0$ contains at least three intersection points. The number of components of the first type is clearly finite. Note that each $C_i$ is an intersection circle and  contains finitely many cut points of $T_{U_i}$. Then the number of components of the second type is also finite.  
	To complete the proof of the lemma, it suffices to verify the following claim.
	
	\vskip 0.2cm
	\emph{Claim}. Let $\Omega_1, \ldots, \Omega_n,$ with $n\geq 2$, be pairwise disjoint disks such that $B:=\bigcup_{i=1}^n\ov{\Omega_i}$ is connected. Considering the components of $\cbar\sm B$ that are disks,  the boundaries of all but finitely many of these components contain exactly two \emph{intersection points} of $B$, i.e., points belonging to at least two of $\partial{\Omega_1},\ldots,\partial{\Omega_n}$.
	\vskip 0.2cm

	First, suppose that $n=2$. If $\#(\partial\Omega_1\cap\partial\Omega_2)=1$, then $\cbar\sm B$ is connected, and $\partial B$ contains only one intersection point. If $\#(\partial\Omega_1\cap\partial\Omega_2)>1$, then the boundary of any component of $\cbar\sm B$ contains exactly two intersection points. Thus, the claim holds for $n=2$.
	
	By induction, we assume that the claim holds for $n\ge 2$. Let $\Omega_{0}$ be a disk disjoint from $\Omega_1,\ldots,\Omega_n$  such that both $\bigcup_{i=0}^n\ov{\Omega_i}$ and $B=\bigcup_{i=1}^n\ov{\Omega_i}$ are connected. Then $\Omega_0$ is contained in a component $D$ of $\cbar\sm B$.  The intersection points of $B\cup\ov{\Omega_0}$ are the union of the intersection points of $B$ together with $\partial \Omega_0\cap \partial D$.
	
	For any component $D'$ of $\cbar\sm B$ other than $D$, the points in $\partial{\Omega_0}\cap \partial D'$ are the intersection points of $B$ in $\partial D'$. Thus, it suffices to verify that the boundaries of all but finitely many components of $D\setminus\ov{\Omega_0}$ contain two intersection points of $\ov{\Omega_0}\cup B$.
	
	If $\partial D\cap \partial\Omega_0$ is a singleton, then $D\setminus \ov{\Omega_0}$ is connected. If $\#(\partial D\cap \partial\Omega_0)\geq2$,  except for finitely many ones, every component of $D\setminus\ov{\Omega_0}$ is a disk whose boundary contains exactly two points of $\partial D\cap \partial \Omega_0$ and consists of one open arc in $\partial D$ and the other in $\partial\Omega_0$. Thus, there exist finitely many components of $D\setminus \ov{\Omega_0}$ whose boundaries contain more than two intersection points of $\ov{\Omega_0}\cup B$, since $\partial D$ has finitely many intersection points of $B$. The claim is proved.
\end{proof}

\begin{proof}[Proof of Proposition \ref{prop:graph1}]
	We use a similar argument as in the proof of Theorem \ref{thm:local}. For a regular boundary circle $C$ of $T$, let $C^\pm$ denote the two components of $C\sm X_0$, and let $B(C^-)=B(C^+)$ denote the closure of the component of $\cbar\setminus T$ whose boundary is $C$. Set
	$$
	G_1=T\sm \bigcup C^-,
	$$
	where the union is taken over all regular boundary circles of $T$. By Lemma \ref{lem:regular}, $G_1$ is a graph serving as a skeleton of $\partial E$ rel $X_0$.
	
	Now, we construct $G_2\subset f^{-p}(G_1)$. For each $n\geq1$, set $X_n:=f^{-np}(X_0)$. Then $X_n\subset X_{n+1}$.  Note that if $z\in X_1\cap G_1$, then $f^p(z)\in X_0\cap T\subset G_1$. Thus, for a component $\alpha_1$ of $G_1\sm X_1$, its image $f^p(\alpha_1)$ is a component of $T\sm X_0$. 
\begin{itemize}
\item If $f(\alpha_1)=C^-$ for a regular boundary circle $C$ of $T$, since $C^+$ and $C^-$ are isotopic rel $X_0$, there exists a unique component $\alpha_1^+$ of $f^{-p}(C^+)$  isotopic to $\alpha_1$ rel $X_1$.  Such an arc $\alpha_1$  is called a {\it deformation arc of $G_1$}. Denote by $B(\alpha_1)$  the component of $f^{-p}(B(C^-))$ containing $\alpha_1$. Then $B(\alpha_1)$ is a closed disk such that $B(\alpha_1)\cap G_1=\ov{\alpha_1}$ and $B(\alpha_1)\cap X_1=\{\alpha_1(0),\alpha_1(1)\}$.
\item In the other case, we have $f^p(\alpha_1)\subset G_1$ by the construction of $G_1$.
\end{itemize}
	  
	Define the graph $G_2$ as 
	$$
	G_2:=\left(G_1\sm\bigcup\alpha_1\right)\cup\bigcup\alpha^+_1,
	$$
	where the union is taken over all deformation arcs of $G_1$. From the previous discussion, we have $f^p(G_2)\subset G_1$, and
	there exists an isotopy $\Theta^1:\cbar\times [0,1]\to\cbar$ rel $P$ such that $\Theta^1_t:=\Theta^1(\cdot,t)$ satisfies
	\begin{enumerate}
\item $\Theta^1_0=id$ on $\cbar$;
	
\item $\Theta^1_t(z)=z$ on a neighborhood of attracting cycles of $f$ for $t\in [0,1]$;
	
\item if $z\in G_1$ is not in any deformation arc, then $\Theta^1_t(z)=z$ for $t\in [0,1]$; and
	
\item if $\alpha_1$ is a deformation arc of $G_1$, then $\Theta^1_1(\alpha_1)=\alpha_1^+$ and $\Theta^1(\ov{\alpha_1}\times[0,1])=B(\alpha_1)$. 
	\end{enumerate}
	\noindent Consequently, $\theta_1(G_1)=G_2$ with $\theta_1:=\Theta^1_1$.
	
	By inductively applying Lemma \ref{lem:lift}, we obtain an isotopy $\Theta^n:\cbar\times [0,1]\to\cbar$ rel $P$ and a graph $G_{n+1}$ for each $n\geq1$, such that $\Theta^n_0=id$ and  $\Theta^{n}_t\circ f^p(z) =f^p\circ \Theta^{n+1}_t(z)$  for all $z\in\ov\C$, $t\in [0,1]$, and $G_{n+1}=\theta_n(G_n)$ with $\theta_n:=\Theta^{n}_1$.
	Thus, $f^p(G_{n+1})\subset G_{n}$. Besides, there exist some components of $G_n\setminus X_n$, called the {\it deformation arcs of $G_n$ $($under $\Theta^n)$}, such that
	\begin{itemize}
\item  if $z\in G_n$ is not in any deformation arc of $G_n$, then $\Theta^n_t(z)=z$ for $t\in [0,1]$;
\item  if $\alpha_n$ is a deformation arc of $G_n$, then 
	the deformation  of $\ov{\alpha_n}$ under $\Theta^n$, denoted by $B(\alpha_n)$, is a closed disk such that $B(\alpha_n)\cap G_n=\ov{\alpha_n}$ and $B(\alpha_n)\cap X_n=\{\alpha_n(0),\alpha_n(1)\}$.
	\end{itemize}
	
	Denote $\phi_n=\theta_{n-1}\circ\cdots\circ\theta_0$ for $n\ge 1$ with $\theta_0:=id$. Then $G_{n}=\phi_n(G_1)$. By Lemma \ref{thm:isotopy}, $\{\phi_n\}$ uniformly converges to a quotient map $\varphi$ of $\cbar$. It follows that $f^p(G)\subset G$ with $G:=\varphi(G_1)$.
	
	Fix a deformation arc $\alpha_n$ of $G_n$, $n\geq1$, and set $\alpha_{n-k}:=f^{kp}(\alpha_n)$ for  $0\leq k\leq n$. From the lifting construction of $\Theta^n$, it follows that   $\alpha_{n-k}$ is a deformation arc of $G_{n-k}$ and  $f^{kp}(B(\alpha_n))=B(\alpha_{n-k})$ for $0\leq k\leq n-1$, and that $\alpha_0=C^-$ for a regular boundary circle $C$ of $T$ and 
	$f^{np}:B(\alpha_n)\to B(\alpha_0)$ is a homeomorphism.

	\begin{proposition}\label{prop:pearl1}
		Let $\alpha_m$ and $\beta_n$ be two distinct deformation arcs of $G_m$ and $G_n$, respectively, with $m\geq n\geq1$. Then either $B(\alpha_m)\subset B(\beta_n)$, or $\#\,(B(\alpha_m)\cap B(\beta_n))\leq 2$.
	\end{proposition}
	
	\begin{proof}
		Set $\beta_0:=f^{np}(\beta_n)$ and $\alpha_{m-n}:=f^{np}(\alpha_m)$. 
		We claim that either $B(\alpha_{m-n})\subset B(\beta_0)$, or $\#\,(B(\alpha_{m-n})\cap B(\beta_0))\leq 2$.
		Note that $\beta_0=C^-$ for a regular boundary circle $C$ of $T$.  The two open arcs  $C^\pm$ are contained in the boundaries of Fatou domains $U_1,U_2\subset E$, respectively.
		
		If $U_1=U_2$, the interior of $B(\beta_0)$ is a component of $\ov{\C}\setminus \ov{U_1}$, and $B(\alpha_{m-n})\subset \ov{D}$ for a component $D$ of $\ov{\C}\setminus \ov{U_1}$. Thus, either $B(\beta_0)=\ov D$ or $\#(B(\beta_0)\cap \ov{D})\leq 1$ by Lemma \ref{lem:circle}. Then the claim holds.
		
		If $U_1\not= U_2$, there exists a component $D$ of $\ov\C\setminus \ov{U_1}$ such that $U_2\subset D$ and  the interior of $B(\beta_0)$ is a component of $D\setminus \ov {U_2}$. Moreover, there exists a component $W$ of $\ov\C\setminus (\ov{U_1}\cup \ov{U_2})$ with $B(\alpha_{m-n})\subset \ov {W}$. If $W$ is a component of $\ov\C\setminus \ov{U_1}$ or $\ov\C\setminus \ov{U_2}$, then $\#(\ov {W}\cap B(\alpha_0))\leq 1$ by Lemma \ref{lem:circle}. Otherwise, $W$ is a component of $D\setminus\ov{U_2}$. In this case, either $\ov{W}=B(\beta_0)$, or $\ov{W}\cap B(\beta_0)$ consists of at most two intersection points in $X_0\cap C$. Then the claim also holds.
		
		The proposition follows directly from the above claim and a pullback argument.
	\end{proof}
	
	The remaining parts of the proof of Proposition \ref{prop:graph1} are the same as the corresponding parts of the proofs of Theorem \ref{thm:local} and Corollary \ref{cor:C}. We omit the details.
\end{proof}

\begin{corollary}\label{cor:disjoint}
	Suppose that $K\not= K'$ are  periodic level-$1$ extremal chains. Let $G\subset K$ and $G'\subset K'$ be invariant graphs  derived from Proposition \ref{prop:graph1}. Then $G\cap G'=\emptyset$.
\end{corollary}

\begin{proof}
	Without loss of generality, we may assume that both $K$ and $K'$ are $f$-invariant. Let $E$ and $E'$ denote the union of all periodic level-$0$ Fatou chains contained in $K$ and $K'$, respectively. Then $K=\overline{\bigcup_k E_k}$ and $K'=\overline{\bigcup E_k'}$.  Moreover, $E_k\cap E'_m=\emptyset$ for any $k, m\geq 0$.
	
	Suppose, to the contrary, that $G\cap G'$ contains a point $z$.  We can assume that $f^n(z)\not\in E$ for all $n\geq0$ since $E\cap E'=\emptyset$.
	Since  $E\cap E_k'=\emptyset$ for every $k\geq1$,  all $E_k'$ lie in the same component of $\ov\C\setminus E$. On the other hand,
	by Proposition \ref{prop:graph1}, there exist an integer $n_0\geq1$ and a component $D$ of $\ov\C\setminus E$ such that $\ov{D}\cap P=\emptyset$ and $f^{n_0}(z)\in \ov{D}$. Since $f^{n_0}(z)\not\in E$, we obtain $f^{n_0}(z)\in D$. Then $K'$ intersects $D$. It follows that  $E_k'$ intersects $D$ for a sufficiently large integer $k$, and hence $E'\subset D$. However, this contradicts $D\cap P=\emptyset$.
\end{proof}

\begin{corollary}\label{coro:added-corollary}
	Suppose that $K$ is an $f$-invariant level-$1$ extremal chain, and  $E$ is the union of boundaries of periodic Fatou domains in $K$. Let $G\subset K$ be the invariant graph obtained in Proposition \ref{prop:graph1}. Set $S:= E\cup G$.  Then, $S_n\subset K$ for $n\geq 1$, and $G_N$ is a skeleton of $S_n$ for some $N$ and all $n\geq N$, where $S_n$ and $G_n$ are the components of $f^{-n}(S)$ and $f^{-n}(G)$ containing $S$ and $G$, respectively.
\end{corollary}

\begin{proof}
	By the construction of $G$, there exist a graph $\G_0$ serving as a skeleton of $E$ rel $P$  and an isotopy  $\Psi^0:\ov{\C}\times [0,1]\to \ov{\C}$ rel $X_0$  such that $ \Psi^0_0=id$, $ \Psi^0_1(\G_0)=G$, and $ \Psi^0_{s_k}(\G_0)\subset E_k$ for a sequence $\{s_k\}_{k\geq1}\subset (0,1)$ with $s_k\to 1$ as $k\to\infty$.
	
	Fix any $n\geq 1$. By Lemma \ref{lem:deg}, there exists a unique component $\G_n$ of $f^{-n}(\G_0)$ serving as a skeleton of $E_n$. Let $\Psi^n:\ov{\C}\times [0, 1]\to\ov{\C}$ rel $X_0$ be the lift of the isotopy $\Psi^0$ by $f^n$ such that $\Psi^n_0=id$. Then $\Psi^n_{s_k}(\Gamma_n)$ is contained in $ E_{k+n}$ and converges to $\Psi^n_1(\G_n)$ as $k\to\infty$, which is a component of $f^{-n}(G)$. Thus, $\Psi^n_1(\G_n)\subset K$. If $X_0\cap E=\emptyset$, then $\partial K=E$ is a Jordan curve, and this corollary clearly holds. Otherwise, we have $X_0\cap E\subset \Psi^n_1(\G_n)\cap G_n$. Therefore, $G_n=\Psi^n_1(\G_n)\subset K$.
	
	Note that both $E$ and  $G$ serve as skeletons of $S$. By Lemma \ref{lem:deg}, $E_n$ and $G_n$ are the unique components of $f^{-n}(E)$ and $f^{-n}(G)$ contained in $S_n$, respectively. Thus, $S_n= E_n\cup G_n\subset K$.  Finally, by Corollary  \ref{cor:monotone} and Lemma \ref{lem:deg}, there exists an $N>0$ such that $\G_N$ is a skeleton of $E_n$  for every $n\geq N$. Since  $\G_n\sim G_n$ rel $P$, the graph $G_N$ is a skeleton of $S_n$ for every $n\geq N$.
\end{proof}

\subsection{Invariant graphs on extremal chains}\label{sec:n-chain}
Let $(f,P)$ be a marked rational map with $J_f\not=\cbar$. The sketch for the construction of invariant graphs on extremal  chains is as follows.

Suppose that $E$ is the intersection of $J_f$ with a component of the union of all periodic level-$0$ Fatou chains.  Let $K$ be the intersection of $J_f$ with the level-$1$ extremal chain containing $E$. By Proposition \ref{prop:graph1}, there exists an invariant graph $G\subset K$ isotopic to a skeleton of $E$ rel $P$. To construct an invariant graph that serves as a skeleton of
$K$, a natural approach is to add a finite number of arcs to $G$ such that
\begin{enumerate}
\item the combined set of $G$ and the added arcs form a skeleton of $K$; and

\item  each added arc $\g$ is preperiodic with respect to $G$, i.e., there exist $q\geq0$ and $p\geq1$ such that $f^{q+p}(\g)\subset f^q(\g)\cup G$.
\end{enumerate}

Indeed, the first condition can be derived from  Lemma \ref{lem:growing}, while the second one follows from Propositions \ref{lem:curve-to-arc}, \ref{pro:criter}, and \ref{pro:criter1}. 

By employing a similar inductive argument, we can construct an invariant graph on any periodic level-$n$ extremal chain for every $n\geq1$.

\begin{proposition}\label{prop:graph}
	Let $(f, P)$ be a marked rational map, and let $K_1,\ldots,K_m$ be pairwise distinct continua such that each $K_i$ is the intersection of $J_f$ and a periodic level-$n$ extremal chain with $n\geq1$. Suppose that $\mathbf K=\bigcup_{i=1}^m K_i$ is connected and $f(\mathbf K)=\mathbf K$. Then there exists a graph $G$ serving as a skeleton of $\mathbf K$ rel $P$ such that $f(G)\subset G$.
\end{proposition}

This proposition  immediately implies Theorem \ref{thm:graph-maximal}. It is worth mentioning that the proposition is false if the level $n=0$, as shown in Theorem \ref{thm:example}.

\begin{proof}
	The proof goes by induction on the level $n$. First, assume that $n=1$.
	
	For each $1\leq i\leq m$, let $E_i$ denote the union of boundaries of all periodic Fatou domains within $K_i$. By Lemma \ref{lem:dyn-def}, each $K_i$ is the growing continuum generated by $E_i$.
  As indicated at the beginning of Section \ref{sec:topology}, we may assume that $E_i$ is a skeleton of $E_{i,k}$ (rel $P$) for every $k\geq1$, where $E_{i,k}$ denotes the component of $f^{-p_ik}(E_i)$ containing $E_i$ and $p_i$ is the period of $E_i$.
	
	\vskip 0.2cm
	
	\noindent  \emph{Claim.} There exist infinitely growing curves $\g_1,\ldots,\g_r$  in $\textbf K$ with preperiodic terminals such that, by replacing each $E_i$ with $E_{i,  N}$ for a sufficiently large  integer $N$, the set $(\bigcup_{i=1}^m  E_{i})\cup(\bigcup_{j=1}^r \g_j)$ is a skeleton of $\mathbf K$.
	
	\begin{proof}[Proof of the Claim]
		Let $z$ be a marked point in $\mathbf K$. Then $z\in K_i$ for some $1\leq i\leq m$. If $z\not\in E_i$, by Lemma \ref{lem:growing}\,(1), there exists a growing curve $\alpha_z\subset K_i$ joining $E_i$ to $z$. Since $E_i$ is a skeleton of every $E_{i,k}$, it holds that $z\not\in \bigcup_{k>0} E_{i,k}$. Thus, $\alpha_z$ is infinitely growing.
		
		Suppose $x,y\in P$ are separated by $\mathbf K$. Then there exists a smallest integer $s\geq1$ such that, by re-enumerating $K_i$ if necessary, the points $x$ and $y$ are separated by the union of $K_1, \ldots, K_s$.
		
		In the case of $s=1$, if $x$ and $y$ are separated by $E_{1,k}$ for some $k\geq1$, then they are separated by $E_1$ since $E_1$ is a skeleton.  Otherwise, by Lemma \ref{lem:growing}\,(2), there exists a curve $\eta=\beta_-\cdot\beta_+^{-1}\subset K_1$  such that $E_1\cup\eta$ separates $x$ from $y$, where $\beta_\pm$ are growing curves in $K_1$. 
		
		If the common terminal $z$ of $\beta_{\pm}$ is disjoint from $E_{1,k}$ for all $k$, then the curve $\eta$ serves as a two-sided self-link of $K_1$ provided that $z\notin P$. If $z$ is contained in some $E_{1, k_0}$, then one of $\beta_\pm$, say $\beta_-$, is infinitely growing, and $\beta_-\cup E_{1,k_0}$ separates $x$ from $y$. In this case, $\beta_-$ serves as a one-sided self-link of $K_1$, and we  reset $\eta=\beta_-$.
		
		In both cases, we can apply Proposition \ref{pro:criter1} to the self-link $\eta$, and thus obtain a curve
		$\eta_z=\beta_z'\cdot\beta_z^{-1}\subset K_1$ such that the common terminal $z$ of the growing curves $\beta_z'$ and $\beta_z$ is preperiodic, and that $\eta_z\cup E_1$ separates $x$ from $y$. 
		 By replacing $E_1$ with some $E_{1, k}$, we may further assume that each of $\beta'_z$ and $\beta_z$ is either trivial or infinitely growing.
		
		In the case of $s=2$, let $D$ be the component of $\ov{\mathbb{C}}\sm  (K_1\cup K_2)$ containing $x$. Since $\partial D$ is locally connected by Theorem \ref{thm:locally-connected},  a Jordan curve $\alpha\subset\partial D$ separates $x$ from $y$. By the minimum of $s$, there exists a unique arc $\alpha_1$ among components of $\alpha\sm K_2$ such that $\alpha_1\cup K_2$ separates $x$ from $y$. Let $\alpha_2$ be an arc in $K_2$ with the same endpoints as $\alpha_1$. Then $\alpha_1\cup \alpha_2$ forms a Jordan curve that separates $x$ from $y$.
		For $s\geq 3$, with similar arguments, there exist arcs $\alpha_i\subset K_i$, $i=1,\ldots,s$, such that their union is a Jordan curve separating $x$ from $y$. Let $Z$ be the set of endpoints of the arcs $\alpha_1,\ldots,\alpha_s$. 
		
		Fix a point $z\in Z$. There exist exactly two distinct integers $i=i(z)$ and $i'=i'(z)$ among $\{1,\ldots,s\}$ such that $z\in \alpha_i\cap\alpha_i'\subset K_i\cap K_{i'}$. By Lemma \ref{lem:growing}\,(1), there exist growing curves $\tilde{\beta}_z$ and $\tilde{\beta}'_{z}$ in $K_i$ and $K_{i'}$, respectively, with the common terminal $z$.  We can further require that $\tilde{\beta}_z$ (resp., $\tilde{\beta}_z'$) is a trivial curve if $z\in E_{i, k_0}$ (resp., $E_{i', k_0}$) for some $k_0$. 
		
		If $z$ is preperiodic, we set $\beta_z=\tilde{\beta}_z$ and $\beta_z'=\tilde{\beta}_z'$. Otherwise,
		 $\tilde{\eta}_z=\tilde{\beta}_z'\cdot \tilde{\beta}_{z}^{-1}$ is a link between $K_{i}$ and $K_{i'}$. In particular, it is a two-sided link if and only if $z$ is disjoint from  $E_{i, k}$ and $E_{i', k}$ for all $k\geq0$. 
		In this case, we can apply Proposition \ref{pro:criter} to the link $\tilde{\eta}_z$ and obtain a curve ${\eta}_z={\beta}_z'\cdot{\beta}_z^{-1}$ such that $\tilde{\eta}_z$ and $\eta_z$ are homotopic rel $\{x, y\}$ with endpoints fixed, and the common terminal of the growing curves ${\beta}_z\subset K_i$ and ${\beta}_z'\subset K_{i'}$ are preperiodic. 
		
		By the minimality of $s$, for a sufficiently large integer $k_0$, the union of $\eta_z$, $z\in Z,$ and all $E_{j, k_0}$, $1\leq j\leq s$, is connected and separates $x$ from $y$. By replacing each $E_j$ with some $E_{j, k}$, we may assume 
		\begin{itemize}
	\item for each $z\in Z$, either $z\in E_i$ for some $E_i$, or $z$ avoids $E_{i,k}$ for all $1\leq i\leq s$ and $k\geq0$;
		
	\item each ${\beta}_z$ (resp., $\beta_z'$) is either trivial or infinitely growing.
		\end{itemize}
		
		Finally, the required growing curves $\g_1,\ldots,\g_r$ consist of all $\alpha_z$ and the non-trivial curves ${\beta}_z$ and ${\beta}_z'$ described above. Thus, the claim is proved.
	\end{proof}
	
	Let $Q\subset \mathbf{K}$ denote the set  of all points in the orbits of $\g_1(1),\ldots,\g_r(1)$. Then $f(Q)\subset Q$. According to Proposition \ref{prop:graph1} and Corollary \ref{cor:disjoint}, each $K_i$ contains a graph $G_i$ such that
	\begin{itemize}
	\item $G_i$ is a skeleton of $S_i:=G_i\cup E_i$ rel $P$ and  contains $Q\cap E_i$;
	
	\item $f(\bigcup_{i=1}^m G_i)\subset \bigcup_{i=1}^m G_i$ and $S_i\cap S_j=\emptyset$ if $i\not=j$.
	\end{itemize}
	
	By Corollary \ref{coro:added-corollary}, each $K_i$ is also the growing continuum generated by $S_i$. For every $k\geq1$, denote by $S_{i,k}$ and $G_{i,k}$ the components of the $k$-th pre-image by $f^{p_i}$ of $S_i$ and $G_i$, respectively, such that $S_i\subset S_{i,k}$ and $G_i\subset G_{i,k}$.
	
	Let $\De$ be a maximal collection of infinitely growing curves in $K_1,\ldots, K_m$, which have initial points in $\bigcup_{i=1}^m G_i$ and terminals in $Q$ and belong to pairwise distinct accesses. According to Lemma \ref{lem:finite-access}, $\De$ contains finitely many elements. The claim above implies that the union of $G_i$, $i=1,\ldots, m,$ together with all curves in $\De$, is a skeleton of $\mathbf K$ rel $P$.
	
	For any $\de\in \De$ with terminal $z:=\de(1)$, its image $f(\de)$ is an infinitely growing curve to $f(z)\in Q$ by Lemma \ref{lem:lifting}\,(1). By the maximality of $\De$, we obtain a self-map $f_h:\De\to \De$ such that $f_h(\de)$ is defined to be the unique element of $\De$ in the same access  as $f(\de)$.
	
	Mark a curve $\de_*$ in each cycle under $f_h$. Suppose that $\de_*\subset K_i$ with period $p$ under $f_h$. By Proposition \ref{lem:curve-to-arc}, we may assume that
	\begin{itemize}
	\item for any $t\in(0,1)$, there exists an integer $k>1$ such that $\de_*[0,t]\subset G_{i,k}$; and
	
\item  $\de_*$ is an $f^p$-invariant arc in the sense that $f^p(\de_*)\subset \de_*\cup G_i$.
	\end{itemize}
	Since $\Delta$ has finitely many elements, any curve $\de\in\De$ is eventually iterated by $f_h$ to a marked one $\de_*$. Let $q\geq 0$ be the smallest number such that $f_h^q(\de)=\de_*$. Assume $\de(0)\in G_j$.  By Lemma \ref{lem:lifting}\,(2), there exists a lift $\de'$ of $\de_*$ by $f^q$ that lies in the same access  as $\de$ and has the initial point in $G_{j,q}$.
	
	Let $N$ be a sufficiently large integer such that the initial point of each $\de'$ with $\de\in\De$ lies in $\bigcup_{i=1}^m G_{i, N}$.  Define $G:=(\bigcup_{i=1}^m G_{i, N})\cup(\bigcup_{\de\in\De} \de')$. The previous discussion shows that $f(G)\subset G$ and $G$ is a skeleton of $\mathbf K$ rel $P$. 
	
	Since the curves in $\De$ are infinitely growing and lie in pairwise distinct accesses, by Proposition \ref{pro:disjoint}\,(2), there exists  $\epsilon>0$ such that $\delta'[1-\epsilon,1)$ with $\delta\in\De$ are pairwise disjoint, each disjoint from $G_{i, N}, i=1,\ldots,m$. On the other hand, the arcs $\delta'[0,1-\epsilon],\delta\in\De$ are contained in $\bigcup_{i=1}^m G_{i, N_1}$ for some $N_1>N$. Thus, the locally branched points of $G$ are contained in those of $\bigcup_{i=1}^{m} G_{i, N_1}$ together with $Q$, which are finite. Thus, $G$ is a graph. Now, we have proved this proposition in the case of $n=1$.
	
	Suppose that the proposition holds for level-$n$ extremal chains with $n\geq 1$. Let $K_1,\ldots,K_m$ be pairwise distinct continua such that each $K_i$ is the intersection of $J_f$ and a periodic level-$(n+1)$ extremal chain.  For each $i\in\{1,\ldots,s\}$, denote by $E_i$ the intersections of $J_f$ and the union of periodic level-$n$ extremal chains within $K_i$.  
	Then $K_i$ is the growing continuum generated by $E_i$. By induction, there exists a graph $G_i$ serving as a skeleton of $E_i$ such that $f(\bigcup_{i=1}^m G_i)\subset \bigcup_{i=1}^m G_i$.
	
	Note that in this case, we have $G_i\subset  E_i$ and set $S_i:=E_i$. In contrast, in the case of $n=1$, the graph $G_i$ is not necessarily contained in $E_i$, and thus we performed a transformation from $E_i$ to $S_i=E_i\cup G_i$ by Corollary \ref{coro:added-corollary} therein. By a similar argument as in the case of $n=1$, we obtain the desired invariant graph $G\subset \mathbf K$.
\end{proof}

\section{Invariant graphs of rational maps}\label{sec:invariant-graphs}
Let $(f, P)$ be a marked rational map with $J_f\neq\cbar$. As stated in the introduction, it suffices to prove Proposition \ref{pro:pre} in order to construct the invariant graph required by Theorem \ref{thm:main}.

According to Corollary \ref{cor:desired} and Theorem \ref{thm:blow-up}, by possibly enlarging $P$, there exists a stable set $\KKK\subset J_f$ that induces a cluster-\Sie\ decomposition of $(f,P)$, such that the decomposition
$$\cbar=\mathcal{K}\sqcup \mathcal{V}\sqcup \mathcal{A}\sqcup \mathcal{S}$$
satisfies the following properties:
\begin{itemize}
\item [(P1)]  Each component of $\KKK$ contains points of $P$;

\item [(P2)] Every component of $\mathcal{V}$ is complex-type and disjoint from any attracting cycle of $f$; 

\item [(P3)] Every component of $\mathcal{S}$ is a simply connected domain of simple type;

\item [(P4)] Every component $A$ of $\mathcal{A}$ is an annulus of annular type. Moreover, if $A\cap f^{-1}(\mathcal{K})\neq \emptyset$, then $A$ contains an annular-type component of $f^{-1}(\mathcal{K})$.
\end{itemize}
\indent Therefore, it suffices to prove Proposition \ref{pro:pre} under the properties (P1)--(P4).\vspace{2pt}

 The proof of this proposition will be divided into three parts. First, we identify a graph in each component of
$\mathcal{E}=\KKK\sqcup \VVV$ such that their union is $f$-invariant. Next, we construct invariant arcs in $\mathcal{A}$ to connect these graphs together. Finally, we join every marked point in $\mathcal{S}\cap J_f$ to the previous graph.

\begin{proof}[Proof of Proposition \ref{pro:pre}]
	At the beginning, we select several specific marked points.

	In  each cycle of $\VVV$ under $f_{\#}$, we designate a {\it preferred} component $V$. Denote its period by $p$. For each $n\geq0$, let $V_n$ denote the unique complex-type component of $f^{-np}(V)$ contained in $V$.
	By Theorem \ref{thm:blow-up} and property (P2), there exists a marked \Sie\ rational map $(g,Q_g)$ as the blow-up by $\pi$ of the exact sub-system $f^p:V_1\to V$, i.e.,
	\begin{itemize}
		\item $\pi(J_g)=\bigcap \ov{V_{n}}$ and $\pi\circ g=f^p\circ \pi$ on $J_g$;
		\item $\pi$ sends the closure of each Fatou domain onto a component of $\cbar\setminus V_{n}$ for some $n\geq0$.
	\end{itemize}
	\noindent Due to property (P1), the marked set $Q_g$ coincides with the union of $\pi^{-1}(P\cap V)$ and the centers of  Fatou domains outside $\pi^{-1}({V})$.

	By  the conditions of the proposition, let $G_g\supset Q_g$ be a $g$-invariant regulated graph. Then for each Fatou domain $ D $ of $g$, the set $Y_{ D }:=G_g\cap\partial  D $ satisfies:
\begin{itemize}
\item  $g(Y_{ D })\subset Y_{ g( D )}$, and $Y_{ D }\not=\emptyset$ if $ D \cap Q_g\not=\emptyset$;
	
\item $Y_{D}$ is a finite set, and there exist only finitely many Fatou domains $ D $ such that $\#Y_{ D }\geq 3$.

\end{itemize}

	Since $V$ avoids the periodic Fatou domains by property (P2), the choice of $Q_g$ implies that $Y_{ V}:=\bigcup_{ D }\pi(Y_{ D })$ lies in $\partial V$ and each component of $\partial V$ intersects $Y_{V}$, where $ D $ ranges over all marked Fatou domains of $(g,Q_g)$. Moreover, we have $f^p(Y_{V})\subset Y_{V}$.
If $V'$ is another component of $\VVV$ such that $f^q_{\#}(V')=V$, set $Y_{ V'}:=f^{-q}(Y_{ V})\cap\partial V'$. 
Thus, 
$$Y_{\VVV}:=\bigcup Y_{V}$$
is an $f$-invariant and finite set in $\partial \VVV\subset \KKK$, where the union is taken over all components of $\VVV$.	
	
For a finitely connected domain $W$, an {\it oriented boundary component} of $W$ means a component of $\partial W$	equipped with an orientation pointing into $W$. 
	
	Let $\Lambda$ be the collection of oriented boundary components of all annuli in ${\rm Comp}(\AAA)$. Then any two elements of $\Lambda$ are distinct even if they overlap. 
	
	For any $\lambda\in\Lambda$, since $\lambda\subset\KKK$ and $\KKK$ is a stable set, there exists either an annular-type component $A_1$ of $f^{-1}(\AAA)$ or an annular-type component $V_1$ of $f^{-1}(\VVV)$ such that  $\lambda$ is an oriented boundary component of $A_1$ or $V_1$. Thus, its image $f(\lambda)$ is either also an element of $\Lambda$, or  an oriented boundary component of a certain $V\in{\rm Comp}(\VVV)$. Set
	\begin{equation}\label{eq:la}
		\Lambda_*=\{\lambda\in \Lambda: f^n(\lambda)\in\Lambda\text{ for all $n\geq0$}\}.
	\end{equation}
	Since $f(\partial \VVV)\subset\partial\VVV$, the orbit of any $\lambda\in\Lambda\setminus\Lambda_*$ will stay in $\partial\VVV$ after leaving $\Lambda$.

	By Theorem \ref{thm:renorm}, we can assign a point $z_\lambda$ to each element $\lambda\in\Lambda_*$ such that $f(z_{\lambda})=z_{f(\lambda)}$. Then the finite set $\{z_\lambda:\lambda\in \Lambda_*\}$ is $f$-invariant and contained in $\KKK$.
On the other hand, there exists an integer $M>0$ such that $f^{ M}(\lambda)\subset \partial\VVV$ for any $\lambda\in\Lambda\setminus\Lambda_*$. Since $f(Y_{\VVV})\subset Y_{\VVV}\subset \KKK$, we obtain an $f$-invariant and finite set 
$$Q:=(f^{- M}(Y_{\VVV})\cap \KKK)\,\bigcup\,\{z_\lambda:\lambda\in \Lambda_*\}\subset\KKK.$$
	\vskip 0.1cm
	{\bf  Part I. Construct invariant graphs in $\boldsymbol{\mathcal{E}=\KKK\sqcup\VVV}$.}
	\vspace{5pt}

By  Theorems \ref{thm:graph-maximal} and \ref{thm:renorm} and Lemma \ref{lem:deg}, each component $K$ of $\KKK$ contains a graph $G_{K}$ serving as a skeleton of $K$ rel $P\cup Q$ such that the union $\bigcup_{K} G_{K}$ is $f$-invariant.
	
	Let $V$ be a preferred $f_{\#}$-periodic component of $\VVV$ with period $p$. Denote by $\BB$ the collection of the complementary components of $V_{n}$ for all $n>0$.
	
	By Theorem \ref{thm:blow-up}, for each $B\in\BB$, $\pi^{-1}(B)=\overline{D}$ and $\pi^{-1}(\partial B)=\partial D$, where $D$ is a Fatou domain of $g$, and $\pi^{-1}(z)$ is a singleton if $z$ does not belong to any element of $\BB$. 

We set $\G:=\pi(G_g)$ and $Y_{B}:=\pi(Y_{D})$ with $B=\pi(\overline{D})$. According to the properties of $Y_{D}$ presented at the third paragraph of the proof, we have that
\begin{itemize}
\item $Y_{B}\subset\partial B$ and $f^p(Y_{B})\subset Y_{B'}$ if $\partial B'=f^p(\partial B)$;
		
\item $Y_{B}$ is a finite set and there exist only finitely many $B\in\BB$ with $\#Y_{B}\geq 3$;
		
\item $Y_{V}=\bigcup_{B} Y_{B}$ and $Y_{B}\neq\emptyset$, where $B$ is taken over all components of $\overline{\mathbb{C}}\setminus V$; 
\item if $z\in \G\setminus\bigcup_{B\in\BB} B$, then $z\in J_f$ and $f^p(z)\in \G$.
\end{itemize}

To obtain an $f^p$-invariant graph associated with $V$, we need to revise $\G\cap B$ to an appropriate graph $G_{B}$ for each $B\in\BB$ that intersects $\G$.

	If $B$ is a component of $\cbar\setminus V$, then $\partial B\subset K$ for a component $K$ of $\KKK$. Define $G_{B}=G_{K}$. Note that $G_{K}$ contains $Y_{B}$ by the choices of $Q$ and $G_{K}$.
	
	If $B$ is not a component of $\cbar\setminus V$, then $B\cap P=\emptyset$, and there exist a smallest positive integer $k$ and a component $B'$ of $\cbar\setminus V$ such that $\partial B$ is a component of $f^{-kp}(\partial B')$. Let $K$ and $K'$ be the components of $f^{-kp}(\KKK)$ containing $\partial B$ and $\partial B'$, respectively. Then $f^{kp}(K)=K'$.
	
	By Lemma \ref{lem:deg}, the set $\tilde{G}_{ B}=f^{-kp}(G_{B'})\cap K$ is a component of $f^{-kp}(G_{B'})$ contained in $B$. Thus, $\tilde{G}_{ B}$ is a graph. Since $f^{kp}(Y_{B})\subset Y_{B'}$, it follows that $Y_{B}\subset \tilde{G}_{ B}$. Define $G_{B}$ as follows:
	\begin{enumerate}
\item If $\# Y_{ B}\geq3$, set $G_{B}=\tilde{G}_{ B}$; if $\# Y_{ B}=1$, set $G_{B}=Y_{B}$;
	
\item If $\# Y_{ B}=2$, let $G_{B}$ be an arc in $\tilde{G}_{ B}$ joining the two points of $Y_{ B}$ such that $f^{kp}(G_{B})\subset G_{B'}$ and $f^p(G_{B})\subset G_{f^p(B)}$.

\end{enumerate}
	\noindent Thus, we obtain an $f^p$-invariant continuum
$$
G_{V}:=\bigg(\G\setminus \bigcup_{B\in\BB} B\bigg)\bigcup \bigg(\bigcup_{B\in\BB} G_{B}\bigg),
$$
	which lies in $J_f$ and contains $P\cap V$.  Since the diameters of $B\in\BB$ exponentially converge to zero by Lemma \ref{lem:expanding}, the continuum $G_{V}$ is a graph. 
	
	If $V'$ is a component of $\VVV$ such that $f^q_{\#}(V')=V$ for a smallest  $q\geq 1$, 
	then define $G_{V'}=f^{-q}(G_{V})\cap V'.$
	Note that the accumulation set of $G_{V'}$ on $\partial V'$ is contained in $Y_{V'}\subset Q$.

	 Define the set
	$$
	\GGG_{\EEE}:=\bigg(\bigcup_{K\in{\rm Comp}(\KKK)} G_{K}\bigg)\ \ \bigcup\bigg(\bigcup_{V\in{\rm Comp}(\VVV)} G_{V}\bigg),
	$$
	which is $f$-invariant and contains $Q$. Moreover, it satisfies the following two properties: 
	\begin{itemize}
	\item [(a)] For each component $E$ of $\mathcal{E}$,  the set $\GGG_{\mathcal{E}}\cap E$ is a graph serving as a skeleton of $E\cap J_f$ rel $P$;
    \item [(b)] For each component $V$ of $\VVV$ and any  component $V'$ of $f^{-1}(V)$, any pair of distinct boundary components $\lambda_\pm$ of ${V}'$ can be joined 
   by an arc in $f^{-1}(\GGG_{\mathcal{E}})$, which lies in the annulus $A(\lambda_+,\lambda_-)$ bounded by $\lambda_{\pm}$ and has the endpoints in $f^{-1}(Y_{V})$.
   \end{itemize}
  \vspace{-1.5pt}
  
 For property (a), it suffices to show 
the connectivity of $\GGG_{\mathcal{E}}\cap E$. Let $V\subset E$ be any component of $\VVV$. By construction, for each boundary component $\lambda$ of $V$, the accumulation points of $G_{V}$ on $\lambda$ are non-empty and lie in the graph $G_{K}$, where $K$ is a component of $\KKK$ contained in $E$ such that $\lambda\subset K$. This implies that $\GGG_{\EEE}\cap E$ is connected.

To prove property (b), we choose a sequence of domains $V_\epsilon$ compactly contained in $V$ that converges to $V$ as $\epsilon\to 0$, such that $V\setminus \ov{V_\epsilon}$ consists of annuli disjoint from $P$, and that $G_\epsilon=(V_\epsilon \cap G_{V})\cup\partial V_{\epsilon}$ is connected. Then each $G_\epsilon$ is a skeleton of $\ov{V_\epsilon}$ rel $P$, and $\lim_{\epsilon\to0} G_\epsilon=(V\cap G_{V})\cup \partial V$.

Set $V_{\epsilon}'=f^{-1}(V_\epsilon)\cap V'$. Then $V_\epsilon'$ is a domain, and each of its boundary components  is parallel to a component of $\partial V'$, and vice versa. Moreover, $\lim_{\epsilon\to0} \ov{V_\epsilon'}=\ov{V'}$. By Lemma \ref{lem:deg}, $G'_\epsilon:=f^{-1}(G_\epsilon)\cap \ov{V_\epsilon'}$ is connected. Thus, it contains all  components of $\partial V_\epsilon$. Consequently,  the Hausdorff limit $G'$ of ${G}'_\epsilon$ is connected and contains $\partial V'$. Moreover, ${G}'\cap {V}'=f^{-1}(G_{V})\cap {V}'$.

From the previous discussion, there exist pairwise disjoint open arcs $\alpha_1,\ldots,\alpha_m$ in $G'\cap V'$ and  components $\lambda_-=\lambda_1,\ldots,\lambda_{m+1}=\lambda_+$ of $\partial V'$ such that each $\alpha_i$ joins $\lambda_i$ to $\lambda_{i+1}$ and its endpoints belong to $f^{-1}(Y_{V})$. Note that for every $i\in\{2,\ldots,m-1\}$, $\lambda_i$ is contained in a component $K_i\subset A(\lambda_-,\lambda_+)$ of $f^{-1}(\KKK)$. Thus, we can find an arc $\beta_i\subset K_i$ joining $\alpha_{i-1}(1)$ to $\alpha_i(0)$ such that $f(\beta_i)\subset G_{f(K_i)}$. Finally, the arc $(\bigcup_{i=1}^m\alpha_i)\cup(\bigcup_{j=2}^{m-1}\beta_i)$ satisfies property (b).
	
	\vspace{5pt}
	
	{\bf Part II. Connect the graphs in $\boldsymbol{\mathcal{E}}$.}
	\vspace{5pt}
	
	By properties (P2)--(P4), any two components of $\GGG_{\mathcal{E}}$ are separated by a component of $\AAA$, and vice versa. Thus, to obtain a global invariant graph, we need to construct appropriate arcs serving as bridges that cross $\AAA$ and join components of $\GGG_{\mathcal{E}}$ together. 
	
	\vskip 0.15cm
	{\it Step 0. Assign a preperiodic point $x_\lambda\in Q$ to every $\lambda\in\Lambda$.}
\vskip 0.15cm
	
	Recall that $\Lambda$ is the collection of oriented boundary components of all annuli $A\in{\rm Comp}(\AAA)$ and $\Lambda_*\subset \Lambda$ consists of all elements  whose orbits under $f$  stay in $\Lambda$; see \eqref{eq:la}. We have assigned one point $x_\lambda\in\lambda$ for each  $\lambda\in\Lambda_*$ such that $f(x_\lambda)=x_{f(\lambda)}$ and $x_\lambda\in Q$. Thus, it remains to assign a point to each element of $\Lambda\setminus \Lambda_*$.
	
	Fix any $\lambda\in \Lambda\setminus\Lambda_*$. It is an oriented boundary component of a unique component $A$ of $\AAA$.
	
	If $f(\lambda)\subset \partial V$ for a component $V$ of $\VVV$, then there exists an annular-type component $V_1$ of $f^{-1}(V)$ contained in $A$ such that $\lambda$ is  an oriented boundary component of $V_1$. The boundary $\partial V_1$ has the other annular-type component $\lambda'$. By property (b) of $\GGG_{\EEE}$, there exists an open arc  $\beta\subset A(\lambda,\lambda')$ joining $\lambda$ to $\lambda'$, such that $f(\beta)\subset \GGG_{\EEE}$ and the endpoints of $\beta$ lie in $f^{-1}(Y_{V})$. Define $x_\lambda$ to be the endpoint of $\beta$ in $\lambda$. It follows that $x_\lambda$ belongs to $f^{-1}(Y_{V})\cap\KKK\subset Q$.
	
	If  $f(\lambda)\in \Lambda$ and $x_{f(\lambda)}\in f(\lambda)$ has been chosen, we assign a point $x_\lambda\in\lambda$ such that $f(x_\lambda)=x_{f(\lambda)}$.
	Then  $x_\lambda$ belongs to $Q$ by the definition of $Q$.
	\vskip 0.15cm

	{\it Step 1.} Construct the initial graph $G_0$.
	\vskip 0.15cm
	
	For each component $A$ of $\AAA$, we denote its two oriented boundary components  by $\lambda_{\pm, A}$.   Let $z_{\pm, A}\subset \lambda_{\pm, A}$ be the points  assigned to $\lambda_{\pm, A}$, respectively.
	
	If  $A$ intersects $f^{-1}(\KKK)$, we call it {\it intersection-type}; otherwise, $f(A)$ is still a component of $\AAA$.  In the latter case, there exists a smallest integer $n_{A}\geq 1$ such that $f^{n_{A}}(A)$ is an intersection-type component of $\AAA$ since $f$ has no Herman rings.
	
	We claim that there exists an open arc $\gamma_{A}$ joining $z_{\pm, A}$ in each component $A$ of $\AAA$ such that $f(\gamma_{A})=\gamma_{f(A)}$ when $A$ is not intersection-type.
	
First, we choose 	an open arc $\alpha_{A}$  with endpoints $z_{\pm, A}$ in each component $A$ of $\AAA$.
Fix an intersection-type component $A$ of $\AAA$.  For any component $A'$ of $\AAA$ with $f^{n( A')}(A')=A$, the curve $\alpha=f^{n( A')}(\alpha_{A'})$ lies in $A$ and joins $z_{\pm,  A}$. Consequently, $\alpha$ is homotopic to $\alpha_{A}$ with endpoints fixed, up to an $N(A')$-time twist around $A$. Let $N$ be the smallest common multiple of all such numbers $N(A')$ and set $\gamma_{A}=T^N(\alpha_{A})$, where $T(\cdot)$ denotes the twist map around $A$. Then $A'$ contains a unique component $\g_{A'}$ of $f^{- n( A')}(\g_{A})$  with endpoints $z_{\pm,A'}$. The claim is proved.

	Since the endpoints of each $\g_{A}$  belong to $Q\subset \GGG_{\mathcal{E}}$, the arc $\g_{A}$ joins the two components of $\GGG_{\mathcal{E}}$ adjacent to $A$ together. Thus, we obtain the initial graph
	$$
	G_0=\GGG_{\mathcal{E}}\cup\bigcup\g_{A},
	$$
	where $A$ ranges over all components of $\AAA$. The vertices of $G_0$ are composed of the points in $Q\cup(P\cap \GGG_{\EEE})$ and the locally branched points of $\GGG_{\EEE}$. Then each $\g_{A}$ is an edge of $G_0$.
	
	\vskip 0.15cm
	{\it Step 2.} Construct a graph $G_1\subset f^{-1}(G_0)$ isotopic to $G_0$.
	\vskip 0.15cm
	
	We first construct a curve $\gamma_{A}^1$ for each component $A$ of $\AAA$ such that  $\g_{A}^1(0,1)\subset A$, $f(\gamma_{A}^1)\subset G_0$, and $\gamma_{A}^1$ is homotopic to $\g_{A}$ (rel $P$) with endpoints fixed.
	
	If $A$ is not intersection-type, define $\g_{A}^1=\g_{A}$ by the claim in Step 1.
	
	If $A$ is intersection-type, let $A_1, \ldots, A_s$, with $s\geq 2,$  be the annular-type components of $A\setminus f^{-1}(\KKK)$ arranged from left to right by property (P4). Let $\lambda_{\pm, i}$ be the annular-type boundary components of $A_i$. Then $\lambda_{+, i}\cup \lambda_{-, i+1}$ is contained in an annular-type component $K_i$ of $f^{-1}(\KKK)$ for each $1\leq i\leq s-1$. By Lemma \ref{lem:deg}, $\G_i:=f^{-1}(G_{f(K_i)})\cap K_i$ is a graph serving as a skeleton of $K_i$.
	
	If $f(A_1)$ is a component of $\AAA$, let $\alpha_1$ be the lift of $\gamma_{f(A_1)}$ based at $z_{-, A}$. Otherwise, $f(A_1)$ is a component of $\VVV$. By property (b) of $\GGG_{\EEE}$ given in Part I and the choice of $z_{-, A}$ in Step $0$, there exists an open arc $\alpha_1\subset A_1$ that joins $z_{-,  A}$ to $\lambda_{+, 1}$ and satisfies  $f(\alpha_1)\subset \GGG_{\EEE}$. Similarly, we can find an open arc $\alpha_i\subset A_i\cap f^{-1}(G_0)$ for every $i\in\{2,\ldots,s\}$ such that $\alpha_i$ joins $\lambda_{\pm,i}$ and one endpoint of $\alpha_s$ is $z_{+,  A}$. Therefore, 
	the points $z_{\pm, A}$ can be connected by an open arc $\beta_{A}$ in $$\bigcup_{i=1}^{s}\alpha_i\cup\bigcup_{i=1}^{s-1}\G_i,$$ 
	and it holds that $\beta_{A}\subset A\cap f^{-1}(G_0)$.

	Note that $\beta_{A}$ is homotopic to $\gamma_{A}$ with endpoints fixed, up to an $m_{A}$-time twist around $A$. Since $\G_1$ is a skeleton of $K_1$, the graph $\G_1$ separates $\partial A$. Thus, we can find a curve $\beta\subset \G_1$ such that
	$\g_{A}^1=(\beta_{A}\setminus K_1)\cup \beta$
	is a curve homotopic to $\gamma_{A}$ rel $P$ with endpoints fixed.
	
Define a graph
	$$
	G_1:=\GGG_{\mathcal{E}}\cup\bigcup \gamma_{A}^1\subset f^{-1}(G_0),
	$$
	where $A$ ranges over all components of $\AAA$.
	 Although a certain $\g_{A}^1$ may have self-intersections, we also consider it an edge of $G_1$. 
	Thus, each edge of $G_0$ is homotopic rel $P$ to an edge of $G_1$ with endpoints fixed, and the homotopy is the identity when the edge is in $\GGG_{\EEE}$.
	
	For $n\geq0$, let $\AAA_n$ be the union of all annular-type components of $f^{-n}(\mathcal{A})$. Consequently, the components of $\AAA_n$ are annuli, and $\AAA_{n+1}\subset \AAA_n$. By inductively lifting the homotopy of the edges of $G_0$ and $G_1$, we obtain a graph 
	$
	G_n=\GGG_{\mathcal{E}}\bigcup\,(\cup \gamma_{A}^n)
	$
	for every $n\geq0$,	where $A$ runs over all components of $\AAA$, such that $f(G_{n+1})\subset G_n$, and the curves $\g_{A}^{n+1}$ and $\g_{A}^n$ are homotopic rel $P$ with endpoints fixed, which differ only within $\AAA_n$.

	Since the degree of $f^n$ on each component of $\AAA_n$ tends to $\infty$ as $n\to\infty$,  there exists an integer $N\geq 0$ such that the $n$-th lift of each $\gamma^1_{A}$ is an arc for every  $n\geq N$. Therefore, there exists a homeomorphism $h_0:\cbar\to \cbar$ that is isotopic to $id$ rel $\cbar\setminus \AAA_N$ such that $h_0(G_N)=G_{N+1}$. For the sake of simplicity, we assume that $N=0$.
	\vskip 0.15cm
	
	{\it Step 3.} Construct an invariant graph $G'$.
	\vskip 0.15cm
	
	 By Lemma \ref{lem:lift}, we get a sequence of homeomorphisms $\{h_n\}_{n\geq0}$ such that $h_n$ is isotopic to $id$ rel $\cbar\setminus f^{-n}(\AAA)$ and $h_n\circ f=f\circ h_{n+1}$ on $\ov{\C}$. Recursively define the graph $G_{n+1}=h_n(G_n)$. It then follows that
	\begin{equation}\label{eq:isotopy}
		h_n(x)=x \text{ if $x\in G_n\setminus\AAA_n\quad\textup{and}\quad h_n(x)\in \AAA_n$ if $x\in G_n\cap \AAA_n$.}
	\end{equation}
	
 Let $\phi_n:=h_n\circ \cdots \circ h_0$ for $n\geq 0$. By Lemma \ref{thm:isotopy}, $\phi_n$ uniformly converges to a quotient map $\phi:\ov\C\to\ov\C$. Thus, $G_{n+1}=\phi_n(G_0)$ converges to a continuum $G':=\phi(G_0)$ in the sense of the Hausdorff metric. Consequently, $f(G')\subset G'\subset J_f$.
	
In order to prove that $G'$ is a graph, it suffices to show that  $\phi^{-1}(z)\cap G_0$ is connected for any $z\in G'$. In other words, we will verify that, for any two distinct points $x, y\in G_0$ with $\phi(x)=\phi(y)$, there exists an arc $l_{x,y}\subset G_0$ joining $x$ and $y$ such that $\phi(l_{x,y})$ is a singleton. 
	
	Fix a pair of distinct points $x$ and $y$. Denote  $x_n=\phi_{n-1}(x)$ and $y_n=\phi_{n-1}(y)$, which lie in $G_n$. Since $\phi(x)=\phi(y)$, at least one of $x$ and $y$, say $x$, satisfies that $x_n\in \AAA_n$ for all $n$ by \eqref{eq:isotopy}.
	
	If $x_n$ and $y_n$ lie in the closure of the same component of $\AAA_n$ for each $n$, then $\phi([x,y])$ is a singleton, where $[x,y]$ denotes the arc in $G_0\cap \ov{\mathcal{A}}$ joining $x$ and $y$. Indeed, let $A_n$ be the component of $\AAA_n$ such that $x_n,y_n\in \ov{A_n}$. Then $(x_n,y_n)=\phi_{n-1}(x,y)$ is the open arc in $G_n\cap A_n$ joining $x_n$ and $y_n$. Since $f^n[x_n,y_n]$ is an arc contained in $G_0\cap\ov{\AAA}$, by Lemma \ref{lem:expanding}, the diameter of $[x_n,y_n]$ converges to zero as $n\to\infty$. Thus, $\phi[x,y]$ is a singleton.
	
	On the other hand, since $\phi(x)=\phi(y)$, it follows from \eqref{eq:isotopy} that $x_n$ and $y_n$ cannot be separated by components of $\AAA_n$ for each $n$. Hence, we are reduced to the case where there exists some $m\geq0$ such that  $x_m$ and $y_m$ are neither contained in the closure of a component of $\AAA_m$ nor separated by components of $\AAA_m$. Then there exist two possibilities:

Case 1.  $x_m \in A$ and $y_m\in E\setminus \lambda$, where $A$ is a component of $\AAA_m$, $E$ is a component of $\cbar\setminus \AAA_m$, and  $\lambda=E\cap \partial A$ is a boundary component of $A$.

In this case, let $z_\lambda\in\lambda$ be the assigned point to $\lambda$ given in Step 0. Then $z_\lambda\not=y_m$ and $y_m=\phi(y)$. Since $\phi(x)=\phi(y)$, the point $x_{m+k}$ must belong to the unique component of $\AAA_{m+k}$ whose boundary contains $\lambda$, for each $k\geq0$. However, by the previous discussion, we have $\phi(x)=z_\lambda$,  which contradicts the assumption that $\phi(x)=\phi(y)$.
    
Case 2.   $x_m\in A_1$ and $y_m\in A_2$, where  $A_1$ and $A_2$ are distinct components of $\AAA_m$, such that each $A_i$ has a boundary component $\lambda_i$ contained in a component $E$ of $\cbar\setminus\mathcal{A}_m$.\vspace{2pt}
	
	In this case, let $z_1\in\lambda_1$ and $z_2\in\lambda_2$ be the assigned points to $\lambda_1$ and $\lambda_2$, respectively. Similarly as above, the points $x_{m+k}$ and $z_{1}$ (resp., $y_{m+k}$ and $z_{2}$) belong to the closure of the same component of $\AAA_{m+k}$ for each $k\geq 0$. Therefore, $[x_{m+k},z_{1}]$ and $[z_{2},y_{m+k}]$ converge to $z_1$ and $z_2$, respectively. Since $\phi(x)=\phi(y)$, it follows that $z_{1}=z_{2}$. Thus, $\phi(l_{x,y})$ is a singleton with $l_{x,y}=\phi_{m-1}^{-1}([x_{m},z_{1}]\cup[z_{1},y_m])$.
	
	Therefore, $G'$ is an $f$-invariant graph, and by property (P3), its complementary components are all simply connected domains of simple type.
	\vspace{5pt}
	
	{\bf Part III. Completion of the proof of Proposition \ref{pro:pre}.}
	\vspace{5pt}
	
To complete the proof,	it remains to join the marked points in $\SSS\cap J_f$ to the graph $G'$. 

Since each complementary component of $G'$ contains at most one marked point, it follows that $f^{-n}(G')$ is connected for all $n>0$. By replacing $G'$ with $f^{-n}(G')$ if necessary, we may assume that each point of $P$ is either contained in $G'$ or never iterated into $G'$.
	
Let $K$ be the growing continuum generated by $G'$. It is clear that $K=J_f$. Let $z\in J_f$ be a point in $P\setminus G'$ with period $p$. According to Lemma \ref{lem:growing}\,(1), there exists an infinitely growing curve $\g$ in $K$ that joins $G'$ to $z$. Since each complementary component of $G'$ contains at most one point of $P$, the growing curve $f^p(\g)$ belongs to the same access to $z$ as $\g$. Therefore, by
Proposition \ref{lem:curve-to-arc}, we can assume that $\g$ is a growing arc in $K$ such that $f^p(\g)\subset \g\cup G'$. Consequently, the union of $G'$ and $\bigcup_{i=0}^{p-1}f^i(\g)$ is an $f$-invariant graph and  contains the orbit of $z$. 
	
 We repeat the process for each cycle in $(P\setminus G')\cap J_f$ and then take an $m$-th iterated pre-image for a sufficiently large integer $m$. The resulting graph $G$ is an $f$-invariant skeleton of $J_f$ rel $P$. This completes the proof of Proposition \ref{pro:pre}.
\end{proof}

\appendix
\section{}
\subsection{Orbifold metric and homotopic length}\label{app:1}

Let $f$ be a   PCF  rational map. Denote by $P_f'$ the post-critical points of $f$ in the Fatou set. Then there exists a complete metric $\omega$, called the {\bf orbifold metric}, on $\ov\C\setminus P_f'$; see \cite[Apendix A.10]{Mc3} or \cite[Section 19]{Mi1}, as well as \cite[Apendix A.10]{BM}.

This metric is induced by a conformal metric $\omega(z)|dz|$ with $\omega(z)$ smooth in
the complement of $P_f$, and has a singularity of the type
\[\omega=\frac{A(z_0)|dz|}{|z-z_0|^{1-1/n(z_0)}},\quad n(z_0)>1,\]
near each post-critical point $z_0\in J_f$. Moreover, we have $\|f'(z)\|_\omega>1$ when $z,f(z)\in \ov\C\setminus P_f'$; see \cite[Theorem 19.6]{Mi1} for  details.

Fix a compact set $\mathcal{O}\supset J_f$ such that $f^{-1}(\mathcal{O})\subset \mathcal{O}$ and $\ov\C\setminus \mathcal{O}$ is a small neighborhood of $P_f'$. Let $\sigma(z)|dz|$ be the standard spherical metric. There exist constants $C>0$ and $\rho>1$ such that
\begin{equation}\label{eq:555}
	\|f'(z)\|_\omega\geq \rho \quad\text{for } z\in f^{-1}(\mathcal{O}),
\end{equation}
and
\begin{equation}\label{eq:666}
	\sigma(z)\leq C\cdot \omega(z) \quad\text{for } z\in \ov\C\setminus P_f.
\end{equation}

Let $P\subset \cbar$ be a finite set in $\cbar$. Two curves $\g_0,\g_1:[0,1]\to\cbar$ are called {\bf homotopic rel $P$ with endpoints fixed} if there exists a continuous map $H:[0,1]\times [0,1]\to \cbar$ such that
\begin{itemize}
\item $H(\cdot,0)=\g_0$ and $H(\cdot,1)=\g_1$;

\item each curve $\g_s:=H(\cdot,s),s\in[0,1]$ has the same endpoints as  $\g_0$ and $\g_s(0, 1)\subset \cbar\setminus P$.
\end{itemize}

Let $\g:[0, 1]\to\cbar$ be a curve with $\g(0, 1)\cap P_f=\emptyset$. The {\bf homotopic length} of $\g$, denoted by $L_\omega[\g]$, is defined as the infimum of the lengths of curves under the orbifold metric, among all smooth curves that are homotopic to $\g$ rel $P_f$ with endpoints fixed.

By \eqref{eq:666}, we have
\begin{equation}\label{eq:777}
	{\rm dist}(\g(0),\g(1)):={\rm dist}_\sigma(\g(0),\g(1))\leq C\cdot L_\omega[\g].
\end{equation}

For a path-connected set $E\subset\cbar$, its {\bf homotopic diameter}  $\text{H-diam}_{\omega}(E)$ is defined as the supremum of homotopic lengths of all curves in $E$. It follows from \eqref{eq:777}  that
\begin{equation}\label{eq:888}
	{\rm diam}(E):={\rm diam}_\sigma(E)\leq C\cdot \text{H-diam}_\omega(E).
\end{equation}

\begin{lemma}\label{lem:expanding}
	Let $\g_n,\g\subset \mathcal{O}$ be curves such that $\gamma(0, 1)\cap P_f=\emptyset$ and $f^n:\g_n\to\g$ is a homeomorphism. Then $L_\omega[\g_n]\leq L_\omega[\g]/\rho^n.$ Moreover,
	suppose that $E$ and $E_n$ are two path-connected sets in $\mathcal{O}$ such that $f^n:E_n\to E$ is a homeomorphism and $\text{{\rm H}-{\rm diam}}_\omega(E)<\infty$. Then
	$${\rm diam}(E_n)\leq C\cdot \text{{\rm H}-{\rm diam}}_\omega(E_n)\leq C\cdot \text{{\rm H}-{\rm diam}}_\omega(E)/\rho^n.$$
\end{lemma}
\begin{proof}
	The first conclusion follows from inequality \eqref{eq:555}.
	Choose any curve $\alpha_n\subset E_n$. Then $f^n:\alpha_n\to \alpha:=f^n(\alpha_n)\,(\subset E)$ is a homeomorphism. Thus,  $L_\omega[\alpha_n]\leq L_\omega[\alpha]/\rho^n\leq \text{H-diam}_\omega(E)/\rho^n$.  Since $\alpha_n$ is arbitrarily chosen, it holds that $\text{H-diam}_\omega(E_n)\leq \text{H-diam}_\omega(E)/\rho^n$.
\end{proof}

\subsection{Lifts of isotopies}\label{app:2}

Applying the usual homotopy lifting theorem for covering maps (see \cite[Proposition 1.30]{Ha}), it is not difficult to prove the following result about lifts of isotopies by rational maps. The details of the proof can be found in \cite[Proposition 11.3]{BM}.

\begin{lemma}\label{lem:lift}
	Suppose that $f,g:\ov\C\to\ov\C$ are   PCF  rational maps, and $h_0,\wt{h}_0:\ov\C\to\ov\C$ are homeomorphisms such that $h_0=\wt{h}_0$ on $P_f$ and $h_0\circ f=g\circ \wt{h}_0$ on $\ov\C$. Let $H:\ov\C\times [0,1]\to\ov\C $ be an isotopy rel $P_f$ with $H_0=h_0$. Then $H$ can be uniquely lifted to an isotopy $\wt{H}:\ov\C\times [0,1]\to \ov\C$ rel $f^{-1}(P_f)$ such that $\wt{H}_0=\wt{h}_0$ and $H_t\circ f=g\circ \wt{H}_t$ on $\ov\C$ for all $t\in[0,1]$.
\end{lemma}

Let $(f,P)$ be a marked rational map, and let $\mathcal{O}$ be the compact set given in Appendix \ref{app:1}. Then $\DDD:=\ov\C\setminus \mathcal{O}$ is a small neighborhood of $P_f'$.

Let $\theta_0:\ov\C\to\ov\C$ be a homeomorphism  isotopic to $id$ rel $P\cup \DDD$. By Lemma \ref{lem:lift}, there exists a homeomorphism $\theta_1:\ov\C\to\ov\C$  isotopic to $id$ rel $P$ such that $\theta_0\circ f=f\circ\theta_{1}$. Inductively, we have a sequence of homeomorphisms $\{\theta_{n},n\geq1\}$ of $\cbar$ isotopic to $id$ rel $P$ such that $\theta_n\circ f=f\circ\theta_{n+1}$. Denote $\phi_n=\theta_{n-1}\circ\cdots\circ\theta_0$.

 A continuous onto map $\pi:\cbar\to\cbar$ is a {\bf quotient map} if $\pi^{-1}(z)$ is either a singleton or a full continuum for any point $z\in\cbar$.
\begin{lemma}\label{thm:isotopy}
	The sequence $\{\phi_n\}$ uniformly converges to a quotient map of $\cbar$ as $n\to\infty$.
\end{lemma}

\begin{proof}
	Let $\Theta^0:\cbar\times [0,1]\to\cbar$ rel $P$ be an isotopy such that $\Theta^0_0=id$, $\Theta^0_1=\theta_0$, and $\Theta^0_t(z)=z$ for all $z\in P\cup\DDD$ and $t\in [0,1]$.
	By inductively applying Lemma \ref{lem:lift}, for each $n\geq1$, we obtain an isotopy $\Theta^n:\cbar\times [0,1]\to\cbar$ such that
	\begin{itemize}
	\item $\Theta^n_0=id$ and $\Theta^n_1=\theta_n$;
	
	\item $\Theta^n_t(z)=z$ for all $z\in f^{-n}(P\cup \DDD)$ and $t\in [0,1]$; and  
	
	\item $\Theta^{n}_t\circ f =f\circ \Theta^{n+1}_t$  for all $z\in\ov\C$ and $t\in [0,1]$.
	\end{itemize}
	
	For each point $z\in\cbar$, define a curve $\g_z:[0,1]\to\ov\C$ by $\g_z(t):=\Theta^0_t(z)$. From the compactness, there exists a constant $L_0$ such that $L_\omega[\g_z]\le L_0$ for all $z\in\cbar\setminus \DDD$. To prove the lemma, it suffices to show that there exist constants $M>0$ and $\rho>1$ such that for all $z\in\cbar$ and $n\geq 1$,
	$${\rm dist}(\phi_n(z),\phi_{n+1}(z))\leq M\rho^{-n}.$$
	
	Fix any $z\in\ov\C$ and $n\geq1$. Set $w=f^n(\phi_n(z))$. Let $\beta$ be the lift of $\gamma_w$ based at $\phi_n(z)$. The other endpoint of $\beta$ is $\phi_{n+1}(z)$. If $w\in P\cup\DDD$, then $\g_w$ is a singleton, and hence $\phi_n(z)=\phi_{n+1}(z)$. Otherwise, it follows from Lemma \ref{lem:expanding} and equality \eqref{eq:777} that $${\rm dist}(\phi_n(z),\phi_{n+1}(z))\leq CL_\omega[\beta]\leq CL_0\rho^{-n}.$$
	Thus, $\{\phi_n\}$ uniformly converges to a continuous map $\phi_\infty$ of $\cbar$ as $n\to\infty$. Since $\phi_\infty$ is a uniform limit of homeomorphisms, it is a quotient map; see e.g. \cite[Lemma 3.1]{CPT}.
\end{proof}

\subsection{Local connectivity}\label{app:3}
It is known that a continuum $E\subset\cbar$ is locally connected if and only if the boundary of each component of $\cbar\sm E$ is locally connected and the spherical diameters of components of $\cbar\sm E$ converge to zero; see e.g. \cite[Lemma 19.5]{Mi1}. We will show that

\begin{lemma}\label{lem:orbifold}
	Let $f$ be a   PCF  rational map, and let $E$ be a continuum with $\partial E\subset J_f$. Then $E$ is locally connected if and only if the boundary of each component of $\cbar\sm E$ is locally connected and the homotopic diameters of components of $\ov\C\setminus E$ disjoint from $P_f$ converge to zero.
\end{lemma}

\begin{proof}
	First, suppose that $E$ is locally connected. Since the homotopic lengths of curves in $\ov\C\setminus P_f$ vary continuously, each component of $\cbar\sm E$ disjoint from $P_f$ has a finite homotopic diameter. To the contrary, assume that $\{D_n\}$ is a sequence of components of $\cbar\sm E$ disjoint from $P_f$, such that $\text{H-diam}_\omega(D_n)\geq \epsilon_0>0$.  Since ${\rm diam}(D_n)\to 0$ as $n\to\infty$, by taking a subsequence, we may assume that $\{\ov{D_n}\}$ converges to a point $a\in E$.
	
	For any $\epsilon>0$, let $\Delta(\epsilon)$ be the round disk with center $a$ and orbifold radius $\epsilon$. Then $\Delta(\epsilon)$ contains at most one point of $P_f$ when $\epsilon$ is sufficiently small. On the other hand, for sufficiently large $n$, $D_n\subset\Delta(\epsilon_0/3)$. This implies that $\text{H-diam}_\omega(D_n)\le 2\epsilon_0/3$, a contradiction.
	
	The converse part of the lemma follows directly from \eqref{eq:888}.
\end{proof}

The following result is well known; see e.g. \cite[Lemmas 17.17 and 17.18]{Mi1}.

\begin{lemma}\label{lem:milnor}
	Let $X$ be a connected and compact metric space. If $X$ is locally connected, then it is arcwise connected and locally arcwise connected.
\end{lemma}

\begin{lemma}\label{lem:equicontinuous}
	Let $E\subset \C$ be a locally connected continuum. Then there exists a family of curves in $E$ that are equicontinuous such that any two  points of $E$ are joined by a curve in this family.
\end{lemma}

\begin{proof}
	For any component $U$ of $\ov{\C}\setminus E$, we fix a Riemann mapping $\phi_{U}:U\to \D$. Since $\partial U$ is locally connected, $\phi_{U}^{-1}$ has a continuous extension from $\ov{\D}$ to $\ov{U}$. For any {\it crosscut} $\alpha$ of $U$, let $D(\alpha)$ denote the component of $U\setminus\alpha$ with a smaller diameter. Here, a crosscut of $U$ means an arc with its interior in $U$ and its endpoints on $\partial U$. By the local connectivity of $E$, for any $\epsilon>0$, there exists $\rho_{\epsilon}>0$ such that for each component $U$ of $\cbar\sm E$,
	\begin{enumerate}
	\item if the distance between $a,b\in\partial\D$ is bounded above by $\rho_{\epsilon}$, then $|\phi_{U}^{-1}(a)-\phi_{U}^{-1}(b)|<\epsilon$;
	
	\item if the diameter of a crosscut $\alpha$ of $U$ is bounded above by $\rho_{\epsilon}$, then $\textup{diam}(D(\alpha))<\epsilon$.
	\end{enumerate}
	
	Let $\G$ be the collection of all line segments with endpoints in $E$. We will revise each $\g\in\G$ to an arc $\tilde{\g}\subset E$ such that $\{\tilde{\g}:\g\in \G\}$ is equicontinuous.
	
	Fix  $\g\in\G$.  Denote $X_\g:=\{t\in[0,1]:\g(t)\in E\}$. Then for any component $I$ of $[0,1]\setminus X_\g$, the open segment $\alpha=\g(I)$ is a crosscut for some component $U$ of $\ov{\C}\setminus E$. Let $\tilde{\alpha}=\partial\phi_{U}(D(\alpha))\cap\partial \D$. Then there exists a linear map $h_{I}:\ov{\alpha}\to \tilde{\alpha}$. 
	
	Now, define a map $\tilde{\g} : [0, 1] \to E$ as
\[
\tilde{\g}(t) := \left\{
\begin{array}{ll}
	\g(t) & \hbox{if $t \in X_\g$,} \\[5pt]
	\phi_{U}^{-1} \circ h_{I} \circ \gamma(t) & \hbox{if $t \in I$ and $\g(I) \subset U$,}
\end{array}
\right.
\]
where $I$ is the component of $[0, 1] \setminus X_\g$ containing $t$.

We claim that $\tilde{\g}$ is a curve. To see this, let $\{I_n\}$ be a sequence of components of $[0, 1] \setminus X_\g$ converging to a point $t_*$. Let $U_n$ be the component of $\ov{\C} \setminus E$ such that $\alpha_n := \g(I_n)$ is a crosscut of $U_n$. Then $\textup{diam}(\alpha_n) \to 0$ as $n \to \infty$ by the continuity of $\g$. 

It follows from point (2) above that ${\rm diam}(D(\alpha_n)) \to 0$ as $n \to \infty$. Since $\tilde{\gamma}(I_n) = \partial D(\alpha_n) \cap \partial U$, it follows that $\tilde{\g}(I_n) \to \tilde{\g}(t_*)$ as $n \to \infty$. Thus, $\tilde{\g}$ is continuous, and the claim is proved.

	We will prove that the family of curves  $\{\tilde{\g},\g\in\G\}$ is equicontinuous. Given any $\epsilon>0$, since the family $\G$ is equicontinuous, there exists a number $\de>0$ such that $|\g(t_1)-\g(t_2)|<\textup{min}\{{\rho^2_\epsilon}/({2\pi}), \epsilon\}$ whenever $|t_1-t_2|<\de$ for every  $\g\in\G$.
	
	Fix any $\g\in\G$. If $t_1,t_2\in X_\g$, then $|\tilde{\g}(t_1)-\tilde{\g}(t_2)|=|\g(t_1)-\g(t_2)|<\epsilon$ whenever $|t_1-t_2|<\de$.
	
	We now assume that $t_1, t_2\in \ov{I}$ for a component $I$ of $[0,1]\setminus X_\g$. Let $\alpha=\gamma(I)$. If $\textup{diam}(\alpha)<\rho_\epsilon$,  point (2) above implies  $|\tilde{\g}(t_1)-\tilde{\g}(t_2)|\leq \textup{diam}(D(\alpha))<\epsilon$. Otherwise, we have $|h_{I}'|<2\pi/\rho_\epsilon$. In this case, if $|t_1-t_2|<\de$, it holds that
	\[|h_I\circ\g(t_1)-h_I\circ \g(t_2)|= |\g(t_1)-\g(t_2)|\cdot|h_I'|< \rho_\epsilon.\]
	It then follows from point (1) above that $|\tilde{\g}(t_1)-\tilde{\g}(t_2)|<\epsilon$.
	
	Finally, assume that $t_1$ and $t_2$ lie in the closures of distinct components $I_1$ and $I_2$ of $[0,1]\setminus X_\g$, respectively. If $|t_1-t_2|<\de$, the two endpoints $t'_1$ and $t_2'$ of $I_1$ and $I_2$ between $t_1$ and $t_2$ satisfy that $|t_1-t'_1|<\de$ and $|t_2-t'_2|<\de$. Then according to the previous two cases,
	\[|\tilde{\g}(t_1)-\tilde{\g}(t_2)|\leq |\tilde{\g}(t_1)-\tilde{\g}(t'_1)|+|\tilde{\g}(t'_1)-\tilde{\g}(t'_2)|+|\tilde{\g}(t'_2)-\tilde{\g}(t_2)|<3\epsilon.\]
	Therefore, the family $\{\tilde{\g},\g\in\G\}$ is equicontinuous.
\end{proof}


\end{document}